\immediate\write18{makeindex \jobname.nlo -s nomencl.ist -o \jobname.nls}
\documentclass[12pt,reqno]{amsart}
\usepackage[lmargin=35mm,rmargin=35mm,tmargin=47mm, bmargin=47mm,lines=36]{geometry}

\makeatletter
\def\chaptermark#1{}

\def\chapter{%
  \if@openright\cleardoublepage\else\clearpage\fi
  \thispagestyle{plain}\global\@topnum\z@
  \@afterindenttrue \secdef\@chapter\@schapter}

\def\@chapter[#1]#2{\refstepcounter{chapter}%
  \ifnum\c@secnumdepth<\z@ \let\@secnumber\@empty
  \else \let\@secnumber\thechapter \fi
  \typeout{\chaptername\space\@secnumber}%
  \def\@toclevel{0}%
  \ifx\chaptername\appendixname \@tocwriteb\tocappendix{chapter}{#2}%
  \else \@tocwriteb\tocchapter{chapter}{#2}\fi
  \chaptermark{#1}%
  \addtocontents{lof}{\protect\addvspace{10\p@}}%
  \addtocontents{lot}{\protect\addvspace{10\p@}}%
  \@makechapterhead{#2}\@afterheading}
\def\@schapter#1{\typeout{#1}%
  \let\@secnumber\@empty
  \def\@toclevel{0}%
  \ifx\chaptername\appendixname \@tocwriteb\tocappendix{chapter}{#1}%
  \else \@tocwriteb\tocchapter{chapter}{#1}\fi
  \chaptermark{#1}%
  \addtocontents{lof}{\protect\addvspace{10\p@}}%
  \addtocontents{lot}{\protect\addvspace{10\p@}}%
  \@makeschapterhead{#1}\@afterheading}
\newcommand\chaptername{Chapter}

\def\@makechapterhead#1{\global\topskip 7.5pc\relax
  \begingroup
  \fontsize{\@xivpt}{18}\bfseries\centering
    \ifnum\c@secnumdepth>\m@ne
      \leavevmode \hskip-\leftskip
      \rlap{\vbox to\z@{\vss
          \centerline{\normalsize\mdseries
              \uppercase\@xp{\chaptername}\enspace\thechapter}
          \vskip 3pc}}\hskip\leftskip\fi
     #1\par \endgroup
  \skip@34\p@ \advance\skip@-\normalbaselineskip
  \vskip\skip@ }
\def\@makeschapterhead#1{\global\topskip 7.5pc\relax
  \begingroup
  \fontsize{\@xivpt}{18}\bfseries\centering
  #1\par \endgroup
  \skip@34\p@ \advance\skip@-\normalbaselineskip
  \vskip\skip@ }
\def\appendix{\par
  \c@chapter\z@ \c@section\z@
  \let\chaptername\appendixname
  \def\thechapter{\@Alph\c@chapter}}

\newcounter{chapter}
\newif\if@openright

\makeatother

\makeatletter
\@addtoreset{section}{chapter}
\makeatother

\usepackage[rgb]{xcolor}
\usepackage{tikz}
\usepackage{verbatim}
\usepackage{amsmath}
\usepackage{amsxtra}
\usepackage{amscd}
\usepackage{amsthm}
\usepackage{amsfonts}
\usepackage{amssymb}
\usepackage{eucal}

\usepackage[linktocpage=true]{hyperref}
\usepackage[all, cmtip]{xy}
\usepackage{stmaryrd}
\usepackage{color}

\hypersetup{colorlinks,linkcolor=blue, urlcolor=,
citecolor=blue}

\newtheorem{thm}{Theorem} [chapter]
\newtheorem{cor}[thm]{Corollary}
\newtheorem{lem}[thm]{Lemma}
\newtheorem{prop}[thm]{Proposition}
\newtheorem{conjecture}[thm]{Conjecture}

\theoremstyle{definition}
\newtheorem{definition}[thm]{Definition}

\theoremstyle{remark}
\newtheorem{rem}[thm]{Remark}

\numberwithin{equation}{chapter}

\newcommand{\nc}{\newcommand}
 \nc{\Z}{{\mathbb Z}}
 \nc{\C}{{\mathbb C}}
 \nc{\N}{{\mathbb N}}
 \nc{\F}{{\mf F}}
 \nc{\Q}{\mathbb{Q}}
 \nc{\la}{\lambda}
 \nc{\ep}{\epsilon}
 \nc{\h}{\mathfrak h}
 \nc{\n}{\mf n}
 \nc{\G}{{\mathfrak g}}
 \nc{\DG}{\widetilde{\mathfrak g}}
 \nc{\SG}{\breve{\mathfrak g}}
 \nc{\La}{\Lambda}
 \nc{\is}{{\mathbf i}}
 \nc{\V}{\mf V}
 \nc{\bi}{\bibitem}
 \nc{\E}{\mc E}
 \nc{\ba}{\tilde{\pa}}
 \nc{\half}{\frac{1}{2}}
 \nc{\hgt}{\text{ht}}
 \nc{\mc}{\mathcal}
 \nc{\mf}{\mathfrak} 
 \nc{\hf}{\frac{1}{2}}
 \nc{\hgl}{\widehat{\mathfrak{gl}}}
 \nc{\gl}{{\mathfrak{gl}}}
 \nc{\so}{{\mathfrak{so}}}
 \nc{\hz}{\hf+\Z}
 \nc{\ospO}{\mathcal{O}^{2m+1|2n}}
 \nc{\CatO}{\mathcal{O}}
\nc{\Omn}{\mc{O}^{m|n}}
\nc{\OO}{\mc{O}}

\nc{\U}{\bold{U}}
\nc{\SU}{\overline{\mathfrak u}}
\nc{\ov}{\overline}
\nc{\ul}{\underline}
\nc{\wt}{\widetilde}
\nc{\I}{\mathbb{I}}
\nc{\X}{\mathbb{X}}
\nc{\Y}{\mathbb{Y}}
\nc{\hh}{\widehat{\mf{h}}}
\nc{\aaa}{{\mf A}}
\nc{\xx}{{\mf x}}
\nc{\wty}{\widetilde{\mathbb Y}}
\nc{\ovy}{\overline{\mathbb Y}}
\nc{\vep}{\bar{\epsilon}}
\nc{\wotimes}{\widehat{\otimes}}

\nc{\ch}{\text{ch}}
\nc{\glmn}{\mf{gl}(m|n)}
\nc{\Uq}{{\mathcal U}_q}
\nc{\Uqgl}{\bold{U}_q(\gl_{2n+2})}
\nc{\VV}{\mathbb V}
\nc{\WW}{\mathbb W}
 \nc{\TL}{{\mathbb T}_{\mathbb L}}
 \nc{\TU}{{\mathbb T}_{\mathbb U}}
 \nc{\TLhat}{\widehat{\mathbb T}_{\mathbb L}}
 \nc{\TUhat}{\widehat{\mathbb T}_{\mathbb U}}
 \nc{\TLwt}{\widetilde{\mathbb T}_{\mathbb L}}
 \nc{\TUwt}{\widetilde{\mathbb T}_{\mathbb U}}
 \nc{\TLC}{\ddot{\mathbb T}_{\mathbb L}}
 \nc{\TUC}{\ddot{\mathbb T}_{\mathbb U}}

\nc{\osp}{\mathfrak{osp}}
 \nc{\be}{e}
 \nc{\bff}{f}
 \nc{\bk}{k}
 \nc{\bt}{t}
\nc{\Ui}{{\bold{U}^{\imath}}}
\nc{\Uidot}{{\dot{\bold{U}}^\imath}}
\nc{\id}{\text{id}}
\nc{\Ihf}{\I^\imath}
\nc{\Udot}{\dot{\U}}
\nc{\one}{\bold{1}}
\nc{\Uwedge}{\dot{\U}^{\wedge}}
\nc{\Uj}{\bold{U}^{\jmath}}
\nc{\ospV}{\osp(2m+1|2n|\infty)}
\nc{\ospW}{\mathfrak{osp}(2m+1|2n+\infty)}
\nc{\epinftyV}{\epsilon^0_{\infty}}
\nc{\epinftyW}{\epsilon^1_{\infty}}
\nc{\wtlV}{X^{\ul \infty, +}_{{\bf b},0}}
\nc{\wtlW}{X^{\ul \infty, +}_{{\bf b},1}}
\nc{\wtlVk}{X^{\ul k, +}_{{\bf b},0}}
\nc{\wtlWk}{X^{\ul k, +}_{{\bf b},1}}
\nc{\f}{\bold{f}}
\nc{\fprime}{\bold{'f}}
\nc{\Qq}{\Q(q)}
\nc{\qq}{(q^{-1}-q)}
\nc{\uqsl}{\bold{U}_q({\mf{sl}_{2n+2}})}
\nc{\BLambda}{{\Lambda_{\inv}}}
\nc{\ThetaA}{\Theta}
\nc{\ThetaB}{\Theta^{\imath}}
\nc{\ThetaC}{\Theta^{\jmath}}
\nc{\B}{\bold{B}}
\nc{\Abar}{\psi}
\nc{\Bbar}{\psi_\imath}
\nc{\HBm}{\mc{H}_{B_m}}
\nc{\ua}{\mf{u}}
\nc{\nb}{u}
\nc{\wtA}{\texttt{wt}}
\nc{\Uilambda}{\Lambda^\imath}
\nc{\Dupsilon}{\Upsilon^{\vartriangle}}
\nc{\inv}{\theta}
\nc{\VVm}{\VV^{\otimes m}}
\nc{\mA}{\mathcal{A}}
\nc{\Tb}{{\mathbb{T}}^{\bf b}}
\nc{\wTb}{\widehat{\mathbb{T}}^{\bf b}}

\nc{\qqq}{(1-q^{-2})^{-1}}

\nc{\Ujdot}{\dot{\U}^{\jmath} }
\nc{\Iint}{\mathbb{I}^{\jmath} }
\nc{\ibe}[1]{e_{\alpha_{#1}}}
\nc{\ibff}[1]{f_{\alpha_{#1}}}
\nc{\ibk}[1]{k_{\alpha_{#1}}}
\nc{\wtC}{\texttt{wt}_{\jmath}}
\nc{\Cbar}{\psi_{\jmath}}

\title[Quantum symmetric pairs and Kazhdan-Lusztig theory]
{A new approach to \\Kazhdan-Lusztig theory of type $B$ \\via quantum symmetric pairs
}
 
 \author[Huanchen Bao]{Huanchen Bao}
\address{Department of Mathematics, University of Maryland, College Park, MD 20742}
\email{huanchen@math.umd.edu}

\author[Weiqiang Wang]{Weiqiang Wang}
\address{Department of Mathematics, University of Virginia, Charlottesville, VA 22904}
\email{ww9c@virginia.edu}

\keywords{Kazhdan-Lusztig theory, Lie superalgebras, canonical bases, quantum symmetric pairs} 

\subjclass[2010]{Primary 17B10, 17B20, 17B37}
\begin{document}

\begin{abstract}
We show that Hecke algebra of type $B$ and a coideal subalgebra of the type $A$ quantum group  
satisfy a double centralizer property,  generalizing the Schur-Jimbo duality in type $A$.  
The quantum group of type $A$ and its coideal subalgebra form a quantum symmetric pair.
A new theory of canonical bases arising from  quantum symmetric pairs is initiated. 
It is then applied to formulate and establish for the first time a Kazhdan-Lusztig theory for the BGG category $\mathcal O$ of the ortho-symplectic Lie superalgebras $\mathfrak{osp}(2m+1|2n)$. In particular, 
our approach provides a new formulation of the Kazhdan-Lusztig theory for Lie algebras of type $B/C$.  
\\
\\
R\'esum\'e: On d\'emontre que les alg\`ebres de Hecke de type $B$ et des coideaux du groupes quantiques de type A satisfont une propri\'et\'e de double centralisateur qui g\'en\'eralise la dualit\'e de Schur-Jimbo en type $A$. Le groupe quantique de type $A$ et son coid\'eal forment une paire sym\'etrique quantique. Une nouvelle th\'eorie de bases canonique associ\'ee aux paires sym\'etriques quantiques est d\'evelopp\'ee. Elle est appliqu\'ee pour formuler et \'etablir une th\'eorie \`a la Kazhdan-Lusztig pour la cat\'egorie O de BGG de la super-alg\`ebre de Lie ortho-symplectique $\mathfrak(2m+1|2n)$. Notre approche donne en particulier une nouvelle formulation de la th\'eorie de Kazhdan-Lusztig pour les alg\`ebres de Lie de type $B/C$.
\end{abstract}

\maketitle


\tableofcontents

\chapter*{Introduction}

\section{Background}

A milestone in representation theory was the Kazhdan-Lusztig (KL) theory 
initiated in \cite{KL} (and completed in \cite{BB, BK}), 
which offered a powerful solution
to the difficult problem of determining the irreducible characters 
in the BGG category $\mc O$ of a semisimple Lie algebra $\G$. 
The Hecke algebra $\mc H_W$ associated to the Weyl group 
$W$ of $\G$ plays a central role in the KL formulation,  which can be
paraphrased as follows: the simple modules of the principal block in $\mc O$ 
correspond to the Kazhdan-Lusztig basis  
of  $\mc H_W$
while the Verma modules correspond to the standard basis of $\mc H_W$. 
The characters of the simple modules in singular blocks are then determined 
from those in the principal block via translation functors \cite{So90},
and the characters of tilting modules were subsequently determined in \cite{So98}. 

Though the classification of finite-dimensional simple Lie superalgebras 
over $\C$ was achieved in 1970's by \cite{Kac77}, the study of representation theory 
such as the BGG category $\mc O$ for a Lie superalgebra turns out to be very challenging and 
the progress has been made only in recent years. 
One fundamental reason is that the Weyl group (of the even part) of a Lie superalgebra alone 
is not sufficient to control the linkage principle in $\mc O$,
and hence the corresponding Hecke algebra cannot  
play the same crucial role as in the classical Kazhdan-Lusztig theory.
Among all basic Lie superalgebras, the infinite series $\glmn$ and $\osp(m|2n)$ 
are arguably the most fundamental ones. 
Since these Lie superalgebras specialize to Lie algebras when one of the parameters $m$ or $n$ is zero,
any possible (conjectural) approach on the irreducible character problem 
in the BGG category of such a Lie superalgebra has to first
provide a new formulation for a classical Lie algebra in which the Hecke algebra does not feature directly.

Brundan \cite{Br03} in 2003 formulated a conjecture on the irreducible and tilting characters
for the BGG category $\mc O$ for the general linear Lie superalgebra $\glmn$,
using Lusztig's canonical basis.
In this case, fortunately Schur-Jimbo duality \cite{Jim} between a 
Drinfeld-Jimbo quantum group $\U$ and a Hecke algebra of type $A$
enables one to reformulate the KL theory of type $A$ in terms of 
Lusztig's canonical basis on some Fock space $\VV^{\otimes m}$, where
$\VV$ is the natural representation of $\U$.
Brundan's formulation for $\glmn$ makes a crucial use
of the Fock space $ \VV^{\otimes m} \otimes \WW^{\otimes n}$, where
$\WW$ denotes the restricted dual to $\VV$. 
The longstanding conjecture of Brundan was settled in \cite{CLW12},
where a super duality approach developed earlier \cite{CW08, CL} (cf. \cite[Chapter 6]{CW12}) plays a key role. 
(For a more recent and different proof of Brundan's conjecture see Brundan, Losev, and Webster \cite{BLW}.)

Finding a general formulation for a Kazhdan-Lusztig theory for the BGG category $\mc O$
of the ortho-symplectic Lie superalgebras is one of the most intriguing open problems
in super representation theory. 
There was no conjecture available in the literature, and the reason should have become clear as we explain above:
no alternative approach to KL theory of type $BCD$ had been discovered  
without using Hecke algebras directly.

A super duality approach was developed in \cite{CLW11} which solves the irreducible 
character problem for some {\em distinguished} parabolic BGG
categories of the $\osp$ Lie superalgebras. This approach was not sufficient to attack the problem in
the full BGG category for $\osp$, and a Brundan-type formulation was not available. 
There has been a completely different approach developed by Gruson and Seganova \cite{GS} toward the finite-dimensional 
irreducible characters for the $\osp$ Lie superalgebras. 
One of the implications of the super duality which is important to us though is that the 
linkage for the distinguished parabolic categories of $\osp(2m+1|2n)$-modules 
is controlled by Hecke algebra of type $B_\infty$, and so one hopes that
it remains to be so for the full BGG category of $\osp(2m+1|2n)$-modules.

\section{The goal}

The goal of this paper is to give a complete and conceptual solution to problem on irreducible characters
in the BGG category $\mc O$ of modules 
of integer and half-integer weights
for the ortho-symplectic Lie superalgebras $\osp(2m+1|2n)$ of type $B(m,n)$. 
The case of Lie superalgebra $\osp(2m|2n)$ is treated in \cite{Bao}.
In particular, the non-super specialization of our work amounts 
to a new formulation to Kazhdan-Lusztig theory
of Lie algebras of classical type in which Hecke algebras are not used directly.

To achieve the goal, we are led to initiate in Part~1
a new theory of quasi-$\mc{R}$-matrix and new canonical basis (called {\em $\imath$-canonical basis})
arising from {\em quantum symmetric pairs} $(\U, \Ui)$.
We show that the coideal subalgebra $\Ui$ of $\U$ and the Hecke algebra of type $B_m$ form
double centralizers on $\VV^{\otimes m}$, generalizing the Schur-Jimbo duality. 
A new formulation of the KL theory for Lie algebras of type $B$ is then made possible
by such a new duality.
 Part~1 (which consists of Chapters~\ref{sec:prelim}-\ref{sec:QSPc})
has nothing to do with Lie superalgebras and should be of independent interest.

We develop in Part~2 an infinite-rank version of the constructions in Part~1,
and then relate the $\imath$-canonical basis to the BGG category $\mc O_{\bf b}$ of $\osp(2m+1|2n)$-modules
of (half-)integer weights relative to a Borel subalgebra whose type is
specified by a $0^m1^n$-sequence $\bf b$.
In this approach, the role of Kazhdan-Lusztig basis is played by the (dual)
$\imath$-canonical basis for a suitable completion of the $\Ui$-module 
$\mathbb{T}^{\bf b}$ associated to $\bf b$;
Here $\mathbb{T}^{{\bf b}}$ is a tensor space which is a variant of
$\VV^{\otimes m} \otimes \WW^{\otimes n}$. 
%
%
%
%
%

\section{An overview of Part 1}

Our starting point is actually natural and simple. 
The generalization of Schur duality beyond type $A$ in the literature is not suitable to our goal,
since it replaces the Lie algebra/group of type $A$
by its classical counterpart and modifies the symmetric group to become a Brauer algebra 
(or a quantum version of such). For our purpose, 
as we look for a substitute for KL theory where
the Hecke algebras have played a key role, we ask for some quantum group like object 
with a coproduct (not Schur type algebra) 
which centralizes the Hecke algebra of type $B_n$
when acting on $\VV^{\otimes n}$.  We recognized
such a quantum group like object as a coideal subalgebra of the quantum group $\U$,  a quantum version of the enveloping algebra of 
the subalgebra of $\mf{sl}(\VV)$ fixed by some involution,
which forms a quantum symmetric pair with $\U$.

Note that the formulation of Part~1 is in the setting that $\VV$ is finite-dimensional, while it is most natural to set $\VV$ to be
infinite-dimensional when making connection with category $\mc O$ in Part~2.

The structure theory of quantum symmetric pairs was systematically developed by Letzter and then Kolb
(see  \cite{Le03}, \cite{Ko12} and the references therein).
Though our coideal subalgebra can be identified with some particular examples in literature by an explicit (anti-)isomorphism, 
the particular form of our presentation and its embedding into $\U$ are different and new.
The coideal subalgebra in our presentation manifestly admits a bar involution, and
the specialization at $q=1$ of our presentation 
has a natural interpretation in terms of translation functors in category $\mc O$.
Depending on whether the dimension of $\VV$ is even or odd, we denote the (right) coideal subalgebra by $\Ui$ or $\Uj$, respectively.
The two cases are similar but also have quite some differences, and the case of $\Ui$ is more challenging as
it contains an unconventional generator which we denote by $t$ (besides the Chevalley-like generators $e_{\alpha_i}$ and $f_{\alpha_i}$). 
We mainly restrict our discussion to $\Ui$ (and so $\dim \VV$ is even) below. 

Recall that the coproduct $\Delta: \U \rightarrow \U\otimes \U$ is not compatible
with the bar involution $\Abar$ on $\U$ and $\Abar \otimes \Abar$ on $\U \otimes \U$,  and Lusztig's quasi-$\mc{R}$-matrix $\Theta$ is designed 
to intertwine $\Delta$ and $\ov\Delta$, where $\ov\Delta(u) := (\Abar \otimes \Abar) {\Delta(\Abar(u))}$, for $u \in \U$. 
Lusztig's construction of $\Theta$ is a variant of Drinfeld's construction of universal $\mc R$-matrix \cite{Dr86}, and it
takes great advantage of  the triangular decomposition and a 
natural bilinear form of $\U$. The bar involution
on $\VVm$ was then constructed by means of the quasi-$\mc{R}$-matrix $\Theta$. 
Inspired by the type $A$ reformulation of KL theory (cf., e.g., \cite{VV99, Br03, CLW12}), 
as an alternative of the Kazhdan-Lusztig theory without using Hecke algebras
we ask for a canonical basis theory arising from quantum symmetric pairs.  

The embedding $\imath: \Ui \rightarrow \U$ which makes $\Ui$ a coideal subalgebra of $\U$
does not commute with the bar involution $\Bbar$ on $\Ui$ and $\Abar$ on $\U$. 
We have a coproduct of the coideal form $\Delta: \Ui \rightarrow \Ui \otimes \U$.
Define $\ov{\Delta}: \Ui \rightarrow \Ui \otimes \U$ by 
$\ov{\Delta}(u) = (\Bbar \otimes \Abar) {\Delta(\Bbar(u))}$, for all $u \in \Ui$. 
Note that the $\ov\Delta$ here is not a restriction of Lusztig's $\ov\Delta$.
Toward our goal, in place of Lusztig's quasi-$\mc{R}$-matrix  for $\U$
one would need a quasi-$\mc{R}$-matrix $\ThetaB$ 
which intertwines $\Delta$ and $\ov\Delta$ for $\Ui$.
The problem here is that $\Ui$ does not have any obvious triangular decomposition or bilinear form as for $\U$. 

Our key strategy is to ask first for some suitable intertwiner $\Upsilon$ which intertwines
$\imath$ and $\ov\imath: \Ui \rightarrow \U$, where 
\[
\ov\imath (u):=\Abar \big( \imath (\Bbar (u))\big), \quad \text{ for } u\in \Ui.
\] 
Note the remarkable analogy with a key property of Lusztig's $\ThetaA$. 
We succeed in constructing such an intertwiner $\Upsilon$ in some completion of 
the negative half $\U^-$ of $\U$
and show that it is unique up to a scalar multiple  (see Theorem~\ref{thm:Upsilon}).
Then
by combining $\Upsilon$ with Lusztig's $\ThetaA$ we are able to construct the quasi-$\mc{R}$-matrix $\ThetaB$,
which lies in some completion of $\Ui \otimes \U^-$. The crucial properties 
 \[
 \Upsilon \ov{\Upsilon} =1,
 \qquad \ThetaB \ov{\ThetaB} =1
 \]
 hold.
The intertwiner $\Upsilon$ can also be applied to turn an involutive $\U$-module into an $\imath$-involutive $\Ui$-module
(see Proposition~\ref{prop:compatibleBbar}, Definitions~\ref{def:involutive} and \ref{def:involutive-i}).

It turns out to be a subtle problem to show that $\Upsilon$ lies in (a completion of) the integral $\mA$-form $\U_\mA^-$,
where $\mA =\Z[q,q^{-1}]$.
We are led to study the simple lowest weight $\U$-modules ${}^\omega L(\la)$ for $\la\in \La^+$ regarded
as $\Ui$-modules. By a detailed study on the behavior of the generator $t$ in $\Ui$ in the rank one case,
we show that $\Upsilon$ preserves the $\mA$-form ${}^\omega L_\mA (\la)$ for all $\la \in \La^+$,
and this eventually allows us to establish the integrality of $\Upsilon$ (see Theorem~\ref{thm:UpsiloninZ}).
We then construct the {\em $\imath$-canonical basis}  of ${}^\omega L(\la)$ which
is $\Bbar$-invariant and admits a triangular decomposition with respect to Lusztig's canonical basis
on ${}^\omega L(\la)$ with coefficients in $\Z[q]$ (see Theorem~\ref{thm:BCB}). 
Consequently, we construct in Theorem~\ref{thm:iCBtensor} an $\imath$-canonical basis for any tensor product of
several finite-dimensional simple $\U$-modules, which differs from and is related to Lusztig's canonical basis on the same tensor product.

Generalizing the Schur-Jimbo duality in type $A$, we show that the action of
the coideal algebra $\Ui$ and Hecke algebra $\HBm$
on $\VVm$ form double centralizers, where $\VV$ is the natural representation of $\U$ 
 (see Theorem~\ref{thm:SchurB}).
With $\Upsilon$ and $\ThetaB$ at hand, we are able to construct a bar involution $\Bbar$
on the $(\Ui, \HBm)$-bimodule $\VVm$ which is compatible with the bar involutions on $\Ui$ and $\HBm$ 
(see Theorem~\ref{thm:samebar}).
In particular, the $\imath$-canonical basis on the involutive $\Ui$-module  $(\VVm, \Bbar)$ alone is sufficient
to reformulate the KL theory of type $B$.

\section{An overview of Part 2}
 
Part 2 is very close to \cite{CLW12} in spirit, where the category $\mc O$ of $\glmn$-modules was addressed.
In this part, we take the $\Qq$-space
$\VV$ to be the direct limit as $r \mapsto  \infty$ of the $2r$-dimensional ones considered in Part~1.
Also let $\U$ and $\Ui$ be the corresponding infinite-rank limits of their finite-rank counterparts in Part~1.

For an $0^m1^n$-sequence $\bf b$ (which consists of $m$ zeros and $n$ ones), we define 
a tensor space $\mathbb{T}^{\bf b}$ using $m$ copies of $\VV$ and $n$ copies of $\WW$ with the tensor order 
prescribed by $\bf b$ (with $0$ corresponds to $\VV$); for instance,
associated to ${\bf b}^{\text{st}} =(0, \ldots, 0, 1, \ldots, 1)$, we have
$\mathbb{T}^{{\bf b}^{\text{st}}} =\VV^{\otimes m} \otimes \WW^{\otimes n}$.
Such a tensor space (called Fock space) was an essential ingredient in the formulation of
Kazhdan-Lusztig-type conjecture for $\gl(m|n)$ and its generalizations \cite{Br03, Ku, CLW12}.
In this approach, $\mathbb{T}^{\bf b}$ at $q=1$ (denoted by $\mathbb{T}^{\bf b}_\Z$) is identified with the Grothendieck group of the BGG
category of $\gl(m|n)$-modules (relative to a Borel subalgebra of type $\bf b$), 
and the (dual) canonical bases of the $\U$-module $\mathbb{T}^{\bf b}$
play the role of Kazhdan-Lusztig basis which solves the irreducible and tilting character
problem in the BGG category for $\gl(m|n)$. 

Now with the intertwiner $\Upsilon$ and the quasi-$\mc R$-matrix $\ThetaB$ for the quantum symmetric
pair $(\U, \Ui)$ at disposal,
we are able to construct the $\imath$-canonical and dual $\imath$-canonical bases for $\mathbb{T}^{\bf b}$
(or rather in its suitable completion with respective to a Bruhat ordering); see Theorem~\ref{thm:iCBb}. 
In the finite-rank setting, this was already proved in Part~1. Nevertheless, the infinite-rank setting
requires much care and extra work to deal with suitable completions, similar to \cite{CLW12} (see also \cite{Br03}). 
A simple but crucial fact is that the partial ordering for $\mathbb{T}^{\bf b}$ used in \cite{CLW12} is coarser than
the one used in this paper and this allows various constructions in {\em loc. cit.} to carry over to the current setting. 
We will ignore the completion issue completely in the remainder of the Introduction. 

Our main theorem (Theorem~\ref{thm:iKL}), which will be referred to as ($\bf b$-KL) here, states that there exists an isomorphism
between the Grothendieck group of the BGG category 
$\mc O_{\bf b}$ of $\osp(2m+1|2n)$-modules of integer weights 
(relative to a Borel subalgebra of type $\bf b$) and $\mathbb{T}^{\bf b}_\Z$,
which sends the Verma, simple, and tilting modules to the standard monomial, 
dual $\imath$-canonical, and $\imath$-canonical bases (at $q=1$), respectively.
In other words, the entries of the transition matrix between (dual) $\imath$-canonical basis and monomial basis
play the role of Kazhdan-Lusztig polynomials in our category $\mc O_{\bf b}$. 
 
Granting the existence of the (dual) $\imath$-canonical bases of $\mathbb{T}^{\bf b}$,
the overall strategy of a proof of ($\bf b$-KL) follows the one employed in \cite{CLW12}
in establishing Brundan's Kazhdan-Lusztig-type conjecture, which is done by induction on $n$ with the
base case solved by the classical Kazhdan-Lusztig theory of type $B$ \cite{KL, BB, BK} (as reformulated above in
terms of the $\imath$-involutive $\Ui$-module $\VVm
$). 
There are two main steps in the proof.
First, we need (an easy generalization of) the super duality developed in \cite{CLW11} for $\osp$,
which is an equivalence of parabolic categories of $\osp(2m+1|2n+\infty)$-modules and
$\osp(2m+1|2n|\infty)$-modules. We establish the corresponding combinatorial super duality which states that there
is an $\Ui$-isomorphism between $\mathbb{T}^{\bf b} \otimes \wedge^\infty \VV$ and
$\mathbb{T}^{\bf b} \otimes \wedge^\infty \WW$, which matches the corresponding 
standard monomial, $\imath$-canonical, and dual $\imath$-canonical bases.
The super duality is used to establish the $\bf b$-KL for one distinguished $0^m1^n$-sequence.

The second step is a comparison of ($\bf b$-KL) and ($\bf b'$-KL) for adjacent sequences $\bf b$ and $\bf b'$
(here ``adjacent" means differing exactly by an adjacent pair $01$).
Let us assume for simplicity that the first entries of $\bf b$ and $\bf b'$ are both
$0$ here (see Remarks~\ref{rem:osprank2} and \ref{rem:remove0} for the removal of this assumption),
as this is sufficient in solving the irreducible and tilting character problems for $\osp(2m+1|2n)$-modules. 
Thanks to the coideal property of $\Ui$, the iterated coproduct for $\Ui$ 
has images in $\Ui \otimes \U  \otimes \ldots \otimes \U$. Therefore
the comparison of ($\bf b$-KL) and ($\bf b'$-KL) for adjacent $\bf b$ and $\bf b'$
can be carried out exactly as in the type $A$ setting \cite{CLW12} since the exchange
of the adjacent $0$ and $1$ does not affect the first tensor factor and hence will not use $\Ui$.
The upshot is that the validity of the statement ($\bf b$-KL) for one $0^m1^n$-sequence implies
the validity for an arbitrary $0^m1^n$-sequence.  

The infinite-rank version of the other quantum symmetric pairs $(\U, \Uj)$ 
and its $\jmath$-canonical basis theory is used to solve a variant of the BGG category $\mc O$
of $\osp(2m+1|2n)$-modules, now of {\em half-}integer weights; see Chapter~\ref{sec:BGGj}.

\section{Some future works} 
\label{sec:future}

This work will serve as a new starting point in several (closely related) directions. 

The constructions of this paper is adapted in \cite{Bao} to provide the  irreducible character formula in the BGG category $\mc O$ for  
Lie superalgebras $\osp(2m|2n)$, settling another longstanding open problem in Lie superalgebras since 1970s.

Recall the Schur-Jimbo duality   has a natural geometric realization in terms of 
partial flag varieties of type $A$ due to Grojnowski and Lusztig.  It is natural to ask for a geometric interpretation
of the type $B$ duality as well as $\imath$-canonical basis
 developed algebraically and categorically in this paper.
This turns out to have a classical answer  in \cite{BKLW}, which settles another old open problem of understanding the
quantum algebra arising from partial flag varieties of classical type. 
(This generalizes the classic work of Beilinson, Lusztig, and MacPherson \cite{BLM} for type $A$.)

While we have developed adequately a theory for $\imath$-canonical 
basis for quantum symmetric pairs to solve the irreducible character problem in the category $\mc O_{\bf b}$, 
a full fledged theory of canonical basis for quantum symmetric pairs remains to be developed.
The quantum symmetric pairs $(\U,\Ui)$ and $(\U, \Uj)$ are just two examples of general 
quantum symmetric pairs of finite or more generally Kac-Moody type (see \cite{Le03, Ko12}). The general 
quantum symmetric pairs afford presentations similar to 
the ones given in this paper which admit a natural bar involution. 
A theory of $\imath$-canonical bases for the general quantum symmetric pairs will be pursued in a separate publication \cite{BW16}.
While several key steps developed in this paper will be generalized suitably, further new ideas are also needed.

One influential and persuasive philosophy in the last two decades, supported by the 
quiver variety construction of Nakajima
and reinforced by the categorification program of Chuang, Rouquier, Khovanov and Lauda, 
is that various constructions in general settings are of  ``type $A$" locally. 
A  general philosophical message of this paper and \cite{BKLW}
is that there exists a whole range of new yet classical $\imath$-constructions,
algebraic, geometric and categorical, which are of  ``type $A$ with involution".

The most significant quantum symmetric pairs beyond $\Ui$ and $\Uj$ in our view
would be the ones associated to the quantum group of affine type $A$, whose $\imath$-canonical basis theory
is expected to be closely related to the irreducible character problem in modular
representation theory of algebraic groups or  quantum groups of classical types. 

The geometric aspects of the finite or affine coideal algebras and $\imath$-icanonical bases will be developed by Yiqiang Li and his collaborators.
A KLR type $\imath$-categorification will be addressed elsewhere.

\section{Organization}

The paper is divided into two parts. Part~1, which consists of
Chapters~\ref{sec:prelim}-\ref{sec:QSPc}, provides various 
foundational constructions on quantum symmetric pairs, where $\dim \VV$
is assumed to be finite. 
Part~2, which consists of
Chapters~\ref{sec:BGG}-\ref{sec:BGGj}, extends the $\imath$-canonical basis and dual $\imath$-canonical basis to
the setting where $\VV$ is infinite-dimensional and uses this to solve the irreducible and tilting character problems
of category $\mc O$ for Lie superalgebra $\osp(2m+1|2n)$.

 In the preliminary Chapter~\ref{sec:prelim}, we review various basic constructions for quantum group $\U$.
 We also introduce the involution $\inv$ on the root system and integral weight lattice of $\U$
 and a ``weight lattice" $\La_\inv$ which will be used in quantum symmetric pairs. 
 
In Chapter~\ref{sec:qsp}, 
we introduce the right coideal subalgebra $\Ui$ of $\U$
and an algebra embedding $\imath: \Ui \rightarrow \U$. The algebra $\Ui$ is endowed with
a natural bar involution. 
Then we construct an intertwiner $\Upsilon =\sum_\mu \Upsilon_\mu$, which lies in a completion $\widehat{\U}^-$,
for the two bar involutions on 
$\Ui$ and $\U$ under $\imath$, and show it is unique once we fix the normalization $\Upsilon_0=1$.
We prove that $\Upsilon \ov{\Upsilon} =1$. Note the remarkable similarity of $\Upsilon$ with Lusztig's quasi-$\mc{R}$-matrix for quantum groups. The intertwiner $\Upsilon$ is used to construct
a $\Ui$-module isomorphism $\mc T$ on any finite-dimensional $\U$-module, which should be viewed
as an analogue of $\mc R$-matrix on the tensor product of $\U$-modules. 

In Chapter~\ref{sec:QuasiR},
we define $\ThetaB$ for $\Ui$, which will play an analogous role as
Lusztig's quasi-$\mc{R}$-matrix for $\U$ on tensor product modules. Our first definition of $\ThetaB$ is simply obtained 
by combining the intertwiner $\Upsilon$ and $\Theta$. 
More detailed analysis is required to show that (a normalized version of) $\ThetaB$ lies in a completion of $\Ui \otimes \U^-$.
We prove that $\ThetaB \ov{\ThetaB} =1$. Then we use $\Upsilon$ to construct an $\imath$-involutive
module structure on an involutive $\U$-module, and then use $\ThetaB$ to construct
an involution on a tensor product of a $\Ui$-module with a $\U$-module.

In Chapter~\ref{sec:UpsiloninZ}, 
we first  study the rank one case of $\U$ and $\Ui$ in detail, which turns out to be nontrivial. 
In the rank one setting, we easily show that $\Upsilon$ is integral and then construct the $\imath$-canonical bases
 for simple $\U$-modules ${}^\omega L(s)$ for $s\ge 0$. We formulate a $\Ui$-homomorphism
from ${^{\omega}{L}(s+2)}$ to $ {^{\omega}{L}(s)}$ and use it to study the relation of $\imath$-canonical bases
on ${^{\omega}{L}(s+2)}$ and ${^{\omega}{L}(s)}$, which surprisingly depends on the parity of $s$. 
This allows us to establish the $\imath$-canonical basis for $\Ui$ in two parities, 
which is shown to afford integrality and should be regarded as ``divided powers" of the generator $t$. 
 
 Then we apply the rank one results to study the general higher rank case.
 We show that the intertwiner $\Upsilon$ is integral and hence the bar involution $\Bbar$ on
 the simple $\U$-module  $ {^{\omega}{L}(\la)}$ preserves its $\mA$-form.
Then we  construct the $\imath$-canonical basis for ${}^\omega L(\la)$ for $\la\in \La^+$. 

In Chapter~\ref{sec:HeckeB},
we recall Schur-Jimbo duality between quantum group $\U$ and Hecke algebra of type $A$.
Then we establish a commuting action of $\Ui$ and Hecke algebra $\HBm$ of type $B$ on $\VV^{\otimes m}$,
and show that they form double centralizers. Just as Jimbo showed that the generators of
Hecke algebra of type $A$ are realized by $\mc R$-matrices, we show that
the extra generator of Hecke algebra of type $B$ is realized 
via the $\Ui$-homomorphism $\mc T$ introduced in Chapter~\ref{sec:qsp}.
We then show the existence of a bar involution on $\VV^{\otimes m}$ which is compatible with the bar involutions
on $\Ui$ and $\HBm$. 
This allows a reformulation of Kazhdan-Lusztig theory for Lie algebras of type $B/C$ via the 
involutive $\Ui$-module $\VV^{\otimes m}$ (without referring directly to the Hecke algebra). 

In Chapter~\ref{sec:QSPc},
we consider the other quantum symmetric pair $(\U, \Uj)$ with $\U$ of type  $A_{2r}$, so its natural representation
$\VV$ is odd-dimensional. We
formulate the counterparts of the main results from Chapter~\ref{sec:qsp} through Chapter~\ref{sec:HeckeB}
where $\U$ was of type $A_{2r+1}$ and $\dim \VV$ was even.
The proofs are similar and often simpler for $\Uj$ since 
it does not contain a generator $t$ as $\Ui$ does, and hence will be omitted almost entirely.

\vspace{.3cm}

In Part~2, which consists of
Chapters~\ref{sec:BGG}-\ref{sec:BGGj}, we switch to infinite-dimensional $\VV$ and infinite-rank
quantum symmetric pair $(\U, \Ui)$. 

In the preliminary Chapter~\ref{sec:BGG},  
we set up variants of BGG categories of the ortho-symplectic Lie superalgebras,
allowing possibly infinite-rank and/or parabolic versions. 

In Chapter~\ref{sec:Fockspaces},  
we formulate precisely the infinite-rank limit of various constructions in Part~1, such as $\VV$, 
$\U$, $\Ui$, $\Upsilon$, $\Bbar$, and so on. We transport the Bruhat ordering from the BGG category $\mc O_{\bf b}$
for $\osp(2m+1|2n)$ to the Fock space $\Tb$ by noting a canonical bijection of the indexing sets.
We formulate the $q$-wedge versions of the Fock spaces, which correspond to parabolic versions of the BGG categories.

In Chapter~\ref{osp:sec:cbanddcb},  we construct the $\imath$-canonical bases and dual
$\imath$-canonical bases in various completed Fock spaces, where the earlier detailed work on
completion of Fock spaces in \cite{CLW12} plays a fundamental role. 

In Chapter~\ref{sec:compareCB},  we are able to compare (dual) $\imath$-canonical bases in three different settings:
a tensor space versus its (partially) wedge subspace, a Fock space versus an adjacent one,
and a Fock space with a tensoring factor $\wedge^\infty \VV$  versus another with $\wedge^\infty \WW$.

In Chapter~\ref{sec:b-KL},  we show that
the coideal subalgebra $\Ui$ at $q=1$ is realized by translation functors in the BGG categories. 
This underlies the importance of the coideal subalgebra $\Ui$. Then
we put all the results in earlier chapters of Part~2 together to
prove the main theorem which solves the irreducible and tilting character problem for $\osp(2m+1|2n)$-modules
of integer weights. 

The last Chapter~\ref{sec:BGGj}  deals with a variant of the BGG category  
of $\osp(2m+1|2n)$-modules 
with half-integer weights.  The Kazhdan-Lusztig theory of this half-integer variant is formulated and solved
by the quantum symmetric pair $(\U, \Uj)$, an infinite-rank version of the ones formulated in the last chapter of Part 1.

\vspace{.2cm}

\noindent {\bf 
Convention and notation.}  We shall denote by $\N$ the set of nonnegative integers, and by $\Z_{>0}$ the set of
positive integers. 
In Part~1,
where $\dim \VV =2r+2$ (except in Chapter~\ref{sec:QSPc} where $\dim \VV=2r+1$),
$r$ is  fixed and so will not show up in most of the notations (such as $\VV$, $\U, \Ui$, $\Upsilon$, $\Bbar$ and so on).
In Part~2 (more precisely in Chapter~\ref{sec:Fockspaces}-\ref{osp:sec:cbanddcb}), 
subscripts and superscripts are added to the notation used in Part~1 to indicate the dependence on $r$
(e.g., $\VV_r$, $\U_{2r+1}$, $\U^\imath_r$, $\Upsilon^{(r)}$, $\Bbar^{(r)}$ and so on).
In this way we shall consider $\VV$ as a direct limit $\varinjlim \VV_r$, and various
constructions including the intertwiner $\Upsilon$ as well as the bar involution $\Bbar$ arise as limits of their counterparts in Part~1.

\vspace{.2cm}

\noindent {\bf Acknowledgement.} 
This research is partially supported
by WW's NSF grants DMS-1101268 and DMS-1405131. 
We are indebted to Shun-Jen Cheng for his generous helps in many ways and thank Institute of Mathematics, Academia Sinica, 
Taipei for providing an excellent working environment and support, where part of this project was carried out. 

\vspace{.2cm}
\noindent {\bf Notes added.} 
This final version of our paper is not much different from the version originally posted in arXiv:1310.0103. 

In the preprint \cite{ES}
Ehrig and Stroppel simultaneously and independently discovered 
connections between the coideal algebras and category $\mathcal O$ of type $D$.
They also independently obtained the bar-invariant presentations of the coideal algebras. 

 In the preprint \cite{BaKo} 
Balagovic and Kolb have generalized our construction of the intertwiner in this paper for general quantum symmetric pairs (this generalization has overlap with our forthcoming paper \cite{BW16}, where it is used toward a general construction of $\imath$-canonical bases.)
Balagovic and Kolb have showed that the notion of intertwiner leads to solutions to the reflection equation, just as Drinfeld's universal $R$-matrix provides solutions
to Yang-Baxter equation. 

\newpage

\part{Quantum symmetric pairs}

\chapter{Preliminaries on quantum groups}
 \label{sec:prelim}
 
 In this preliminary chapter, we review some basic definitions and constructions on quantum groups
 from Lusztig's book,
 including the braid group action, canonical basis and quasi-$\mc R$-matrix. 
 We also introduce the involution $\inv$ and the lattice $\La_\inv$ which will be used in quantum symmetric pairs. 
 
\section{The involution $\inv$ and the lattice $\Lambda_\inv$}
  \label{subsec:theta}

Let $q$ be an indeterminate. 
For $r \in \N$, we define the following index sets:
\begin{align}
  \label{eq:I}
\begin{split}
\I_{2r+1} &= \{i \in \Z \mid -r \le i \le r\},  
 \\
\I_{2r}  &= \big\{i \in \Z+\hf \mid -r < i < r \big\}.
\end{split}
\end{align}

Set $k =2r+1$ or $2r$, and we use the shorthand notation $\I =\I_k$ in the remainder of Chapter~\ref{sec:prelim}.
Let  
\[
\Pi= \big \{\alpha_i=\varepsilon_{i-\hf}-\varepsilon_{i+\hf} \mid i \in \I \big \}
\]
be the simple system of type $A_{k}$, and
let $\Phi$ be the associated root system. Denote by 
$$
\Lambda = \sum_{{i \in \I} } \big( \Z \varepsilon_{i - \hf} + \Z \varepsilon_{i + \hf} \big)
$$  
the integral weight lattice, and denote by $(\cdot, \cdot)$ the standard bilinear pairing on 
$\Lambda$ such that $(\varepsilon_a, 
\varepsilon_b) = \delta_{ab}$ for all $a,b$.  
For any $\mu = \sum_{i}c_i\alpha_i \in \N {\Pi}$, 
set $\hgt(\mu) = \sum_{i} c_i$. 

Let $\inv$ be the involution of the weight lattice $\Lambda$ such that 
\[
\inv(\varepsilon_{i-\hf}) = - \varepsilon_{-i+\hf}, \quad \text{ for all } i \in \I. 
\]
We shall also write $\lambda^{\inv} = \inv(\lambda)$, for $\lambda \in \Lambda$. The involution
$\inv$ preserves the bilinear form $(\cdot,\cdot)$ on the weight lattice $\Lambda$
and induces an automorphism on the root system $\Phi$ such that $ \alpha^{\inv}_i = \alpha_{-i}$ for all $i \in \I$. 

Denote by $\Lambda^{\inv} = \{\mu \in \Lambda \mid \mu^{\inv} =\mu\}$ the subgroup of $\theta$-fixed points in $\Lambda$.  
It is easy to see that the quotient group
\begin{equation}  \label{eq:bwl}
\Lambda_{\inv} := \Lambda/\Lambda^{\inv}
\end{equation}
is a lattice. For $\mu \in \Lambda$, denote by $\ov{\mu}$ the image of $\mu$ under the quotient map. 
There is a well-defined bilinear 
pairing  $\Z[\alpha_i - \alpha_{-i}]_{i \in \I} \times \BLambda  \rightarrow \Z$, such that 
$(\sum_{i > 0}a_i(\alpha_i-\alpha_{-i}), \ov{\mu}) :=   \sum_{i > 0}a_i(\alpha_i-\alpha_{-i}, \mu)$ for any $\ov{\mu} \in \BLambda$
with any preimage $\mu\in \Lambda$.

\section{The algebras $\fprime$, $\f$ and $\U$}
  \label{subsec:f}

Consider a free $\Qq$-algebra $'\f$ generated by $F_{\alpha_i}$ for ${i \in \I}$ associated with the 
Cartan datum of type $(\I, (\cdot,\cdot))$ \cite{Lu94}. As a $\Qq$-vector space, $'\f$ has a direct sum decomposition as 
\[
'\f = \bigoplus_{\mu \in {\N}{\Pi}}~ '\f_{\mu},
\]
where $F_{\alpha_i}$ has weight $\alpha_i$ for all $i \in \I$.
For any $x \in \fprime_\mu$, we set $|x| = \mu$.

For each $i \in \I$, we define $r_{i}, {}_i r$ to be the unique $\Q(q)$-linear maps on $\fprime$ such that 
\begin{align}  \label{eq:rr}
\begin{split}
r_{i}(1) = 0, \quad r_{i}(F_{\alpha_j}) = \delta_{ij},
\quad r_{i}(xx') = xr_{i}(x') + q^{(\alpha_i, \mu')}r_{i}(x)x',
 \\
{}_{i}r(1) = 0, \quad {}_{i}r(F_{\alpha_j}) = \delta_{ij},
\quad {}_{i}r(xx') =q^{(\alpha_i, \mu)}x \,_{i}r(x') +{ _{i}r(x)x'},
\end{split}
\end{align}
for all $x \in \fprime_{\mu}$ and $x' \in \fprime_{\mu'}$. 
The following lemma is well known (see \cite{Lu94} and \cite[Section~ 10.1]{Jan}).

\begin{lem}\label{lem:rlrr=rrrl}
The $\Q(q)$-linear map $r_j$ and ${}_{i}r$ commute; that is,
$r_{j} \,_{i}r  = \,_{i}r \,r_{j}$ for all $i$, $j \in \I$.
\end{lem}

\begin{prop}   \label{prop:Luform}    \cite{Lu94}
There is a unique symmetric bilinear form $( \cdot, \cdot)$ on $\fprime$ which satisfies that, for all $x, x' \in \fprime$,
\begin{enumerate}
\item $(F_{\alpha_i}, F_{\alpha_j}) =\delta_{ij} \qqq$,
\item $(F_{\alpha_i}x, x') = (F_{\alpha_i}, F_{\alpha_i})(x,\,_{i}r(x'))$,
\item $(xF_{\alpha_i}, x') = (F_{\alpha_i}, F_{\alpha_i})(x, r_{i}(x'))$.
\end{enumerate}
\end{prop}


Let ${\bf I}$ be the radical of the bilinear form $(\cdot,\cdot)$ on $\fprime$. It is known in \cite{Lu94} that ${\bf I}$ is 
generated by the quantum Serre relators $S_{ij}$, for $ i \neq  j \in \I$, where
\begin{equation}  \label{eq:Sij}
S_{i j} = 
\begin{cases}
F_{\alpha_i}^2 F_{\alpha_j} +F_{\alpha_j} F_{\alpha_i}^2 
- (q+q^{-1}) F_{\alpha_i} F_{\alpha_j} F_{\alpha_i}, \qquad &\text{ if }  |i-j|=1;\\
F_{\alpha_i} F_{\alpha_j}- F_{\alpha_j} F_{\alpha_i}, \qquad &\text{ if } |i-j| >1.
\end{cases}
\end{equation}
Let $\f =\fprime/\bf{I}$.
By \cite{Lu94}, we have 
\begin{equation}  \label{eq:r=0}
r_{\ell}(S_{ij}) =\,_{\ell}r(S_{ij}) =0, \qquad \forall \ell, i , j \in \I \; (i\neq j).
\end{equation}
Hence $r_{\ell}$ and $_{\ell}r$ descend to well-defined $\Qq$-linear maps on $\f$.

We introduce the divided power $F^{(a)}_{\alpha_i} = F^a_{\alpha_i}/[a]!$, where $a \ge 0$, 
$[a] = (q^{a}- q^{-a})/(q-q^{-1})$ and $[a]! = [1][2] \cdots [a]$. Let $\mA =\Z[q,q^{-1}]$.
Let $\f_\mA$ be the $\mA$-subalgebra of $\f$ 
generated by $F^{(a)}_{\alpha_i}$ for various $a \ge 0$ and $i \in \I$.

 The quantum group $\U = \U_q(\mf{sl}(k+1))$ is defined to be the associative $\Q(q)$-algebra  
 generated by $E_{\alpha_i}$, $F_{\alpha_i}$, $K_{\alpha_i}$, $K^{-1}_{\alpha_i}$, $i \in \I$, subject to
the following relations for $i$, $j \in \I$:
\begin{align*}
 K_{\alpha_i} K_{\alpha_i}^{-1} &= K_{\alpha_i}^{-1} K_{\alpha_i} =1,
  \\
 K_{\alpha_i} K_{\alpha_j} &= K_{\alpha_j} K_{\alpha_i},
  \\
 K_{\alpha_i} E_{\alpha_j} K_{\alpha_i}^{-1} &= q^{(\alpha_i, \alpha_j)} E_{\alpha_j}, \\
 K_{\alpha_i} F_{\alpha_j} K_{\alpha_i}^{-1} &= q^{-(\alpha_i, \alpha_j)} F_{\alpha_j}, \\
 E_{\alpha_i} F_{\alpha_j} -F_{\alpha_j} E_{\alpha_i} &= \delta_{i,j} \frac{K_{\alpha_i}
 -K^{-1}_{\alpha_i}}{q-q^{-1}}, \\
 E_{\alpha_i}^2 E_{\alpha_j} +E_{\alpha_j} E_{\alpha_i}^2 
 &= (q+q^{-1}) E_{\alpha_i} E_{\alpha_j} E_{\alpha_i},  \quad &&\text{if } |i-j|=1, \\
 E_{\alpha_i} E_{\alpha_j} &= E_{\alpha_j} E_{\alpha_i},  \,\qquad\qquad &&\text{if } |i-j|>1, \\
 F_{\alpha_i}^2 F_{\alpha_j} +F_{\alpha_j} F_{\alpha_i}^2 
 &= (q+q^{-1}) F_{\alpha_i} F_{\alpha_j} F_{\alpha_i},  \quad\, &&\text{if } |i-j|=1,\\
 F_{\alpha_i} F_{\alpha_j} &= F_{\alpha_j} F_{\alpha_i},  \qquad\ \qquad &&\text{if } |i-j|>1.
\end{align*}

Let $\U^+$, $\U^0$ and $\U^-$ be the $\Qq$-subalgebra of $\U$ generated by $E_{\alpha_i}$, $K^{\pm 1}_{\alpha_i}$, 
and $F_{\alpha_i}$  respectively, for  $i \in \I$. 
Following \cite{Lu94}, we can identify $\f \cong \U^-$ by matching the 
generators in the same notation. This identification induces
a bilinear form $(\cdot, \cdot)$ on $\U^{-}$ and $\Qq$-linear maps $r_i, {}_i r$ $(i\in \I)$ on $\U^-$.
Under this identification, we let $\U_{-\mu}^-$ be the image of $\f_\mu$,
and let $\U^-_\mA$ be the image of $\f_\mA$. The following Serre relation holds in $\U^-$:
\begin{equation}  \label{eq:Serre}
S_{ij} =0, \qquad \forall i,j \in \I \; (i\neq j). 
\end{equation}
Similarly we have $\f \cong \U^+$ by identifying each generator $F_{\alpha_i}$ with $E_{\alpha_i}$. 
Similarly we let $\U^+_\mA$ denote the image of $\f_\mA$ under this isomorphism, 
which is generated by all divided powers $E^{(a)}_{\alpha_i} =E_{\alpha_i}^a/[a]!$.

\begin{prop}
\label{prop:invol}
\begin{enumerate}\label{prop:tauA}
\item There is an  involution $\omega$ on the $\Qq$-algebra $\U$
such that $\omega(E_{\alpha_i}) =F_{\alpha_i}$, $\omega(F_{\alpha_i}) =E_{\alpha_i}$, 
and $\omega(K_{\alpha_i}) = K^{-1}_{\alpha_i}$ for all $i \in \I$.

\item There is an anti-linear ($q \mapsto q^{-1}$) bar involution of the $\Q$-algebra $\U$
such that $\ov{E}_{\alpha_i}= E_{\alpha_i}$, $\ov{F}_{\alpha_i}=F_{\alpha_i}$, 
and $\ov{K}_{\alpha_i}=K_{\alpha_i}^{-1}$ for all $i \in \I$.

(Sometimes we denote the bar involution on $\U$ by $\psi$.)
%
\end{enumerate}
\end{prop}

Recall that $\U$ is a Hopf algebra with a coproduct 
%
\begin{align}  \label{eq:coprod}
\begin{split}
\Delta:  &\U \longrightarrow \U \otimes \U,
 \\
 \Delta (E_{\alpha_i}) &= 1 \otimes E_{\alpha_i} + E_{\alpha_i} \otimes K^{-1}_{\alpha_i}, \\
 \Delta (F_{\alpha_i}) &= F_{\alpha_i} \otimes 1 +  K_{\alpha_i} \otimes F_{\alpha_i},\\
 \Delta (K_{\alpha_i}) &= K_{\alpha_i} \otimes K_{\alpha_i}.
 \end{split}
\end{align}
The coproduct $\Delta$ here (which is chosen to be convenient for the connection with category $\mathcal O$)
differs from the one used in \cite{Lu94};  this results 
to a switching between positive and negative parts of $\U$ for quasi-$\mathcal R$-matrix, and between
highest and lowest weight modules.
There is a unique $\Qq$-algebra homomorphism
$\epsilon: \U \rightarrow \Qq$, called counit, 
such that $\epsilon(E_{\alpha_i}) = 0$, $\epsilon(F_{\alpha_i}) = 0$, and $\epsilon(K_{\alpha_i}) =1$.

\section{Braid group action and canonical basis} 
\label{subsec:CB}

Let $W := W_{A_{{k}}} = \mf{S}_{{k}+1}$ be the Weyl group of type $A_{{k}}$. 
Recall \cite{Lu94} for each $\alpha_i$ and each finite-dimensional $\U$-module $M$, a linear operator 
$T_{\alpha_i}$ on $M$ is defined by,  for $\lambda \in \Lambda$ and $m \in M_\lambda$,
\[
T_{\alpha_i}(m) = \sum_{a, b, c \ge 0; -a+b-c
=(\lambda, \alpha_i)}(-1)^b q^{b-ac}E^{(a)}_{\alpha_i}F^{(b)}_{\alpha_i}E^{(c)}_{\alpha_i} m.
\]
These $T_{\alpha_i}$'s induce automorphisms of $\U$, denoted by $T_{\alpha_i}$ as well,  such that 
\[
T_{\alpha_i}(um) = T_{\alpha_i}(u)T_{\alpha_i}(m), \qquad \text{ for all } u \in \U, m \in M.
\] 
As automorphisms on $\U$ and as $\Qq$-linear isomorphisms on $M$,
the $T_{\alpha_i}$'s satisfy the braid group relation (\cite[Theorem 39.4.3]{Lu94}):
\begin{align*}
T_{\alpha_i}T_{\alpha_j} &= T_{\alpha_j}T_{\alpha_i}, &\text{ if } |i-j| >1 ,\\
T_{\alpha_i}T_{\alpha_j}T_{\alpha_i} &= T_{\alpha_j}T_{\alpha_i}T_{\alpha_j}, &\text{ if } |i-j| =1,
\end{align*}
Hence for each $w \in W$, $T_w$ can be defined independent of the choices of reduced expressions of $w$. 
(The $T_{\alpha_i}$ here is consistent with $T_{\alpha_i}$ in \cite{Jan}, and it is $T''_{i,+}$ in \cite{Lu94}).

Denote by $\ell (\cdot)$ the length function of $W$, and let $w_0$ be the longest element of $W$.

\begin{lem}\label{lem:Tw0}
The following identities hold for $i \in \I$:
\[
T_{w_0}(K_{\alpha_i}) =K^{-1}_{\alpha_{-i}}, \quad T_{w_0}(E_{\alpha_i}) = -F_{\alpha_{-i}}K_{\alpha_{-i}}, \quad 
T_{w_0}(F_{\alpha_{-i}}) = - K^{-1}_{\alpha_i}E_{\alpha_i}.
\]
\end{lem}
\begin{proof}
The identity $T_{w_0}(K_{\alpha_i}) =K^{-1}_{\alpha_{-i}}$ is clear (see \cite{Lu94} or \cite{Jan}). 

Let us show that $T_{w_0}(E_{\alpha_i}) = -F_{\alpha_{-i}}K_{\alpha_{-i}}$, for any given $i \in \I$.
Indeed, we can always write
$w_0 = w s_i$ with $\ell(w) = \ell(w_0) -1$. Then we have $T_{w_0} = T_{w} T_{s_i}$, and
\[
T_{w_0} (E_{\alpha_i}) = T_{w} (T_{s_i}(E_{\alpha_i})) =  T_{w} (-F_{\alpha_i}K_{\alpha_i}) =  - T_{w}(F_{\alpha_i}) K_{\alpha_{-i}} 
= -F_{\alpha_{-i}}K_{\alpha_{-i}},
\]
where the last identity used $w(-\alpha_i)= w_0 (\alpha_i) =- \alpha_{-i}$ 
and \cite[Proposition 8.20]{Jan}. 

The identity $T_{w_0}(F_{\alpha_{-i}}) = - K^{-1}_{\alpha_i}E_{\alpha_i}$ can be similarly proved. 
\end{proof}

Let 
$$\Lambda^+ = \{\lambda \in \Lambda \mid 2(\alpha_i , \lambda)/(\alpha_i,\alpha_i) \in {\N}, \forall i \in \I\}$$ 
 be the set of dominant weights. Note that $\mu \in \Lambda^+$ if and only if $\mu ^{\inv} \in \Lambda^+$, 
 since the bilinear pairing $(\cdot,\cdot)$ on $ \Lambda$ is invariant under $\inv : \Lambda \rightarrow \Lambda$.

Let $M(\lambda)$ be the Verma module of $\U$ with highest weight $\lambda\in \Lambda$ and
with a highest weight vector denoted by $\eta$ or $\eta_{\lambda}$.  
We define a $\U$-module $^\omega M(\lambda)$, which 
has the same underlying vector space as $M(\lambda)$ but 
with the action twisted by the involution $\omega$ given in Proposition~\ref{prop:invol}.
When considering $\eta$ as a vector in $^\omega M(\lambda)$, 
we shall denote it by $\xi$ or $\xi_{-\lambda}$. 
The Verma module $M(\lambda)$ associated to dominant $\la \in \La^+$ has a unique finite-dimensional 
simple quotient $\U$-module, denoted by $L(\lambda)$. 
Similarly we define the $\U$-module $^\omega L(\lambda)$. 
For $\la \in \Lambda^+$, we let $L_\mA(\lambda) =\U^-_\mA \eta$ and $^\omega L_\mA(\lambda) =\U^+_\mA \xi$ 
be the $\mA$-submodules of $L(\lambda)$ and $^\omega L(\lambda) $, respectively. 

In \cite{Lu90, Lu94} and \cite{Ka}, the canonical basis $\bold{B}$ of $\f_\mA$ is constructed. Recall that we can identify $\f$ with both $\U^-$ and $\U^+$. 
For any element $b \in \B$, when considered as an element in $\U^-$ or $\U^+$, we shall denote it by $b^-$ or $b^+$,  respectively. 
In \cite{Lu94}, subsets $\B(\lambda)$ of $\B$ is also constructed for each $\lambda \in \Lambda^+$, 
such that $\{b^-\eta_{\lambda} \mid b \in \B(\lambda)\}$ gives the canonical basis of $L_\mA(\lambda)$. 
Similarly $\{b^+ \xi_{-\lambda} \mid b \in \B(\lambda)\}$ gives the canonical basis of ${^{\omega}L(\lambda)}$.  
By \cite[Proposition 21.1.2]{Lu94},  we can identify ${^{\omega}L(\lambda)}$ with $L(\lambda^{\inv}) = L(-w_0\lambda) $
such that the set $\{b^+\xi_{-\lambda} \mid b \in \B(\lambda)\}$ is identified with the set
$ \{b^- \eta_{\lambda^{\inv}} \mid b \in \B(\lambda^{\inv})\} = \{b^- \eta_{-w_0\lambda} \mid b \in \B(-w_0\lambda)\}  $. 
We shall identify ${^{\omega}L(\lambda)}$ with $L(\lambda^{\inv})$ in this way throughout this paper.

\section{Quasi-$\mc{R}$-matrix $\ThetaA$}
 \label{sec:quasiR}
 
\begin{prop}\cite[Theorem 4.1.2]{Lu94}\label{prop:ThetaA}
There exists a unique family of elements $\ThetaA_{\mu}$ in $\U_\mu^+\otimes \U_{-\mu}^-$ 
with $\mu \in {\N}{\Pi}$, such that $\ThetaA_{0}= 1\otimes 1$ and the following identities hold for all $\mu$ and all $i$:
\begin{align*}
(1 \otimes E_{\alpha_i}) \ThetaA_{\mu} + (E_{\alpha_i} \otimes K^{-1}_{\alpha_i})\ThetaA_{\mu - \alpha_i}
 &=  \ThetaA_{\mu}(1 \otimes E_{\alpha_i} )+ \ThetaA_{\mu - \alpha_i}(E_{\alpha_i} \otimes K_{\alpha_i}),\\
(F_{\alpha_i} \otimes 1)  \ThetaA_{\mu} +  (K_{\alpha_i} \otimes F_{\alpha_i})\ThetaA_{\mu - \alpha_i}
 &=   \ThetaA_{\mu} (F_{\alpha_i} \otimes 1) + \ThetaA_{\mu - \alpha_i} (K^{-1}_{\alpha_i} \otimes F_{\alpha_i}),\\
(K_{\alpha_i} \otimes K_{\alpha_i}) \ThetaA_{\mu} &= \ThetaA_{\mu} (K_{\alpha_i} \otimes K_{\alpha_i}) .
\end{align*}
\end{prop}

\begin{rem}
We adopt the convention in this paper that $\ThetaA_\mu$  lies in $\U^{+} \otimes \U^-$ 
due to our different choice of the coproduct $\Delta$ from \cite{Lu94}.
(In contrast the $\ThetaA_\mu$ in \cite{Lu94} lies in $\U^- \otimes \U^+$.) The convention here is adopted 
in order to be more
compatible with the application to category $\mc O$  in Part 2.
\end{rem}

Lusztig's quasi-$\mc R$-matrix for $\U$ is defined to be
\begin{equation}\label{eq:ThetaA}
\ThetaA  := \sum_{\mu\in \N \Pi} \ThetaA_\mu.
\end{equation}
For any finite-dimensional $\U$-modules $M$ and  $M'$, the action of $\ThetaA$ on $M \otimes M'$ is well defined.
Proposition \ref{prop:ThetaA} implies that
\begin{equation}\label{eq:propThetaA}
\Delta(u) \ThetaA (m \otimes m') 
= \ThetaA \ov{\Delta ( \ov {u})} (m \otimes m'),  
\end{equation}
for all $m \in M,  m' \in M' ,\text{ and } u \in \U.$
By \cite[Corollary 4.1.3]{Lu94},
we have 
\begin{equation}\label{eq:ThetaAinv}
\ThetaA \ov{\ThetaA} (m \otimes m') = m \otimes m', \quad \text{ for all } m \in M \text{ and } m' \in M'.
\end{equation}
In \cite[32.1.5]{Lu94}, a $\U$-module isomorphism 
$$
\mc{R} = \mc{R}_{M, M'} : M' \otimes M \longrightarrow M \otimes M'
$$ 
is constructed. 
As an operator, $\mc{R}$ can be written as $\mc{R}= \ThetaA \circ \widetilde{g} \circ P$ 
where $\widetilde{g}: M \otimes M' \rightarrow M \otimes M'$ 
is the map $\widetilde{g}(m \otimes m') = q^{(\lambda,\mu)} m \otimes m'$ 
for all $ m \in M_{\lambda}, m' \in M_{\mu}'$, 
and $P : M' \otimes M \rightarrow M \otimes M'$ is a $\Qq$-linear isomorphism such that
$P(m \otimes m')= m' \otimes m$. 

\begin{definition}  \label{def:involutive}
A $\U$-module $M$ equipped with an anti-linear involution $\Abar$ is called {\em involutive} if
$$
\Abar(u m) = \Abar(u) \Abar(m), \qquad \forall u \in \U, m \in M.
$$ 
 \end{definition}

Given two involutive $\U$-modules $(M,\psi_1)$ and $(M_2, \psi_2)$, 
following Lusztig we define  a map $\Abar$ on $M_1 \otimes M_2$ by 
\begin{equation}   \label{eq:Lupsi}
\Abar(m \otimes m') : = \Theta(\psi_1(m) \otimes \psi_2(m')).
\end{equation}
By Proposition \ref{prop:ThetaA}, we have 
$\Abar(u (m \otimes m')) = \Abar(u) \Abar(m \otimes m')$ for all $u \in \U$,
and the identity \eqref{eq:ThetaAinv} implies that the map $\psi$ on $M_1\otimes M_2$ is an anti-linear involution. 
This proves the following result of Lusztig (though the terminology of involutive modules is new here).

\begin{prop}  \cite[27.3.1]{Lu94}   \label{prop:Lu27.3.1}
Given two involutive $\U$-modules $(M,\psi_1)$ and $(M_2, \psi_2)$,  $(M_1 \otimes M_2, \psi)$
is  an involutive $\U$-module with $\psi$ given in \eqref{eq:Lupsi}.
\end{prop}

It follows by induction that $M_1 \otimes \cdots \otimes M_s$ is naturally an involutive $\U$-module 
for given involutive $\U$-modules $M_1,\ldots, M_s$; see \cite[27.3.6]{Lu94}.

 As in \cite{Lu94}, there is a unique anti-linear involution $\Abar$ on ${^{\omega}L}(\lambda)$  such that
 $\Abar(u \xi) = \Abar(u)\xi$  for all
$u \in \U$. Similarly there is a unique anti-linear involution $\Abar$ on $L(\lambda)$  such that 
$\Abar( u \eta) = \Abar(u) \eta$  for all
$u \in \U$. Therefore ${^{\omega}L}(\lambda)$ and $L(\lambda)$ are both involutive $\U$-modules.

\chapter{Intertwiner for a quantum symmetric pair}
  \label{sec:qsp}

In Chapters~\ref{sec:qsp}-\ref{sec:HeckeB}, we will formulate and study in depth the quantum symmetric pair $(\U, \Ui)$
for $\U$ of type $A_k$ with  ${k} = 2r+1$ being an odd integer. 
In these chapters, we shall use the shorthand notation 
$$
\I
=\I_{2r+1} =\{-r, \ldots, -1,0,1,\ldots, r\}
$$ 
as given in \eqref{eq:I}, and set 
\begin{equation}  \label{eq:Ihf}
\Ihf  :=\Z_{>0} \cap \I =\{1, \ldots, r\}.
\end{equation}

In this chapter, we will introduce the right coideal subalgebra $\Ui$ of $\U$
and an algebra embedding $\imath: \Ui \rightarrow \U$.
Then we construct an intertwiner $\Upsilon$ for the two bar involutions on 
$\Ui$ and $\U$ under $\imath$, and  use it to construct
a $\Ui$-module isomorphism $\mc T$ on any finite-dimensional $\U$-module.

\section{Definition of the algebra $\Ui$}\label{sec:bun}

The algebra $\Ui = \U^\imath_r$ is defined to be the associative algebra over $\Q(q)$ generated by  
$\be_{\alpha_i}$, $\bff_{\alpha_i}$, $\bk_{\alpha_i}$, $\bk^{-1}_{\alpha_i}$ ($i \in \Ihf$) , and $\bt$, 
subject to the following relations for $i$, $j \in \Ihf$:
\begin{align*}
 \bk_{\alpha_i} \bk_{\alpha_i}^{-1} &= \bk_{\alpha_i}^{-1} \bk_{\alpha_i} =1,\displaybreak[0]\\
 \bk_{\alpha_i} \bk_{\alpha_j} &= \bk_{\alpha_j} \bk_{\alpha_i}, \displaybreak[0]\\
 \bk_{\alpha_i} \be_{\alpha_j} \bk_{\alpha_i}^{-1} &= q^{(\alpha_i-\alpha_{-i}, \alpha_j)} \be_{\alpha_j}, \displaybreak[0]\\
 \bk_{\alpha_i} \bff_{\alpha_j} \bk_{\alpha_i}^{-1} &= q^{-(\alpha_i-\alpha_{-i}, \alpha_j)}
 \bff_{\alpha_j}, \displaybreak[0]\\
 \bk_{\alpha_i}\bt\bk^{-1}_{\alpha_i} &= \bt, \displaybreak[0]\\
 \be_{\alpha_i} \bff_{\alpha_j} -\bff_{\alpha_j} \be_{\alpha_i} &= \delta_{i,j} \frac{\bk_{\alpha_i}
 -\bk^{-1}_{\alpha_i}}{q-q^{-1}}, \displaybreak[0]\\
 \be_{\alpha_i}^2 \be_{\alpha_j} +\be_{\alpha_j} \be_{\alpha_i}^2 &= (q+q^{-1}) \be_{\alpha_i} \be_{\alpha_j} \be_{\alpha_i}, 
   \ \ \quad\quad &\text{if }& |i-j|=1, \displaybreak[0]\\
 \be_{\alpha_i} \be_{\alpha_j} &= \be_{\alpha_j} \be_{\alpha_i}, \ \qquad\qquad\ \ \ \ \qquad &\text{if }& |i-j|>1, \displaybreak[0]\\
 \bff_{\alpha_i}^2 \bff_{\alpha_j} +\bff_{\alpha_j} \bff_{\alpha_i}^2 &= (q+q^{-1}) \bff_{\alpha_i} \bff_{\alpha_j} \bff_{\alpha_i},
    \ \ \quad\quad &\text{if }& |i-j|=1,\displaybreak[0]\\
 \bff_{\alpha_i} \bff_{\alpha_j}  &= \bff_{\alpha_j}  \bff_{\alpha_i},  \ \qquad\qquad\ \ \ \ \qquad &\text{if }& |i-j|>1, \displaybreak[0]\\
 \be_{\alpha_i}\bt &=\bt\be_{\alpha_i}, \quad\qquad\quad &\text{if }&  i > 1, \displaybreak[0]\\
 \be_{\alpha_1}^2\bt + \bt\be_{\alpha_1}^2 &= (q+q^{-1}) \be_{\alpha_1}\bt\be_{\alpha_1},\displaybreak[0]\\
 \bt^2\be_{\alpha_1} + \be_{\alpha_1}\bt^2 &= (q + q^{-1}) \bt\be_{\alpha_1}\bt + \be_{\alpha_1},\displaybreak[0]\\
 \bff_{\alpha_i}\bt &=\bt\bff_{\alpha_i}, \quad\qquad&\text{if }&  i > 1, \displaybreak[0]\\
 \bff_{\alpha_1}^2\bt + \bt\bff_{\alpha_1}^2 &= (q+q^{-1}) \bff_{\alpha_1}\bt\bff_{\alpha_1},\displaybreak[0]\\
 \bt^2\bff_{\alpha_1} + \bff_{\alpha_1}\bt^2 &= (q + q^{-1}) \bt\bff_{\alpha_1}\bt + \bff_{\alpha_1}.\displaybreak[0]
\end{align*}
We introduce the divided powers $\be^{(a)}_{\alpha_i} = \be^a_{i} / [a]!$, $\bff^{(a)}_{\alpha_i} = \bff^a_{i} / [a]!$ for  $a \ge 0$, $i \in \Ihf$.

\begin{lem}   \label{lem:3inv}
\begin{enumerate}
\item 
The $\Qq$-algebra $\Ui$ has an involution $\omega_\imath$ such that 
$\omega_\imath (\bk_{\alpha_i}) = \bk^{-1}_{\alpha_i}$, 
$\omega_\imath (\be_{\alpha_i}) = \bff_{\alpha_i}$, 
$\omega_\imath (\bff_{\alpha_i}) = \be_{\alpha_i}$, 
and $\omega_\imath (\bt) = \bt$ for all $i \in \Ihf$. 

\item
The $\Qq$-algebra $\Ui$ has an anti-involution  $\tau_\imath$ such that 
$\tau_\imath(\be_{{\alpha_i}}) =\be_{\alpha_i} , \tau_\imath(\bff_{\alpha_i}) = \bff_{\alpha_i}, \tau_\imath(\bt)=\bt $, 
and $\tau_\imath(\bk_{\alpha_i}) =\bk^{-1}_{\alpha_i}$ for all $i \in \Ihf$. 

\item
The $\Q$-algebra $\Ui$ has an anti-linear ($q \mapsto q^{-1}$) bar involution such that 
$\ov{\bk}_{\alpha_i} = \bk^{-1}_{\alpha_i}$, 
$\ov{\be}_{\alpha_i} = \be_{\alpha_i}$, $\ov{\bff}_{\alpha_i} = \bff_{\alpha_i}$, and  $\ov{\bt} = \bt$ for all $i \in \Ihf$. 

(Sometimes we denote the bar involution on $\Ui$ by $\Bbar$.)
\end{enumerate}
\end{lem}

\begin{proof}
Follows by a direct computation from the definitions. 
\end{proof}

\section{Quantum symmetric pair $(\U, \Ui)$}

The Dynkin diagram of type $A_{2r+1}$ together with the involution $\inv$
can be depicted as follows: 

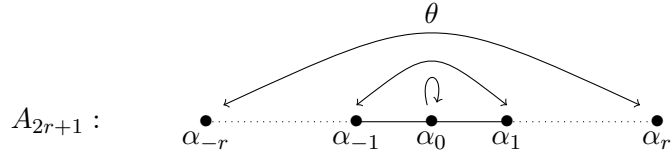
\begin{center}
\begin{tikzpicture}

\draw (-2,0) node {$A_{2r+1}:$};
 \draw[dotted]  (0,0) node[below] {$\alpha_{-r}$} -- (2,0) node[below] {$\alpha_{-1}$} ;
 \draw (2,0) -- (3,0) node[below]  {$\alpha_{0}$}
 -- (4,0) node[below] {$\alpha_{1}$};
 \draw[dotted] (4,0) -- (6,0) node[below] {$\alpha_{r}$} ;
\draw (0,0) node (-r) {$\bullet$};
 \draw (2,0) node (-1) {$\bullet$};
\draw (3,0) node (0) {$\bullet$};
\draw (4,0) node (1) {$\bullet$}; 
\draw (6,0) node (r) {$\bullet$};
\draw[<->] (-r.north east) .. controls (3,1.5) .. node[above] {$\theta$} (r.north west) ;
\draw[<->] (-1.north) .. controls (3,1) ..  (1.north) ;
\draw[<->] (0) edge[<->, loop above] (0);
\end{tikzpicture}
\end{center}

A general theory of quantum symmetric pairs via the notion of coideal subalgebras was 
developed systematically by Letzter \cite{Le03} (also see \cite{KP, Ko12}). 
As the properties in Propositions~\ref{prop:embedding}  and  \ref{prop:coproduct} below indicate,
the algebra $\Ui$ is a (right) coideal subalgebra of $\U$ and that $(\U, \Ui)$ forms a quantum symmetric pair.  
 
\begin{prop} \label{prop:embedding}   
There is an injective $\Qq$-algebra homomorphism $\imath :  \Ui \rightarrow \U$
which sends
 \begin{align*}
\bk_{\alpha_i} &\mapsto K_{\alpha_i}K^{-1}_{\alpha_{-i}},\\
\be_{\alpha_i} &\mapsto  E_{\alpha_i} + K^{-1}_{\alpha_i}F_{\alpha_{-i}},\\
\bff_{\alpha_i} &\mapsto F_{\alpha_i} K^{-1}_{\alpha_{-i}}+ E_{\alpha_{-i}},\\
\bt &\mapsto E_{\alpha_0} +qF_{\alpha_0}K^{-1}_{\alpha_0} + K^{-1}_{\alpha_0}
\end{align*}
 for all $i \in \Ihf$.
\end{prop}

\begin{proof}
This proposition is a variant of a general property for
quantum symmetric pairs which can be found in \cite[Theorem 7.1]{Le03}.  
Hence we will not repeat the proof, except noting how to covert the result therein to the form used here. 

It follows from a direct computation that $\imath$ is a homomorphism of $\Q(q)$-algebras. 

We shall compare $\imath$ with the embedding in 
\cite[Proposition 4.1]{KP} (as modified by \cite[Remark~4.2]{KP}), which is a version of  \cite[Theorem 7.1]{Le03}.
Set $\U_{\C} = \C(q^{\hf}) \otimes_{\Qq} \U$.  
Recall from \cite[\S4]{KP}  a $\Qq$-subalgebra $U_q'(\mf{k})$ of $\U_{\C}$  with a generating set $\mf S$ consisting of
$F_{\alpha_0} - K^{-1}_{\alpha_0}E_{\alpha_0} + q^{-\hf}K^{-1}_{\alpha_0}$,
$K_{\alpha_i}K^{-1}_{\alpha_{-i}},
F_{\alpha_{-i}} - K^{-1}_{\alpha_{-i}}E_{\alpha_{i}},
F_{\alpha_{i}}-E_{\alpha_{-i}} K^{-1}_{\alpha_{i}},$ 
for all $0\neq i \in \Ihf$.

{\bf Claim.} The algebras $\C(q^{\hf}) \otimes_{\Qq} \imath(\Ui)$ and $\C(q^{\hf}) \otimes_{\Qq} U_q'(\mf{k})$
are anti-isomorphic.

Consider the $\C(q^{\hf})$-algebra anti-automorphism $\kappa: \U_\C \rightarrow \U_\C$ such that 
\[
E_{\alpha_i} \mapsto \sqrt{-1}F_{\alpha_{-i}}, \quad F_{\alpha_{i}} \mapsto -\sqrt{-1}
E_{\alpha_{-i}}, \quad K_{\alpha_i} \mapsto K_{\alpha_{-i}}, 
\]
 for all $0 \neq i \in \I,$ and 
\[
E_{\alpha_0} \mapsto \sqrt{-1}q^{\hf}F_{\alpha_0}, \quad F_{\alpha_0} \mapsto -\sqrt{-1}q^{-\hf}
E_{\alpha_0}, \quad K_{\alpha_0} \mapsto K_{\alpha_{0}}.
\]

A direct computation shows that $\kappa$ sends
 \begin{align*}
K_{\alpha_i}K^{-1}_{\alpha_{-i}} &\mapsto K_{\alpha_i}K^{-1}_{\alpha_{-i}},
\\
E_{\alpha_i} + K^{-1}_{\alpha_i}F_{\alpha_{-i}}&\mapsto  \sqrt{-1}(F_{\alpha_{-i}} - K^{-1}_{\alpha_{-i}}E_{\alpha_{i}}),
\\
F_{\alpha_i} K^{-1}_{\alpha_{-i}}+ E_{\alpha_{-i}} &\mapsto \sqrt{-1}( F_{\alpha_{i}}-E_{\alpha_{-i}} K^{-1}_{\alpha_{i}}),
\\
E_{\alpha_0} +qF_{\alpha_0}K^{-1}_{\alpha_0} + K^{-1}_{\alpha_0} 
 &\mapsto \sqrt{-1} q^{\hf} (F_{\alpha_0} - K^{-1}_{\alpha_0}E_{\alpha_0} + q^{-\hf}K^{-1}_{\alpha_0}).
\end{align*}
Hence, $\kappa$ restricts to an algebra anti-isomorphism between 
$\C(q^{\hf}) \otimes_{\Qq} \imath(\Ui)$ and $\C(q^{\hf}) \otimes_{\Qq} U_q'(\mf{k})$, whence the claim.

We observe that \cite[Proposition 4.1]{KP} provides a presentation of the algebra 
$U_q'(\mf{k})$ with the generating set $\mf S$ and a bunch of relations, which correspond  under $\kappa$
exactly to (the images of) the defining relations of $\Ui$. In other words, the composition
$\C(q^{\hf}) \otimes_{\Qq} \Ui \stackrel{\imath}{\rightarrow} \C(q^{\hf}) \otimes_{\Qq} \imath(\Ui)  \stackrel{\kappa}{\rightarrow} 
\C(q^{\hf}) \otimes_{\Qq} U_q'(\mf{k})$ is an anti-isomorphism.
Hence $\imath: \Ui \rightarrow \U$ must be an embedding. 
\end{proof}

\begin{rem}
Note that the coproduct for $\U$ used in \cite{KP}
 follows Lusztig \cite{Lu94} and hence differs from the one used in this paper;
 this leads to somewhat different presentations of the quantum symmetric pairs. 
Our choices are determined by the application we have in mind:
 the  $(\Ui, \mc{H}_{B_m})$-duality  in Chapter~\ref{sec:HeckeB} and 
 the translation functors for category $\mc O$ in Part~2.  One crucial advantage of
 our presentation is the existence of a natural bar involution as given in Lemma~\ref{lem:3inv}(3). 
\end{rem}

Any $\U$-module $M$ can be naturally regarded as a $\Ui$-module via the embedding $\imath$. 

\begin{rem}    
    \label{rem:bars}
The bar involution on $\Ui$ and the bar involution on $\U$ are {\em not}  compatible through $\imath$, i.e., 
$\ov{\imath(u)} \neq \imath(\ov{u})$ for $u \in \Ui$ in general. For example, 
\begin{align*}
\imath(\ov{\be}_{\alpha_i}) 
 &= \imath(\be_{\alpha_i}) = E_{\alpha_i} +  K^{-1}_{\alpha_i} F_{\alpha_{-i}},
 \\ 
\ov{\imath(\be_{\alpha_i})}  
  &= \ov{E_{\alpha_i} +  F_{\alpha_{-i}} K^{-1}_{\alpha_i}}
     = E_{\alpha_i} + F_{\alpha_{-i}} K_{\alpha_i}.
\end{align*}
\end{rem}

Note that $E_{\alpha_i} (K^{-1}_{\alpha_i}F_{\alpha_{-i}}) = q^{2} (K^{-1}_{\alpha_i}F_{\alpha_{-i}}) E_{\alpha_i}$ 
for all $ 0 \neq i \in \I$. Using the quantum binomial formula \cite[1.3.5]{Lu94},  we have, for all $i \in \Ihf$, $a \in \N$,
\begin{align}
\label{eq:beZ}
\imath(\be^{(a)}_{\alpha_i}) &= \sum^{a}_{j=0}  q^{j(a-j)}F^{(j)}_{\alpha_{-i}}K^{-j}_{\alpha_i} E^{(a-j)}_{\alpha_i},
 \\
\label{eq:bffZ}
\imath(\bff^{(a)}_{\alpha_i}) &= \sum^{a}_{j=0}  q^{j(a-j)}F^{(j)}_{\alpha_{i}}K^{-j}_{\alpha_{-i}} E^{(a-j)}_{\alpha_{-i}}.
\end{align}

\begin{prop} \label{prop:coproduct}
The coproduct $\Delta : \U \rightarrow \U \otimes \U$ restricts via the embedding $\imath$ to a $\Qq$-algebra homomorphism 
$\Delta : \Ui \mapsto \Ui \otimes \U$ such that, for all $i \in \Ihf$, 
\begin{align*}
\Delta(\bk_{\alpha_i}) &= \bk_{\alpha_i} \otimes K_{\alpha_i} K^{-1}_{\alpha_{-i}},
 \\
\Delta({\be_{\alpha_i}}) &= 1 \otimes E_{\alpha_i} + \be_{\alpha_i} \otimes K^{-1}_{\alpha_i} 
+ \bk^{-1}_{\alpha_i} \otimes K^{-1}_{\alpha_i}F_{\alpha_{-i}},
 \\
\Delta (\bff_{\alpha_i}) &= \bk_{\alpha_i} \otimes F_{\alpha_i}K^{-1}_{\alpha_{-i}} + \bff_{\alpha_i} 
\otimes K^{-1}_{\alpha_{-i}} + 1 \otimes E_{{\alpha_{-i}}},
 \\
\Delta(\bt) &= \bt \otimes K^{-1}_{\alpha_0} + 1 \otimes q F_{\alpha_0}K^{-1}_{\alpha_0}+ 1 \otimes E_{\alpha_0}.
\end{align*}
Similarly, the counit $\epsilon$ of $\U$ induces a $\Qq$-algebra homomorphism $ \epsilon : \Ui \rightarrow \Qq$ such that 
$\epsilon(\be_{\alpha_i}) =\epsilon(\bff_{\alpha_i})=0$, $\epsilon(\bt) =1$, and $\epsilon(\bk_{\alpha_i}) =1$ for all $i \in \Ihf$.
\end{prop}
\begin{proof}
This follows from a direct computation.
\end{proof}

\begin{rem}
Propositions~\ref{prop:embedding}  and  \ref{prop:coproduct} imply that $\Ui$ (or rather
$\imath(\Ui)$) is a {\em (right) coideal subalgebra} of $\U$ in the sense of \cite{Le03}. 
There exists a $\Qq$-algebra embedding $\imath_{L} : \Ui \rightarrow \U$ which makes $\Ui$ (or rather $\imath_{L}(\Ui)$) 
a {\em left coideal subalgebra} of $\U$; that is, the coproduct $\Delta : \U \rightarrow \U \otimes \U$ restricts 
via $\imath_L$ to a $\Qq$-algebra homomorphism $\Delta: \Ui \rightarrow \U \otimes \Ui$. 
We will not use the left variant in this paper. 
\end{rem}

\begin{rem}
The pair $( \U, \Ui)$ forms a quantum symmetric pair in the sense of \cite{Le03}. At the limit  $q \mapsto 1$,
it reduces to  a classical symmetric pair $(\mf{sl}(2r+2), \mf{sl}(2r+2)^{w_0})$; here $w_0$ is 
the involution on $\mf{gl}(2r+2)$ which sends $E_{i,j}$ to $E_{-i,-j}$
 and its restriction to $\mf{sl}(2r+2)$ if we label the rows and columns
of $\mf{sl}(2r+2)$ by $\{-r-1/2, \ldots, -1/2,1/2,\ldots, r+1/2\}$. 
\end{rem}

The following corollary follows immediately from the Hopf algebra structure of $\U$.

\begin{cor} \label{cor:counit}
Let $m : \U \otimes \U \rightarrow \U$ denote the multiplication map. Then we have 
\[
m (\epsilon \otimes 1)\Delta = \imath : \Ui \longrightarrow \U.
\]
\end{cor}

The map $\Delta : \Ui \mapsto \Ui \otimes \U$ is clearly coassociative, i.e., we have
$(1 \otimes \Delta) \Delta = (\Delta \otimes 1)\Delta: \Ui \longrightarrow \Ui \otimes \U \otimes \U.
$
This $\Delta$ 
will be called the {\em coproduct} of $\Ui$, and 
$ \epsilon : \Ui \rightarrow \Qq$ will be called the {\em counit} of $\Ui$.  
The counit map $\epsilon$ makes $\Qq$ a $\Ui$-module. 
We shall call this {\em the trivial representation} of $\Ui$. 

\begin{rem}
The $1$-dimensional space $\Qq$ can be realized as $\Ui$-modules in different (non-isomorphic) ways. 
For example, we can consider the $\Qq$-algebras homomorphism $\epsilon' : \Ui \rightarrow \Qq$, such that 
$\epsilon'(\be_{\alpha_i}) =\epsilon'(\bff_{\alpha_i})=0$, $\epsilon'(\bk_{\alpha_i}) =1$ for all $i \in {\Z_{> 0}}$, 
and $\epsilon'(\bt) =x$  for any $x \in \Qq$. 
We shall only consider the one induced by $\epsilon$ as the trivial representation of $\Ui$, 
which is compatible with the trivial representation of $\U$ via $\imath$. 
\end{rem}

\section{The intertwiner $\Upsilon$} 
   \label{subsec:Upsilon}

Let $\widehat{\U}$ be the completion of the $\Qq$-vector space $\U$ 
with respect to the following descending sequence of subspaces 
$\U^+ \U^0 \big(\sum_{\hgt(\mu) \geq N}\U_{-\mu}^- \big)$,   for $N \ge 1.$
Then we have the obvious embedding of $\U$ into $\widehat{\U}$. 
We let $\widehat{\U}^-$ be the closure of $\U^-$ in $\widehat{\U}$, and so $\widehat{\U}^- \subseteq \widehat{\U}$. 
By continuity the $\Q(q)$-algebra structure 
on $\U$ extends to a $\Q(q)$-algebra structure on $ \widehat{\U}$.  The bar involution \,$\bar{\ }$\, on $\U$ extends 
by continuity to an anti-linear involution on $\widehat{\U}$, also denoted by \,$\bar{\ }$. 
Recall the bar involutions on $\Ui$ and $\U$ are not compatible via the embedding $\imath: \Ui \rightarrow \U$, by
Remark~\ref{rem:bars}.

\begin{thm}   \label{thm:Upsilon}
There is a unique family of elements $\Upsilon_\mu \in \U_{-\mu}^-$ for $\mu \in {\N}{\Pi}$ such that 
$\Upsilon = \sum_{\mu}\Upsilon_\mu \in \widehat{\U}^-$ 
intertwines the bar involutions on $\Ui$ and $\U$ via the embedding $\imath$
and $\Upsilon_0 = 1$;  
that is, $\Upsilon$ satisfies the following identity (in $\widehat{\U}$): 
\begin{equation}\label{eq:star}
 \imath(\ov{u}) \Upsilon = \Upsilon\  \overline{\imath(u)}, \quad \text{ for all } u \in \Ui.
\end{equation}
Moreover, $\Upsilon_\mu = 0$ unless $\mu^{\inv} = \mu$.
\end{thm}

\begin{rem}  \label{rem:intertw}
Define 
$\ov\imath: \Ui \rightarrow \U$, where $\ov\imath (u):=\Abar \big( \imath (\Bbar (u))\big)$, for $u\in \Ui$.
Then the identity \eqref{eq:star} can be equivalently reformulated as
\begin{equation}\label{eq:intertw}
 \imath(u) \Upsilon = \Upsilon\  \overline{\imath}(u), \quad \text{ for all } u \in \Ui.
\end{equation}
This reformulation makes it more transparent to observe the remarkable analogy with Lusztig's $\ThetaA$; see \eqref{eq:propThetaA}. 
\end{rem}

Sometimes it could be confusing to use \,$\bar{\ }$\,  to denote the two distinct bar involutions on $\U$ and $\Ui$.
Recall that we set in Chapter~ \ref{subsec:f} that $\Abar({u}) = \ov{u}$ for all $u\in \U$, and set in Chapter~\ref{sec:bun}
that $\Bbar(u)  = \ov{u} \in \Ui$ for $u \in \Ui$. 
In the $\psi$-notation the identities \eqref{eq:star} and \eqref{eq:intertw} read
$$
\imath(\Bbar(u)) \Upsilon = \Upsilon  \Abar (\imath(u)), 
\quad
\imath(u) \Upsilon = \Upsilon  \Abar \big(\imath(\Bbar(u)) \big), 
\quad \text{ for all } u \in \Ui.
$$

\begin{definition}  \label{def:Upsilon}
The element $\Upsilon$ in Theorem~\ref{thm:Upsilon} is called
{\em the intertwiner for the quantum symmetric pair $(\U, \Ui)$}.
\end{definition}

As we shall see, the intertwiner $\Upsilon$ leads to the construction of
what we call quasi-$\mc{R}$-matrix for $\Ui$,
which plays an analogous role as Lusztig's quasi-$\mc{R}$-matrix  for $\U$. 
We shall prove later on that $\Upsilon_\mu \in \U^-_\mA$ for all $\mu$; see Theorem~\ref{thm:UpsiloninZ}.

The proof of Theorem \ref{thm:Upsilon} will be given in \S\ref{sec:proofUpsilon} below.
Here we note immediately a fundamental  property of $\Upsilon$. 

\begin{cor}  \label{cor:Upsiloninv}
We have $\Upsilon \cdot \ov{\Upsilon} =1.$
\end{cor}

\begin{proof}
Clearly $\Upsilon$ is invertible in $\widehat{\U}$. 
Multiplying $\Upsilon^{-1}$ on both sides of  the identity~ \eqref{eq:star} in Theorem \ref{thm:Upsilon}, we have
\[
\Upsilon^{-1} \imath(\ov{u}) =\overline{\imath(u)} \Upsilon^{-1}, \qquad \forall u \in \Ui.
\]
Applying \,$\bar{\ }$\, to the above identity and replacing $\ov{u}$ by $u$, we have
\[
\ov{\Upsilon}^{-1} \ov{\imath({u})} ={\imath(\ov{u})} \ov{\Upsilon}^{-1}, \qquad \forall u \in \Ui.
\]
Hence $\ov{\Upsilon}^{-1}$ (in place of $\Upsilon$) satisfies the identity~ \eqref{eq:star} 
as well.
Thanks to the uniqueness of $\Upsilon$ in Theorem \ref{thm:Upsilon}, 
we must have $\ov{\Upsilon}^{-1} = \Upsilon$, whence the corollary. 
\end{proof}

\section{Constructing  $\Upsilon$}
   \label{sec:proofUpsilon}
  
The goal here is to construct $\Upsilon$ and establish Theorem~\ref{thm:Upsilon}.

The set of all $u \in \Ui$ that satisfy the identity \eqref{eq:star} is clearly a subalgebra of $\Ui$. 
Hence it suffices to consider the identity \eqref{eq:star} when $u$ is one of the generators 
$\be_{\alpha_i}$, $\bff_{\alpha_i}$, $\bk_{\alpha_i}$, and $\bt$ in $\Ui$, that is, 
the following identities for all $\mu \in\N\Pi$ and $0 \neq i \in \I$:
\begin{align*}
K_{\alpha_i}K^{-1}_{\alpha_{-i}} \Upsilon_\mu &= \Upsilon_\mu K_{\alpha_i}K^{-1}_{\alpha_{-i}}\displaybreak[0],
 \\
F_{\alpha_i} K^{-1}_{\alpha_{-i}} \Upsilon_{\mu- \alpha_i-\alpha_{-i}} + E_{\alpha_{-i}}\Upsilon_\mu
 &=\Upsilon_{\mu-\alpha_i-\alpha_{-i}} F_{\alpha_i}K_{\alpha_{-i}} + \Upsilon_\mu E_{\alpha_{-i}}  \displaybreak[0],
  \\
qF_{\alpha_0}K^{-1}_{\alpha_0}\Upsilon_{\mu-2\alpha_0} + K^{-1}_{\alpha_0}\Upsilon_{\mu - \alpha_0}
 & + E_{\alpha_0}\Upsilon_\mu 
  \\
  &=q^{-1}\Upsilon_{\mu -2\alpha_0}F_{\alpha_0}K_{\alpha_0} 
  + \Upsilon_{\mu-\alpha_0}K_{\alpha_0} + \Upsilon_\mu E_{\alpha_0 }.
\end{align*}

Using \cite[Proposition 3.1.6]{Lu94}, we can rewrite the above identities in terms of $_{-i}r$ and $r_{-i}$ as follows:
\begin{align}
K_{\alpha_i}K^{-1}_{\alpha_{-i}} \Upsilon_\mu - \Upsilon_\mu K_{\alpha_i}K^{-1}_{\alpha_{-i}}
  &=0,\displaybreak[0] \label{aux:UpsilonK}
 \\
(q^{-1} -q) q^{(\alpha_{-i} , \mu-\alpha_{-i}-\alpha_{i})} \Upsilon_{\mu-\alpha_i-\alpha_{-i}}F_{\alpha_i}
  + \,_{-i}r(\Upsilon_\mu) &= 0,\displaybreak[0]\label{aux:Upuniq1}
 \\
(q^{-1} -q) q^{(\alpha_{-i} , \mu-\alpha_{-i}-\alpha_{i})} F_{\alpha_i}\Upsilon_{\mu-\alpha_i-\alpha_{-i}} 
  + \, r_{-i}(\Upsilon_\mu) &=0,\displaybreak[0]
  \\
(q^{-1}-q) q^{(\alpha_0 , \mu - \alpha_0)} (q^{-1}\Upsilon_{\mu-2\alpha_0}F_{\alpha_0}
  +\Upsilon_{\mu - \alpha_0}) + \,_{0}r(\Upsilon_\mu) &= 0,\displaybreak[0]
  \\
(q^{-1}-q) q^{(\alpha_0 , \mu - \alpha_0)} (q^{-1}F_{\alpha_0}\Upsilon_{\mu - 2\alpha_0}
 +\Upsilon_{\mu - \alpha_0} )+ r_{0} (\Upsilon_\mu) &= 0.\displaybreak[0]\label{aux:Upuniq4}
\end{align}

Recall the non-degenerate bilinear form $(\cdot, \cdot)$ on $\U^-$ in Section~ \ref{subsec:f};
see Proposition~ \ref{prop:Luform}. The identities \eqref{aux:Upuniq1}-\eqref{aux:Upuniq4} can be shown to  be equivalent to the
following identities \eqref{aux:T1}-\eqref{aux:T4}:
\begin{align}
(\Upsilon_\mu, F_{\alpha_{-i}}z) 
 &=\qqq q^{(\alpha_{-i} , \mu-\alpha_{-i}-\alpha_{i})+1} (\Upsilon_{\mu - \alpha_i - \alpha_{-i}},r_{i}(z)),
    \label{aux:T1}\\
(\Upsilon_\mu,z F_{\alpha_{-i}})  
  &=\qqq q^{(\alpha_{-i} , \mu-\alpha_{-i}-\alpha_{i})+1}  (\Upsilon_{\mu - \alpha_i - \alpha_{-i}}, \,_{i}r(z)),
    \\
(\Upsilon_\mu, F_{\alpha_0}z) &= \qqq q^{(\alpha_0 , \mu-\alpha_0)}
   (\Upsilon_{\mu - 2\alpha_0}, r_{0}(z)) 
  \notag  \\
   &\qquad\qquad\qquad\qquad + q^{(\alpha_0 , \mu - \alpha_0)+1}(\Upsilon_{\mu - \alpha_0}, z),
     \\
(\Upsilon_\mu, zF_{\alpha_0}) &= \qqq q^{(\alpha_0 , \mu - \alpha_0 ) } 
   (\Upsilon_{\mu - 2\alpha_0}, \,_{0}r(z)) 
  \notag \\
   &\qquad\qquad\qquad \qquad + q^{(\alpha_0 , \mu - \alpha_0)+1}(\Upsilon_{\mu - \alpha_0}, z)
   \label{aux:T4},
\end{align}
for all $z \in \U^-_{-\nu}$, $\nu \in \N \Pi$, $\mu \in\N\Pi$, and  $0 \neq i \in \I$.
For example, the equivalence between \eqref{aux:T1} and \eqref{aux:Upuniq1} is shown as follows:
\begin{align*}
\eqref{aux:Upuniq1}
&\Leftrightarrow ( \,_{-i}r(\Upsilon_\mu) , z)= -(q^{-1} -q) 
 q^{(\alpha_{-i} , \mu-\alpha_{-i}-\alpha_{i})}(\Upsilon_{\mu-\alpha_i-\alpha_{-i}}F_{\alpha_i}, z)  \quad \forall z,\\
& \Leftrightarrow (F_{\alpha_{-i}}, F_{\alpha_{-i}})^{-1} (\Upsilon_\mu, F_{\alpha_{-i}}z) \\
&\qquad = -(q^{-1} -q) q^{(\alpha_{-i} , \mu-\alpha_{-i}-\alpha_{i})} (F_{\alpha_i}, F_{\alpha_i}) 
  (\Upsilon_{\mu - \alpha_i - \alpha_{-i}},r_{i}(z)) \quad \forall z, \\
&\Leftrightarrow \eqref{aux:T1}  \quad \forall z.
\end{align*}
The remaining cases are similar. 

Summarizing, we have established the following.

\begin{lem}  \label{lem:identityeq}
\begin{enumerate}
\item
The validity of the identity \eqref{eq:star} is equivalent to the validity of the identities 
\eqref{aux:UpsilonK} and \eqref{aux:Upuniq1}-\eqref{aux:Upuniq4}.

\item
The validity of the identity \eqref{eq:star} is equivalent to the validity of the identities 
\eqref{aux:UpsilonK} and \eqref{aux:T1}-\eqref{aux:T4}.
\end{enumerate}
\end{lem}

Let $\fprime^*$ (respectively, $(\U^-)^*$) be the non-restricted dual of $\fprime$ (respectively, of $\U^-$). 
In light of Lemma~\ref{lem:identityeq}(2), we define $\Upsilon^*_L$ and $\Upsilon^*_R$ in $\fprime^*$, 
inductively on weights, by the following formulas:
\begin{align}
\Upsilon^*_L(1) &= \Upsilon^*_R(1) =1,
 \notag \\
\Upsilon^*_L(F_{\alpha_{-i}}z)&=\qqq q^{( \alpha_{-i} , \nu-\alpha_i)+1}\Upsilon^*_L(r_{i}(z)),
  \notag \\
\Upsilon_L^*( F_{\alpha_0}z) &= \qqq q^{(\alpha_0 , \nu )} \Upsilon^*( r_{0}(z)) + q^{( \alpha_0, \nu)+1}\Upsilon^*(z),
  \label{eq:LRgeneral} \\
\Upsilon^*_R(zF_{\alpha_{-i}})&=\qqq q^{(\alpha_{-i} , \nu - \alpha_i)+1}\Upsilon^*_L(_{i}r(z)),
 \notag \\
\Upsilon_R^*(zF_{\alpha_0}) &= \qqq  q^{( \alpha_0 , \nu)} \Upsilon^*( _{0}r(z)) + q^{( \alpha_0, \nu)+1}\Upsilon^*( z),
 \notag
\end{align}
for all $i\in \I$ and $z \in \f_{\nu}$ with $\nu \in \N{\Pi}$. 
(The formulas \eqref{eq:LRgeneral} are presented here only for the sake of latter reference 
as they also make sense in the case of $\Uj$.)

Note that since $(\alpha_i, \alpha_{-i}) = 0$ for all $i \neq 0$, we can simplify the definition 
\eqref{eq:LRgeneral} of $\Upsilon^*_L$ and $\Upsilon^*_R$ as follows:
\begin{align}
\Upsilon^*_L(1) &= \Upsilon^*_R(1) =1,
 \notag \\
\Upsilon^*_L(F_{\alpha_{-i}}z)&=\qqq q^{( \alpha_{-i} , \nu)+1}\Upsilon^*_L(r_{i}(z)),
  \notag \\
\Upsilon_L^*( F_{\alpha_0}z) &= \qqq q^{(\alpha_0 , \nu )} \Upsilon^*( r_{0}(z)) + q^{( \alpha_0, \nu)+1}\Upsilon^*(z),
  \label{eq:LR} \\
\Upsilon^*_R(zF_{\alpha_{-i}})&=\qqq q^{(\alpha_{-i} , \nu)+1}\Upsilon^*_L(_{i}r(z)),
 \notag \\
\Upsilon_R^*(zF_{\alpha_0}) &= \qqq  q^{( \alpha_0 , \nu)} \Upsilon^*( _{0}r(z)) + q^{( \alpha_0, \nu)+1}\Upsilon^*( z),
 \notag
\end{align}
for all $i\in \I$ and $z \in \f_{\nu}$ with $\nu \in \N{\Pi}$.   
%

\begin{lem}   \label{lem:Up0}
For all $x \in \fprime_\mu$ with $\mu^{\inv} \neq \mu$,  we have 
\[
\Upsilon^*_L(x) =\Upsilon^*_R (x)= 0.
\]
\end{lem}

\begin{proof}
We will only prove that 
$\Upsilon^*_L (x) = 0$ for all $x \in \fprime_\mu$ with $\mu^{\inv} \neq \mu$, as the proof 
for the identity $\Upsilon^*_R (x)= 0$ is the same.
By definition of $\Upsilon^*_L$ \eqref{eq:LR},  the value of $\Upsilon^*_L (x)$ for $x \in \fprime_\mu$
is equal to (up to some scalar multiple) $\Upsilon^*_L (x')$ for some $x' \in \fprime_{\mu'}$,
where $\mu' =\mu -\alpha_i -\alpha_{-i}$ for some $i$; here we recall $\theta(\alpha_i) =\alpha_{-i}$. 
Also by definition \eqref{eq:LR}, we have 
$\Upsilon^*_L(F_{\alpha_{i}})=0$ for all $i\in \I$. Now the claim follows by an  induction on weights. 
\end{proof}

\begin{lem}\label{lem:claim1}
We have $\Upsilon^*_L=\Upsilon^*_R$.
\end{lem}

\begin{proof}
We shall prove the identity $\Upsilon^*_L(x)=\Upsilon^*_R(x)$  for all homogeneous elements $x \in \fprime$, 
by  induction on ${\rm ht}(|x|)$. 

When ${\rm ht}(|x|) =0 \text{ or }1$, this is trivial by definition. 
Assume the identity holds for all $x$ with ${\rm ht}(|x|) \leq k$, for $k \geq 1$. 
Let  $x' = F_{\alpha_{-i}}x''F_{\alpha_{-j}} \in \fprime_{\nu + \alpha_{-i} + \alpha_{-j}}$ with ${\rm ht}(|x'|) = k+1 \geq 2$. 
We can further assume that $\inv(\nu + \alpha_{-i}+\alpha_{-j}) =\nu+\alpha_{-i}+\alpha_{-j}$, 
since otherwise $\Upsilon^*_L(x') = \Upsilon^*_R(x') = 0$ by Lemma~\ref{lem:Up0}. 
The proof  is divided into four cases (1)-(4).  

(1) Assume that $i, j \neq 0$. Then we have 
\begin{equation*}
\Upsilon^*_L(x') = \qqq q^{( \alpha_{-i} , \nu + \alpha_{-j})+1} \Upsilon^*_L(r_{i}(x''F_{\alpha_{-j}}))
=L_1 + L_2,
\end{equation*}
where
\begin{align*}
L_1 &=\qqq  q^{(\alpha_{-i}, \nu + \alpha_{-j})+ (\alpha_{i} , \alpha_{-j})+1} \Upsilon^*_L(r_{i}(x'')F_{\alpha_{-j}}),
  \\
L_2 &=\qqq  q^{(\alpha_{-i}, \nu + \alpha_{-j})+1} \delta_{i,-j}\Upsilon^*_L (x'').
\end{align*}
We also have 
\begin{equation*}
\Upsilon^*_R(x') = \qqq q^{(\alpha_{-j}, \nu + \alpha_{-i})+1}\Upsilon^*_R(\,_{j}r(F_{\alpha_{-i}}x''))\\
=R_1 +R_2,
\end{equation*}
where
\begin{align*}
R_1 &= \qqq q^{(\alpha_{-j}, \nu + \alpha_{-i})+ (\alpha_{j}, \alpha_{-i})+1} \Upsilon^*_R(F_{\alpha_{-i}}~_{j}r(x'')),
  \\
R_2 &=\qqq q^{( \alpha_{-j}, \nu + \alpha_{-i})+1} \delta_{i,-j}\Upsilon^*_R (x'').
\end{align*}

Applying the induction hypothesis to $r_{i}(x'')F_{\alpha_{-j}}$ and $F_{\alpha_{-i}}~_{j}r(x'')$ gives us
\begin{align*}
L_1
&= (1 -q^{-2})^{-2} q^{(\alpha_{-i}, \nu + \alpha_{-j})+ (\alpha_{i}, \alpha_{-j})+ (\alpha_{-j}, \nu-\alpha_i)+2}  \Upsilon^*_L(\,_{j}r(r_{i}(x'')))
 \\
&= (1 -q^{-2})^{-2} q^{(\alpha_{-i}, \nu) + (\alpha_{-j}, \nu)+(\alpha_{-i}, \alpha_{-j}) +2}\Upsilon^*_L(_{j}r(r_{i}(x'')));
 \\
R_1
&=   (1 -q^{-2})^{-2} q^{(\alpha_{-j}, \nu + \alpha_{-i})+ (\alpha_{j},\alpha_{-i})+ (\alpha_{-i}, \nu-\alpha_j)+2}  \Upsilon^*_R( r_{i}(_{j}r(x'')) )
 \\
&= (1 -q^{-2})^{-2} q^{(\alpha_{-i}, \nu) + (\alpha_{-j}, \nu)+(\alpha_{-j}, \alpha_{-i})+2}\Upsilon^*_R(r_{i}(_{j}r(x''))).
\end{align*}
Note that $_{j}r(r_{i}(x'')) =r_{i}(_{j}r(x''))$ by Lemma \ref{lem:rlrr=rrrl} and ${\rm ht}(|{}_{j}r(r_{i}(x''))|) < {\rm ht}(|x'|)$. By
the induction hypothesis, $\Upsilon^*_L(_{j}r(r_{i}(x''))) =\Upsilon^*_R(r_{i}(_{j}r(x'')))$.  Hence $L_1=R_1$. 

By the induction hypothesis, we also have $\Upsilon^*_L (x'')= \Upsilon^*_R(x'')$.  When $i =-j$, we have $\nu^{\inv} =\nu$, 
and hence
\begin{align*}
\qqq q^{(\alpha_{-i}, \nu + \alpha_{-j})+1}
&=\qqq q^{(\alpha_{-i}, \nu + \alpha_{i})+1}\\
&= \qqq q^{(\alpha^{\inv}_{-i} , \nu^{\inv} + \alpha^{\inv}_{i})+1}\\
&=\qqq q^{(\alpha_{i}, \nu+ \alpha_{-i})+1}\\
&= \qqq q^{(\alpha_{-j}, \nu+\alpha_{-i})+1}.
\end{align*}
Hence we have $L_2 =R_2$. 

Summarizing, we have  $\Upsilon^*_L(x') =L_1 +L_2 =R_1 +R_2 = \Upsilon^*_R(x')$ in this case.

(2) 
Assume that $i=0$ and  $j\neq 0$. Then we have
\begin{align*}
& \Upsilon^*_L(x') \\
&=\qqq q^{(\alpha_{0}, \nu + \alpha_{-j})} \Upsilon^*_L(r_{0}(x''F_{\alpha_{-j}})) 
+ q^{(\alpha_{0}, \nu + \alpha_{-j})+1}\Upsilon^*_L(x''F_{\alpha_{-j}})
  \\
&= \qqq q^{(\alpha_{0}, \nu + \alpha_{-j})} \Upsilon^*_L(q^{(\alpha_{0} ,\alpha_{-j})}r_{0}(x'')F_{\alpha_{-j}})
+ q^{( \alpha_{0}, \nu + \alpha_{-j})+1}\Upsilon^*_L(x''F_{\alpha_{-j}})
 \\
&=\qqq q^{(\alpha_{0}, \nu + \alpha_{-j}) + (\alpha_{0}, \alpha_{-j})} 
 \Upsilon^*_L(r_{0}(x'')F_{\alpha_{-j}}) + q^{(\alpha_{0}, \nu + \alpha_{-j})+1}\Upsilon^*_L(x''F_{\alpha_{-j}})
\end{align*}
Applying the induction hypothesis to $r_{0}(x'')F_{\alpha_{-j}}$ and $x''F_{\alpha_{-j}}$, we have
\begin{align*}
\Upsilon^*_L(r_{0}(x'')F_{\alpha_{-j}}) &=\qqq q^{(\alpha_{-j}, \nu - \alpha_0)+1}\Upsilon^*_L(_{j}r(r_{0}(x'')),
  \\
\Upsilon^*_L(x''F_{\alpha_{-j}}) &= \qqq q^{(\alpha_{-j}, \nu)+1}\Upsilon^*_L(_{j}r(x'')).
\end{align*}

Hence we obtain
\begin{align*}
\Upsilon^*_L(x') =&(1-  q^{-2})^{-2} q^{(\alpha_{0}, \nu ) +  (\alpha_{-j}, \nu) + (\alpha_{-j},\alpha_{0}) + 1}\Upsilon^*_L(_{j}r(r_{0}(x''))
  \\
&+\qqq q^{(\alpha_0, \nu) + ( \alpha_{-j}, \nu) + (\alpha_{-j}, \alpha_{0})+2}\Upsilon^*_L(_{j}r(x'')).
\end{align*}
From a similar computation we obtain
\begin{align*}
\Upsilon^*_R(x') =&(1-  q^{-2})^{-2} q^{( \alpha_0, \nu) + (\alpha_{-j}, \nu)+(\alpha_{0},\alpha_{-j}) +1}\Upsilon^*_R(r_{0}(_{j}r(x'')) \\
&+\qqq q^{(\alpha_0, \nu) + ( \alpha_{-j}, \nu)+(\alpha_{0},\alpha_{-j})+2}\Upsilon^*_R(_{j}r(x'')).
\end{align*}
It follows by Lemma~ \ref{lem:rlrr=rrrl} that $r_{0}(_{j}r(x'')) =\,_{j}r(r_{0}(x'')$.
Then, by the induction hypothesis on $r_{0}(_{j}r(x''))$, ${}_{j}r(r_{0}(x'')$, and $_{j}r(x'')$, 
we obtain $\Upsilon^*_L(x') = \Upsilon^*_R(x')$ in this case.

(3) 
Similar computation works for the case where $j=0, i\neq 0$ as in Case (2).

(4)
 At last, consider the case where $i=j=0$. 
\begin{align*}
\Upsilon^*_L(x') &= \qqq q^{(\alpha_{0}, \nu +\alpha_{0})} 
\Upsilon^*_L(r_{0}(x''F_{\alpha_{0}}))+ q^{(\alpha_{0}, \nu + \alpha_{0})+1}\Upsilon^*_L(x''F_{\alpha_{0}})
  \\ 
&=\qqq q^{(\alpha_{0}, \nu +\alpha_{0}) + (\alpha_0, \alpha_0)} \Upsilon^*_L(r_{0}(x'')F_{\alpha_{0}})
  \\
&\text{ } \quad \qqq q^{(\alpha_{0}, \nu +\alpha_{0})}  \Upsilon^*_L (x'') 
  + q^{( \alpha_{0}, \nu +\alpha_{0})+1}  \Upsilon^*_L(x''F_{\alpha_{0}}).
\end{align*}
Applying the induction hypothesis to $r_{0}(x'')F_{\alpha_0}$ and $x''F_{\alpha_0}$, we have
\begin{align*}
\Upsilon^*_L(r_{0}(x'')F_{\alpha_0}) &= \qqq q^{(\alpha_0, \nu -\alpha_0)}\Upsilon^*_L(_{0}r(r_{0}(x''))) 
+ q^{( \alpha_0, \nu-\alpha_0)+1} \Upsilon^*_L(r_{0}(x'')),
  \\
\Upsilon^*_L(x''F_{\alpha_0}) &= \qqq q^{(\alpha_0 , \nu)}\Upsilon^*_L(_{0}r(x'')) + q^{(\alpha_0, \nu)+1}\Upsilon^*_L(x'').
\end{align*}
Hence we have 
\begin{align*}
 \Upsilon^*_L(x') 
  &=(1-q^{-2})^{-2} q^{(\alpha_0, \nu) + (\alpha_0, \nu)+(\alpha_0,\alpha_0) }\Upsilon^*_L(_{0}r(r_{0}(x'')))
    \\
&\quad+\qqq q^{(\alpha_0, \nu) + (\alpha_0, \nu)+(\alpha_0,\alpha_0) +1}
  \Upsilon^*_L(r_{0}(x''))     \\
&\quad+ \qqq q^{(\alpha_{0}, \nu +\alpha_{0})}  \Upsilon^*_L (x'')
    \\
&\quad +\qqq q^{(\alpha_0, \nu) + (\alpha_0, \nu)+(\alpha_0,\alpha_0)+1}
  \Upsilon^*_L(_{0}r(x''))     \\
&\quad+ q^{(\alpha_0, \nu) + (\alpha_0, \nu)+(\alpha_0,\alpha_0)+2}\Upsilon^*_L(x'').
\end{align*}
Similarly we have
\begin{align*}
 \Upsilon^*_R(x') 
 &=(1-q^{-2})^{-2} q^{(\alpha_0, \nu) + (\alpha_0, \nu)+(\alpha_0,\alpha_0)}\Upsilon^*_R(r_{0}(_{0}r(x'')))
   \\
&\quad+\qqq q^{(\alpha_0, \nu) + (\alpha_0, \nu)+(\alpha_0,\alpha_0) +1}\Upsilon^*_R(_{0}r(x'')) 
      \\
&\quad+ \qqq q^{(\alpha_{0}, \nu +\alpha_{0})} \Upsilon^*_R (x'') 
    \\
&\quad +\qqq q^{(\alpha_0, \nu) + (\alpha_0, \nu)+(\alpha_0,\alpha_0)+1} \Upsilon^*_R(r_{0}(x''))
      \\
&\quad+  q^{(\alpha_0, \nu) + (\alpha_0, \nu)+(\alpha_0,\alpha_0)+2}\Upsilon^*_R(x'').
\end{align*}
Therefore $\Upsilon^*_L(x') = \Upsilon^*_R(x')$ in this case too by induction and by Lemma~ \ref{lem:rlrr=rrrl}.

This completes the proof of Lemma \ref{lem:claim1}.
\end{proof}

We shall simply denote $\Upsilon^*_L =\Upsilon^*_R$ by $\Upsilon^*$ thanks to Lemma~\ref{lem:claim1}. 
Recall $\fprime/{\bf I} = \U^-$, where ${\bf I} = \langle S_{ij} \rangle.$

\begin{lem}\label{lem:claim2}
We have $\Upsilon^*({\bf I}) =0$; hence we may regard $\Upsilon^* \in (\U^-)^*$.
\end{lem}

\begin{proof}
Recall $r_{k}(S_{ij}) = \,_{k}r(S_{ij}) =0$, for all $i$, $j$, $k$. 
Any element in ${\bf I}$ is a $\Q(q)$-linear combination of elements of the form
$$
F_{\alpha_{m_1}} \dots F_{\alpha_{m_h}} S_{ij} F_{\alpha_{n_1}} \dots F_{\alpha_{n_l}}.
$$
So it suffices to prove $\Upsilon^*(F_{\alpha_{m_1}} \dots F_{\alpha_{m_h}} S_{ij} F_{\alpha_{n_1}} \dots F_{\alpha_{n_l}})=0$,
by induction on $h+l.$

Recall the Serre relator $S_{ij}$, for $i\neq j\in \I$, from \eqref{eq:Sij}. 
Let us verify that $\Upsilon^*(S_{ij}) =0$, which is the base case of the induction.
If $|i-j|=1$, the weight of $S_{ij}$ is  $-2\alpha_i-\alpha_j$, which is not $\theta$-invariant.
If $|i-j|>1$, the weight of $S_{ij}$ is  $-\alpha_i-\alpha_j$, which is not $\theta$-invariant unless $i=-j$.
In case of $i=-j$, a quick computation by definition \eqref{eq:LR} gives us that $\Upsilon^*(S_{ij}) =0$.
In the remaining cases, it follows by Lemma~\ref{lem:Up0} that $\Upsilon^*(S_{ij}) =0$.

If $h > 0$, by \eqref{eq:LR}, \eqref{eq:rr} and \eqref{eq:r=0} we have
\begin{align*}
\Upsilon^*(F_{\alpha_{m_1}} & \dots F_{\alpha_{m_h}} S_{ij} F_{\alpha_{n_1}} \dots F_{\alpha_{n_l}}) 
  \\
= &\Upsilon^*(r_{-m_1}(F_{\alpha_{m_2}} \dots F_{\alpha_{m_h}} S_{ij} F_{\alpha_{n_1}} \dots F_{\alpha_{n_l}}))
 \\
= &\Upsilon^* \Big(\sum c_{m'n'}F_{\alpha_{m'_1}} \dots F_{\alpha_{m'_{h'}}} S_{ij} F_{\alpha_{n'_1}} \dots F_{\alpha_{n'_{l'}}} \Big)
 \\
&+ \delta_{-m_1,0}c' \Upsilon^*(F_{\alpha_{m_2}} \dots F_{\alpha_{m_h}} S_{ij} F_{\alpha_{n_1}} \dots F_{\alpha_{n_l}}),
\end{align*}
for some scalars $c_{m'n'}$ and $c'$. 
Similarly if $l> 0$, we have 
\begin{align*}
\Upsilon^*(F_{\alpha_{m_1}} & \dots F_{\alpha_{m_h}} S_{ij} F_{\alpha_{n_1}} \dots F_{\alpha_{n_l}})
  \\
= &\Upsilon^*(_{-n_l}r(F_{\alpha_{m_1}} \dots F_{\alpha_{m_h}} S_{ij} F_{\alpha_{n_1}} \dots F_{\alpha_{n_{l-1}}}))
  \\
= &\Upsilon^* \Big(\sum c_{m''n''}F_{\alpha_{m''_1}} \dots F_{\alpha_{m''_{h''}}} S_{ij} F_{\alpha_{n''_1}} \dots F_{\alpha_{n''_{l''}}} \Big) 
  \\
&+ \delta_{-n_l,0}c''\Upsilon^*(F_{\alpha_{m_1}} \dots F_{\alpha_{m_h}} S_{ij} F_{\alpha_{n_1}} \dots F_{\alpha_{n_{l-1}}}).
\end{align*}
for some scalars $c_{m''n''}$ and $c''$. 
In either case, we have $h'+l' = h''+l'' < h+l$. Therefore by induction on $h+l$, Lemma~ \ref{lem:claim2} is proved.
\end{proof}

Now we are ready to prove Theorem \ref{thm:Upsilon}.

\begin{proof}[Proof of Theorem \ref{thm:Upsilon}]
We first prove the existence of $\Upsilon$ satisfying the identity \eqref{eq:star}. Set $\Upsilon_{\mu} = 0$ if $\mu \not \in \N\Pi$. 
Let $B =\{b\}$ be a basis of $\U^-$ such that $B_\mu = B \cap \U^-_{-\mu}$ is a basis of $\U^-_{-\mu}$. 
Let $B^* = \{b^*\}$ be the dual basis of $B$ with respect to the bilinear pairing $(\cdot,\cdot)$ in Section~ \ref{subsec:f}.
Define $\Upsilon$ by 
\[
\Upsilon := \sum_{b \in B} \Upsilon^*(b^*)b = \sum_{\mu}\Upsilon_\mu.
\]
As functions on $\U^-$, $(\Upsilon, \cdot) = \Upsilon^*$.
Clearly $\Upsilon$ is in $\widehat{\U}^-$ and $\Upsilon_0 =1$. Also $\Upsilon$ satisfies the identities 
in \eqref{aux:T1}-\eqref{aux:T4} by the definition of $\Upsilon^*$.
For any $x \in \U^{-}_{\nu}$, it follows by Lemma~\ref{lem:Up0} that $\Upsilon^*_L(x) =\Upsilon^*_R(x) =0$ if $\nu^{\inv} \neq \nu$. 
It follows that \eqref{aux:UpsilonK} is satisfied. Therefore, 
by Lemma~\ref{lem:identityeq}(2), $\Upsilon$ satisfies the desired identity \eqref{eq:star} in the theorem.

By Lemma~\ref{lem:identityeq}(1) and the definition of $\Upsilon$, 
the identity \eqref{aux:Upuniq1} holds for $\Upsilon$, and so
${}_{-i} r (\Upsilon_\mu)$ is determined by $\Upsilon_\nu$ with weight $\nu\prec \mu$.
By \cite[Lemma 1.2.15]{Lu94}, if an element $x \in \U^{-}_{-\nu}$ with $\nu \neq 0$ satisfies 
${}_{-i} r(x) = 0$ for all $i \in \I$ then $x =0$. 
Therefore, by induction on weight, the identity \eqref{aux:Upuniq1}
together with $\Upsilon_0 =1$ imply the uniqueness of $\Upsilon$. 

The $\Upsilon$ as constructed satisfies the additional property that
$\Upsilon_\mu = 0$ unless $\mu^{\inv} = \mu$,
by Lemmas~\ref{lem:Up0}, \ref{lem:claim1} and \ref{lem:claim2}.
The theorem is proved.
\end{proof}

\section{The isomorphism $\mc{T}$}

Consider a function $\zeta$ on $\Lambda$ such that 
\begin{align}
\zeta (\mu+\alpha_0)&=-q \zeta (\mu), 
 \notag \\ 
\zeta (\mu+\alpha_i) &= -q^{(\alpha_i -\alpha_{-i}, \mu+\alpha_i)} \zeta (\mu), 
  \label{eq:zeta0} \\
\zeta (\mu+\alpha_{-i}) &= -q^{(\alpha_{-i}, \mu+\alpha_{-i}) - (\alpha_{i}, \mu)} \zeta (\mu),
\quad \forall \mu \in \Lambda, \; i \in \Ihf. 
\notag
\end{align}
Noting that $(\alpha_i, \alpha_{-i})=0 $ for all $i \in \Ihf$, we see that $\zeta$ satisfying \eqref{eq:zeta0}
is equivalent to $\zeta$ satisfying
 \begin{align}\label{eq:zeta}
\begin{split} 
\zeta (\mu+\alpha_0) &=-q \zeta (\mu),
 \\
\zeta (\mu+\alpha_i) &= -q^{(\alpha_i, \mu+\alpha_i) - (\alpha_{-i}, \mu)} \zeta (\mu),
 \quad \forall \mu \in \Lambda,\;  0 \neq i \in \I.
\end{split}
\end{align}
Such $\zeta$ clearly exists. For any weight $\U$-module $M$, 
define a $\Qq$-linear map on $ M$ 
\begin{align}
  \label{eq:zetatd}
\begin{split}
\widetilde{\zeta}&: M \longrightarrow M, 
 \\
\widetilde{\zeta}  (m &)  =  \zeta (\mu)m, \quad \forall   m \in M_{\mu}.
\end{split}
\end{align}

Recall that $w_0$ is the longest element of $W$ and   $T_{w_0}$  is the associated braid group element
from Section~ \ref{subsec:CB}. 

\begin{thm}\label{thm:mcT}
For any finite-dimensional $\U$-module $M$, 
the composition map
\[
\mc{T} := \Upsilon\circ \widetilde{\zeta} \circ T_{w_0}: M \longrightarrow M
\] 
is a $\Ui$-module isomorphism.
\end{thm}

\begin{proof}
The map $\mc{T}$ is clearly a  $\Qq$-linear isomorphism. So it remains to verify that $\mc{T}$ commutes with the action of $\Ui$; 
we shall check this on generators of $\Ui$ by applying repeatedly Lemma \ref{lem:Tw0}. 

Let $m \in M_{w_0(\mu)}$ and $i \in \Ihf$. Then we have
\begin{align*}
\mc{T} (\bk_{\alpha_i}m) &=   \Upsilon\circ \widetilde{\zeta} \circ T_{w_0}(\imath(\bk_{\alpha_i}))T_{w_0}(m) \\
& = \Upsilon\circ \widetilde{\zeta} \circ T_{w_0}(K_{\alpha_i}K^{-1}_{\alpha_{-i}})T_{w_0}(m) \\
& = \Upsilon\circ \widetilde{\zeta} K_{\alpha_i}K^{-1}_{\alpha_{-i}}T_{w_0}(m) \\
& = (K_{\alpha_i}K^{-1}_{\alpha_{-i}}) \Upsilon\circ \widetilde{\zeta} \circ T_{w_0} (m)\\
& = \bk_{\alpha_i}\mc{T}(m).
\end{align*}
We also have
\begin{align*}
\mc{T} (\be_{\alpha_i}m) &= \Upsilon\circ \widetilde{\zeta} (T_{w_0}(\imath(\be_{\alpha_i})) T_{w_0}(m))\displaybreak[0]\\
&= \Upsilon\circ \widetilde{\zeta}  (T_{w_0}(E_{\alpha_i} + K^{-1}_{\alpha_i}F_{\alpha_{-i}})  T_{w_0}(m)) \displaybreak[0]\\
&=  - \Upsilon\circ \widetilde{\zeta}  (K^{-1}_{\alpha_i}(K_{\alpha_i}F_{\alpha_{-i}}
 + E_{\alpha_i})K_{\alpha_{-i}}  T_{w_0}(m) )\displaybreak[0]\\
& = - \Upsilon(\zeta(\mu-\alpha_{-i}))q^{(\alpha_{-i},\mu)-(\alpha_i, \mu-\alpha_{-i})} K_{\alpha_i} F_{\alpha_{-i}}T_{w_0}(m) )
\\& \quad-\Upsilon( \zeta(\mu+\alpha_i)q^{(\alpha_{-i},\mu)-(\alpha_i, \mu+{\alpha_i})}E_{\alpha_i}T_{w_0}(m))\displaybreak[0]\\
&\stackrel{(a)}{=}   \Upsilon (E_{\alpha_i} + K_{\alpha_i}F_{\alpha_{-i}}) \zeta(\mu) T_{w_0}(m) \displaybreak[0]\\
&\stackrel{(b)}{=} (E_{\alpha_i} + K^{-1}_{\alpha_i}F_{\alpha_{-i}}) \Upsilon \circ \widetilde{\zeta} \circ T_{w_0} (m)\displaybreak[0]\\
& = \be_{\alpha_i} \mc{T} (m).
\end{align*}
The identity $(a)$ above follows from the definition of $\zeta$ and the identity $(b)$ 
follows from the definition of $\Upsilon$. 

By a similar computation we have
$
\mc{T}\bff_{\alpha_i}(m) = \bff_{\alpha_i}\mc{T}(m)$. 

For the generator $\bt$, we have 
\begin{align*}
\mc{T} (\bt \, m) &=   \Upsilon\circ \widetilde{\zeta} \circ T_{w_0}(\imath(\bt))T_{w_0}(m) \displaybreak[0]\\
& = \Upsilon\circ \widetilde{\zeta} \circ T_{w_0}(E_{\alpha_0} 
+ qF_{\alpha_0}K^{-1}_{\alpha_0}+ K^{-1}_{\alpha_0}) T_{w_0}(m) \displaybreak[0]\\
&= \Upsilon\circ \widetilde{\zeta} (-  F_{\alpha_0}K_{\alpha_0} - q^{-1}E_{\alpha_0} 
 + K_{\alpha_0}) T_{w_0}(m) \displaybreak[0]\\
&= \Upsilon (-q \zeta(\mu -\alpha_0) q^{-1}F_{\alpha_0}K_{\alpha_0} 
 - q^{-1} \zeta({\mu + \alpha_0}) E_{\alpha_0} + \zeta(\mu)K_{\alpha_0})T_{w_0}(m) \displaybreak[0]\\
& \stackrel{(c)}{=} \Upsilon (q^{-1}F_{\alpha_0} K_{\alpha_0} + E_{\alpha_0} 
 + K_{\alpha_0}) \zeta(\mu) T_{w_0}(m)  \displaybreak[0]\\
& \stackrel{(d)}{=} (E_{\alpha_0} +q F_{\alpha_0}K^{-1}_{\alpha_0} 
 + K^{-1}_{\alpha_0})\Upsilon \circ \widetilde{\zeta} \circ T_{w_0}(m) \displaybreak[0]\\
& =\bt \mc{T} (m).
\end{align*}
Here the identity $(c)$ follows from the definition of $\zeta$ and identity $(d)$ follows from the definition of $\Upsilon$. 
Hence the theorem is proved. 
\end{proof}

\chapter{Quasi-$\mc{R}$-matrix for a quantum symmetric pair}  
 \label{sec:QuasiR}

In this chapter, we define a quasi-$\mc{R}$-matrix $\ThetaB$ for $\Ui$, which will play an analogous role as
Lusztig's quasi-$\mc{R}$-matrix for $\U$. Our $\ThetaB$ is constructed from the intertwiner $\Upsilon$
and $\Theta$. 

\section{Definition of $\ThetaB$} 

Recall Lusztig's quasi-$\mc{R}$-matrix $\ThetaA$ from \eqref{eq:ThetaA}. 
It follows by Theorem \ref{thm:Upsilon} that $\Upsilon$ is 
a well-defined operator on finite-dimensional $\U$-modules. For any finite-dimensional $\U$-modules $M$ and $M'$, 
the action of $\Upsilon$ on $M \otimes M'$ is also well defined. So we shall use the formal notation $\Dupsilon$ 
to denote the action of $\Upsilon$ on $M \otimes M'$. Hence the operator
\begin{equation}   \label{eq:ThetaB}
\ThetaB := \Dupsilon \ThetaA(\Upsilon^{-1} \otimes 1)
\end{equation}
on $M \otimes M'$ is well defined. Note that $\ThetaB$ lies in (a completion of) $\U \otimes \U$. 
We shall prove in Proposition~
\ref{prop:ThetainBoA} that it actually lies in (a completion of) $\Ui \otimes \U$.

\begin{definition}
The element $\ThetaB$ is called the quasi-$\mc{R}$-matrix for the quantum symmetric pair $(\U, \Ui)$. 
\end{definition}

Recall that we set in Section~ \ref{subsec:f} that $\Abar({u}) = \ov{u}$ for all $u\in \U$, and in Section~\ref{sec:bun}
that $\Bbar(x) : = \ov{x} \in \Ui$ for $x \in \Ui$. 
We shall also set $\Abar(x) := \ov{\imath(x)} \in \U$ for $x \in \Ui$. 

Define $\ov{\Delta}: \Ui \rightarrow \Ui \otimes \U$ by 
$\ov{\Delta}(u) = (\Bbar \otimes \Abar) {\Delta(\Bbar(u))}$, for all $u \in \Ui$. 
Recall that the bar involution on $\Ui$ is not compatible with the bar involution 
on $\U$ through $\imath$ (see Remark~ \ref{rem:bars});
in particular the $\ov{\Delta}$ here does not coincide with the restriction to $\Ui$ of the map in the same notation
$\ov{\Delta}: \U \rightarrow \U \otimes \U$ in \cite[4.1.1]{Lu94}. 

\begin{prop}\label{prop:quasiTR}
Let $M$ and $M'$ be finite-dimensional $\U$-modules. As linear operators on $M \otimes M'$, we have 
$
\Delta(u)\ThetaB  = \ThetaB\,\ov{\Delta}(u), 
$
for all $u \in \Ui$.
\end{prop}

\begin{proof}
For $u\in \Ui$, we set $\Delta(\ov {u} ) = \sum u_{(1)} \otimes u_{(2)} \in \Ui \otimes \U$. Then, for $m \in M$,  $m' \in M'$, we have 
\begin{align*}
\Dupsilon \ThetaA(\Upsilon^{-1} \otimes 1) \ov{\Delta(\ov {u})}(m \otimes m')
&=  \Dupsilon \ThetaA(\sum \Upsilon^{-1} \imath(\ov{u_{(1)}}) \otimes \ov{u_{(2)}})(m \otimes m')\\
&\stackrel{(a)}{=} \Dupsilon \ThetaA (\sum \ov {\imath( u_{(1)})} \otimes \ov{u_{(2)}}) (\Upsilon^{-1} \otimes 1) (m \otimes m')\\
&\stackrel{(b)}{=}  \Dupsilon \Delta \big(\ov{\imath(\ov{u})} \big) \ThetaA (\Upsilon^{-1} \otimes 1) (m \otimes m')\\
& \stackrel{(c)}
= \Delta(u) \Dupsilon \ThetaA(\Upsilon^{-1} \otimes 1) ( m\otimes m').
\end{align*}
The identities $(a)$ and $(c)$ follow from Theorem~ \ref{thm:Upsilon} and the identity $(b)$ 
follows from \eqref{eq:propThetaA}. 
Note that the bar-notation above translates into the  $\psi$-notation as follows:

$\ov{u} =\Bbar(u)$, $\ov{u_{(1)}} = \Bbar(u_{(1)})$, $\ov{u_{(2)}} =\Abar(u_{(2)})$,
 $\ov {\imath( u_{(1)})} =\Abar(\imath( u_{(1)}))$, 
$\ov{\imath(\ov{u})} = \Abar(\imath(\Bbar(u)))$.

The proposition is proved. 
\end{proof}

\section{Normalizing $\ThetaB$}
\label{subsec:normalTheta}

Our next goal is to understand $\ThetaB$ in a precise sense as an element in a completion of $\U\otimes \U^-$
instead of merely as well-defined operators on $M\otimes M'$ for finite-dimensional $\U$-modules $M, M'$. 

 Let $B =\{b\}$ be a basis of $\U^-$ such that $B_\mu = B \cap \U^-_{-\mu}$ is a basis of $\U^-_{-\mu}$ for each $\mu$. 
 Let $B^* = \{b^*\}$ be the basis of  $\U^-$ dual to $B$ with respect to the bilinear form $(\cdot, \cdot)$
 in Section~ \ref{subsec:f}. For each $N \in \N$, define the $\Q(q)$-linear truncation map 
 $tr_{\leq N}: \fprime \rightarrow \fprime$ such that, for any 
 $i_1, \ldots, i_k \in \I$,
\begin{equation} \label{eq:trN}
tr_{\leq N} (F_{\alpha_{i_1}}\dots F_{\alpha_{i_k}})= 
\begin{cases}
F_{\alpha_{i_1}}\dots F_{\alpha_{i_k}}, \quad &\text{ if } k \leq N,\\
0, \quad &\text{ if } k > N.
\end{cases}
\end{equation}
This induces a truncation map on $\U^- =\fprime/{\bf I}$, also denoted by $tr_{\leq N}$, since ${\bf I}$ is homogeneous. 
Recalling $\ThetaA$ from \eqref{eq:ThetaA}, we denote
\[
\ThetaA_{\leq N} : =\sum_{\hgt(\mu) \leq N}\ThetaA_\mu.
\]
Then we define
\begin{equation}   \label{eq:ThetaleqN}
\ThetaB_{\leq N} := \sum_{\mu}id \otimes tr_{\leq N}(\Delta(\Upsilon_{\mu})  \ThetaA_{\leq N}( \Upsilon^{-1} \otimes 1)),
\end{equation}
which is actually a finite sum, and hence
$\ThetaB_{\leq N}\in \U \otimes \U^-$ and $\ThetaB_{\leq 0} = 1\otimes 1$.
Define
\begin{equation}  \label{eq:ThetaN}
\ThetaB_{ N} : = \ThetaB_{\leq N} - \ThetaB_{\leq N-1} 
= 
\sum_{ b_\mu \in B_\mu, \hgt(\mu) = N} 
a^\mu \otimes b_\mu \in \U \otimes \U^- ,
\end{equation}
where it is understood that $\ThetaB_{\leq -1} = 0$. The following lemma is clear from weight consideration.

\begin{lem}\label{lem:ThetainBoA}
Let $M$ and $M'$ be finite-dimensional $\U$-modules. For all $m\in M$ and $m' \in  M'$, we have  
\[
\ThetaB(m\otimes m') = \ThetaB_{\leq N}(m \otimes m'), \quad \text{for } N \gg 0.
\]

\end{lem}

Note that any finite-dimensional $\U$-module is also a $\widehat{\U}$-module. 
\begin{lem}\label{lem:allfdM}
Let $u \in \widehat{\U}$ be an element that acts as zero on all finite-dimensional $\U$-modules. Then $u =0$.
\end{lem}
\begin{proof}

It is well known that any element $u \in \U$ that acts as zero on all finite-dimensional $\U$-modules 
has to be $0$ (see \cite[Proposition 3.5.4]{Lu94}).
Hence the lemma follows by weight consideration.
\end{proof}
 
 We have the following fundamental property of $\ThetaB_N$.
 
\begin{prop}\label{prop:ThetainBoA}
For any $N \in \N$, we have $\ThetaB_{ N} \in \imath(\Ui) \otimes \U^-$.
\end{prop}

\begin{proof}
The identity in Proposition~ \ref{prop:quasiTR} for $u$ being one of the generators 
$\bk_{\alpha_i}$, $\be_{\alpha_i}$, $\bff_{\alpha_i}$,  and $\bt$ of $\Ui$ 
can be rewritten as the following identities (valid for all $N\ge 0$):
\begin{align*}
(\bk_{\alpha_i}  & \otimes K_{\alpha_i}K^{-1}_{\alpha_{-i}}) \ThetaB_{N}(m \otimes m') 
 =\ThetaB_{N} (\bk_{\alpha_i} \otimes K_{\alpha_i}K^{-1}_{\alpha_{-i}})(m \otimes m')\displaybreak[0],
\end{align*}
\begin{align*}
( (\bk_{\alpha_i}  & \otimes F_{\alpha_i}K^{-1}_{\alpha_{-i}}) \ThetaB_{N-1}
 +  (\bff_{\alpha_i} \otimes K^{-1}_{\alpha_{-i}})\ThetaB_{N} +(1 \otimes E_{\alpha_{-i}})\ThetaB_{N+1})(m \otimes m')\\
& =(\ThetaB_{N-1}(\bk^{-1}_{\alpha_i} \otimes F_{\alpha_i}K_{\alpha_{-i}})+ \ThetaB_{N}(\bff_{\alpha_i}
 \otimes  K_{\alpha_{-i}})+ \ThetaB_{N}( 1 \otimes E_{\alpha_{-i}}))(m \otimes m')\displaybreak[0],\end{align*}
\begin{align*}
( ( \bk^{-1}_{\alpha_i} & \otimes K^{-1}_{\alpha_i}F_{\alpha_{-i}}) \ThetaB_{ N-1}+( \be_{\alpha_i}
 \otimes K^{-1}_{\alpha_{i}}) \ThetaB_{ N} +  (1\otimes E_{\alpha_{i}} )\ThetaB_{ N+1})(m \otimes m')\\
& =(\ThetaB_{ N-1} (\bk_{\alpha_i} \otimes K_{\alpha_i}F_{\alpha_{-i}})+\ThetaB_{ N}( \be_{\alpha_i}
  \otimes  K_{\alpha_{i}})+ \ThetaB_{ N+1} (1\otimes E_{\alpha_{i}}))(m \otimes m') \displaybreak[0],\end{align*}
\begin{align*}
((1 & \otimes qF_{\alpha_0}K^{-1}_{\alpha_0}) \ThetaB_{ N-1}  + (\bt  \otimes K^{-1}_{\alpha_0} )\ThetaB_{ N}
  + (1  \otimes E_{\alpha_0})\ThetaB_{ N+1})(m \otimes m')\\
& =(\ThetaB_{ N-1} (1 \otimes q^{-1}F_{\alpha_0}K_{\alpha_0}) +  \ThetaB_{ N} (\bt \otimes K_{\alpha_0})
 + \ThetaB_{ N+1 } ( 1\otimes E_{\alpha_0}))(m \otimes m')\displaybreak[0],
\end{align*}
for all $0\neq i \in \Ihf$, $m\in M$ and $m' \in M'$, where $M, M'$ are finite-dimensional $\U$-modules. Write 
$$
\ThetaB_{N}= \sum_{b_\mu\in B_\mu, \hgt(\mu) = N}a^\mu \otimes b_\mu \in \U  \otimes \U^-,
$$ 
where $a_\mu$'s are fixed once $B$ is chosen. 
Thanks to Lemma \ref{lem:allfdM}, the above four identities for all $M,M'$ are equivalent to the following four identities: 
\begin{align}
&\sum_{\stackrel{b_\mu}{\hgt(\mu) =N}}\imath(\bk_{\alpha_i})a^\mu \otimes K_{\alpha_i}K^{-1}_{\alpha_{-i}}b_\mu 
= \sum_{\stackrel{b_\mu}{\hgt(\mu) =N}}a^\mu \imath(\bk_{\alpha_i}) \otimes b_\mu K_{\alpha_i}K^{-1}_{\alpha_{-i}}\displaybreak[0],
 \label{eq:auxThetainBoA1}
  \\ \notag
 \\ 
&\sum_{\stackrel{b_{\mu''}}{\hgt(\mu'')=N-1}} \!\!\!\! \imath(\bk_{\alpha_i})a^{\mu''} \otimes F_{\alpha_i}K^{-1}_{\alpha_{-i}}b_{\mu''}
+\!\!\!\!  \sum_{\stackrel{b_{\mu'}}{\hgt(\mu')=N}}\!\!\!\! \imath(\bff_{\alpha_i})a^{\mu'} \otimes K^{-1}_{\alpha_{-i}}b_{\mu'} 
+ \!\!\!\! \sum_{\stackrel{b_\mu}{\hgt(\mu) =N+1}}\!\!\!\! a^{\mu} \otimes E_{\alpha_{-i}}b_{\mu}
 \label{eq:auxThetainBoA2}\\
=& 
\sum_{\stackrel{b_{\mu''}}{\hgt(\mu'')=N-1}} \!\!\!\! a^{\mu''}\imath(\bk^{-1}_{\alpha_i}) \otimes b_{\mu''} F_{\alpha_i}K_{\alpha_{-i}}
+ \!\!\!\! \sum_{\stackrel{b_{\mu'}}{\hgt(\mu')=N}} \!\!\!\! a^{\mu'}\imath(\bff_{\alpha_i}) \otimes b_{\mu'} K_{\alpha_{-i}}
+ \!\!\!\! \sum_{\stackrel{b_\mu}{\hgt(\mu) =N+1}}\!\!\!\! a^{\mu} \otimes b_{\mu} E_{\alpha_{-i}}\displaybreak[0],
 \notag
\\ \notag
\\
& \sum_{\stackrel{b_{\mu''}}{\hgt(\mu'')=N-1}}\!\!\!\! \imath(\bk^{-1}_{\alpha_i})a^{\mu''} \otimes K^{-1}_{\alpha_i}F_{\alpha_{-i}}b_{\mu''}
+ \!\!\!\! \sum_{\stackrel{b_{\mu'}}{\hgt(\mu')=N}} \!\!\!\! \imath(\be_{\alpha_i})a^{\mu'} \otimes K^{-1}_{\alpha_{i}}b_{\mu'} 
+\!\!\!\! \sum_{\stackrel{b_\mu}{\hgt(\mu) =N+1}} \!\!\!\! a^{\mu} \otimes E_{\alpha_{i}}b_{\mu}
\\
=&\notag
\sum_{\stackrel{b_{\mu''}}{\hgt(\mu'')=N-1}}\!\!\!\! a^{\mu''}\imath(\bk_{\alpha_i}) \otimes b_{\mu''} K_{\alpha_i}F_{\alpha_{-i}}
+ \!\!\!\! \sum_{\stackrel{b_{\mu'}}{\hgt(\mu')=N}} \!\!\!\! a^{\mu'}\imath(\be_{\alpha_i}) \otimes b_{\mu'} K_{\alpha_{i}}
+\!\!\!\! \sum_{\stackrel{b_\mu}{\hgt(\mu) =N+1}} \!\!\!\!  a^{\mu}\otimes b_{\mu} E_{\alpha_{i}} \displaybreak[0],
 \\ \notag
 \\
& \sum_{\stackrel{b_{\mu''}}{\hgt(\mu'')=N-1}} a^{\mu''} \otimes qF_{\alpha_0}K^{-1}_{\alpha_0}b_{\mu''} 
+\sum_{\stackrel{b_{\mu'}}{\hgt(\mu')=N}}\imath(\bt) a^{\mu'} \otimes K^{-1}_{\alpha_0}b_{\mu'}
+ \sum_{\stackrel{b_\mu}{\hgt(\mu) =N+1}} a^{\mu} \otimes E_{\alpha_0}b_{\mu}\label{eq:auxThetainBoA4}\\
=&\notag\sum_{\stackrel{b_{\mu''}}{\hgt(\mu'')=N-1}} a^{\mu''} \otimes b_{\mu''} q^{-1}F_{\alpha_0}K_{\alpha_0} 
+ \sum_{\stackrel{b_{\mu'}}{\hgt(\mu')=N}} a^{\mu'}\imath(\bt) \otimes b_{\mu'} K_{\alpha_0}
+ \sum_{\stackrel{b_\mu}{\hgt(\mu) =N+1}} a^{\mu} \otimes b_{\mu} E_{\alpha_0}\displaybreak[0].
\end{align}
A straighforward rewriting of \eqref{eq:auxThetainBoA2}-\eqref{eq:auxThetainBoA4} involves
the commutators  $[E_{\alpha_k}, b_\mu]$ for various $k\in \I$,
which can be expressed in terms of $_{k}r$ and $r_{k}$
by invoking \cite[Proposition 3.1.6]{Lu94}. In this way, using the PBW theorem for $\U$ 
we rewrite the three identities \eqref{eq:auxThetainBoA2}-\eqref{eq:auxThetainBoA4}
as the following six identities:
\begin{align}
&
{\tiny
 \!\! \sum_{\stackrel{b_{\mu''}}{\hgt(\mu'')=N-1}} \!\!\!\!\!\! \imath(\bk_{\alpha_i}) a^{\mu''} \otimes F_{\alpha_i}b_{\mu''}
+ \!\!\!\!\!\! \sum_{\stackrel{b_{\mu'}}{\hgt(\mu')=N}}\!\!\!\!\!\! \imath(\bff_{\alpha_i}) a^{\mu'} \otimes b_{\mu'}
+ \frac{q^{( \alpha_{-i} ,\mu +\alpha_{-i})}}{q^{-1} -q}\!\!\!\!  \sum_{\stackrel{b_\mu}{\hgt(\mu) =N+1}}\!\!\!\! a^{\mu} \otimes r_{-i}(b_{\mu})  =0\displaybreak[0], 
 }
  \label{eq:f}
  \\ \notag
  \\
&
{\tiny
\!\!\sum_{\stackrel{b_{\mu''}}{\hgt(\mu'')=N-1}}\!\!\!\!\!\!  a^{\mu''} \imath(\bk^{-1}_{\alpha_i}) \otimes b_{\mu''} F_{\alpha_i}
+ \!\!\!\!\!\!  \sum_{\stackrel{b_{\mu'}}{\hgt(\mu')=N}} \!\!\!\!\!\!  a^{\mu'} \imath(\bff_{\alpha_i}) \otimes b_{\mu'}
+ \frac{q^{( \alpha_{-i} , \mu +\alpha_{-i})}}{q^{-1} -q}\!\!\!\! \sum_{\stackrel{b_\mu}{\hgt(\mu) =N+1}}\!\!\!\! a^{\mu} \otimes \!\,_{-i}r(b_{\mu}) =0\displaybreak[0],
 }
  \notag
  \\ \notag
  \\ 
&
{\tiny
\sum_{\stackrel{b_{\mu''}}{\hgt(\mu'')=N-1}}\!\!\! \imath(\bk^{-1}_{\alpha_i})a^{\mu''} \otimes F_{\alpha_{-i}}b_{\mu''}
+ \!\!\!\!\!\!  \sum_{\stackrel{b_{\mu'}}{\hgt(\mu')=N}} \!\!\! \imath(\be_{\alpha_i}) a^{\mu'} \otimes b_{\mu'}
+\frac{q^{(\alpha_i , \mu+\alpha_{i})}}{q^{-1}-q}\!\!\! \sum_{\stackrel{b_\mu}{\hgt(\mu) =N+1}}\!\!\! a^{\mu}\otimes r_{i}(b_{\mu}) = 0\displaybreak[0],
 }
  \label{eq:e}
  \\\notag
   \\
&
{\tiny
\sum_{\stackrel{b_{\mu''}}{\hgt(\mu'')=N-1}}\!\!\! a^{\mu''} \imath(\bk_{\alpha_i}) \otimes b_{\mu''} F_{\alpha_{-i}}
+\!\!\!\!\!\!  \sum_{\stackrel{b_{\mu'}}{\hgt(\mu')=N}}\!\!\!  a^{\mu'}\imath(\be_{\alpha_i}) \otimes b_{\mu'}
+\frac{q^{(\alpha_i , \mu+\alpha_{i})}}{q^{-1}-q}\!\!\! \sum_{\stackrel{b_\mu}{\hgt(\mu) =N+1}} \!\!\! a^{\mu}\otimes \! \,_{i}r(b_{\mu}) = 0\displaybreak[0],
 }
  \notag
  \\\notag
  \\ 
&
{\tiny
\sum_{\stackrel{b_{\mu''}}{\hgt(\mu'')=N-1}}\!\!\! a^{\mu''} \otimes q^{-1}F_{\alpha_0}b_{\mu''}
+\sum_{\stackrel{b_{\mu'}}{\hgt(\mu')=N}}\!\!\! \imath(\bt) a^{\mu'} \otimes b_{\mu'}
+ \frac{q^{(\alpha_0 , \mu+\alpha_0)}}{q^{-1} - q}\!\!\! \sum_{\stackrel{b_\mu}{\hgt(\mu) =N+1}}\!\!\! a^{\mu} \otimes r_{0}(b_{\mu}) = 0\displaybreak[0],
 }
  \label{eq:t}
  \\\notag
   \\
&
{\tiny
\sum_{\stackrel{b_{\mu''}}{\hgt(\mu'')=N-1}}\!\!\! a^{\mu''}\otimes q^{-1}b_{\mu''}  F_{\alpha_0} 
+\sum_{\stackrel{b_{\mu'}}{\hgt(\mu')=N}} \!\!\!a^{\mu'} \imath(\bt) \otimes b_{\mu'}
+ \frac{q^{(\alpha_0 , \mu+\alpha_0)}}{q^{-1} - q}\!\!\! \sum_{\stackrel{b_\mu}{\hgt(\mu) =N+1}}\!\!\! a^{\mu} \otimes _{0}r(b_{\mu}) = 0\displaybreak[0].
 \notag
 }
\end{align}

So far we have the flexibility in choosing the dual bases $B$ and $B^*$ of $\U^-$. 
Now let us be more specific by fixing 
$B^*=\{b^*\}$ to be a monomial basis of $\U^-$ which consists of monomials 
in the Chevalley generators $F_{\alpha_i}$; for example, we can take the $\U^-$-variant of
the basis $\{E(({\bf c}))\}$ in \cite[pp.476]{Lu90} where Lusztig worked with $\U^+$.
Let $B =\{b\}$ be the dual basis of $B^*$  
with respect to $(\cdot, \cdot)$, and write $B_\mu =B\cap \U_{-\mu} =\{b_\mu\}$ as before. 
Fix an arbitrary basis element $\tilde{b}_\mu \in B_\mu$ (with $\mu \neq 0$),
with its dual basis element 
written as  $\tilde{b}^*_\mu =xF_{\alpha_{-i}}$, 
for some $x \in \U^-$ and some $i$. We now apply $1 \otimes (x, \cdot)$ to the identities \eqref{eq:f}, \eqref{eq:e} and \eqref{eq:t}, 
depending on whether $i$ is positive, zero or negative. 

We will treat in detail the case when $i$ is positive, 
while the other cases are similar. Applying $1 \otimes (x, \cdot)$ to the identity $\eqref{eq:f}$ above, we have
\begin{align*}
\sum_{\stackrel{b_{\mu''}}{\hgt(\mu'')=N-1}}   & \imath(\bk_{\alpha_i})a^{\mu''} \otimes (x, F_{\alpha_i}b_{\mu''})
+ \sum_{\stackrel{b_{\mu'}}{\hgt(\mu')=N}} \imath(\bff_{\alpha_i})a^{\mu'} \otimes (x, b_{\mu'})\\
&
+ \frac{q^{( \alpha_{-i} , \mu +\alpha_{-i})}}{q^{-1} -q}\sum_{\stackrel{b_{\mu}}{\hgt(\mu)=N+1}} a^{\mu} \otimes (x, r_{-i}(b_{\mu})) = 0.
\end{align*}
Since $(x, r_{-i}(b_{\mu})) = (1- q^{-2} )(xF_{\alpha_{-i}}, b_\mu) = (1 - q^{-2}) \delta_{b_{\mu},\tilde{b}_\mu}$, we have
\begin{align}  \label{eq:atilde}
 \sum_{\stackrel{b_{\mu''}}{\hgt(\mu'')=N-1}} & \imath(\bk_{\alpha_i}) a^{\mu''} (x, F_{\alpha_i}b_{\mu''})
\notag \\
&+ \sum_{\stackrel{b_{\mu'}}{\hgt(\mu')=N}} \imath(\bff_{\alpha_i}) a^{\mu'} (x, b_{\mu'})
- q^{(\alpha_{-i} , \mu +\alpha_{-i})-1}\tilde{a}^{\mu}  = 0.
\end{align}
By an easy induction on height based on \eqref{eq:atilde} 
(where the base case is $\ThetaB_{0}= 1 \otimes 1$), we conclude that 
$
a^\mu \in \imath(\Ui) \text{ for all } \mu$; that is, $\ThetaB_{ N} \in \imath(\Ui) \otimes \U^-$.
\end{proof}

By Proposition~\ref{prop:ThetainBoA}
we have $\imath^{-1} (\ThetaB_N) \in \Ui \otimes \U$ for each $N$. 
For any finite-dimensional $\U$-modules $M$ and $M'$, 
the action of $\imath^{-1} (\ThetaB_N)$ coincides with the action of $\ThetaB_N$
on $M \otimes M'$. 

{\bf As we only need to use $\imath^{-1} (\ThetaB_N) \in \Ui \otimes \U$ rather than $\ThetaB_N$,
we shall write $\ThetaB_N$ in place of  $\imath^{-1} (\ThetaB_N)$  
and regard $\ThetaB_N \in \Ui \otimes \U$ from now on.}  

\section{Properties of $\ThetaB$}

Let $(\Ui \otimes \U^-)^\wedge$ be the completion of the $\Qq$-vector space $\Ui \otimes \U^-$ 
with respect to the following descending sequence of subspaces 
\[
{H}_N^\imath := \Ui \otimes \Big(\sum_{\hgt(\mu) \geq N}\U_{-\mu}^- \Big), \quad \text{ for } N \ge 1.
\]
The $\Qq$-algebra structure on $\Ui \otimes \U^-$ extends by continuity to
a $\Qq$-algebra structure on $(\Ui \otimes \U^-)^\wedge$, and we have an embedding
$\Ui \otimes \U^- \hookrightarrow (\Ui \otimes \U^-)^\wedge$.

The actions of $\sum_{N\ge 0} \ThetaB_N$ (which is well defined by Lemma \ref{lem:ThetainBoA})
and of $\ThetaB$ coincide on any tensor product of finite-dimensional $\U$-modules. 
From now on, we may and shall identify 
\begin{equation}  \label{eq:ThTh}
\ThetaB = \sum_{N\ge 0}  \ThetaB_N \in (\Ui \otimes \U^-)^\wedge,
\end{equation}
(or alternatively, one may regard this as a normalized definition of $\ThetaB$).

The following theorem is a generalization of Proposition~ \ref{prop:quasiTR}.

\begin{thm}\label{thm:ThetaB}
Let $L$ be a finite-dimensional  $\Ui$-module and $M$ be a finite-dimensional $\U$-module. 
Then as linear operators on $L \otimes M$, we have
\[
\Delta(u)\ThetaB = \ThetaB\ov{\Delta}(u), \qquad \text{ for all } u \in \Ui.
\]
\end{thm}

\begin{proof}
By the identities \eqref{eq:auxThetainBoA1}-\eqref{eq:auxThetainBoA4} 
in the proof of Proposition \ref{prop:ThetainBoA}, 
there exists $N_0>0$ (depending on $L$ and $M$) such that for $N\geq N_0$ we have
\begin{equation}\label{eq:ThN}
\Delta(u) \ThetaB_{\leq N} - \ThetaB_{\leq N} \ov{\Delta}(u)
 = 0 \qquad \text{ on } L\otimes M,
\end{equation}
where  $u$ is one of the generators 
$\bk_{\alpha_i}$, $\be_{\alpha_i}$, $\bff_{\alpha_i}$,  and $\bt$ of $\Ui$.
We then note that, for $u_1, u_2 \in \Ui$, 
\begin{align} \label{eq:u12}
\begin{split}
\Delta (u_1u_2) & \ThetaB_{\leq N} - \ThetaB_{\leq N} \ov{\Delta}(u_1u_2)
 \\
  =&
\Delta(u_1) \big(\Delta(u_2) \ThetaB_{\leq N} - \ThetaB_{\leq N} \ov{\Delta}(u_2) \big)
\\
&+ 
\big( \Delta(u_1) \ThetaB_{\leq N} - \ThetaB_{\leq N} \ov{\Delta}(u_1) \big )
\ov{\Delta} (u_2).
\end{split}
\end{align}
Then by an easy induction using \eqref{eq:u12}, we conclude that
\eqref{eq:ThN} holds for all $u\in \Ui$ and  $N \ge N_0$. 
The theorem now follows from \eqref{eq:ThTh}. 
\end{proof}

\begin{prop}  \label{prop:ThetaBinv}
We have $\ThetaB \ov{\ThetaB} =1$ (an identity in $(\Ui \otimes \U^-)^\wedge$). 
\end{prop}

\begin{proof}
By construction, $\ThetaB =\sum_{N\ge 0} \ThetaB_N$ (with $\ThetaB_0 =1\otimes 1$) is clearly invertible in 
$(\Ui \otimes \U^-)^\wedge$. Write $'\ThetaB= (\ThetaB)^{-1}$.
 
Multiplying $'\ThetaB$ on both sides of  the identity in Theorem~ \ref{thm:ThetaB}, we have
\[
'\ThetaB \Delta(\ov{u}) =\overline{\Delta (u)} \; {}'\ThetaB, \qquad \forall u \in \Ui.
\]
Applying \,$\bar{\ }$\, to the above identity and replacing $\ov{u}$ by $u$, we have
\[
\ov{'\ThetaB}\,  \ov{\Delta({u})} ={\Delta(\ov{u})} \, \ov{'\ThetaB}, \qquad \forall u \in \Ui.
\]
Hence $\ov{'\ThetaB}$ (in place of $'\ThetaB$) satisfies the same identity in Theorem~ \ref{thm:ThetaB}
as well; note that $\ov{'\ThetaB} \in (\Ui \otimes \U^-)^\wedge$ has constant term $1\otimes 1$.

By reexamining the proof of Proposition~\ref{prop:ThetainBoA} and especially \eqref{eq:atilde},
we note that the element $\ThetaB \in (\Ui \otimes \U^-)^\wedge$ (with constant term $1\otimes 1$) satisfying the identity in
Proposition~ \ref{prop:quasiTR} (and thus Theorem \ref{thm:Upsilon}) is unique.
Hence we must have $\ThetaB =\ov{\ThetaB}^{-1}$, and equivalently,  $\ThetaB \ov{\ThetaB} =1$. 
 \end{proof}

Recall that $m (\epsilon \otimes 1)\Delta =\imath$ from Corollary~ \ref{cor:counit},
where $\epsilon$ is the counit and $m$ denotes the multiplication in $\U$. 

\begin{cor}\label{cor:ThetatoT}
The intertwiner $\Upsilon$ can be recovered from the quasi-$\mc R$-matrix $\ThetaB$ as 
 $m(\epsilon \otimes 1)(\ThetaB)  = \Upsilon$.
\end{cor}

\begin{proof}
Applying $m(\epsilon \otimes 1)$ to the identities 
\eqref{eq:auxThetainBoA1}-\eqref{eq:auxThetainBoA4}, we obtain an identity in $\widehat{\U}$:
\begin{equation}\label{eq:ThetatoT}
 \imath(\ov{u}) \Big(\sum_{N \geq 0} m(\epsilon \otimes 1)(\ThetaB_N) \Big)
 =\Big(\sum_{N \geq 0} m(\epsilon \otimes 1)(\ThetaB_N) \Big)
 \overline{\imath(u)},
\end{equation}
for all $u \in \Ui.$
The corollary now follows from \eqref{eq:ThTh}, \eqref{eq:ThetatoT} and the uniqueness of $\Upsilon$  in Theorem \ref{thm:Upsilon},
as  we have $m(\epsilon \otimes 1)(\ThetaB_{0}) = 1$.
\end{proof}

\section{The bar map on $\Ui$-modules}   \label{subsec:bars}

In this section we shall assume all the modules are finite dimensional. 
Recall the bar map on $\U$ and on its modules is denoted by $\Abar$, and
the bar map on $\Ui$ is also denoted by $\Bbar$. 
It is also understood that $\Abar(u) =\Abar(\imath(u))$ for $u\in \Ui$.


\begin{definition}  \label{def:involutive-i}
A $\Ui$-module $M$ equipped with an anti-linear involution $\Bbar$ 
is called {\em involutive} (or {\em $\imath$-involutive} to avoid possible ambiguity) if
$$
\Bbar(u m) = \Bbar(u) \Bbar(m), \qquad \forall u \in \Ui, m \in M.
$$ 
 \end{definition}

\begin{prop}\label{prop:compatibleBbar}
Let $M$ be an involutive $\U$-module. Then $M$ is an $\imath$-involutive  
$\Ui$-module with involution
$\Bbar := \Upsilon \circ \Abar$.
\end{prop}

\begin{proof}
By Theorem \ref{thm:Upsilon}, we have 
$
\imath( \Bbar(u)) \Upsilon = \Upsilon \Abar(u),  \text{ for all } u \in \Ui.$
By definition the action of $\Bbar(u)$ on $M$ is the same as the action of $\imath(\Bbar(u))$ on $M$. Therefore we have
\[
\Bbar(u m) = \Upsilon \Abar( u m) = \Upsilon \Abar(u) \Abar(m) = \imath( \Bbar (u) ) \Upsilon \Abar (m) = \Bbar (u) \Bbar(m),
\]
for all $ u \in \Ui$  and $m \in M$.

It remains to verify that $\Bbar$ is an involution on $M$. Indeed,  for $m \in M$, we have
 \[
\Bbar(\Bbar(m)) = \Upsilon \Abar(\Upsilon \Abar (m)) = \Upsilon \ov{\Upsilon} \Abar (\Abar(m)) = \Upsilon \ov{\Upsilon} m = m,
 \]
where the last identity follows from Corollary \ref{cor:Upsiloninv}. 
\end{proof}

\begin{cor}   \label{cor:Li-invol}
As $\Ui$-modules, $L(\la)$ and ${^{\omega}L}(\lambda)$ are $\imath$-involutive, for $\la\in \Lambda^+$. 
\end{cor}

\begin{rem}  \label{rem:xi inv}
We can and will choose  $\xi_{-\la} \in {^{\omega}L}(\lambda)$ to be $\Abar$-invariant.
It follows that  $\xi_{-\la}$ is also $\Bbar$-invariant,
since $\Bbar =\Upsilon \Abar$ and $\Upsilon$ lies in a completion of $\U^-$ with constant term $1$.
Because of this, it is more convenient to work with a lowest weight vector instead of a highest weight vector in
a finite-dimensional simple $\U$-module.
\end{rem}
Recall the quasi-$\mc{R}$-matrix $\ThetaB$ from \eqref{eq:ThetaB}.
Given an involutive $\Ui$-module $L$ and an involutive $\U$-module $M$, we define 
$\Bbar: L \otimes M \rightarrow L \otimes M$ by letting
\begin{equation}   \label{eq:Bbartensor}
\Bbar(l \otimes m) := \ThetaB(\Bbar(l) \otimes \Abar(m)),  \qquad  \text{ for all } l \in L, m \in M.
\end{equation}

\begin{prop}  \label{prop:Bbartensor}
Let $L$ be an involutive $\Ui$-module and let $M$ be an involutive $\U$-module. Then 
$(L \otimes M, \Bbar)$ is an involutive $\Ui$-module.
\end{prop}

\begin{proof}
For all $l \in L$, $m \in M$, $u \in \Ui$, using \eqref{eq:Bbartensor} twice we have
\begin{align*}
\Bbar(u (l \otimes m)) 
&= \ThetaB \big( \ov{\Delta(u)} (\Bbar (l) \otimes \Abar (m)) \big)
 \\
&=  \Delta(\ov{u}) \ThetaB  (\Bbar (l) \otimes \Abar (m)) 
 \\
&= \Bbar(u) \, \Bbar(l \otimes m).
\end{align*}
The second equality  in the above computation uses Theorem \ref{thm:ThetaB} and the first equality holds  since
$L$ and $M$ are involutive modules.

It remains to verify that $\Bbar$ is an involution on $L \otimes M$. It is occasionally convenient to use the bar-notation
to denote the involution $\Bbar \otimes \Abar$ on $\Ui \otimes \U$ below. 
Indeed,  for $l\in L$ and $m \in M$, using \eqref{eq:Bbartensor} twice we have
 \begin{align*} 
\Bbar(\Bbar(l\otimes m)) 
&= \ThetaB (\Bbar \otimes \Abar). \big(\ThetaB ( \Bbar(l) \otimes \Abar (m) ) \big)
  \\
&= \ThetaB \ov{ \ThetaB} ( \Bbar^2(l) \otimes \Abar^2 (m)) 
 = l \otimes m,
 \end{align*}
where the last equality follows from Proposition \ref{prop:ThetaBinv} and 
the second equality holds since
$L$ and $M$ are involutive modules.
\end{proof}

\begin{rem}   \label{rem:sameinv}
Given two involutive $\U$-modules $(M_1, \psi_1)$ and $(M_2,\psi_2)$, the $\U$-module $M_1\otimes M_2$ is involutive 
with the involution given by $\ThetaA \circ (\psi_1\otimes  \psi_2)$, 
(see \cite[27.3.1]{Lu94} or Proposition~\ref{prop:Lu27.3.1}). Now there are two natural
ways to define an anti-linear involution on the $\Ui$-module $M_1\otimes M_2$: 
\begin{itemize}
\item[(i)]  apply Proposition~\ref{prop:compatibleBbar} to the involutive $\U$-module 
$(M_1\otimes M_2, \ThetaA \circ (\psi_1\otimes \psi_2))$;

\item[(ii)] 
apply Proposition~\ref{prop:Bbartensor} by regarding $M_1$ as an $\imath$-involutive $\Ui$-module
with involution $\Upsilon \circ \psi_1$.
\end{itemize}
One checks that the resulting involutions on the $\Ui$-module $M_1\otimes M_2$ in two different ways coincide. 
\end{rem}

The following proposition implies that different bracketings on the tensor product 
of several involutive $\U$-modules   give rise to the same $\Bbar$. (Recall a similar property holds for
Lusztig's bar involution on tensor products of $\U$-modules \cite{Lu94}.)

\begin{prop}\label{prop:barinThetaB}
Let $M_1$, $\ldots$, $M_k$ be involutive $\U$-modules with $k\ge 2$. We have 
\[
\Bbar(m_1 \otimes \cdots \otimes m_k)
= \ThetaB (\Bbar (m_1 \otimes \cdots \otimes m_{k'}) \otimes \Abar(m_{k'+1} \otimes \cdots \otimes m_k)),
\]
for any $1 \leq k' < k$.
\end{prop}

\begin{proof}
Recall $\ThetaB = \Dupsilon\ThetaA (\Upsilon^{-1} \otimes 1)$. 
Unraveling the definition $\Bbar = \Upsilon \Abar$ on $M_1 \otimes \cdots \otimes M_{k'}$, 
we have
\begin{align*}
\ThetaB (\Bbar & (m_1 \otimes \cdots \otimes m_{k'}) \otimes \Abar(m_{k'+1} \otimes \cdots \otimes m_k)) \\
= &\Dupsilon\ThetaA (\Upsilon^{-1} \otimes 1) (\Upsilon \Abar (m_1 \otimes \cdots \otimes m_{k'}) 
   \otimes \Abar(m_{k'+1} \otimes \cdots \otimes m_k) )\\
= &\Dupsilon\ThetaA (\Abar (m_1 \otimes \cdots \otimes m_{k'}) \otimes \Abar(m_{k'+1} \otimes \cdots \otimes m_k) )\\
= & \Dupsilon \Abar (m_1 \otimes \cdots \otimes m_{k'} \otimes m_{k'+1} \otimes \cdots \otimes m_k)\\
= & \Bbar(m_1 \otimes \cdots \otimes m_k).
\end{align*}
The proposition follows.
\end{proof}

\chapter{The integrality of $\Upsilon$ and the $\imath$-canonical basis of ${}^\omega L(\la)$}   \label{sec:UpsiloninZ}
 
 In this chapter, we first  construct the $\imath$-canonical bases
 for simple $\U$-modules and then for the algebra
 $\Ui$ in the rank one case. Then we use the rank one results to study the general higher hank case.
 We show that
the intertwiner $\Upsilon$ is integral and construct the $\imath$-canonical basis
for ${}^\omega L(\la)$ for $\la\in \La^+$.

\section{The homomorphism $\pi_{\lambda, \mu}$}

Though only the rank one case
of the results in this section will be needed in this paper, 
it is natural and causes no extra work to formulate in the full generality below.

\begin{lem}\label{lem:bunxi} 
Let $\lambda \in \Lambda^{+}$. We have
$\Ui \xi_{-\lambda} = {^{\omega}L}(\lambda)$
and $\Ui \eta_{\lambda} = L(\lambda)$.
\end{lem}

\begin{proof}
We shall only prove $\Ui \xi_{-\lambda} = {^{\omega}L}(\lambda)$. 
The proof for the second identity is similar and will be skipped.

We write $\xi = \xi_{-\lambda}$. Let $h \in {^{\omega}L}(\lambda)_{\mu}$.
We shall prove $h \in \Ui \xi$ by induction on ${\rm ht}(\mu +\lambda)$. 
When ${\rm ht}(\mu + \lambda) =0$, the claim is clear since $h$ must be a scalar multiple of $\xi$. 
Thanks to $\U^+ \xi = {^{\omega}L}(\lambda)$, there exists $y \in \U^+$ such that $y \xi = h$. 
Writing $y$ as a linear combination of PBW basis elements for $\U^+$ and
replacing $E_{\alpha_0}$, $E_{\alpha_i}$, $E_{\alpha_{-i}}$  (for all $i \in \Ihf$) 
by $\bt$, $\be_{\alpha_i}$,  $\bff_{\alpha_i}$  in such a linear combination, respectively,
we obtain an element $u=u(y) \in \Ui$. Setting $\imath(u)= y + z$ for $z\in \U$, we have $u\xi = h + z\xi$. 
By construction, $z\xi$ is a $\Qq$-linear combination of elements in ${^{\omega}L}(\lambda)$
of weight lower than $h$. Hence by the induction hypothesis, we have $z\xi \in \Ui \xi$, and so 
is $h = u\xi -z\xi$. 
\end{proof}

Recall from Section ~\ref{subsec:CB} that ${^{\omega}L}(\lambda)$ for $\lambda \in \Lambda^+$ 
is identified with $L(\lambda^{\inv})=L(-w_0\lambda)$, $\xi_\la$ is the lowest weight vector of  ${^{\omega}L}(\lambda)$,
and $\eta_{\la^\inv}$ is the highest  weight vector of  $L(\lambda^\inv)$.

\begin{lem}\label{lem:uniqueT}
For  $\lambda \in \Lambda^+$,  there is an isomorphism of $\Ui$-modules
\[
\mc{T} : {^{\omega}L}(\lambda) \longrightarrow {^{\omega}L}(\lambda) = L( \lambda^{\inv})
\]
such that $\mc{T} (\xi_{\-\lambda}) = \sum_{ b \in \B(\lambda)} g_b b^- \eta^{}_{ \lambda^{\inv}}$  
where $g_b \in \Qq$ and $g_1 =1$. Moreover, the isomorphism $\mc T$ is uniquely determined by 
the image $\mc{T} (\xi_{\-\lambda})$.
\end{lem}

\begin{proof}
Recall the isomorphism 
$\mc{T}= \Upsilon \circ \widetilde{\zeta} \circ T_{w_0} : {^{\omega}L}(\lambda) \rightarrow {^{\omega}L}(\lambda)$ 
of $\Ui$-modules from Theorem \ref{thm:mcT}.  The existence of $\mc{T}$ satisfying the lemma follows by 
fixing the weight function $\zeta$ such that $\mc{T} (\xi_{\-\lambda}) =\eta_{\la^\theta} +$ terms in lower weights. 

The uniqueness of such $\mc T$ follows from Lemma~ \ref{lem:bunxi}.
\end{proof}

The following proposition can be found in \cite[Chapter 25]{Lu94}.

\begin{prop}  \label{prop:shrinkA}
Let $\lambda$, $\lambda' \in \Lambda^+$.
\begin{enumerate}
\item
There exists a unique homomorphism of  $\U$-modules
$$\chi = \chi_{\lambda,\lambda'} : {^{\omega}L}(\lambda+\lambda')
 \longrightarrow {^{\omega}L}(\lambda) \otimes {^{\omega}L}(\lambda')$$\
such that $\chi(\xi_{-\lambda-\lambda'}) = \xi_{-\lambda} \otimes \xi_{-\lambda'}$.

\item
For $b \in \bold{B}(\lambda+\lambda')$, we have 
$\chi(b^{+}\xi_{-\lambda-\lambda'}) = \sum_{b_1, b_2} f(b;b_1, b_2) b^+_1  \xi_{-\lambda} \otimes b^+_2 \xi_{-\lambda'}$, 
summed over $b_1 \in \B(\lambda)$ and $b_2 \in \B(\lambda')$, with $f(b; b_1,b_2) \in \Z[q]$. 
If $b^+ \xi_{-\lambda'} \neq 0$, then $f(b; 1, b) =1 $ and $f(b; 1, b_2)=0$ for any $b_2 \neq b$. 
If $b^+  \xi_{-\lambda'} =0$, then $f(b; 1, b_2)=0$ for any $b_2$. 

\item 
There is a unique  homomorphism of  $\U$-modules 
$\delta=\delta_{\lambda}: L(\lambda) \otimes {^{\omega}L}(\lambda) \rightarrow \Qq$, 
where $\Qq$ is the trivial representation of $\U$, such that 
$\delta(\eta_{\lambda} \otimes \xi_{-\lambda}) =1$. 
Moreover, for $b_1, b_2 \in \B(\lambda)$, $\delta(b^-_1 \eta_{\lambda} \otimes b^{+}_2\xi_{-\lambda})$ 
is equal to $1$ if $b_1 = b_2 =1$ and is in $q\Z[q]$ otherwise. In particular, 
$\delta(b^-_1 \eta_{\lambda} \otimes b^{+}_2\xi_{-\lambda}) =0$ if $|b_1| \neq |b_2|$.
\end{enumerate}
\end{prop}

\begin{prop}\label{prop:shrink}
Let $\lambda$,$\mu \in \Lambda^+$. There is a unique homomorphism of $\Ui$-modules 
\[
\pi_{\lambda , \mu}: {^{\omega}L} ( \mu^{\inv} + \mu + \lambda ) \longrightarrow {^{\omega}L}(\lambda)
\] 
such that $\pi_{\lambda , \mu} (\xi^{}_{-\mu^{\inv}- \mu - \lambda }) = \xi_{-\lambda}$. 
\end{prop}

\begin{proof}
The uniqueness of the map is clear, thanks to Lemma \ref{lem:bunxi}. 

We shall prove the existence of $\pi_{\lambda , \mu}$. 
Recall that any homomorphism of $\U$-modules  is naturally a homomorphism of $\Ui$-modules. 
Note that ${^{\omega}L}(\mu^{\inv}) = L(-w_0\mu^{\inv})= L(\mu)$. 
Let $\pi_{\lambda , \mu}$ be the composition of the following homomorphisms of $\Ui$-modules:
\[
\xymatrixcolsep{3pc}\xymatrix{
{^{\omega}L} (\mu^{\inv}+ \mu +\lambda) \ar^-{\chi}[r] 
\ar[rrdd]_{\pi_{\lambda , \mu}}
&  {^{\omega}L}(\mu^{\inv}+\mu)  \otimes  {^{\omega}L} (\lambda) \ar^-{\chi \otimes \id }[r] 
& {^{\omega}L}(\mu^{\inv}) \otimes  {^{\omega}L}(\mu)  \otimes  {^{\omega}L} (\lambda) \ar[d]^{\mc{T}\otimes \id \otimes \id} 
  \\
&&L(\mu) \otimes  {^{\omega}L}(\mu)  \otimes  {^{\omega}L} (\lambda) \ar^{\delta \otimes \id}[d]
  \\
&& {^{\omega}L}(\lambda)
}
\]
where $\mc{T}$ is the map from Lemma~ \ref{lem:uniqueT}.
First, we have 
\[
(\chi \otimes \id)\chi(\xi^{}_{-\mu^{\inv} - \mu - \lambda }) 
 = \xi^{}_{-\mu^{\inv}} \otimes \xi_{-\mu} \otimes \xi_{-\lambda}.
\]
Then applying $\mc{T} \otimes \id \otimes \id$, by Lemma~ \ref{lem:uniqueT} we have 
\begin{align*}
(\mc{T} \otimes  \id \otimes & \id )(\xi^{}_{-\mu^{\inv}} \otimes \xi_{-\mu} \otimes \xi_{-\lambda}) 
  \\
 &= \eta_{\mu} \otimes \xi_{-\mu} \otimes \xi_{-\lambda} + \sum_{1 \neq b \in \B(\mu)} g(1;b)
   b^-{\eta_{\mu}} \otimes \xi_{\mu} \otimes \xi_{-\lambda} .
\end{align*}
Applying $\delta \otimes 1$ to the above identity, we conclude that 
$$\pi_{\lambda , \mu}(\xi^{}_{- \mu^{\inv}- \mu - \lambda } ) = \xi_{-\lambda}.
$$ 
\end{proof}


\begin{lem}\label{lem:piBbar=Bbarpi}
Retain the notation in Proposition~\ref{prop:shrink}. The homomorphism $\pi_{\lambda , \mu}$ commutes
with the involution $\Bbar$; that is, 
$\pi_{\lambda , \mu} \Bbar = \Bbar \pi_{\lambda , \mu}$.
\end{lem}

\begin{proof}
In this proof, we write $\pi = \pi_{\lambda , \mu}$, $\xi = \xi^{}_{- \mu^{\inv} - \mu - \lambda }$, 
and $\xi' = \xi_{-\lambda}$.  
Then $\pi (\xi) = \xi'$ by Proposition~\ref{prop:shrink}.
An arbitrary element in ${^{\omega}L} (\mu^{\inv}+ \mu + \lambda )$ is of the form $u \xi$ for some $u \in \Ui$, 
by Lemma~ \ref{lem:bunxi}. 
Since $\xi$ and $\xi'$ are both $\Bbar$-invariant (see Remark~\ref{rem:xi inv}), we have
\[
\pi\Bbar(u\xi) = \pi\Bbar(u) (\xi) = \Bbar(u) \pi(\xi) = \Bbar(u) \xi'. 
\]
On the other hand, we have
\[
\Bbar\pi(u \xi) = \Bbar(u\xi') = \Bbar(u) \Bbar (\xi')= \Bbar(u) \xi'.
\]
The lemma is proved.
\end{proof}

\section{The $\imath$-canonical bases at rank one}  \label{subsec:rank1}

In this section we shall consider the rank 1 case of the algebra $\Ui$, i.e., 
$\Ui = \Qq[t]$,  the polynomial algebra in $t$.  
In order to simplify the notation, we shall write 
\[
E = E_{\alpha_0}, 
\quad
F = F_{\alpha_0}, 
\quad
\text{and } 
K = K_{\alpha_0}
\]
for the generators of $\U =\U_{q}({\mf{sl}_2})$.
By Proposition~ \ref{prop:embedding}, we have an algebra embedding 
$\imath: \Qq[t] \rightarrow \U_{q}({\mf{sl}_2})$ 
such that $\imath(t) = E + q F K^{-1}  + K^{-1}$. 

In the rank one case, $\Lambda^+$ can be canonically identified with ${\N}$. 
The finite-dimensional irreducible $\U$-modules  are of the form ${^{\omega}L}(s)$ of lowest weight $-s$,
with $s \in {\N}$. Recall \cite{Lu94}
the canonical basis of ${^{\omega}L}(s)$ consists of $\{E^{(a)} \xi_{-r} \mid 0 \le a \le s\}$.
We denote by ${^{\omega}\mc{L} (s)}$ the $\Z[q]$-submodule of ${^{\omega}L}(s)$ 
generated by $\{E^{(a)} \xi_{-s} \mid 0 \le a \le s\}$.
Also denote by ${^{\omega}L_\mA (s)}$ the $\mA$-submodule of ${^{\omega}L}(s)$ 
generated by $\{E^{(a)} \xi_{-s} \mid 0 \le a \le s\}$.

In the current rank one setting, we can write the intertwiner $\Upsilon = \sum_{k\ge 0}\Upsilon_{k}$, 
with $\Upsilon_{k} = \Upsilon_{k \alpha_0}= c_k F^{(k)}$ for $c_k \in \Qq$, and $c_0=1$. 

\begin{lem} \label{rank1:lem:upsiloninZ}
We have $\Upsilon_k \in \U^{-}_\mA$, for $k \ge 0$.
\end{lem}

\begin{proof}
It is equivalent to prove that $c_k \in \mA=\Z[q, q^{-1}]$ for all $k \ge  0$. The equation \eqref{eq:star} for $u=t$ implies that
\[
qF K^{-1} \Upsilon_{k-2} + K^{-1} \Upsilon_{k-1} + E \Upsilon_{k}=  q^{-1}\Upsilon_{k-2}F K + \Upsilon_{k-1}K  + \Upsilon_{k} E ,
\]
for all $k \ge 0$. Solving this equation, we have the following recursive formula for $c_k$:
\[
c_k = (-q^{k-1})\qq(q^{-1} [k-1]c_{k-2} + c_{k-1}), \quad \text{ for all } k \ge 1,
\]
where $c_{-1} = 0$ and $c_0 =1$. Then it follows by induction on $k$ that $c_k \in \mA$.
\end{proof}
One can show by the recursive relation in the above proof that
\begin{align}
\label{Ups1}
\Upsilon =\sum_{k \ge 0} 
q^{k(k+1)} \Big( \prod_{i=1}^k (q^{2i-1}-q^{1-2i})  F^{(2k)}   
+ \prod_{i=1}^{k+1} (q^{2i-1}-q^{1-2i})    F^{(2k+1)} \Big).
\end{align}

\begin{prop}  \label{rank1:BCB}
Let $s \in \N$. 
\begin{enumerate}
\item
The $\Ui$-module ${^{\omega}L}(s)$ admits a unique $\Qq$-basis 
$\B^\imath(s) = \{T^s_{a} \mid 0 \le a \le s \}$ which satisfies $\Bbar(T^s_a) =T^s_a$ and
\begin{equation}
T^s_{a} = E^{(a)} \xi_{-s} +\sum_{a' < a} t^s_{a;a'} E^{(a')} \xi_{-s}, 
\label{eq:Tsa}
\end{equation}
where $t^s_{a;a'} \in q\Z[q].$ (We also set $t^s_{a;a}=1$.)

\item
$\B^\imath(s)$ forms an $\mA$-basis for the $\mA$-lattice ${^{\omega} L_\mA (s)}$.

\item
$\B^\imath(s)$ forms a $\Z[q]$-basis for the $\Z[q]$-lattice ${^{\omega}\mc{L} (s)}$.
\end{enumerate}
\end{prop}
We call $\B^\imath(s)$  the  {\em $\imath$-canonical basis}  of
the $\Ui$-module ${^{\omega}L}(s)$.

\begin{proof}
Parts (2) and (3) follow immediately from (1) by noting \eqref{eq:Tsa}. 

It remains to prove (1).  Since $\Bbar = \Upsilon  \Abar$ and $ \Abar (E^{(a)} \xi_{-s})
 = E^{(a)} \xi_{-s}$, we have
\begin{equation*}
\Bbar(E^{(a)} \xi_{-s}) 
 = \Upsilon  (E^{(a)} \xi_{-s}) 
 = E^{(a)} \xi_{-s} +\sum_{a' < a} \rho^s_{a;a'} E^{(a')} \xi_{-s},
\end{equation*}
for some scalars $\rho^s_{a;a'} \in \mA$. 
As $\Bbar$ is an involution, Part~(1) follows by an application of   
\cite[Lemma~ 24.2.1]{Lu94} to our setting.
\end{proof}

\begin{lem}\label{rank1:lem:qZq}
Write $x \equiv x'$ if $x - x' \in q\; {^{\omega}\mc{L}(s)}$ with $s\in \N$. The $\Ui$-homomorphism 
$\pi= \pi_{s,1}:  {^{\omega}{L}(s+2)} \rightarrow {^{\omega}{L}(s)}$
from Proposition~ \ref{prop:shrink} satisfies that, for $a\ge 0$, 
\[
\pi(E^{(a)}\xi_{-s-2}) \equiv
\begin{cases}
E^{(a-1)}\xi_{-s} , & \text{ if } s= a-1;\\
E^{(a)} \xi_{-s} , & \text{ otherwise}.
\end{cases}
\]
\end{lem}

\begin{proof}
Recall Proposition~\ref{prop:shrink}, Proposition~\ref{prop:shrinkA},
and $\pi = (\delta \otimes \id) (\mc{T} \otimes \id \otimes \id) (\chi \otimes \id) \chi$.
It is easy to compute the action of $\mc{T}$ on ${^\omega L(1)} = L(1)$ is given by 
\[
\mc{T} (\xi_{-1}) = E\xi_{-1} -(q^{-1}-q) \xi_{-1} \quad \text{ and } \quad \mc{T} (E\xi_{-1}) = \xi_{-1}.
\]
For the map $\delta \otimes \id : {L(1)} \otimes {^\omega L(1)}
 \otimes {^\omega L(s)} \rightarrow {^\omega L(s)}$, it is easy to compute that 
\[
\delta (E\xi_{-1} \otimes \xi_{-1}) = 1, 
\quad
\delta(\xi_{-1} \otimes E\xi_{-1}) = -q, 
\]
and 
\[
\delta(\xi_{-1} \otimes \xi_{-1}) =\delta(E\xi_{-1} \otimes E\xi_{-1})=0.
\]

For the map $(\chi \otimes \id) \chi : {^\omega L(s+2)} \rightarrow {^\omega L(1)}
 \otimes {^\omega L(1)}
 \otimes {^\omega L(s)}$, we have 
\begin{align*}
&(\chi \otimes \id)\chi (E^{(a)}\xi_{-s-2}) \\
= &\sum_{a_1+a_2+a_3 = a}q^{-a_1a_2-a_1a_3-a_2a_3+a_1+sa_1+sa_2} E^{(a_1)} \xi_{-1} \otimes 
 E^{(a_2)} \xi_{-1}\otimes E^{(a_3)} \xi_{-s}
\\
= &  \xi_{-1} \otimes  \xi_{-1}\otimes E^{(a)} \xi_{-s}
 +  q^{-a+1+s}\xi_{-1} \otimes E \xi_{-1}\otimes E^{(a-1)} \xi_{-s}
\\
&+ q^{-a+2+s}E \xi_{-1} \otimes  \xi_{-1}\otimes E^{(a-1)} \xi_{-s}
 + q^{2s-2a+4}E \xi_{-1} \otimes E \xi_{-1}\otimes E^{(a-2)} \xi_{-s}.
\end{align*}
Then by applying $\mc{T} \otimes \id \otimes \id$, we have
\begin{align*}
&(\mc{T} \otimes \id \otimes \id) (\chi \otimes \id)\chi (E^{(a)}\xi_{-s-2}) 
\\
= & E \xi_{-1} \otimes  \xi_{-1}\otimes E^{(a)} \xi_{-s} - (q^{-1}-q) \xi_{-1} \otimes  \xi_{-1}\otimes E^{(a)} 
 \xi_{-s}
\\
&+ q^{-a+2+s}E \xi_{-1} \otimes  E\xi_{-1}\otimes E^{(a-1)} \xi_{-s} 
\\
& - q^{-a+1+s}(q^{-1}-q) \xi_{-1} 
 \otimes  E\xi_{-1}\otimes E^{(a-1)}\xi_{-s}
\\
&+ q^{-a+1+s} \xi_{-1} \otimes  \xi_{-1}\otimes E^{(a-1)} \xi_{-s}
 + q^{2s-2a+4} \xi_{-1} \otimes E \xi_{-1}\otimes E^{(a-2)} \xi_{-s}.
\end{align*}
At last, by applying $\delta \otimes 1$, we have 
\begin{align*}
&\pi (E^{(a)}\xi_{-s-2})
 \\
= &E^{(a)}\xi_{-s} + 0 + 0 +q^{-a+2+s}(q^{-1}-q)E^{(a-1)}\xi_{-s} + 0 - q^{2s-2a+5}E^{(a-2)}\xi_{-s}
 \\
= &E^{(a)}\xi_{-s} + q^{-a+1+s}E^{(a-1)}\xi_{-s}   -q^{-a+3+s}E^{(a-1)}\xi_{-s}- q^{2s-2a+5}E^{(a-2)}\xi_{-s}.
\end{align*}

The lemma follows.
\end{proof}

We adopt the convention that $T^s_a = 0$ if $ s < a$.  

\begin{prop}\label{rank1:prop:BCBcompatibility}
The homomorphism 
$\pi= \pi_{s,1}:  {^{\omega} L(s+2)} \rightarrow {^{\omega} L(s)}$ 
sends $\imath$-canonical basis elements to $\imath$-canonical basis elements or zero.  
More precisely, we have
\[
\pi(T^{s+2}_{a}) = 
\begin{cases}
T^{s}_{a-1}, & \text{ if } s = a - 1 ;\\
T^{s}_{a},  & \text{ otherwise}.
\end{cases}
\]
\end{prop}

\begin{proof}
By Proposition~\ref{rank1:BCB} and Lemma~ \ref{rank1:lem:qZq}, 
the difference of the two sides of the identity in the proposition lies in $q \,{^{\omega}\mc{L}(s)}$
and hence is a $q\Z[q]$-linear combination of $\B^\imath(s)$. 
Lemma~ \ref{lem:piBbar=Bbarpi} implies that such a difference is fixed by the anti-linear involution $\Bbar$ and
hence  it must be zero. The proposition follows.
\end{proof}

\begin{lem}\label{rank1:lem:degree}
Let $f(t) \in \Ui =\Qq[t]$ be nonzero. Then $f(t)\xi_{-s}  \neq 0$ for  all $s \ge \deg f$. 
\end{lem}

\begin{proof}
We write $\xi = \xi_{-s}$. Write $a=\deg f$, and $f(t) = \sum^a_{i = 0} c_i t^i$ with $c_a \neq 0$. 
Then $\imath(f(t))= c_a E^a + x$, where $x$ is a linear combination of elements in $\U$ with weights 
lower than that of $E^a$. It follows that 
\[
f(t) \xi = c_a E^a \xi + x \xi \neq 0 \quad \text{ for }s\ge a, 
\]
since 
$c_a E^a \xi \neq 0$ and it cannot be canceled out by $x \xi$ for weight reason.  
\end{proof}

\begin{prop}  \label{rank1:prop:BCBodd} 
There exists a unique $\Qq$-basis $\{T^{\rm odd}_a \mid a\in \N\}$ of $\Ui =\Qq[t]$ 
with $\deg T^{\rm odd}_a =a$ such that
\begin{eqnarray}
 \label{eq:Todd}
T^{\rm odd}_a \xi_{-s} = 
\begin{cases}
T^{s}_{a-1}, & \text{ if } s = a-1;\\
T^{s}_a, &\text{ otherwise},
\end{cases}
\end{eqnarray}
for each $s \in 2{\N} +1$.
Moreover, we have
$
\ov{T^{\rm odd}_a} = T^{\rm odd}_a$.
\end{prop}

\begin{proof}
By going over carefully the proof of Lemma \ref{lem:bunxi} in the rank one case, 
we can prove the following refinement of Lemma~ \ref{lem:bunxi}:

($\heartsuit_{a}^{s}$) \;\;
{\em Whenever $s  \ge a$, there exists a {\em unique} element $T_a(s) \in \Ui =\Qq[t]$ {\em of degree $a$} such that 
$T_a(s) \xi_{-s} = T^s_a$. }

Let $s \ge a$ and take $l \ge 0$. 
Since $\pi_{s,l}$ is a $\Ui$-homomorphism with $\pi_{s ,l}(\xi_{-(s+2l)}) =\xi_{-s}$ (see Proposition~\ref{prop:shrink}), we have
by Proposition~ \ref{rank1:prop:BCBcompatibility} 
\begin{align*}
T_a (s+2l) \xi_{-s} &=\pi_{s ,l}(T_a(s+2l)\xi_{-(s+2l)}) 
  \\
 & \stackrel{\heartsuit_{a}^{s+2l}}{=}   \pi_{s ,l} (T^{s+2l}_a) 
 = T^s_a 
   \stackrel{\heartsuit_{a}^{s}}{=}   T_a(s) \xi_{-s}.
\end{align*}
Hence $T_a(s+2l) = T_a(s)$ for all $l \ge 0$ and $s \ge a$, thanks to the uniqueness of $T_a(s)$ in ($\heartsuit_{a}^{s}$). Hence,
\[
T^{\rm odd}_a := \lim_{l \mapsto \infty} T_a(1+2l) \in \Ui  
\]
is well defined. 
It follows by Proposition~ \ref{rank1:prop:BCBcompatibility} that $T^{\rm odd}_a$ satisfies \eqref{eq:Todd}. 

We now show that $T^{\rm odd}_a$  is unique (for a given $a$). 
Let $'T^{\rm odd}_a$ be another such element satisfying \eqref{eq:Todd}. 
Then $(T^{\rm odd}_a -{} 'T^{\rm odd}_a) \xi_{-s} = 0$ for all $s \in 2{\N}+1$. 
It follows  by Lemma~ \ref{rank1:lem:degree} that ${T}^{\rm odd}_a = {}'T^{\rm odd}_a$.  

Applying $\Bbar$ to both sides of \eqref{eq:Todd} and using 
Corollary~\ref{cor:Li-invol}, we conclude that $\ov{T^{\rm odd}_a}$ satisfies \eqref{eq:Todd} as well. 
Hence by the uniqueness we have $T^{\rm odd}_a = \ov{T^{\rm odd}_a}$.  
\end{proof}

A similar argument gives us the following proposition.
\begin{prop} 
There exists a unique $\Qq$-basis $\{T^{\rm ev}_a \mid a\in \N\}$ of $\Ui =\Qq[t]$ 
with $\deg T^{\rm ev}_a =a$ such that
\[
T^{\rm ev}_a \xi_{-s} = 
\begin{cases}
T^{s}_{a-1}, & \text{ if } a = s+1;\\
T^{s}_a, &\text{ otherwise},
\end{cases}
\]
for each $s \in 2{\N}$.
Moreover, we have
$\ov{T^{\rm ev}_a} = T^{\rm ev}_a$.
\end{prop}

Clearly we have $T^{\rm odd}_0 = T^{\rm ev}_0 =1$. It is also easy to see that
$T^{\rm odd}_a$ and $T^{\rm ev}_a$ for $a\ge 1$ are both of the form 
\begin{equation}\label{eq:dividedt}
\frac{t^a}{[a]!} + g(t), \quad \text{ where }  \deg g < a.
\end{equation}

We have the following conjectural formula (which is not needed in this paper).
\begin{conjecture}
For $a\in \N$, we have 
\begin{align*}
T^{\rm odd}_{2a} &= {t(t - [-2a+2])(t-[-2a+4]) \cdots (t-[2a-4]) (t - [2a-2]) \over [2a]!},\\
T^{\rm odd}_{2a+1} & = {  (t - [-2a])(t-[-2a+2]) \cdots (t-[2a-2]) (t - [2a])  \over [2a+1]!},\\
T^{\rm ev}_{2a} &= { (t - [-2a+1] ) (t - [-2a+3]) \cdots (t - [2a-3]) (t-[2a-1]) \over [2a]!},\\
T^{\rm ev}_{2a+1} &= { t(t - [-2a+1] ) (t - [-2a+3]) \cdots (t - [2a-3]) (t-[2a-1]) \over [2a+1]!}.
\end{align*}
\end{conjecture}

\section{Integrality at rank one}

\begin{lem}  \label{rank1:lem:pi-}
Let  $s, l \in \N$.
\begin{enumerate}
\item There exists a unique homomorphism of $\Ui$-modules 
\[{\pi}^- = {\pi}^-_{s,l}: {^{\omega}L}(s+2l) \longrightarrow L(l) \otimes {^{\omega}L}(s+l) 
\]
such that ${\pi}^-(\xi_{-s-2l}) = \eta_{l} \otimes \xi_{-s-l}$.

\item ${\pi}^-$ induces a homomorphism of $\mA$-modules 
\[
{\pi}^- = {\pi}^-_{s,l} : {^{\omega}L}_\mA(s+2l) \longrightarrow L_\mA(l) \otimes {^{\omega}L}_\mA(s+l).
\]
\end{enumerate}
\end{lem}

\begin{proof}
The uniqueness of such a homomorphism is clear, since $\Ui \xi_{-s-2l} = {^{\omega}L}(s+2l)$ by Lemma \ref{lem:bunxi}. 

We let ${\pi}^- = \mc{T}^{-1} \chi $ be the composition of the $\Ui$-homomorphisms
\[
\xymatrix{ {^{\omega}L}(s +2l) \ar[r]^-{\chi} & {^{\omega}L}(l) \otimes {^{\omega}L}(s+l) \ar[r]^-{\mc{T}^{-1} \otimes 1} 
 &L(l) \otimes {}^{\omega}L}(s+l),
\]
where $\chi$ is the $\Ui$-homomorphism from Proposition~\ref{prop:shrinkA} and 
$\mc{T} = \Upsilon \circ \widetilde{\zeta} \circ T_{w_0}$ is the $\Ui$-homomorphism from Theorem~ \ref{thm:mcT}. 
As the automorphism $T_{w_0}$ preserves the $\mA$-forms, we can choose the weight function $\zeta$ 
in \eqref{eq:zeta} with suitable value $\zeta(l) \in q^\Z$
such that 
$
T_{w_0}^{-1}  \widetilde{\zeta}^{-1}  (\xi_{-l}) = \eta_{l}.
$
It follows by \eqref{eq:zeta} that $\zeta$ must be $\mA$-valued.
Then ${\pi}^- = \mc{T}^{-1} \chi $ is the map satisfying (1) since $\chi (\xi_{-s-2l}) =\xi_{-l} \otimes \xi_{-s-l}$. 

By Proposition~\ref{prop:shrinkA} $\chi$ maps ${^{\omega}L}_\mA(s+2l)$ to  $L_\mA(l) \otimes {^{\omega}L}_\mA(s+l)$. 
It is also well known that $T_{w_0}$ is an automorphism of 
the $\mA$-form  ${^{\omega}L}_\mA(l)$.   By Lemma~ \ref{rank1:lem:upsiloninZ},  
$\Upsilon^{-1} = \ov{\Upsilon}$ preserves the $\mA$-form ${^{\omega}L}_\mA(l)$ as well. 
As a composition of all these maps, 
$\pi^- =(\Upsilon \circ \widetilde{\zeta} \circ T_{w_0})^{-1} \chi $ preserves the  $\mA$-forms, whence (2). 
\end{proof}

The following lemma is a variant of Lemma~\ref{rank1:lem:pi-} and can be proved in the same way. 

\begin{lem}\label{rank1:lem:pi+}
Let $s, l \in \N$. 
\begin{enumerate}
\item There exists a unique homomorphism of $\Ui$-modules 
\[{\pi}^+ = {\pi}^+_{s,l}: {^{\omega}L}(s+2l) \longrightarrow L(s+l) \otimes {^{\omega}L}(l), 
\]
such that ${\pi}^+(\xi_{-s-2l}) = \eta_{s+l} \otimes \xi_{-l}$.

\item ${\pi}^+$ induces a homomorphism of $\mA$-modules   
\[
{\pi}^+ : {^{\omega}L}_\mA(s+2l) \longrightarrow L_\mA(s+l) \otimes {^{\omega}L}_\mA(l).
\]
\end{enumerate}
\end{lem}

Recall  that a modified $\Qq$-algebra $\Udot$ as well as its $\mA$-form 
$\Udot_\mA$ are defined in \cite[Chapter~ 23]{Lu94}.  
Any finite-dimensional unital $\Udot$-module is naturally a weight $\U$-module, 
and vice versa (see \cite[23.1.4]{Lu94}). 
In the rank one setting, $\Udot$ (or $\Udot_\mA$) is generated by $E, F$ 
and the idempotents $\one_s$ for $s\in \Z$. 
As $\Udot$ is naturally a $\U$-bimodule,  $\imath(T^{\rm odd}_a) \one_s$ and $\imath(T^{\rm ev}_a) \one_s$  
make sense as elements in $\Udot \one_s$, for $a\in \N$ and $s \in \Z$.    

\begin{prop}  \label{rank1:CBinZ}
\begin{enumerate}
\item We have $\imath(T^{\rm odd}_a) \one_{s} \in  \dot{\U}_\mA$, for all $a \in \N$, $s \in 2\Z+1$.
\item We have $\imath(T^{\rm ev}_a) \one_{s} \in  \dot{\U}_\mA$, for all $a \in \N$, $s \in 2\Z$.
\end{enumerate}
\end{prop}

\begin{proof}
 (1). Let $s \in 2\N +1$. Fix an arbitrary $a \in \N$. 
Recall Lusztig's canonical basis $\{b \Diamond_{-s} b'\}$ 
of $\dot{\U} 1_{-s}$ in \cite[Theorem~25.2.1]{Lu94}. We write
\[
\imath(T^{\rm odd}_a) \one_{-s} = \sum_{b,b'} c_{b,b'} b \Diamond_{-s} b',
\] 
for some scalars $c_{b,b'}  \in \Qq$. 
 Consider the map 
\[
\pi^- : {^{\omega}L}_\mA(s+2l) \longrightarrow L_\mA(l) \otimes {^{\omega}L}_\mA(s+l)
\] 
in Lemma~ \ref{rank1:lem:pi-} for all $l \ge 0$. We have $T^{\rm odd}_a \xi_{-s-2l} \in {^{\omega}L}_\mA(s+2l)$ 
by Propositions~ \ref{rank1:BCB} and \ref{rank1:prop:BCBodd}. Therefore we have 
\begin{align*}
\imath(T^{\rm odd}_a) \one_{-s}(\eta_{l} \otimes \xi_{-s-l}) 
&=  T^{\rm odd}_a (\eta_{l} \otimes \xi_{-s-l}) 
\\
&= \pi^- (T^{\rm odd}_a \xi_{-s-2l}) 
\in L_\mA(l) \otimes {^{\omega}L}_\mA(s+l).  
\end{align*}
Hence we have (in Lusztig's notation \cite[Theorem~25.2.1]{Lu94})
\[
\sum_{(b,b')}  c_{b,b'} (b \Diamond b')_{l,s+l} = \imath(T^{\rm odd}_a) \one_{-s}(\eta_{l} \otimes \xi_{-s-l}) 
\in  L_\mA(l) \otimes {^{\omega}L}_\mA(s+l).
\]
Since this holds for all $l$ and $(b \Diamond b')_{l,s+l} \neq 0$ for $l\gg 0$, all $c_{b,b'}$ must belong to $\mA$. 
Hence $\imath(T^{\rm odd}_a) \one_{-s} \in \dot{\U}_\mA$. 

By considering the map 
\[
\pi^+ : {^{\omega}L}_\mA(s+2l) \longrightarrow L_\mA(s+l) \otimes {^{\omega}L}_\mA(l)
\]
in Lemma \ref{rank1:lem:pi+} for all $l \ge 0$, we can show that $\imath(T^{\rm odd}_a) \one_{s} \in \dot{\U}_\mA$ 
for $s \in 2\N +1$ in a similar way. This proves (1).
The proof of (2)  is similar and will be skipped.
\end{proof}


\section{The integrality of $\Upsilon$}   \label{subsec:UpsiloninZ:proof}

Back to the general higher rank case, we
are now ready to prove the following crucial lemma with the help of Proposition~\ref{rank1:CBinZ}.

\begin{lem}\label{lem:UppreZ}
For each $\lambda \in \Lambda^+$, 
we have $\Upsilon ({^{\omega}L}_\mA(\lambda)) \subseteq {^{\omega}L}_\mA(\lambda)$.
\end{lem}

\begin{proof}
We write $\xi = \xi_{-\lambda}$.
We shall prove that $\Upsilon x \in {^{\omega}L}_\mA(\lambda)$ by induction on the height $\hgt (\mu+\la)$,
for an arbitrary weight vector $x \in {^{\omega}L}_\mA(\lambda)_\mu$. 
It suffices to consider $x$ of the form 
$x = E^{(a_1)}_{\alpha_{i_1}} E^{(a_2)}_{\alpha_{i_2}} \cdots E^{(a_s)}_{\alpha_{i_s}} \xi$
which is $\Abar$-invariant. 

The base case when $\hgt (\mu+\la)=0$ is clear, since $x=\xi$ and $\Upsilon \xi =\xi$.

Denote $x' = E^{(a_2)}_{\alpha_{i_2}} \cdots E^{(a_s)}_{\alpha_{i_s}} \xi  \in {^{\omega}L}_\mA(\lambda)$,
and so $x = E^{(a_1)}_{\alpha_{i_1}} x'$. 
The induction step is divided into three cases depending on whether $i_1 > 0$, $i_1 < 0$, or $i_1 = 0$.
Recall that, for any $u\in \Ui$, the actions of $u$ and $\imath(u)$ on ${^{\omega}L}_\mA(\lambda)$ are the same by definition. 

(1)
Assume that $i_1 >0$ (i.e., $i_1 \in \Ihf$). 
Replacing $E^{(a_1)}_{\alpha_{i_1}}$ in the expression of $x$ by $\be^{(a_1)}_{\alpha_{i_1}}$,
we introduce a  new element $x'' = \be^{(a_1)}_{\alpha_{i_1}} x'$ 
which lies in ${^{\omega}L}_\mA(\lambda)$ thanks to \eqref{eq:beZ}.
Then $y := x'' -x \in {^{\omega}L}_\mA(\lambda)$ is a linear combination of elements of weights 
lower than the weight of $x$. 

We shall consider $\Bbar( x'')$ in two ways. By Corollary~\ref{cor:Li-invol}, 
${^{\omega}L}_\mA(\lambda)$ is $\imath$-involutive. Since
$\be^{(a_1)}_{\alpha_{i_1}}$ is $\Bbar$-invariant and $\Bbar =\Upsilon \psi$, we have
\[
\Bbar( x'') =\Bbar (\be^{(a_1)}_{\alpha_{i_1}} x')
=\be^{(a_1)}_{\alpha_{i_1}} \Bbar(x') = \be^{(a_1)}_{\alpha_{i_1}} \Upsilon \Abar( x').
\]
It is well known (cf. \cite{Lu94})  that $\Abar$ preserves ${^{\omega}L}_\mA(\lambda)$, and so 
$\Abar(x') \in {^{\omega}L}_\mA(\lambda)$. Since $\Abar(x')$ has weight lower than $x$,
we have $\Upsilon \Abar(x') \in {^{\omega}L}_\mA(\lambda)$ by the induction hypothesis. 
Equation~\eqref{eq:beZ} implies that $\Bbar( x'') = \be^{(a_1)}_{\alpha_{i_1}} \Upsilon \Abar( x') \in {^{\omega}L}_\mA(\lambda)$. 

On the other hand, we have
\[
\Bbar( x'') = \Bbar(x) + \Bbar(y) = \Upsilon \Abar(x) + \Upsilon \Abar (y) = \Upsilon x + \Upsilon \Abar (y) .
\]
Since $\Abar (y) \in {^{\omega}L}_\mA(\lambda)$ has weight lower than $x$, 
we have $\Upsilon \Abar (y) \in {^{\omega}L}_\mA(\lambda)$ by 
the induction hypothesis. Therefore we conclude that $\Upsilon x =\Bbar( x'')-\Upsilon \Abar (y)   \in {^{\omega}L}_\mA(\lambda)$.

(2)
Assume that $i_1 < 0$. In this case, replacing $E^{(a_1)}_{\alpha_{i_1}}$ in the expression of $x$ 
by $\bff^{(a_1)}_{\alpha_{-i_1}}$ instead, 
we consider a  new element $x'' = \bff^{(a_1)}_{\alpha_{i_1}} x'$ which also lies in ${^{\omega}L}_\mA(\lambda)$ by \eqref{eq:bffZ}. 
Then an argument parallel to (1) shows that $\Upsilon x \in {^{\omega}L}_\mA(\lambda)$. 

(3)
Now consider the case where $i_1 =0$.
Set $\beta = \sum^s_{p=2} a_i \alpha_{i_p} -\la$. We decide into two subcases (i)-(ii), depending on 
whether $(\alpha_0, \beta)$ is odd or even.

\underline{Subcase~(i)}.  Assume that $(\alpha_0, \beta)$ is an odd integer.
Replacing $E^{(a_1)}_{\alpha_{i_1}}$ in the expression of $x$ 
by the element $T^{\rm odd}_{a_1}$ defined in Proposition~\ref{rank1:prop:BCBodd}, we introduce a new element 
$x'' = T^{\rm odd}_{a_1} x'$, which belongs to 
${^{\omega}L}_\mA(\lambda)$ by Proposition~\ref{rank1:CBinZ} 
(as we can write $x'' =T^{\rm odd}_{a_1} \one_{(\alpha_0,\beta)}x'$). Thanks to \eqref{eq:dividedt}, 
$y:=x''-x \in {^{\omega}L}_\mA(\lambda)$ is a linear combination of elements of weights lower than $x$.  
Then similarly as in case $(1)$, we have 
\[
\Bbar( x'')  =\Bbar(T^{\rm odd}_{a_1} x')
 = T^{\rm odd}_{a_1}  \Bbar( x')
= T^{\rm odd}_{a_1} \Upsilon \Abar( x').
\]
As in (1), we have $\Upsilon \Abar (x') \in {^{\omega}L}_\mA(\lambda)$. 
Recall from Theorem~ \ref{thm:Upsilon} that  $\Upsilon = \sum_{\mu} \Upsilon_\mu$, 
where $\Upsilon_\mu \neq 0$ only if $\mu^{\inv} = \mu$. 
Note that $(\alpha_0, \mu)$ must be an even integer if $\mu^{\inv} = \mu$. 
Hence $(\alpha_0, \mu + \beta)$ is always odd whenever $\mu^{\inv} = \mu$. Therefore by 
Proposition~ \ref{rank1:CBinZ}, we have 
$$
\Bbar( x'') = T^{\rm odd}_{a_1} \Upsilon \Abar( x') 
= \sum_{\mu: \mu^\inv =\mu} T^{\rm odd}_{a_1} \one_{(\alpha_0, \mu + \beta)} \Upsilon_\mu \Abar( x') 
\in {^{\omega}L}_\mA(\lambda).
$$

Now by the induction hypothesis we have $\Upsilon \Abar (y)  \in {^{\omega}L}_\mA(\lambda)$,
and hence $\Upsilon x = \Bbar( x'')-\Upsilon \Abar (y)  \in {^{\omega}L}_\mA(\lambda)$.

\underline{Subcase~(ii)}.  Assume that $(\alpha_0, \beta)$ is an even integer.
In this subcase, we replace $E^{(a_1)}_{\alpha_{i_1}}$ by $T^{\rm ev}_{a_1}$.
The rest of the argument is the same as Subcase (i) above.

This completes the induction and the proof of the lemma.
\end{proof}

\begin{thm}   \label{thm:UpsiloninZ}
We have $\Upsilon_\mu \in \U^-_\mA$, for all $\mu \in \N\Pi$.
\end{thm}

\begin{proof}
Recall Lusztig's canonical basis $\B$ of $\f$ in Section~ \ref{subsec:CB} with $\B_\mu = \B \cap \f_\mu$. We write 
$\Upsilon_\mu = \sum_{b \in \B_\mu}c_b b^{-}$ for some scalars $c_b \in \Qq$.
By Lemma~ \ref{lem:UppreZ}, we have
\[
\Upsilon_\mu \eta_{\lambda} =\sum_{b \in \B_\mu}c_b b^{-}\eta_{\lambda} \in L_\mA(\lambda), \quad \text{ for all } \lambda \in \Lambda^+.
\]
For an arbitrarily fixed $b\in \B_\mu$, $b^{-}\eta_{\lambda} \neq 0$ for $\la$ large enough,
and hence we must have $c_b \in \mA$.  Therefore  $\Upsilon_\mu \in \U^-_\mA$.
\end{proof}

\section{The $\imath$-canonical basis of a based $\U$-module
} \label{subsec:CBandDCB} 

By Corollary~\ref{cor:Li-invol}, ${^{\omega}L}(\lambda)$ for $\la \in \Lambda^+$ 
is an $\imath$-involutive $\Ui$-module with involution $\Bbar =\Upsilon \Abar$.

\begin{lem}\label{lem:Zform}
The bar map $\Bbar$ preserves the $\mA$-form ${^{\omega}L}_\mA(\lambda)$, for $\la \in \Lambda^+$.
\end{lem}
\begin{proof}
It is well known (cf. \cite{Lu94})  that $\Abar$ preserves ${^{\omega}L}_\mA(\lambda)$. 
As ${^{\omega}L}_\mA(\lambda)$ is preserved by $\Upsilon$ by Lemma \ref{lem:UppreZ}, 
it is also preserved by $\Bbar =\Upsilon \Abar$.
\end{proof}

Define a partial ordering $\preceq$ on the set $\B(\la)$ of canonical basis for $\la\in\Lambda^+$ as follows: 
\begin{equation}  \label{eq:order}
\begin{split}
b_1 \preceq b_2 
\text{ if the images of } |b_1|,  |b_2| \text{ are the same in } \Lambda_{\inv}
\\
\text{ and } |b_2| - |b_1| \in \N \Pi.
\end{split}
\end{equation} 
(Recall  that $|b|$ denotes the weight of $b$ as in \S\ref{subsec:f}). 

For any $b \in \B(\lambda)$, we have 
\begin{equation}  \label{eq:triang}
\Bbar(b^+ \xi_{-\lambda}) = \Upsilon  \Abar (b^+ \xi_{-\lambda} ) 
 = \Upsilon  (b^+ \xi_{-\lambda} ) 
 = \sum_{b' \in \B(\lambda)} \rho_{b; b'} b'^+ \xi_{-\lambda}, 
\end{equation}
where $\rho_{b; b'} \in \mA$ by Theorem~\ref{thm:UpsiloninZ}. 
Since $\Upsilon$ lies in a completion of $\U^-$ satisfying $\Upsilon_\mu =0$ unless $\mu^\inv =\mu$ 
(see Theorem~\ref{thm:Upsilon}), 
we have  $ \rho_{b; b} =1$ and $\rho_{b; b'} = 0$
unless $b' \preceq b$.
As $\Bbar$ is an involution, we can apply 
\cite[Lemma~ 24.2.1]{Lu94} to our setting to establish the following theorem,
which is a generalization of Proposition~\ref{rank1:BCB} in the rank one case. 

\begin{thm}\label{thm:BCB}
Let $\la \in \La^+$. 

\begin{enumerate}
\item
The $\Ui$-module ${^{\omega}L}_\mA(\lambda)$ admits a unique basis 
\[
\B^\imath(\lambda) := \{T^{\lambda}_{b} \mid b \in \B(\lambda)\}
\]
which is $\Bbar$-invariant and of the form
\[
T^{\lambda}_{b} = b^+ \xi_{-\lambda} +\sum_{b' \prec b}
t^{\lambda}_{b;b'} b'^+  \xi_{-\lambda},
\quad \text{ for }\;  t^{\lambda}_{b;b'} \in q\Z[q].
\]

\item
$\B^\imath(\la)$ forms an $\mA$-basis for the $\mA$-lattice ${^{\omega} L_\mA (\la)}$.

\item
$\B^\imath(\la)$ forms a $\Z[q]$-basis for the $\Z[q]$-lattice ${^{\omega}\mc{L} (\la)}$.
\end{enumerate}
\end{thm}

\begin{definition}
$\B^\imath(\lambda)$ is called the $\imath$-canonical basis of
the $\Ui$-module ${^{\omega}L(\lambda)}$.
\end{definition}

\begin{rem}
The $\imath$-canonical basis $\B^\imath(\lambda)$ is not homogenous in terms of the weight lattice $\Lambda$,
though it is homogenous in terms of $\Lambda_\theta$. 
\end{rem}

\begin{rem}
Lusztig's canonical basis $\B (\la)$ is computable algorithmically. 
As $\Upsilon$ is constructed recursively in \S\ref{sec:proofUpsilon}, 
there is an algorithm to compute
the structure constants $\rho_{b; b'}$ in \eqref{eq:triang} and then $t^{\lambda}_{b;b'}$.
\end{rem}

Set $t^{\lambda}_{b;b}= 1$, and $t^{\lambda}_{b;b'}= 0$ if $b$, $b' \in \B(\lambda)$ satisfy $b' \npreceq b$.

\begin{conjecture} [Positivity of $t^{\lambda}_{b;b'}$]
 \label{conj:N}
We have $t^{\lambda}_{b;b'} \in {\N}[q]$, for $b,b' \in \B(\la)$.
\end{conjecture}
One can hope for similar positivity
in the general setting of based $\U$-modules below.

Recall \cite[Chapter 27]{Lu94} has developed a theory of finite-dimensional based $\U$-modules $(M,B)$ (for general quantum groups $\U$ of finite type).
The basis $B$ generates a $\Z[q]$-submodule $\mc{M}$ and an $\mA$-submodule ${}_\mA M$ of $M$.
 Applying the same argument for Theorem~\ref{thm:BCB} above, we have established the following.

\begin{thm}   \label{thm:iCBbased}
Let $(M,B)$ be a finite-dimensional based $\U$-module. 
 
\begin{enumerate}
\item
The $\Ui$-module $M$ admits a unique basis (called $\imath$-canonical basis)
$
B^\imath  := \{T_{b} \mid b \in B \}
$
which is $\Bbar$-invariant and of the form
\begin{equation} \label{iCB}
T_{b} = b +\sum_{b' \in B, b' \prec b}
t_{b;b'} b',
\quad \text{ for }\;  t_{b;b'} \in q\Z[q].
\end{equation}

\item
$B^\imath$ forms an $\mA$-basis for the $\mA$-lattice ${}_\mA M$, and
$B^\imath$ forms a $\Z[q]$-basis for the $\Z[q]$-lattice $\mc{M}$.
\end{enumerate}
\end{thm}   

Recall that a tensor product of finite-dimensional simple $\U$-modules is a based $\U$-module
by \cite[Theorem~27.3.2]{Lu94}. Theorem~\ref{thm:iCBbased} implies now the following.

\begin{thm}  \label{thm:iCBtensor}
Let $\la_1, \ldots, \la_r \in \La^+$. The tensor product of finite-dimensional simple $\U$-modules
${^{\omega} L (\la_1)}  
\otimes \ldots \otimes {^{\omega} L (\la_r)}$
admits a unique $\Bbar$-invariant basis of the form \eqref{iCB} (called $\imath$-canonical basis).
\end{thm}
 
 \begin{rem}
 One can hope for a positivity similar to Conjecture \ref{conj:N}
in the above general setting of tensor product $\U$-modules.
\end{rem}

\chapter{The $(\Ui, \mc{H}_{B_m})$-duality and compatible bar involutions} \label{sec:HeckeB}

In this chapter, we recall Schur-Jimbo duality between quantum group $\U$ and Hecke algebra of type $A$.
Then we establish a duality between $\Ui$ and Hecke algebra $\HBm$ of type $B$ acting on $\VV^{\otimes m}$,
and show the existence of a bar involution on $\VV^{\otimes m}$ which is compatible with the bar involutions
on $\Ui$ and $\HBm$. 
This allows a reformulation of Kazhdan-Lusztig theory for Lie algebras of type $B/C$ via the 
involutive $\Ui$-module $\VV^{\otimes m}$. 

\section{Schur-Jimbo duality}

Recall the notation $\I_{2r}$ from \eqref{eq:I}, and we set 
$$I =\I_{2r+2} 
= \Big\{-r-\hf, \ldots, -\frac32, -\hf, \hf, \frac32, \ldots, r+\hf \Big\}.
$$ 
Let the $\Qq$-vector space $\VV := \sum _{a \in I}\Qq v_{a} $ be the natural representation of $\U$. 
We shall call $\VV$ the natural representation of $\Ui$ (by restriction) as well. 
For $m \in {\Z_{> 0}}$, the tensor space
$\VV^{\otimes m}$ is naturally a $\U$-module (and a $\Ui$-module) via the coproduct $\Delta$. The $\U$-module
$\VV$ is involutive with $\Abar$ defined by
\[
\Abar(v_a) :=v_a, \quad \text{ for all } a \in I.
\]
Then  $\VV^{\otimes m}$ is an involutive  $\U$-module and hence an $\imath$-involutive $\Ui$-module by 
Proposition~\ref{prop:compatibleBbar} and Remark~\ref{rem:sameinv}.

We view $f \in I^m$ as a function $f: \{1, \ldots, m\} \rightarrow I$. 
For any $f \in I^m$, we define
\[
M_f := v_{f(1)} \otimes \cdots \otimes v_{f(m)}.
\]
Then $\{ M_f \mid f\in I^m\}$ forms a basis for $\VVm$. 

Let $W_{B_m}$ be the Coxeter groups of type $\text{B}_m$ with simple reflections $s_j, 0 \leq j \leq m-1$,
where the subgroup  generated by $s_i$, $1\leq i \leq m-1$ is isomorphic to $W_{A_{m-1}} \cong \mathfrak{S}_m$. 
The group  $W_{B_m}$ and its subgroup $S_m$
act naturally on $I^m$ on the right as follows: for any $f \in I^m$, $ 1 \leq i \leq m$, we have 
\begin{equation}  \label{eq:rightW}
f \cdot s_j = 
 \begin{cases}
 (\dots, f(j+1), f(j), \dots) , &\text{if }  j > 0,\\
 (-f(1), f(2), \dots, f(m)),& \text{if } j =0.
 \end{cases}
 \end{equation}
 
Let $\mathcal{H}_{B_m}$ be the Iwahori-Hecke algebra of type $B_m$ over $\mathbb Q(q)$. 
It is generated by $H_0, H_1, H_2, \dots , H_{m-1}$, subject to the following relations,
\begin{align*}
(H_i -q^{-1})(H_i +q) &= 0,   & \text{for } i \geq 0, &
 \\
H_i H_{i+1} H_i &= H_{i+1} H_i H_{i+1}, & \text{for } i> 0,&
 \\
H_i H_j &= H_j H_i, & \text{for } |i-j| >1, &
\\
H_0H_1H_0H_1&=H_1H_0H_1H_0. &
\end{align*}
Associated to $\sigma \in W_{B_m}$ with a reduced
expression $\sigma=s_{i_1} \cdots s_{i_k}$, we define
 $
H_\sigma :=H_{i_1} \cdots H_{i_k}. $ The bar involution on $\mathcal
H_{B_m}$ is the unique anti-linear automorphism defined by
$\overline{H_\sigma} =H_{\sigma^{-1}}^{-1}, \overline{q} =q^{-1},$
for all $\sigma \in W_{B_m}$.

There is a right action of the Hecke algebra $\mathcal{H}_{B_m}$  on the $\Qq$-vector space $\VV^{\otimes m}$ as follows: 
\begin{equation}  \label{eq:HBm}
 M_f H_a=
 \begin{cases}
 q^{-1} M_f, & \text{if $a>0, f(a) = f(a+1)$};\\
 M_{f \cdot  s_a}, & \text{if $a > 0, f(a) < f(a+1)$};\\
 M_{f \cdot  s_a} + (q^{-1} - q)M_{f}, & \text{if $a > 0, f(a) > f(a+1)$};\\
 M_{f \cdot  s_0}, & \text{if $a = 0 , f(1) > 0 $};\\
 M_{f \cdot s_0} + (q^{-1} - q)M_{f}, & \text{if $a = 0, f(1) < 0$}.
 \end{cases}
 \end{equation}
One can alternatively set $f(0) =0$, then we only need the first three cases above without any condition on $a$.

Identified as the subalgebra generated by $H_1, H_2, \dots, H_{m-1}$ of $\mathcal{H}_{B_m}$,
the Hecke algebra $\mc H_{A_{m-1}}$ inherits a right action on $\VV^{\otimes m}$. Note
that the bar involution on $\mc{H}_{A_{m-1}}$ is just the restriction of the bar involution on $\mc{H}_{B_m}$.

Recall from Section \ref{sec:quasiR} the operator $\mc{R}$. 
We define the following operator on $\VV^{\otimes m}$ for each $1 \leq i \leq m-1$:
\[
\mc{R}_i : = id^{i-1} \otimes \mc{R} \otimes id^{m-i-1} : \VV^{\otimes m} \longrightarrow \VV^{\otimes m}.
\]
The following basic result was due to Jimbo.

\begin{prop}{\cite{Jim}}\label{prop:SchurA}
\begin{enumerate}
\item
The action of $\mc{R}^{-1}_i$ coincides with the action of $H_i$ on $\VV^{\otimes m}$ for $1 \leq i \leq m-1$. 
\item
The actions of $\U$ and $\mathcal{H}_{A_{m-1}}$ on $\VV^{\otimes m}$ 
commute with each other,  and they form double centralizers. 
\end{enumerate}
\end{prop}

\section{The $(\Ui, \mc{H}_{B_m})$-duality}

Introduce the $\Qq$-subspaces of $\VV$:
\begin{align*}
\VV_{-}  &=  \bigoplus_{0 \leq i \leq r} \Qq (v_{-i-\hf} - q^{-1}v_{i+\hf}),
 \\
\VV_{+} &=   \bigoplus_{0 \leq i \leq r} \Qq (v_{-i-\hf} + q v_{i+\hf} ).
\end{align*}

\begin{lem}  \label{lem:V+-}
The subspace $\VV_-$ is a $\Ui$-submodule of $\VV$ generated by $v_{-\hf} - q^{-1}v_{\hf}$
and $\VV_+$ is a $\Ui$-submodule of $\VV$ generated by $v_{-\hf} +q v_{\hf}$.
Moreover, we have 
$
\VV = \VV_- \oplus \VV_+.
$
\end{lem}

\begin{proof}
Follows by a direct computation. 
\end{proof}

Now we fix 
the function $\zeta$ in \eqref{eq:zeta} with $\zeta (\varepsilon_{-r-\hf}) = 1$ so that
$$\zeta({\varepsilon_{r+\hf -i}}) = (-q)^{i-2r-1}, 
\qquad \text{ for }  0 \le i \le 2r+1.
$$
Let us compute the $\Ui$-homomorphism
$\mc{T} = \Upsilon\circ \widetilde{\zeta} \circ T_{w_0}$ (see Theorem~\ref{thm:mcT}) 
on the $\U$-module $\VV$; we remind that $w_0$ here is associated to $\U$ instead of $W_{B_m}$ or $W_{A_{m-1}}$. 

\begin{lem}  \label{lem:T}
The $\Ui$-isomorphism $\mc{T}^{-1}$ on $\VV$ acts as a scalar $(-q) \id$ on the submodule $\VV_-$ and 
as $q^{-1} \id$ on the submodule $\VV_+$. 
\end{lem}

\begin{proof}
First one computes that the action of $T_{w_0}$ on $\VV$ is given by 
$$
T_{w_0} (v_{-r-\hf+i}) =(-q)^{2r+1-i}v_{r+\hf-i}, 
\qquad \text{ for }  0 \le i \le 2r+1.
$$ 
Hence 
\begin{equation}\label{eq:mcTtos0}
\widetilde{\zeta} \circ T_{w_0}(v_a) = v_{ a\cdot s_0}, \quad \text{ for all } a \in I.
\end{equation}

We have $\Upsilon_{\alpha_0} = -\qq F_{\alpha_0}$ from the proof of Theorem~ \ref{thm:Upsilon}
in \S\ref{sec:proofUpsilon}.
Therefore, using $\mc{T} = \Upsilon\circ \widetilde{\zeta} \circ T_{w_0}$ we have
\begin{align}
\mc{T}^{-1} ( v_{-\hf} - q^{-1}v_{\hf})  
 &= -q (v_{-\hf} - q^{-1}v_{\hf})\label{eq:mcT1},
   \\
\mc{T}^{-1} (v_{-\hf}+qv_{\hf}) 
 &= q^{-1}(v_{-\hf}+qv_{\hf})\label{eq:mcT2}.
\end{align}
The lemma now follows from Lemma~\ref{lem:T}
since  $\mc{T}^{-1}$ is a $\Ui$-isomorphism.
\end{proof}

We have the following generalization of Schur-Jimbo duality in Proposition \ref{prop:SchurA}.

\begin{thm}   [$(\Ui, \mc{H}_{B_m})$-duality]
 \label{thm:SchurB}
\begin{enumerate}
\item
The action of $\mc{T}^{-1} \otimes \id^{m-1}$ coincides with the action of 
$H_0 \in \mathcal{H}_{B_m}$ on $\VV^{\otimes m}$. 

\item
The actions of $\Ui$ and $\mathcal{H}_{B_m}$ on $\VV^{\otimes m}$ 
commute with each other, and they form double centralizers. 
\end{enumerate}
\end{thm}

\begin{proof}
Part (1) follows from Lemma~\ref{lem:T} and the action \eqref{eq:HBm} of $H_0\in \mathcal{H}_{B_m}$ on $\VV^{\otimes m}$. 

By Proposition \ref{prop:SchurA},
the actions of $\Ui$ and $\mc{H}_{A_{m-1}}$ on $\VV^{\otimes m}$ commute with each other. 
The action of $\Ui$ on $\VV^{\otimes m}$ comes from the iterated coproduct
$\Ui \rightarrow \Ui \otimes \U^{\otimes m-1}$. Since $\mc T^{-1}: \VV \rightarrow \VV$ is a $\Ui$-homomorphism,
we conclude that the actions of $\mc{T}^{-1} \otimes \id^{m-1}$ and $\Ui$
on $\VV^{\otimes m}$ commute with each other. 
Hence by (1) the actions of $\Ui$ and $\mathcal{H}_{B_m}$ on $\VV^{\otimes m}$ 
commute with each other.

The double centralizer property is equivalent to a multiplicity-free decomposition
of $\VV^{\otimes m}$ as an $\Ui \otimes \mathcal{H}_{B_m}$-module. The latter
follows by the same multiplicity-free decomposition claim at
the specialization $q\mapsto 1$, in which case $\Ui$ specializes to the enveloping algebra of
$\mf{sl}(r+1) \oplus \gl(r+1)$
and $\mathcal{H}_{B_m}$ 
to the group algebra of $W_{B_m}$. Then $\VV =\VV_+ \oplus \VV_-$ at $q=1$ becomes the natural module of $\mf{sl}(r+1) \oplus \gl(r+1)$,
on which $s_0\in W_{B_m}$ acts as $(\id_{\VV_+} , - \id_{\VV_-})$. A multiplicity-free decomposition of $\VV^{\otimes m}$ at $q=1$
can be established by a standard method with the simples parameterized by ordered pairs of partitions $(\la, \mu)$
such that $\ell(\la) \le r+1, \ell(\mu) \le r+1$ and $|\la| +|\mu| =m$. 
\end{proof}

\begin{rem}
The homomorphism $\mc{T}$ (or $\mc{T}^{-1}$) is not needed in Theorem~\ref{thm:SchurB}(2), 
as one can check directly that the action of $H_0$ 
commutes with the action of $\Ui$. However, it is instructive to note that
the action of $H_0$ arises from $\mc T$ which plays an analogous role as the $\mc R$-matrix.
\end{rem}

\begin{rem}
A version of the duality in Theorem~\ref{thm:SchurB} was given in \cite{Gr}, where a Schur-type algebra
was in place of $\Ui$ here. For the applications to BGG categories in Part~2, it is essential for us
to work with the ``quantum group" $\Ui$ equipped with a coproduct. 
\end{rem}

\section{Bar involutions and duality}

\begin{definition}\label{def:antidominant}
An element $f \in I^m$ is called {\em anti-dominant} (or {\em $\imath$-anti-dominant}), if $0 < f(1) \leq f(2) \leq \ldots \leq f(m)$.
\end{definition}

\begin{thm}   \label{thm:samebar}
There exists an anti-linear bar involution $\Bbar: \VV^{\otimes m} \rightarrow \VV^{\otimes m}$ 
which is compatible 
with both the bar involution of $\mc{H}_{B_m}$ and the bar involution of $\Ui$; 
that is, for all $v \in \VV^{\otimes m}$,  $\sigma \in W_{B_m}$, and $u \in \Ui$, we have 
\begin{equation}  \label{eq:barfit}
\Bbar(u v H_\sigma) =\Bbar(u) \, \Bbar(v) \ov{H}_\sigma.
\end{equation}
Such a bar involution 
is unique by requiring $\Bbar(M_f) = M_f$ for all $\imath$-anti-dominant $f$.
\end{thm}

\begin{proof}
Applying the general construction
in \S\ref{subsec:bars} to our setting, we have an $\imath$-involutive $\Ui$-module
$(\VV^{\otimes m}, \Bbar)$; in other words, we have constructed an anti-linear involution 
$\Bbar: \VV^{\otimes m} \rightarrow \VV^{\otimes m}$ 
which  is compatible with the bar involution of $\Ui$.

As the $\mc{H}_{B_m}$-module $\VV^{\otimes m}$ is a direct sum of 
permutation modules of the form $\HBm/\mc H_J$ for various Hecke subalgebras $\mc H_J$,
there exists a unique anti-linear involution on $\VV^{\otimes m}$, denoted by $\Bbar'$, such that 
\begin{enumerate}
\item $\Bbar'(M_f) = M_f$, if $f$ is $\imath$-anti-dominant;
\item $\Bbar'(M_g H_\sigma) = \Bbar'(M_g) \ov{H}_\sigma $, for 
all $g \in I^m$ and $\sigma \in W_{B_m}$.
\end{enumerate}

To  prove the compatibility of $\Bbar$ with the bar involution of $\mc{H}_{B_m}$, it suffices to prove $\Bbar$ satisfies 
the conditions (1)-(2) above; note that it suffices to consider $\sigma$ in (2) to be the simple reflections.  

By the construction
in \S\ref{subsec:bars}, the bar involution $\Bbar: \VV^{\otimes m} \rightarrow \VV^{\otimes m}$ is given by 
$\Bbar =\Upsilon \psi$, where $\psi: \VV^{\otimes m} \rightarrow \VV^{\otimes m}$ is 
a bar involution of type $A$. 
The following compatibility of the bar involutions in the type $A$ setting is well known (see, e.g., \cite{Br03})
(Here we note that our $\imath$-anti-dominant condition is stronger than the type $A$ anti-dominant condition):

\begin{itemize}
\item[($1'$)] 
$\Abar(M_f) = M_f$, if $f$  is $\imath$-anti-dominant;

\item[($2'$)] 
$\Abar(M_g H_\sigma) = M_g \ov{H}_{\sigma}$,  for any $g \in I^m$ and any $H_\sigma \in \mc{H}_{A_{m-1}}$.
\end{itemize}

The $\U$-weight of $M_f$ is $\wtA(f) :=\sum^m_{a=1} \varepsilon_{f(a)} \in \La$. 
Define the $\Ui$-weight of $M_f$ 
$\texttt{wt}_0(f) : = \sum^m_{a=1} \ov{\varepsilon}_{f(a)} \in \La_\theta,$
which is the image of $\wtA(f)$ in $\BLambda =\La/ \La^\theta$
(here we have denoted by $\ov{\varepsilon}_k$ the image of $\varepsilon_k$ in $\La_\inv$).
Defined the following partial ordering $\preceq$ on $I^m$ (which is only used in this proof): 
\begin{equation*}
g \preceq f \quad \Leftrightarrow \quad \texttt{wt}_0(g) =\texttt{wt}_0(f) \text{ and } \wtA(f) - \wtA(g) \in \N \Pi.
\end{equation*}

Applying the intertwiner $\Upsilon =\sum_{\mu\in \N \Pi} \Upsilon_\mu$ from Theorem \ref{thm:Upsilon}, 
we can write for any $f\in I^m$ that
\[
\Upsilon(M_f) = \sum_{g\in I^m} c_g M_g, \quad \text{ for  } c_g \in \Qq.
\]
Here the sum
can be restricted to $g$ with $\texttt{wt}_0(g) = \texttt{wt}_0(f)$ (since $\Upsilon_\mu=0$ unless $\mu^\inv =\mu$ by Theorem~\ref{thm:Upsilon});
hence we have $\wtA(gf) - \wtA(g) \in \N \Pi$ (since $\Upsilon_{\mu} \in \U^-$). Therefore we have
\[
\Upsilon(M_f) = M_f + \sum_{g \prec f} c_g M_g, \quad \text{ for } c_g \in \Qq.
\]
So if $f$ is $\imath$-anti-dominant then  we have $\Upsilon (M_f) = M_f$, and thus
by Proposition \ref{prop:compatibleBbar} and ($1'$) above, 
$\Bbar(M_f) = \Upsilon \Abar(M_f) =\Upsilon (M_f) = M_f$. Hence $\Bbar$ satisfies Condition~(1).

To verify Condition $(2)$ for $\Bbar$, let us first consider the special case when $m=1$. 
Note that $\Abar(v_a) =v_a $ and hence $\Bbar(v_a) = \Upsilon(v_a)$ for all $a$. 
By Definition \ref{def:antidominant}, $a$ is $\imath$-anti-dominant if and only if $a >0$.
Thus  we have 
\begin{equation}   \label{eq:barva>}
\Bbar(v_a) =v_a = \Bbar'(v_a), \qquad \text{ for } a>0.
\end{equation}
 On the other hand, by \eqref{eq:mcTtos0} and Lemma~\ref{lem:T} we have
\begin{align}   \label{eq:barva}
\begin{split}
\Bbar(v_a) 
&= \Upsilon(v_a) 
=\Upsilon \circ \widetilde{\zeta} \circ T_{w_0} (v_{ a\cdot s_0}) 
  \\
 &= \mc{T} (v_{a \cdot s_0}) 
= v_{a \cdot s_0} H^{-1}_{0} = \Bbar'(v_a), \qquad \text{ for } a<0. 
\end{split}
\end{align}
Hence $\Bbar =\Bbar'$ and \eqref{eq:barfit} holds when $m=1$.

Now consider general $m \in {\Z_{> 0}}$. For $1\le i \le m-1$, by applying Proposition \ref{prop:compatibleBbar},   
the identity ($2'$) above, and Proposition \ref{prop:SchurA} in a row, we have, for $g\in I^m$,
\[
\Bbar (M_g H_i) = \Upsilon \Abar (M_g H_i)   = \Upsilon (\Abar(M_g) \ov{H}_{i}) = \Bbar(M_g) \ov{H}_i.
\]
When $i =0$, we write $M_g = v_{g(1)} \otimes M_{g'} $, and hence 
\begin{align*}
\Bbar(M_g H_0) 
&= \Bbar(v_{g(1)} H_0  \otimes M_{g'}) &
  \\
&= \ThetaB(\Bbar(v_{g(1)} H_0) \otimes \Abar (M_{g'}))  & \text{ by Proposition}~ \ref{prop:barinThetaB},
  \\
&= \ThetaB(\Bbar(v_{g(1)}) \ov{H}_0 \otimes \Abar (M_{g'}))  & \text{ by \eqref{eq:barfit} in case $m=1$} ,
  \\
&= \ThetaB(\Bbar(v_{g(1)})  \otimes \Abar (M_{g'})) \ov{H}_0 & \text{ by Theorem}~\ref{thm:SchurB},
  \\
&= \Bbar (M_g) \ov{H}_0   & \text{ by Proposition}~ \ref{prop:barinThetaB}.
\end{align*}
This proves $\Bbar =\Bbar'$ in general,  and hence
completes the proof of the compatibility of all these bar involutions.

The uniqueness of $\Bbar$ in the theorem follows from the uniqueness of $\Bbar'$ above. 
\end{proof}

\begin{rem}
\label{rem:samebar}
The anti-linear involution $\Bbar$ defined on $\VV^{\otimes m}$ 
from the Hecke algebra side gives rise to the Kazhdan-Lusztig theory of type B. 
Theorem~ \ref{thm:samebar} implies that the (induced) Kazhdan-Luszig basis on
$\VV^{\otimes m}$ coincides with its $\imath$-canonical basis (see Theorem~\ref{thm:iCBtensor}). 
Hence Kazhdan-Lusztig theory of type $B$ 
can be reformulated from the algebra $\Ui$ side through $\Bbar$ 
without referring to the Hecke algebra; see Theorem~\ref{thm:KL}.
\end{rem}

\begin{rem}
It follows by \eqref{eq:barva>} and \eqref{eq:barva} that
$\big \{v_{i+\hf}, (v_{-i-\hf} - q^{-1}v_{i+\hf})  \mid 0 \leq i \leq r\big \}$ forms a $\Bbar$-invariant basis of $\VV$.
Also $\big \{v_{i+\hf}, (v_{-i-\hf} + q v_{i+\hf} ) \mid 0 \leq i \leq r \big\}$ forms another $\Bbar$-invariant basis of $\VV$,
which must be the $\imath$-canonical basis by the characterization in Theorem~\ref{thm:BCB}. 
 \end{rem}

\chapter{The quantum symmetric pair $(\U, \Uj)$}
\label{sec:QSPc}

In this chapter we consider the quantum symmetric pair $(\U, \Uj)$ with $\U$ of type  $A_{2r}$. We
formulate the counterparts of the main results from Chapter~\ref{sec:qsp} through Chapter~\ref{sec:HeckeB}
where $\U$ was of type $A_{2r+1}$.
The proofs are similar and often simpler for $\Uj$ since 
it does not contain a generator $t$ as $\Ui$ does, and hence will be omitted almost entirely.

\section{The coideal subalgebra $\Uj$}

We shall write $\I = \I_{2r}$ as given in \eqref{eq:I}  in this chapter. 
We define 
$$
\I^{\jmath}  =\I^{\jmath} _{r} = (\hf +\N) \cap \I 
= \Big\{\hf, \frac32, \ldots, r-\hf \Big\}.
$$ 
The Dynkin diagram of type $A_{2r}$ together with the involution $\inv$ are depicted as follows:
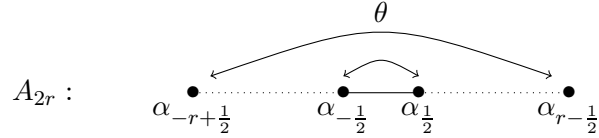
\begin{center}
\begin{tikzpicture}
\draw (-1.5,0) node {$A_{2r}:$};
 \draw[dotted]  (0.5,0) node[below]  {$\alpha_{-r+\hf}$} -- (2.5,0) node[below]  {$\alpha_{-\hf}$} ;
 \draw (2.5,0)
 -- (3.5,0) node[below]  {$\alpha_{\hf}$};
 \draw[dotted] (3.5,0) -- (5.5,0) node[below] {$\alpha_{r-\hf}$} ;
\draw (0.5,0) node (-r) {$\bullet$};
 \draw (2.5,0) node (-1) {$\bullet$};
\draw (3.5,0) node (1) {$\bullet$}; 
\draw (5.5,0) node (r) {$\bullet$};
\draw[<->] (-r.north east) .. controls (3,1) .. node[above] {$\theta$} (r.north west) ;
\draw[<->] (-1.north) .. controls (3,0.5) ..  (1.north) ;
\end{tikzpicture}
\end{center}

The algebra $\Uj$ is defined to be the associative algebra over $\Q(q)$ 
generated by $\be_{\alpha_i}$, $\bff_{\alpha_i}$, $\bk_{\alpha_i}$, $\bk^{-1}_{\alpha_i}$, $i \in \I^{\jmath} $,
subject to the following relations for $i, j \in \I^{\jmath} $:
\begin{align*}
 \ibk{i} \ibk{i}^{-1} &= \ibk{i}^{-1} \ibk{i} =1, \displaybreak[0]\\
 \ibk{i} \ibk{j} &=  \ibk{j}  \ibk{i}, \displaybreak[0]\\
 \ibk{i} \ibe{j} \ibk{i}^{-1} &= q^{(\alpha_i-\alpha_{-i},\alpha_j)} \ibe{j}, \displaybreak[0]\\
 \ibk{i} \ibff{j} \ibk{i}^{-1} &= q^{-(\alpha_i-\alpha_{-i},\alpha_j)}\ibff{j}, \displaybreak[0]\\
 \be_{\alpha_i} \ibff{j} -\ibff{i} \ibe{j} &= \delta_{i,j} \frac{\bk_{\alpha_i}
 -\bk^{-1}_{\alpha_i}}{q-q^{-1}},         \qquad\; \qquad \text{if } i, j \neq \hf, \displaybreak[0]\\
 \ibe{i}^2 \ibe{j} +\ibe{j} \ibe{i}^2 &= (q+q^{-1}) \ibe{i} \ibe{j} \ibe{i},  \qquad \text{if }  |i-j|=1, \displaybreak[0]\\
 \ibff{i}^2 \ibff{j} +\ibff{j} \ibff{i}^2 &= (q+q^{-1}) \ibff{i} \ibff{j} \ibff{i}, \qquad \text{if } |i-j|=1,\displaybreak[0]\\
 \ibe{i} \ibe{j} &= \ibe{j} \ibe{i},     \quad\qquad\qquad\qquad\; \text{if } |i-j|>1, \displaybreak[0]\\
 \ibff{i} \ibff{j}  &=\ibff{j}  \ibff{i},   \quad\qquad\qquad\qquad\; \text{if } |i-j|>1,\displaybreak[0]\\
 %
 %
  \ibff{\hf}^2\ibe{\hf} + \ibe{\hf}\ibff{\hf}^2
    &= (q+q^{-1}) \Big(\ibff{\hf}\ibe{\hf}\ibff{\hf}-q\ibff{\hf}\ibk{\hf}^{-1}-q^{-1}\ibff{\hf}\ibk{\hf} \big),\displaybreak[0]\\
 \ibe{\hf}^2\ibff{\hf} + \ibff{\hf}\ibe{\hf}^2
   &= (q+q^{-1}) \Big(\ibe{\hf}\ibff{\hf}\ibe{\hf}-q^{-1}\ibk{\hf}\ibe{\hf} -q\ibk{\hf}^{-1}\ibe{\hf} \Big).\displaybreak[0]
\end{align*}
We introduce the divided powers 
$\be^{(a)}_{\alpha_i} = \be^a_{\alpha_i} / [a]!$, $\bff^{(a)}_{\alpha_i} = \bff^a_{\alpha_i} / [a]!$.

The following is a counterpart of Lemma~\ref{lem:3inv}.

\begin{lem} 
\begin{enumerate}
\item 
The algebra $\Uj$ has an involution $\omega_\jmath$ such that  

$\omega_\jmath (\bk_{\alpha_i}) 
=q^{-\delta_{i , \hf}} \bk^{-1}_{\alpha_i}$, $\omega_\jmath (\be_{\alpha_i}) = \bff_{\alpha_i}$, and  $\omega_\jmath (\bff_{\alpha_i}) 
= \be_{\alpha_i}$,  for all $i \in \I^{\jmath}$.

\item  
The algebra $\Uj$ has an anti-involution  $\tau_\jmath$ such that 

$\tau_\jmath(\be_{{\alpha_i}}) 
=\be_{\alpha_i}$, $\tau_\jmath(\bff_{\alpha_i}) = \bff_{\alpha_i}$, and 
$\tau_\jmath(\bk_{\alpha_i}) = q^{-\delta_{i , \hf}} \bk^{-1}_{\alpha_i}$, for all $i \in \I^{\jmath} $. 

\item
 The algebra $\Uj$ has an anti-linear ($q \mapsto q^{-1}$) bar involution such that 
 
 $\ov{\bk}_{\alpha_i} = \bk^{-1}_{\alpha_i}$, $\ov{\be}_{\alpha_i} = \be_{\alpha_i}$, 
 and $\ov{\bff}_{\alpha_i} = \bff_{\alpha_i}$,  for all $i \in \I^{\jmath} $. 

(Sometimes we denote the bar involution on $\Uj$ by $\Cbar$.)
\end{enumerate}
\end{lem}

The following is a counterpart of Proposition~\ref{prop:embedding}, the proof of which
relies on  \cite[Proposition 4.1]{KP} and \cite[Theorem 7.1]{Le03}. 

\begin{prop} \label{int:prop:embedding} 
There is an injective $\Qq$-algebra homomorphism $\jmath :  \Uj \rightarrow \U$ defined by, for all  $i \in \I^{\jmath}$,
 \begin{align*}
\bk_{\alpha_i} &\mapsto K_{\alpha_i}K^{-1}_{\alpha_{-i}},\\
\be_{\alpha_i} &\mapsto  E_{\alpha_i} + K^{-1}_{\alpha_i} F_{\alpha_{-i}},\\
\bff_{\alpha_i} &\mapsto F_{\alpha_i}K^{-1}_{\alpha_{-i}} + E_{\alpha_{-i}}.
\end{align*}
\end{prop}

Note that $E_{\alpha_i} (K^{-1}_{\alpha_i}F_{\alpha_{-i}}) 
= q^{2} (K^{-1}_{\alpha_i}F_{\alpha_{-i}}) E_{\alpha_i}$ for $i \in \I^{\jmath} $. 
We have for $i \in \I^{\jmath} $, 
\begin{align}
\label{int:eq:beZ} \jmath(\be^{(a)}_{\alpha_i}) 
 &= \sum^{a}_{j=0}  q^{j(a-j)}{(K^{-1}_{\alpha_i}F_{\alpha_{-i}})^j \over [j]!}\frac{E^{a-j}_{\alpha_i}}{[a-j]!},
   \\
\label{int:eq:bffZ} \jmath(\bff^{(a)}_{\alpha_i}) 
 &= \sum^{a}_{j=0}  q^{j(a-j)}{(F_{\alpha_{i}}K^{-1}_{\alpha_{-i}})^j \over [j]!}\frac{E^{a-j}_{\alpha_{-i}}}{[a-j]!}.
\end{align}

The following is a counterpart of Proposition~\ref{prop:coproduct}.

\begin{prop} \label{int:prop:coproduct}
The coproduct $\Delta : \U \rightarrow \U \otimes \U$ restricts under the embedding $\jmath$ to
a $\Qq$-algebra homomorphism $\Delta : \Uj \mapsto \Uj \otimes \U$ such that for all $i \in \I^{\jmath} $, 
\begin{align*}
\Delta(\bk_{\alpha_i}) &= \bk_{\alpha_i} \otimes K_{\alpha_i} K^{-1}_{\alpha_{-i}},
 \\
\Delta({\be_{\alpha_i}}) &= 1 \otimes E_{\alpha_i} + \be_{\alpha_i} \otimes 
 K^{-1}_{\alpha_i} + \bk^{-1}_{\alpha_i} \otimes K^{-1}_{\alpha_i}  F_{\alpha_{-i}},
  \\
\Delta (\bff_{\alpha_i}) &= \bk_{\alpha_i} \otimes F_{\alpha_i}K^{-1}_{\alpha_{-i}}
  + \bff_{\alpha_i} \otimes K^{-1}_{\alpha_{-i}} + 1 \otimes E_{{\alpha_{-i}}}.
\end{align*}
Similarly, the counit $\epsilon$ of $\U$ induces a $\Qq$-algebra homomorphism $ \epsilon : \Uj \rightarrow \Qq$ such that 
$\epsilon(\be_{\alpha_i}) =\epsilon(\bff_{\alpha_i})=0$ and $\epsilon(\bk_{\alpha_i}) =1$ for all $i \in \I^{\jmath} $.
\end{prop}
It follows by Proposition~\ref{int:prop:coproduct}
that $\Uj$  is a  (right) coideal subalgebra of $\U$. 
The map $\Delta : \Uj \rightarrow \Uj \otimes \U$ will be called the coproduct of $\Uj$ 
and $ \epsilon : \Uj \rightarrow \Qq$ will be called the  counit of $\Uj$.  
The coproduct $\Delta : \Uj \rightarrow \Uj \otimes \U$  is coassociative, i.e.,
$
(1 \otimes \Delta) \Delta = (\Delta \otimes 1)\Delta: \Uj \rightarrow \Uj \otimes \U \otimes \U.
$
The counit map $\epsilon$ makes $\Qq$ a (trivial) $\Uj$-module.  
Let $m : \U \otimes \U \rightarrow \U$ denote the multiplication map. 
Just as in Corollary~\ref{cor:counit}, we have 
$
m (\epsilon \otimes 1)\Delta = \jmath : \Uj \longrightarrow \U.
$

\section{The intertwiner $\Upsilon$ and the isomorphism $\mc{T}$}

As in \S\ref{subsec:Upsilon},
we let $\widehat{\U}$ be the completion of the $\Qq$-vector space $\U$.
We have the obvious embedding of $\U$ into $\widehat{\U}$. 
By continuity the $\Q(q)$-algebra structure on $\U$ extends to the $\Q(q)$-algebra structure on $ \widehat{\U}$.  
The bar involution \,$\bar{\ }$\, on $\U$ extends by continuity 
to an anti-linear involution on $\widehat{\U}$, which is also denoted by \,$\bar{\phantom{c}}$. 
The following is a counterpart of Theorem~\ref{thm:Upsilon}.

\begin{thm}  \label{int:thm:Upsilon}
There is a unique family of elements $\Upsilon_\mu \in \U_{-\mu}^-$ for $\mu \in {\N}{\Pi}$ such that 
$\Upsilon = \sum_{\mu}\Upsilon_\mu \in \widehat{\U}$ 
intertwines the bar involutions on $\Uj$ and $\U$ via the embedding $\jmath$
and $\Upsilon_0 = 1$;  
that is, $\Upsilon$ satisfies the following identity (in $\widehat{\U}$): 
\begin{equation}\label{int:eq:star}
 \jmath(\ov{u}) \Upsilon = \Upsilon\  \overline{\jmath(u)}, \quad \text{ for all } u \in \Uj.
\end{equation}
Moreover, $\Upsilon_\mu = 0$ unless $\mu^{\inv} = \mu$.
\end{thm}

The following is a counterpart of Corollary~\ref{cor:Upsiloninv}.

\begin{cor}\label{int:cor:Upsiloninv}
We have $\Upsilon \cdot \ov{\Upsilon} =1 .$
\end{cor}

Consider a function $\zeta$ on $\Lambda$ such that
\begin{align}  \label{int:eq:zeta}
\begin{split}
 \zeta (\mu+\alpha_i) 
 & = -q^{(\alpha_i -\alpha_{-i}, \mu+\alpha_i)} \zeta (\mu), 
  \\
\zeta (\mu+\alpha_{-i}) 
 &= -q^{(\alpha_{-i}, \mu+\alpha_{-i}) - (\alpha_{i}, \mu)} \zeta (\mu),
 \end{split}
 \end{align}
for all $\mu \in \Lambda$, $i \in \I^{\jmath} $.
Such $\zeta$ exists. For any $\U$-module $M$, define a $\Qq$-linear map $\widetilde{\zeta}  : M \rightarrow M$ by 
\[
\widetilde{\zeta} (m) = \zeta (\mu)m, \quad \text{ for all } m \in M_{\mu}.
\]

Let $w_0$ denote the longest element of the Weyl group $W$ of type $A_{2r}$. 
As in Section~ \ref{subsec:CB} we denote by $T_{w_0}$ the braid group  element. 
The following is a counterpart of Theorem~\ref{thm:mcT}.

\begin{thm}\label{int:thm:mcT}
Given any finite-dimensional $\U$-module $M$, 
the composition map
\[
\mc{T} := \Upsilon\circ \widetilde{\zeta} \circ T_{w_0}: M \longrightarrow M
\] 
is an isomorphism of $\Uj$-modules.
\end{thm}

\section{Quasi-$\mc{R}$ matrix on $\Uj$}

It follows by Theorem \ref{int:thm:Upsilon} that $\Upsilon$ 
is a well-defined operator on finite-dimensional $\U$-modules. 
For any finite-dimensional $\U$-modules $M$ and $M'$, we shall use the formal notation $\Dupsilon$ to denote the well-defined 
action of $\Upsilon$ on $M \otimes M'$. Hence the operator
\begin{equation}  \label{eq:ThetaC}
\ThetaC := \Dupsilon \ThetaA(\Upsilon^{-1} \otimes 1)
\end{equation}
on $M \otimes M'$ is well defined. 
Define 
$$
\ov{\Delta}: \Uj \longrightarrow \Uj \otimes \U
$$ 
by letting $\ov{\Delta}(u) = \ov{\Delta(\ov{u})}$, for all $u \in \Uj$. 

The construction in \S\ref{subsec:normalTheta} carries over with little modification, and we will be sketchy.
For each $N \in {\N}$, we have a truncation map $tr_{\leq N}$ on $\U^-$ as in \eqref{eq:trN}. 
Then the same formulas as in \eqref{eq:ThetaleqN} and \eqref{eq:ThetaN} give us
$\ThetaC_{\leq N}$ and 
$\ThetaC_{ N}$  in $\U  \otimes \U^-$.
%
The following is a counterpart of Proposition~\ref{prop:ThetainBoA}.

\begin{prop}\label{int:prop:ThetainBoA}
For any $N \in {\N}$, we have $\ThetaC_{ N} \in \jmath(\Uj) \otimes \U^-$.
\end{prop}

Proposition~\ref{int:prop:ThetainBoA}
allows us to make sense of $\jmath^{-1} (\ThetaC_N) \in \Uj \otimes \U$ for each $N$. 
For any finite-dimensional $\U$-modules $M$ and $M'$, 
the action of $\jmath^{-1} (\ThetaC_N)$ coincides with the action of $\ThetaC_N$
on $M \otimes M'$. 
{\bf As we will only need to use $\jmath^{-1} (\ThetaC_N) \in \Uj \otimes \U$ rather than $\ThetaC_N$,
we will simply write $\ThetaC_N$ for  $\jmath^{-1} (\ThetaC_N)$  
and regard $\ThetaC_N \in \Uj \otimes \U$ from now on.}
Similarly, it is now understood that $\ThetaC_{\le N} = \sum_{r=0}^N \ThetaC_r \in \Uj \otimes \U$.
The actions of $\sum_{N\ge 0} \ThetaC_N$ and of $\ThetaC$ coincide on any tensor product of finite-dimensional $\U$-modules. 
From now on, we may and shall identify 
$$
\ThetaC = \sum_{N\ge 0}  \ThetaC_N
$$ 
(or alternatively, use this as a normalized
definition of $\ThetaC$) as an element
in a completion $(\Uj \otimes \U^-)^\wedge$ of $\Uj \otimes \U^-$.

The following is a counterpart of Theorem~\ref{thm:ThetaB}.

\begin{thm}\label{int:thm:ThetaB}
Let $L$ be a finite-dimensional  $\Uj$-module and let $M$ be a finite-dimensional $\U$-modules. 
Then as linear operators on $L\otimes M$, we have
\[
\Delta(u)\ThetaC = \ThetaC\ov{\Delta}(u), \qquad \text{ for }  u \in \Uj.
\]
\end{thm}
 
The following is the counterpart of Proposition~\ref{prop:ThetaBinv}.

\begin{prop}\label{int:prop:ThetaBinv}
We have $\ThetaC \ov{\ThetaC} =1$. 
\end{prop}

The following is the counterpart of Corollary~\ref{cor:ThetatoT}.

\begin{cor}\label{int:cor:ThetatoT}
We have
$m(\epsilon \otimes 1)\ThetaC  = \Upsilon$.
\end{cor}

\section{The $\jmath$-involutive modules} 
 \label{int:subsec:CBandDCB}  
 
In this section we shall assume all modules are finite dimensional. 
Recall the bar map on $\U$ and its modules is also denoted by $\Abar$, and
the bar map on $\Uj$ is also denoted by $\Cbar$. 
It is also understood that $\Abar(u) =\Abar(\jmath(u))$ for $u\in \Uj$.
 
 \begin{definition}  \label{def:involutive-j}
A $\Uj$-module $M$ equipped with an anti-linear involution $\Cbar$ 
is called {\em involutive} (or {\em $\jmath$-involutive}) if
$$
\Cbar(u m) = \Cbar(u) \Cbar(m), \qquad \forall u \in \Uj, m \in M.
$$ 
 \end{definition}

The following is a counterpart of Proposition~\ref{prop:compatibleBbar}.

\begin{prop}\label{int:prop:compatibleBbar}
Let $M$ be an involutive $\U$-module. Then $M$ is an $\jmath$-involutive  
$\Uj$-module with involution
$\Cbar := \Upsilon \circ \Abar$.
\end{prop}

The following is a counterpart of Corollary~\ref{cor:Li-invol}.

\begin{cor}   \label{int:cor:Li-invol}
Let $\la\in \Lambda^+$. Regarded as $\Uj$-modules, $L(\la)$ and ${^{\omega}L}(\lambda)$ are $\jmath$-involutive. 
\end{cor}

Given an involutive $\Uj$-module $L$ and an involutive $\U$-module $M$, we define 
$\Cbar: L \otimes M \rightarrow L \otimes M$ by letting
\begin{equation}   \label{eq:Cbartensor}
\Cbar(l \otimes m) := \ThetaC(\Cbar(l) \otimes \Abar(m)),  \qquad  \text{ for all } l \in L, m \in M.
\end{equation}

The following is a counterpart of Proposition~\ref{prop:Bbartensor}.

\begin{prop}  \label{int:prop:Bbartensor}
Let $L$ be an involutive $\Uj$-module and let $M$ be an involutive $\U$-module. Then 
$(L \otimes M, \Cbar)$ is an involutive $\Uj$-module.
\end{prop}

\begin{rem}   \label{int:rem:sameinv}
Given two involutive $\U$-modules $(M_1, \psi_1)$ and $(M_2,\psi_2)$, the two different ways,
via  Proposition~\ref{int:prop:compatibleBbar}  or  Proposition~\ref{int:prop:Bbartensor}, 
of defining an $\jmath$-involutive $\Uj$-module structure on  $M_1\otimes M_2$
coincide; compare with Remark~\ref{rem:sameinv}. 
\end{rem}

The following proposition, which is a counterpart of Proposition~\ref{prop:barinThetaB}, 
implies that different bracketings on the tensor product 
of several involutive $\U$-modules   give rise to the same $\Cbar$. 

\begin{prop}\label{int:prop:barinThetaB}
Let $M_1$, $\ldots$, $M_k$ be involutive $\U$-modules with $k\ge 2$. We have 
\[
\Cbar(m_1 \otimes \cdots \otimes m_k)
= \ThetaC (\Cbar (m_1 \otimes \cdots \otimes m_{k'}) \otimes \Abar(m_{k'+1} \otimes \cdots \otimes m_k)),
\]
for any $1 \leq k' < k$.
\end{prop}

\section{Integrality of $\Upsilon$}

Similar to Lemma~\ref{lem:bunxi} for $\Ui$, we can show that 
$$
\Uj \xi_{-\lambda} = {^{\omega}L}(\lambda),
\qquad \Uj \eta_{\lambda} = L(\lambda).
$$ 
%
%
The following is a counterpart of Lemma~\ref{lem:uniqueT}. 

\begin{lem}\label{int:lem:uniqueT}
For any $\lambda \in \Lambda^+$,  there is a unique isomorphism of $\Uj$-modules
\[
\mc{T} : {^{\omega}L}(\lambda) \longrightarrow {^{\omega}L}(\lambda) = L( \lambda^{\inv}),
\]
such that $\mc{T} (\xi_{\-\lambda}) = \sum_{ b \in \B(\lambda)} g_b b^- \eta^{}_{ \lambda^{\inv}}$  
for $g_b \in \Qq$ and $g_1 =1$.
\end{lem}

\begin{prop}\label{int:prop:shrink}
Let $\lambda, \mu \in \Lambda^+$. There is a unique homomorphism of $\Uj$-modules 
\[
\pi_{\lambda , \mu }: {^{\omega}L} ( \mu^{\inv} + \mu + \lambda ) \longrightarrow {^{\omega}L}(\lambda),
\] 
such that $\pi_{\la,\mu} (\xi^{}_{-\mu^{\inv}- \mu - \lambda }) = \xi_{-\lambda}$. 
\end{prop}

Recall that ${^{\omega}L} (\mu^{\inv} +\mu+\lambda)$ and ${^{\omega}L}(\lambda)$ are both 
$\jmath$-involutive $\Uj$-modules with $\Cbar= \Upsilon \circ \Abar$.
Similar to Lemma~\ref{lem:piBbar=Bbarpi},
the $\Uj$-homomorphism $\pi_{\la,\mu}$ commutes with the bar involution $\Cbar$, i.e., 
$\pi_{\la,\mu} \Cbar = \Cbar\pi_{\la,\mu} $. 

The following is a counterpart of Lemma~\ref{lem:UppreZ}, with a much easier proof. 
Indeed, since the identities \eqref{int:eq:beZ} and \eqref{int:eq:bffZ} give us all the divided powers we need, 
we can bypass the careful study of the rank one case as in \S\ref{subsec:rank1}
for  $\Ui$.

\begin{lem}  \label{int:lem:UppreZ}
For each $\lambda \in \Lambda^+$, 
we have $\Upsilon ({^{\omega}L}_\mA(\lambda)) \subseteq {^{\omega}L}_\mA(\lambda)$.
\end{lem}

The following is a counterpart of Theorem~\ref{thm:UpsiloninZ}.

\begin{thm}\label{int:thm:UpsiloninZ}
We have $\Upsilon_\mu \in \U^-_\mA$, for all $\mu \in \N\Pi$.
\end{thm}

\section{The $\jmath$-canonical basis of ${^{\omega}L(\lambda)}$} 

The following is a counterpart of Lemma~\ref{lem:Zform}, which now
follows from Theorem \ref{int:thm:UpsiloninZ} and Proposition \ref{int:prop:compatibleBbar}. 
Note that we do not need the input from the rank one case here. 

\begin{lem}  \label{int:lem:Zform}
The bar map $\Cbar$ preserves the $\mA$-form ${^{\omega}L}_\mA(\lambda)$, for $\la \in \La^+$.
\end{lem}

Recall a partial ordering $\preceq$ on the set $\B(\la)$ of canonical basis for $\la\in\Lambda^+$
from \eqref{eq:order}.
For any $b \in \B(\lambda)$, recalling $\Cbar =\Upsilon \Abar$, we have 
\begin{equation}  \label{int:eq:triang}
\Cbar(b^+ \xi_{-\lambda}) 
 = \Upsilon  (b^+ \xi_{-\lambda} ) 
 = \sum_{b' \in \B(\lambda)} \rho_{b; b'} b'^+ \xi_{-\lambda}
\end{equation}
where $\rho_{b; b'} \in \mA$ by Theorem~\ref{int:thm:UpsiloninZ}. 
Since $\Upsilon$ lies in a completion of $\U^-$ satisfying $\Upsilon_\mu =0$ unless $\mu^\inv =\mu$ 
(see Theorem~\ref{int:thm:Upsilon}), 
we have  $ \rho_{b; b} =1$ and $\rho_{b; b'} = 0$
unless $b' \preceq b$.
Since $\Cbar$ is an involution, we can apply 
\cite[Lemma~ 24.2.1]{Lu94} to our setting to establish the following counterpart of 
Theorem~\ref{thm:BCB}.

\begin{thm}\label{int:thm:BCB}
Let $\la \in \La^+$. The $\Uj$-module  ${^{\omega}L}(\lambda)$ admits a unique basis 
\[
\B^\jmath(\lambda) := \{T^{\lambda}_{b} \mid b \in \B(\lambda)\}
\]
which is $\Cbar$-invariant and of the form
\[
T^{\lambda}_{b} = b^+ \xi_{-\lambda} +\sum_{b' \prec b}
t^{\lambda}_{b;b'} b'^+  \xi_{-\lambda},
\quad \text{ for }\;  t^{\lambda}_{b;b'} \in q\Z[q].
\]
\end{thm}

\begin{definition}  \label{int:def:CB}
$\B^\jmath(\lambda)$ is called the $\jmath$-canonical basis of
the $\Uj$-module ${^{\omega}L(\lambda)}$.
\end{definition}

Just as in Section~\ref{subsec:CBandDCB}, we can generalize Theorem~\ref{int:thm:BCB} 
to any based $\U$-module $(M, B)$ (in the sense of Lusztig \cite[Chapter 27]{Lu94}). 
The basis $B$ generates a $\Z[q]$-submodule $\mc{M}$ and an $\mA$-submodule ${}_\mA M$ of $M$.
Recall again Lusztig has shown that the tensor product of several finite-dimensional simple $\U$-modules
is a based module. 
Thus we have the following counterparts of Theorem~\ref{thm:iCBbased} and \ref{thm:iCBtensor}.

\begin{thm}  
Let $(M,B)$ be a finite-dimensional based $\U$-module. (For example, take
$M ={^{\omega} L (\la_1)}   \otimes \ldots \otimes {^{\omega} L (\la_r)}$, for $\la_1, \ldots, \la_r \in \La^+$.)
 
\begin{enumerate}
\item
The $\Uj$-module $M$ admits a unique basis (called $\jmath$-canonical basis)
$
B^\jmath  := \{T_{b} \mid b \in B \}
$
which is $\Cbar$-invariant and of the form
\begin{equation*}  
T_{b} = b +\sum_{b' \in B, b' \prec b}
t_{b;b'} b',
\quad \text{ for }\;  t_{b;b'} \in q\Z[q].
\end{equation*}

\item
$B^\jmath$ forms an $\mA$-basis for the $\mA$-lattice ${}_\mA M$, and
$B^\jmath$ forms a $\Z[q]$-basis for the $\Z[q]$-lattice $\mc{M}$.
\end{enumerate}
\end{thm}

\section{The $(\Uj, \mc{H}_{B_m})$-duality} 

Again in this section $\U$ is of type $A_{2r}$ with simple roots parametrized by $\I_{2r}$ in \eqref{eq:I}.
Recall the notation $\I_{2r+1}$  from \eqref{eq:I}, and we set 
$$I =\I_{2r+1}  =\{-r,\ldots, -1, 0, 1, \ldots, r \}. 
$$ 
Let the $\Qq$-vector space $\VV := \sum _{a \in I}\Qq v_{a} $ be the natural representation of $\U$, 
hence a $\Uj$-module. We shall call $\VV$ the natural representation of $\Uj$ as well. 
For $m \in {\Z_{> 0}}$, 
$\VV^{\otimes m}$ becomes a natural $\U$-module (hence a $\Uj$-module) 
via the iteration of the coproduct $\Delta$. 
Note that $\VV$ is an involutive  $\U$-module with $\Abar$ defined as 
\[\Abar(v_a) :=v_a, \quad \text{ for all } a \in I.
\]
Therefore $\VV^{\otimes m}$ is an involutive  $\U$-module 
and hence a $\jmath$-involutive $\Uj$-module by Proposition~\ref{int:prop:barinThetaB}.

For any $f \in I^m$, we define
$
M_f= v_{f(1)} \otimes \cdots \otimes v_{f(m)}.
$
The Weyl group $W_{B_m}$ acts on $I^m$ by \eqref{eq:rightW} as before. 
Now the Hecke algebra $\HBm$  acts on the $\Qq$-vector space $\VV^{\otimes m}$ as follows: 
\begin{equation}  \label{int:eq:HBm}
 M_f H_a=
 \begin{cases}
 q^{-1}M_f, & \text{ if } a>0, f(a) = f(a+1);\\
 M_{f \cdot  s_a}, & \text{ if } a > 0, f(a) < f(a+1);\\
 M_{f \cdot  s_a} + (q^{-1} - q) M_{f}, & \text{ if } a > 0, f(a) > f(a+1);\\
 M_{f \cdot  s_0}, & \text{ if } a = 0 , f(1) > 0;\\
 M_{f \cdot s_0} + (q^{-1} - q) M_{f}, & \text{ if } a = 0, f(1) < 0;\\
 q^{-1} M_{f}, &\text{ if } a=0, f(1)=0.
 \end{cases}
\end{equation}
Identified as the subalgebra generated by $H_1, H_2, \dots, H_{m-1}$ of $\mathcal{H}_{B_m}$,
the Hecke algebra $\mc H_{A_{m-1}}$ inherits a right action on $\VV^{\otimes m}$. 
The Schur-Jimbo duality as formulated in Proposition~\ref{prop:SchurA} remains to be valid 
in the current setting, i.e., 
the actions of $\U$ and $\mathcal{H}_{A_{m-1}}$ on $\VV^{\otimes m}$ 
commute with each other and they form double centralizers. 

Introduce the $\Qq$-subspaces of $\VV$:
\begin{align*}
\VV_{-}  &=\bigoplus_{1 \leq i \leq r} \Qq (v_{-i} - q^{-1}v_{i}),
  \\
\VV_{+}  &= \Qq v_0 \bigoplus \bigoplus_{1 \leq i \leq r} \Qq (v_{-i} + q v_{i} ).
\end{align*}
The following is a counterpart of Lemma~\ref{lem:V+-}.

\begin{lem}  \label{int:lem:V+-}
$\VV_-$ is a $\Uj$-submodule of $\VV$ generated by $v_{-1} - q^{-1}v_{1}$
and $\VV_+$ is a $\Uj$-submodule of $\VV$ generated by $v_0$.
Moreover, we have   
$
\VV = \VV_- \oplus \VV_+.
$
\end{lem}

Now we fix 
$\zeta$ in \eqref{int:eq:zeta}  such that $\zeta (\varepsilon_{-r}) = 1$. It follows that
$$\zeta({\varepsilon_{r -i}}) =
\begin{cases}
 (-q)^{-2r+i}, & \text{ if } i \neq r;\\
 q \cdot (-q)^{-r}, &\text{ if } i = r.
 \end{cases}
 $$
Let us compute the $\Uj$-homomorphism
$\mc{T} = \Upsilon\circ \widetilde{\zeta} \circ T_{w_0}$ (see Theorem~\ref{int:thm:mcT}) 
on the $\U$-module $\VV$; we remind that $w_0$ here is associated to $\U$ instead of $W_{B_m}$ or $W_{A_{m-1}}$. 

\begin{lem}  \label{int:lem:T}
The $\Uj$-isomorphism $\mc{T}^{-1}$ on $\VV$ acts as a scalar $(-q) \id$ on the submodule $\VV_-$ and 
as $q^{-1} \id$ on the submodule $\VV_+$. 
\end{lem}
 
\begin{proof}
First one computes that the action of $T_{w_0}$ on $\VV$ is given by 
$$
T_{w_0} (v_{-r+i}) =(-q)^{2r-i}v_{r-i}, 
\qquad \text{ for }  0 \le i \le 2r.
$$ 
Hence 
\begin{equation}\label{eq:mcTtos0C}
\widetilde{\zeta} \circ T_{w_0}(v_a) =
\begin{cases} 
v_{ a\cdot s_0}, &\text{ if } a \neq 0; \\
q v_{0}, &\text{ if } a = 0.
\end{cases}
\end{equation}

One computes that $\Upsilon_{\alpha_{-\hf}+\alpha_{\hf}} = -\qq F_{\alpha_{\hf}}F_{\alpha_{-\hf}}$.
Therefore using $\mc{T} = \Upsilon\circ \widetilde{\zeta} \circ T_{w_0}$ we have 
\begin{align}
\mc{T}^{-1} v_0 &= q^{-1} v_0\label{int:eq:mcT1},\\
\mc{T}^{-1} ( v_{-1} - q^{-1}v_{1}) &= (-q) (v_{-1} - q^{-1}v_{1})\label{int:eq:mcT2},\\
\mc{T}^{-1} (v_{-1}+qv_{1}) &= q^{-1}(v_{-1}+qv_{1})\label{int:eq:mcT3}.
\end{align}
The lemma now follows from Lemma~\ref{int:lem:T}
since  $\mc{T}^{-1}$ is a $\Uj$-isomorphism.
\end{proof}
 
We have the following generalization of Schur-Jimbo duality,
which is a counterpart of Theorem~\ref{thm:SchurB}.
The proof is almost identical as the one for Theorem~\ref{thm:SchurB},
and for Part~(1) we now use Lemma~\ref{int:lem:T} and 
the action \eqref{int:eq:HBm} of $H_0\in \mathcal{H}_{B_m}$ on $\VV^{\otimes m}$. 

\begin{thm}  [$(\Uj, \mc{H}_{B_m})$-duality]
  \label{int:thm:SchurB}
\begin{enumerate}
\item
The action of $\mc{T}^{-1} \otimes \id^{m-1}$ coincides with the action of 
$H_0 \in \mathcal{H}_{B_m}$ on $\VV^{\otimes m}$. 

\item
The actions of $\Uj$ and $\mathcal{H}_{B_m}$ on $\VV^{\otimes m}$ 
commute with each other, and they form double centralizers. 
\end{enumerate}
\end{thm} 
  
\begin{definition}\label{def:antidominant2}
An element $f \in I^m$ is anti-dominant (or $\jmath$-anti-dominant) if 
$$0 \leq  f(1) \leq f(2) \leq \cdots \leq f(m).
$$
\end{definition}
The following is the counterpart of Theorem~\ref{thm:samebar}. 

\begin{thm}\label{int:thm:samebar}
There exists an anti-linear involution $\Cbar: \VV^{\otimes m} \rightarrow \VV^{\otimes m}$
which is compatible with both the bar involution of $\HBm$ and the bar involution of $\Uj$; 
that is, for all $v \in \VV^{\otimes m}$,  $H_\sigma \in \HBm$, and $u \in \Uj$, we have 
\[
\Cbar(u v H_\sigma) = \Cbar(u) \, \Cbar(v) \ov{H}_\sigma.
\]
Such a bar involution  is unique
by requiring $\Cbar(M_f) = M_f$ for all $\jmath$-anti-dominant $f$.

\end{thm}

\newpage

\part{Representation theory}

In this Part~2, we shall focus on the infinite-rank limit ($r \to \infty$) 
of the algebras and spaces formulated in Part~ 1. 
In  Chapter~\ref{sec:BGG} through Chapter~ \ref{sec:b-KL}
we will mainly treat in detail the counterparts of Chapter~\ref{sec:qsp} through Chapter~\ref{sec:HeckeB}
where $\U$ was of type $A_{2r+1}$ in Part~1. 
In Chapter~\ref{sec:BGGj} we deal with a variation of BGG category with half-integer weights 
which corresponds to the second quantum symmetric pair $(\U, \Uj)$ in Chapter~\ref{sec:QSPc} 
where $\U$ was of type $A_{2r}$. 

As it becomes necessary to keep track of the finite ranks, we shall add subscripts and superscripts to various notation 
introduced in Part~1 to indicate the
dependence on $r\in \N$. Here are the new notations in place of those in Part~1 without superscripts/subscripts
(Chapter~\ref{sec:qsp} through Chapter~\ref{sec:HeckeB}):

 $\La_{2r+1}$, $\Pi_{2r+1}$, $\I_{2r+1}$, $\I^{\imath}_r$,
 $\U_{2r+1}$, $\U^\imath_r$, $\Upsilon^{(r)}$, $\VV_r$, $\WW_r$, $\Abar^{(r)}$, $\Bbar^{(r)}$, $\ThetaA^{(r)}$.

Part 2 of this paper follows closely \cite{CLW12} with new input from Part~1. 
The notations here often have different meaning from 
the same notations used in \cite{CLW12}, as the current ones are often ``of type $B$".

\chapter{BGG categories for ortho-symplectic Lie superalgebras}
\label{sec:BGG}

In this chapter, we recall the basics on the ortho-symplectic Lie superalgebras such as linkage principle and
Bruhat ordering. We formulate various versions of (parabolic) BGG categories and truncation functors.

\section{The Lie superalgebra $\mf{osp}(2m+1|2n)$}
   \label{subsec:osp}
   
We recall some basics on ortho-symplectic Lie superalgebras and set up notations to be used later on (cf. \cite{CW12} for 
more on Lie superalgebras). Fix integers $m\ge 1$ and $n\ge 0$ throughout this paper. 

Let $\Z_2 = \{\ov{0}, \ov{1}\}$.
 Let $\C^{2m+1|2n}$ be a superspace of dimension $(2m+1|2n)$ with basis 
 $\{e_i \mid 1 \leq i \leq 2m+1\} \cup \{e_{\ov j} \mid 1 \leq j \leq 2n\}$, 
 where the $\Z_2$-grading is given by the following parity function:
\[
p(e_i) = \ov 0, \qquad p(e_{\ov j}) = \ov 1 \quad (\forall i,j).
\] 
Let $B$ be a non-degenerate even supersymmetric bilinear form on $\C^{2m+1|2n}$.
The general linear Lie superalgebra $\mf{gl}(2m+1|2n)$ is the Lie superalgebra of linear transformations on $\C^{2m+1|2n}$
(in matrix form with respect to the above basis).
 For $s \in \Z_2$, we define
\begin{align*}
\osp(2m+1|2n)_s &:= \{ g \in \mf{gl}(2m+1|2n)_s \mid  
  \\
  &\qquad\qquad\qquad  B(g(x), y) = -(-1)^{s \cdot p(x)}B(x, g(y))\},\\
\osp(2m+1|2n) &:=\osp(2m+1|2n)_{\ov 0} \oplus \osp(2m+1|2n)_{\ov 1}.
\end{align*}

We now give a matrix realization of the Lie superalgebra $\osp(2m+1|2n)$. Take the supersymmetric bilinear form $B$  
with the following matrix form, with respect to the basis $(e_1, e_2, \dots, e_{2m+1}, e_{\ov 1}, e_{\ov 2}, \dots, e_{\ov{2n}})$: 
\[\mc J_{2m+1|2n} :=
\begin{pmatrix}
0 & I_m & 0 & 0& 0\\
I_m & 0 & 0 & 0 & 0\\
0 & 0 & 1 & 0 & 0 \\
0 & 0 & 0 & 0 & I^n\\
0 & 0 & 0 & -I^n & 0
\end{pmatrix}
\]

 Let $E_{i,j}$, $1 \leq i,j \leq 2m+1$, and $E_{\ov k , \ov h}$, $ 1 \leq k,h \leq 2n$, be the $(i,j)$th and 
 $(\ov k, \ov h)$th elementary matrices, respectively. The Cartan subalgebra of $\osp(2m+1|2n)$ 
 of diagonal matrices is denoted by $\mf h_{m|n}$, which is spanned by
\begin{align*}
&H_i := E_{i,i}-E_{m+i,m+i}, \quad  1 \leq i \leq m,\\
&H_{\ov j} :=E_{\ov j, \ov j} - E_{\ov{n+j}, \ov{n+j}}, \quad 1 \leq j \leq n.
\end{align*}
We denote by $\{\ep_i, \ep_{\ov j} \mid 1 \leq i \leq m, 1 \leq j \leq n \}$ the basis of $\mf h^*_{m|n}$ such that
\[
\ep_{a}(H_b) = \delta_{a,b}, \quad \text{ for  } a, b \in \{i, \ov j \mid 1 \leq i \leq m, 1 \leq j \leq n\}.
\]
We denote the lattice of integral weights of $\osp(2m+1|2n)$ by 
\begin{equation}  \label{eq:Xmn}
X(m|n) := \sum^{m}_{i=1}\Z\ep_{i} + \sum^n_{j=1}\Z\ep_{\ov j}.
\end{equation}
The supertrace form on $\osp(2m+1|2n)$ induces a non-degenerate symmetric bilinear form 
on $\mf h^*_{m|n}$ determined by $(\cdot | \cdot)$, such that
\[
(\ep_{i}\vert\ep_{a}) = \delta_{i,a}, \quad (\ep_{\ov j}| \ep_{a}) 
 = -\delta_{\ov j, a}, \quad \text{ for  } a \in   \{i, \ov j \mid 1 \leq i \leq m, 1 \leq j \leq n\}.
\]
We have the following root system of $\osp(2m+1|2n)$ with respect to $\mf h_{m|n}$
\[\Phi = \Phi_{\ov 0} \cup \Phi_{\ov 1} 
= \{\pm\ep_{i}\pm\ep_{j}, \pm\ep_{p}, \pm\ep_{\ov k}\pm\ep_{\ov l}, \pm2\ep_{\ov q}\} \cup \{\pm\ep_{p}\pm\ep_{\ov q}, \pm\ep_{\ov q}\},
\]
where $1 \leq i < j \leq n$, $1 \leq p \leq n$, $1\leq q \leq m$, $1 \leq k < l\leq m$. 

In this paper we shall need to deal with various Borel subalgebras, hence various simple systems of $\Phi$. 
Let ${\bf b}=(b_1,b_2,\ldots,b_{m+n})$ be a
sequence of $m+n$ integers such that $m$ of the $b_i$'s are equal to
${0}$ and $n$ of them are equal to ${1}$. We call such a sequence a
{\em $0^m1^n$-sequence}. 
Associated to each $0^m1^n$-sequence ${\bf b} =(b_1, \ldots, b_{m+n})$, 
we have the following fundamental system  $\Pi_{\bf {b}}$, and hence a positive system 
$\Phi_{\bf b}^+ =\Phi_{{\bf b},\bar{0}}^+ \cup \Phi_{{\bf b},\bar{1}}^+$, of the root system $\Phi$ of $\mathfrak{osp}(2m+1|2n)$:
\[
\Pi_{\bf {b}} = \{-\ep^{b_1}_1, \ep^{b_i}_i -\ep^{b_{i+1}}_{i+1} \mid 1 \leq i \leq m+n-1\},
\]
where $\ep^{0}_i = \ep_{x}$ for some $1 \leq x \leq m$, $\ep^1_{j} = \ep_{\ov y}$ 
for some $1 \leq y \leq n$, such that $\ep_{x} -\ep_{x+1}$ and $\ep_{\ov y} - \ep_{\ov{y+1}}$ 
are always positive. It is clear that $\Pi_{\bf b}$ is uniquely determined by these restrictions.
The Weyl vector is defined to be
$\rho_{\bf b}:= \hf \sum_{\alpha \in \Phi^+_{{\bf b}, \bar{0}}} \alpha -\hf \sum_{\alpha \in \Phi^+_{{\bf b}, \bar{1}}} \alpha$.

Corresponding to ${\bf b}^{\text{st}} =(0,\ldots, 0,1,\ldots,1)$, we have the following standard 
Dynkin diagram associated to $\Pi_{{\bf b}^{\text{st}}}$:

\begin{center}
\hskip -3cm \setlength{\unitlength}{0.16in}
\begin{picture}(24,4)
\put(5.6,2){\makebox(0,0)[c]{$\bigcirc$}}
\put(8,2){\makebox(0,0)[c]{$\bigcirc$}}
\put(10.4,2){\makebox(0,0)[c]{$\bigcirc$}}
\put(14.85,2){\makebox(0,0)[c]{$\bigotimes$}}
\put(17.25,2){\makebox(0,0)[c]{$\bigcirc$}}
\put(19.4,2){\makebox(0,0)[c]{$\bigcirc$}}
\put(23.5,2){\makebox(0,0)[c]{$\bigcirc$}}
\put(8.35,2){\line(1,0){1.5}} \put(10.82,2){\line(1,0){0.8}}
\put(13.2,2){\line(1,0){1.2}} \put(15.28,2){\line(1,0){1.45}}
\put(17.7,2){\line(1,0){1.25}} \put(19.81,2){\line(1,0){0.9}}
\put(22,2){\line(1,0){1}}
\put(6.8,2){\makebox(0,0)[c]{$\Longleftarrow$}}
\put(12.5,1.95){\makebox(0,0)[c]{$\cdots$}}
\put(21.5,1.95){\makebox(0,0)[c]{$\cdots$}}
\put(5.4,1){\makebox(0,0)[c]{\tiny $-\epsilon_{1}$}}
\put(7.8,1){\makebox(0,0)[c]{\tiny $\epsilon_{1} -\epsilon_{2}$}}
\put(14.7,1){\makebox(0,0)[c]{\tiny $\epsilon_{m} -\epsilon_{\bar{1}}$}}
\put(17.15,1){\makebox(0,0)[c]{\tiny $\epsilon_{\bar{1}} -\epsilon_{\bar{2}}$}}
\put(23.5,1){\makebox(0,0)[c]{\tiny $\epsilon_{\ov{n-1}} -\epsilon_{\bar{n}}$}}
\end{picture}
\end{center}
As usual, $\bigotimes$ stands for an isotropic simple odd root, $\bigcirc$ stands for an even simple root,
and $\bullet$ stands for a non-isotropic odd simple root. A direct computation shows that
\begin{equation}  \label{eq:rhobst}
\rho_{{\bf b}^{\text{st}}} = -\hf \epsilon_1 -\frac32 \epsilon_2 -\ldots -(m-\hf) \epsilon_m
+ (m-\hf) \epsilon_{\bar{1}} 
+\ldots + (m-n+\hf) \epsilon_{\bar{n}}.
\end{equation}

More generally, associated to a sequence $\bf b$ which starts with $0$ is a Dynkin diagram which always starts on the left with a short
{\em even} simple root:

\begin{center}
\hskip -3cm \setlength{\unitlength}{0.16in}
\begin{picture}(24,4)
\put(.6,2){\makebox(0,0)[c]{($\star$)}}
\put(5.6,2){\makebox(0,0)[c]{$\bigcirc$}}
\put(8,2){\makebox(0,0)[c]{$\bigodot$}}
\put(10.4,2){\makebox(0,0)[c]{$\bigodot$}}
\put(14.85,2){\makebox(0,0)[c]{$\bigodot$}}
\put(17.25,2){\makebox(0,0)[c]{$\bigodot$}}
\put(19.4,2){\makebox(0,0)[c]{$\bigodot$}}
\put(23.5,2){\makebox(0,0)[c]{$\bigodot$}}
\put(8.35,2){\line(1,0){1.5}} \put(10.82,2){\line(1,0){0.8}}
\put(13.2,2){\line(1,0){1.2}} \put(15.28,2){\line(1,0){1.45}}
\put(17.7,2){\line(1,0){1.25}} \put(19.81,2){\line(1,0){0.9}}
\put(22,2){\line(1,0){1}}
\put(6.8,2){\makebox(0,0)[c]{$\Longleftarrow$}}
\put(12.5,1.95){\makebox(0,0)[c]{$\cdots$}}
\put(21.5,1.95){\makebox(0,0)[c]{$\cdots$}}
\put(5.4,1){\makebox(0,0)[c]{\tiny $-\epsilon_{1}$}}
\end{picture}
\end{center}
Here $\bigodot$ stands for either $\bigotimes$ or $\bigcirc$ depending on $\bf b$.
Associated to a sequence $\bf b$ which starts with $1$ is a Dynkin diagram which always starts on the left with a  
non-isotropic {\em odd} simple root:

\begin{center}
\hskip -3cm \setlength{\unitlength}{0.16in}
\begin{picture}(24,4)
\put(.6,2){\makebox(0,0)[c]{($\star\star$)}}
\put(5.9,2){\makebox(0,0)[c]{$\bullet$}}
\put(8,2){\makebox(0,0)[c]{$\bigodot$}}
\put(10.4,2){\makebox(0,0)[c]{$\bigodot$}}
\put(14.85,2){\makebox(0,0)[c]{$\bigodot$}}
\put(17.25,2){\makebox(0,0)[c]{$\bigodot$}}
\put(19.4,2){\makebox(0,0)[c]{$\bigodot$}}
\put(23.5,2){\makebox(0,0)[c]{$\bigodot$}}
\put(8.35,2){\line(1,0){1.5}} \put(10.82,2){\line(1,0){0.8}}
\put(13.2,2){\line(1,0){1.2}} \put(15.28,2){\line(1,0){1.45}}
\put(17.7,2){\line(1,0){1.25}} \put(19.81,2){\line(1,0){0.9}}
\put(22,2){\line(1,0){1}}
\put(6.8,2){\makebox(0,0)[c]{$\Longleftarrow$}}
\put(12.5,1.95){\makebox(0,0)[c]{$\cdots$}}
\put(21.5,1.95){\makebox(0,0)[c]{$\cdots$}}
\put(5.4,1){\makebox(0,0)[c]{\tiny $-\epsilon_{\bar{1}}$}}
\end{picture}
\end{center}

\begin{rem}  \label{rem:rhob}
For general $\bf b$, one checks that $\rho_{\bf b}$ has a summand $(m-n+\hf) \epsilon_{\bar{n}}$ as for 
$\rho_{{\bf b}^{\text{st}}}$ in \eqref{eq:rhobst} if the Dynkin diagram associated to $\bf b$ has
$\bigcirc$ as its rightmost node, and that $\rho_{\bf b}$ has a summand $(m-n-\hf) \epsilon_{\bar{n}}$
if the Dynkin diagram associated to $\bf b$ has
$\bigotimes$ as its rightmost node.
\end{rem}

Now we can write the non-degenerate symmetric bilinear form on $\Phi$ as follows:
\[
(\ep^{b_i}_i | \ep^{b_j}_j) = (-1)^{b_i} \delta_{ij}, \quad \quad \quad 1 \leq i, j \leq m+n.
\] 
We define $\mathfrak{n}_{\bf b}^\pm$ to be the nilpotent subalgebra spanned by
the positive/negative root vectors in $\osp(2m+1|2n)$. 
Then we obtain a triangular decomposition of $\osp(2m+1|2n)$:
\[
\osp(2m+1|2n) = \mathfrak{n}_{\bf b}^+ \oplus \mathfrak{h}_{m|n} \oplus \mathfrak{n}_{\bf b}^-,
\]
with $\mathfrak{n}_{\bf b}^+ \oplus \mathfrak{h}_{m|n}$ as a Borel subalgebra. 

Fix a $0^m1^n$-sequence ${\bf b}$ and hence a positve system $\Phi^+_{\bf b}$. 
We denote by $Z(\osp(2m+1|2n))$ the center of the enveloping algebra $U(\osp(2m+1|2n))$. 
There exists a standard projection $\phi: U(\osp(2m+1|2n)) \rightarrow U(\mf h_{m|n})$ which is
consistent with the PBW basis associated to the above triangular decomposition 
(\cite[\S 2.2.3]{CW12}). For $\lambda \in \mf h^*_{m|n}$, we define the central character $\chi_{\lambda}$ by letting 
$$
\chi_\lambda(z) :=\lambda(\phi(z)),\quad  \text{ for }z \in Z(\osp(2m+1|2n)).
$$
Denote the Weyl group of (the even subalgebra of) $\osp(2m+1|2n)$ by $W_{\osp}$,
which is isomorphic to $(\Z_{2} \rtimes \mf{S}_{m}) \times (\Z_{2} \rtimes \mf{S}_{n})$. 
Then for $\mu$, $\nu \in \mf h^*_{m|n}$, we say $\mu$, $\nu$ are linked and denote it by $\mu \sim \nu$, 
if there exist mutually orthogonal isotropic odd roots $\alpha_1, \alpha_2, \dots, \alpha_l$, 
complex numbers $c_1, c_2, \dots, c_l$, and an element $w \in W_{\osp}$ satisfying 
\[
\mu + \rho_{\bf b}= w(\nu +\rho_{\bf b} - \sum_{i=1}^{l}c_i\alpha_i), \quad (\nu + \rho_{\bf b}|  \alpha_j)= 0, \quad j=1 \dots, l.
\]
It is clear that $\sim$ is an equivalent relation on $\mf h^*_{m|n}$.  Versions of the following basic fact went back
to Kac, Sergeev, and others.
\begin{prop}
  \cite[Theorem 2.30]{CW12}
Let $\lambda$, $\mu \in \mf h^*_{m|n}$. Then $\lambda$ is linked to $\mu$ if and only if $\chi_{\lambda} = \chi_{\mu}$.
\end{prop}
We define the {\em Bruhat ordering} $\preceq_{\bf b}$ on $\mf h^*_{m|n}$ and hence on $X(m|n)$  as follows:
\begin{equation}  \label{eq:BrX}
\la \preceq_{\bf b} \mu \Leftrightarrow \mu -\la \in \N \Pi_{\bf b} \text{ and } \la \sim \mu, 
\qquad \text{ for } \la, \mu \in \mf h^*_{m|n}.
\end{equation}

\section{Infinite-rank Lie superalgebras}
  \label{subsection:infinite rank Lie superalgebra}

We shall define the infinite-rank Lie superalgebras $\osp(2m+1|2n+\infty)$ and $\osp(2m+1|2n|\infty)$. 
Define the sets 
\begin{align*}
&\widetilde{\mathbb{J}} :=\{1, 2 , \dots, 2m+1, \ov 1, \ov 2, \dots, \ov{2n}\}
 \cup\Big\{ {\small \ul \hf , \ul 1, \ul {\frac{3}{2}}, \dots} \Big \}
 \cup \Big \{ \ul \hf' , \ul 1', \ul {\frac{3}{2}}', \dots \Big \},
\\
&\mathbb{J} := \{1, 2 , \dots, 2m+1, \ov 1, \ov 2, \dots, \ov{2n}\}\cup\{ \ul 1, \ul2, \dots\}\cup \{\ul 1', \ul2', \dots\},
\\
&\breve{\mathbb{J}} :=\{1, 2 , \dots, 2m+1, \ov 1, \ov 2, \dots, \ov{2n}\}
 \cup \Big\{\ul \hf, \ul{\frac{3}{2} }, \dots \Big\}
 \cup \Big\{\ul \hf', \ul{\frac{3}{2} }', \dots \Big\}.
\end{align*}
Let $\widetilde V$ be the infinite-dimensional superspace over $\C$ with the basis $\{e_i \mid i \in \widetilde{\mathbb{J}}\}$, 
whose $\Z_2$-grading is specified as follows:
\begin{align*}
&p(e_i) = \ov 0 ~ (1 \leq i \leq 2m+1), && p(e_{\ov j}) = \ov 1~ (1 \leq j \leq 2n), 
\\
&p(e_{\ul s'})=p(e_{\ul s}) = \ov 0 ~(s \in \Z_{>0}),&& p(e_{\ul t'}) =p(e_{\ul t}) = \ov 1 ~(t \in \hf +\N).
\end{align*}
With respect to this basis, a linear map on $\widetilde{V}$ may be identified with a complex matrix 
$(a_{rs})_{r,s \in \widetilde{\mathbb{J}}}$. Let $\mf {gl}(\widetilde V)$ be the Lie superalgebra 
consisting of $(a_{rs})_{r,s \in \widetilde{\mathbb{J}}}$ with $a_{rs} = 0 $ for almost all but finitely many $a_{rs}$'s. 
The standard Cartan subalgebra of $\mf {gl}(\widetilde V)$ is spanned by $\{E_{rr} \mid r \in \widetilde{\mathbb{J}}\}$, 
with dual basis $\{\ep_r \mid r \in \widetilde{\mathbb{J}}\}$.
The superspaces $V$ and $\breve V$ are defined to be the subspaces of $\widetilde{V}$ with basis $\{e_i\}$ 
indexed by $\mathbb{J}$ and $\breve{\mathbb{J}}$ respectively. Similarly we can define $\mf {gl}(V)$ and $\mf{gl}(\breve V)$.

Recall the supersymmetric non-degenerate bilinear form $B$ define in \S \ref{subsec:osp}. 
We can easily identify $\C^{2m+1|2n}$ as a subspace of $\widetilde V$.  
Define a supersymmetric non-degenerate bilinear form $\tilde B$ on $\widetilde V$ by
\begin{align*}
&\tilde{B}(e_s, e_t) = B(e_s, e_t), && \tilde B(e_s, e_{\ul x}) =  \tilde B(e_s, e_{\ul x'}) = 0,\\
&\tilde B(e_{\ul x}, e_{\ul y})=\tilde B(e_{\ul x'}, e_{\ul y'})  = 0, && \tilde B(e_{\ul x}, e_{\ul y'}) = \delta_{x,y} = (-1)^{p(e_{\ul x}) p(e_{\ul y'})}  \tilde B(e_{\ul y'}, e_{\ul x}),
\end{align*}
where $s, t \in \{i, \ov j \mid 1 \leq i \leq 2m+1, 1 \leq j \leq 2n\}$ and $x, y \in \{\ul \hf , \ul 1, \ul {\frac{3}{2}} \dots\}$. 
By restriction, we can obtain a supersymmetric non-degenerate bilinear form on $V$ and $\breve V$. 

Following \S \ref{subsec:osp}, we define $\osp(V)$ and $\osp(\breve V)$ to be the subalgebra of $\mf {gl}(V)$ 
and $\mf{gl}(\breve V)$ preserving the bilinear forms,  respectively. 
With respect to the standard basis of $V$ and $\breve V$, we identify
\begin{align*}
\ospV = \osp(V),\qquad
\ospW = \osp(\breve V). 
\end{align*}
The standard Cartan subalgebras of $\ospV$ and $\ospW$ are obtained by taking the intersection of the standard 
Cartan subalgebra of $\gl(\widetilde{V})$ with $\ospV$ and $\ospW$, respectively,
which are denoted by $\h_{m|n|\infty}$ and $\h_{m|n+\infty}$. 
For any $0^m1^n$-sequence ${\bf b}$, we assign the following simple system to the Lie superalgebra $\ospV$:
\[ 
\Pi_{{\bf b}, 0} :=\{-\ep^{b_1}_{1}, \ep^{b_i}_{i}-\ep^{b_{i+1}}_{i+1}, \ep^{b_{m+n}}_{m+n}-\ep^0_{\ul 1}, 
\ep^0_{\ul j}-\ep^0_{\ul {j+1}} \mid 1 \leq i \leq m+n-1, 1 \leq j \}.
\]
Similarly, we assign the following simple system to $\ospW$:
\[\Pi_{{\bf b}, 1} :=\{-\ep^{b_1}_{1}, \ep^{b_i}_{i}-\ep^{b_{i+1}}_{i+1}, \ep^{b_{m+n}}_{m+n}-\ep^1_{\ul 1},  
 \ep^1_{\ul j}-\ep^1_{\ul {j+1}} \mid 1 \leq i \leq m+n-1, 1 \leq j\}.
\] 
The $\ep^{b_i}_{i}$'s are defined in the same way as in \S \ref{subsec:osp} and
it is understood that $\ep^1_{\ul j} := \ep_{\ul {j-\hf}}$, $\ep^0_{\ul j} :=\ep_{\ul j}$ for any $1 \le j$. 
We introduce the following formal symbols:
\begin{align*}
\epinftyV := \sum_{j \geq 1}\ep^0_{\ul j},\qquad\quad 
\epinftyW :=\sum_{j \geq 1}\ep^1_{\ul j}.
\end{align*}
Let $\mc P$ be the set of all partitions. We define 
\begin{align}
  \label{eq:wtlV}
\wtlV :=\{\sum^{m+n}_{i=1}\lambda_i\ep^{b_i}_{i} + \sum_{1 \leq j}{^+\lambda_{\ul j}}\ep^0_{\ul j} 
+ d\epinftyV \mid d, \lambda_i \in \Z, ({^+\lambda_{\ul 1}}, {^+\lambda_{\ul 2}}, \dots) \in \mc P\},
\\
  \label{eq:wtlW}
\wtlW := \{\sum^{m+n}_{i=1}\lambda_i\ep^{b_i}_{i} + \sum_{1 \leq j}{^+\lambda_{\ul j}}\ep^1_{\ul j} 
 + d\epinftyW \mid d,\lambda_i \in \Z, ({^+\lambda_{\ul 1}}, {^+\lambda_{\ul 2}}, \dots) \in \mc P\}.
\end{align}

\section{The BGG categories}
\label{subsec:cat}

We shall define various parabolic BGG categories for ortho-symplectic Lie superalgebras.

\begin{definition}
 \label{def:catO}
Let ${\bf b}$ be a $0^m1^n$-sequence. The Bernstein-Gelfand-Gelfand
(BGG) category $\mathcal{O}_{\bf b}$ 
is the category of  
$\h_{m|n}$-semisimple $\mathfrak{osp}(2m+1|2n)$-modules $M$ such that
\begin{itemize}
\item[(i)]
$M=\bigoplus_{\mu\in X(m|n)}M_\mu$ and $\dim M_\mu<\infty$;

\item[(ii)]
there exist finitely many weights ${}^1\la,{}^2\la,\ldots,{}^k\la\in X(m|n)$
(depending on $M$) such that if $\mu$ is a weight in $M$, then
$\mu\in{{}^i\la}-\sum_{\alpha\in{\Pi_{\bf b}}}\N \alpha$, for
some $i$.
\end{itemize}
The morphisms in $\mathcal{O}_{\bf b}$ are all (not necessarily even)
homomorphisms of $\mathfrak{osp}(2m+1|2n)$-modules.
\end{definition}
Similar to \cite[Proposition 6.4]{CLW12}, all these categories $\mc O_{\bf b}$ are identical for various $\bf b$,
since the even subalgebras of the Borel subalgebras 
$\mathfrak{n}_{\bf b}^+ \oplus \mathfrak{h}_{m|n}$ are identical and the odd parts of these Borels 
always act locally nilpotently.

Denote by $M_{\bf b}(\lambda)$ the {\bf b}-Verma modules with highest weight $\lambda$. 
Denote by  $L_{\bf b}(\lambda)$ the unique simple quotient of $M_{\bf b}(\lambda)$. They are both in $\mathcal{O}_{\bf b}$.

It is well known that the Lie superalgebra $\gl(2m+1|2n)$ has an automorphism $\tau$ given by the formula:
\[
\tau(E_{ij}):=-(-1)^{p(i)(p(i)+p(j))}E_{ji}.
\]
The restriction of $\tau$ on $\osp(2m+1|2n)$ gives an automorphism of $\osp(2m+1|2n)$. For an object 
$M = \oplus_{\mu \in X(m|n)}M_\mu \in \mathcal{O}_{\bf b}$, we let
\[
M^{\vee}:=\oplus_{\mu \in X(m|n)}M^*_{\mu}
\]
be the restrictd dual of $M$. We define the action of $\osp(2m+1|2n)$ on $M^{\vee}$ by $(g \cdot f)(x) := -f(\tau(g)\cdot x)$, 
for $ f \in M^\vee, g\in \osp(2m+1|2n)$, and $ x\in M$. We denote the resulting module by $M^\tau$. 

An object $M \in \mathcal{O}_{\bf b}$ is said to have a ${\bf b}$-Verma flag 
(respectively, dual ${\bf b}$-Verma flag), if $M$ has a filtration 
$
0=M_0 \subseteq M_1 \subseteq M_2 \subseteq \dots \subseteq M_t = M,
$
such that $M_i/M_{i-1} \cong M_{\bf b}(\gamma_i), 1 \leq i \leq t$ (respectively, $M_i/M_{i-1} \cong M^\tau_{\bf b}(\gamma_i)$) for some $\gamma_i \in X(m|n)$.

\begin{definition}
  \label{def:tilting}
Associated to each $\lambda \in X(m|n)$, a ${\bf b}$-tilting module $T_{\bf b}(\lambda)$ is an indecomposable 
$\osp(2m+1|2n)$-module in $\mathcal{O}_{\bf b}$ characterized by the following two conditions:
\begin{itemize}
\item[(i)]
$T_{\bf b}(\la)$ has a ${\bf b}$-Verma flag with $M_{\bf b}(\la)$ at
the bottom;
\item[(ii)]
$\text{Ext}^1_{\CatO_{\bf b}}(M_{\bf b}(\mu),T_{\bf b}(\la))=0$, for all
$\mu\in X(m|n)$.
\end{itemize}
\end{definition}

Recall the definition of the infinite-rank Lie superalgebras in \S \ref{subsection:infinite rank Lie superalgebra}.
For a nonempty $0^m1^n$-sequence ${\bf b} =(b_1, b_2, \dots , b_{m+n})$ and $k \in \N \cup \{\infty\}$,  
consider the extended sequence $({\bf b}, 0^k) = (b_1, b_2, \dots , b_{m+n}, 0,  \dots,0)$. 
This sequence corresponds to the following simple system of the Lie superalgebra $\mathfrak{osp}(2m+2k+1|2n)$, which
we shall denote by $\mathfrak{osp}(2m+1|2n|2k)$ throughout this paper to indicate the choice of $\Pi_{({\bf {b}}, 0^k)}$:
\[
\Pi_{({\bf {b}}, 0^k)} = \{-\ep^{b_1}_1, \ep^{b_i}_i -\ep^{b_{i+1}}_{i+1} \mid 1 \leq i \leq m+n+k\}, 
\]
where $b_{i} = 0 \text{ for } i > m+n.$ 
Let $\Pi^{\underline{k}}_{{\bf b},0} = \{\ep^{b_i}_i -\ep^{b_{i+1}}_{i+1} \mid i > m+n \}$. 
Define the following Levi subalgebra and parabolic subalgebra of $\osp(2m+1|2n|2k)$:
\begin{align*}
\mf l^{\underline{k}}_{{\bf b},0} & := \sum_{\alpha \in \Z \Pi^{\underline{k}}_{{\bf b},0}}\osp(2m+1|2n|2k)_\alpha,
\\
\mathfrak{p}^{\underline{k}}_{{\bf b},0} &:=\sum_{\alpha \in \Phi^+_{({\bf b}, 0^k)}\cup \Z \Pi^{\underline{k}}_{{\bf b},0}}\osp(2m+1|2n|2k)_\alpha.
\end{align*}
Let $L_0(\lambda)$ be the irreducible $\mf l^{\underline{k}}_{{\bf b},0}$-module with highest weight $\lambda$. 
It can be extended trivially to a $\mathfrak{p}^{\underline{k}}_{{\bf b},0}$-module. We form the parabolic Verma module
\[
M^{\ul k}_{{\bf b},0}(\lambda) :=\text{Ind}^{\osp(2m+1|2n|2k)}_{\mathfrak{p}^{\underline{k}}_{{\bf b},0}}L_0(\lambda).
\]
 For $k \in \N$, we define
\begin{align}
\wtlVk := \left\{\sum^{m+n}_{i =1}\lambda_{i}\ep^{b_i}_{i} + \sum^k_{j =1} {^+}\lambda_{\ul j}\ep^0_{\ul j} + d\sum_{j=1}^{k} \epsilon^0_{\ul j}
\mid d, \lambda_i \in \Z, (^+\lambda_{\ul 1},^+\lambda_{\ul 2}, \dots ) \in \mc P\right\},\label{align:wtlVk}
\\
\wtlWk := \left\{\sum^{m+n}_{i =1}\lambda_{i}\ep^{b_i}_{i} + \sum^k_{j =1}{^+}\lambda_{\ul j}\ep^1_{\ul j} + d\sum_{j=1}^{k} \epsilon^1_{\ul j}
\mid d, \lambda_i \in \Z, (^+\lambda_{\ul 1},^+\lambda_{\ul 2}, \dots ) \in \mc P\right\}\label{align:wtlWk}.
\end{align}
Recall the definition of $\wtlV$ and $\wtlW$ from \eqref{eq:wtlV}-\eqref{eq:wtlW}.
\begin{definition}

Let ${\bf b}$ be a $0^m1^n$-sequence and $k \in \N \cup \{\infty\}$. Let  $\mathcal{O}^{\ul k}_{{\bf b},0}$ be the category of  
$\h_{m|n|k}$-semisimple $\mathfrak{osp}(2m+1|2n|2k)$-modules $M$ such that
\begin{itemize}
\item[(i)]
$M=\bigoplus_{\mu}M_\mu$ and $\dim M_\mu<\infty$;

\item[(ii)]
M decomposes over $\mathfrak{l}^{\underline{\infty}}_{{\bf b},0}$ into a direct sum of $L_0(\lambda)$ for $\lambda \in \wtlVk$;
\item[(iii)]
there exist finitely many weights ${}^1\la,{}^2\la,\ldots,{}^k\la\in \wtlVk$
(depending on $M$) such that if $\mu$ is a weight in $M$, then
$\mu\in{{}^i\la}-\sum_{\alpha\in{\Pi_{({\bf b},0^{k})}}}\N\alpha$, for
some $i$.
\end{itemize}
The morphisms in $\mathcal{O}^{\ul k}_{{\bf b},0}$ are all (not necessarily even)
homomorphisms of $\mathfrak{osp}(2m+1|2n|2k)$-modules.
\end{definition}

Let $\lambda \in X^{\ul k,+}_{{\bf b},0}$. We shall denote by $L^{\ul k}_{{\bf b},0}(\lambda)$ the simple 
module in $\mathcal{O}^{\ul k}_{{\bf b},0}$ with highest weight $\lambda$. Following Definition~\ref{def:tilting}, 
we can define the tilting module $T^{\ul k}_{{\bf b},0}(\lambda)$ in $\mathcal{O}^{\ul k}_{{\bf b},0}$.

Similar construction exists for the sequence $({\bf b}, 1^k)$, where we consider the Lie 
superalgebras $\osp(2m+1|2n+2k)$ for $k \in \N\cup\{\infty\}$ with the following simple systems:
\[
\Pi_{({\bf {b}}, 1^k)} = \{-\ep^{b_1}_1, \ep^{b_i}_i -\ep^{b_{i+1}}_{i+1} \mid 1 \leq i \leq m+n+k\}, \quad\text{where } b_{i} = 1 \text{ for } i > m+n.
\]

Let $\Pi^{\underline{k}}_{{\bf b},1} = \{\ep^{b_i}_i -\ep^{b_{i+1}}_{i+1} \mid i > m+n \}$. 
Define the following Levi subalgebra and parabolic subalgebra of $\osp(2m+1|2n|2k)$:
\begin{align*}
\mf l^{\underline{k}}_{{\bf b},1}  &:=  \sum_{\alpha \in \Z[\Pi^{\underline{k}}_{{\bf b},1}]}\osp(2m+1|2n|2k)_\alpha,
  \\
\mathfrak{p}^{\underline{k}}_{{\bf b},1} &:= \sum_{\alpha \in \Phi^+_{({\bf b}, 1^k)}
\cup \Z[\Pi^{\underline{k}}_{{\bf b},1}]}\osp(2m+1|2n|2k)_\alpha.
\end{align*}
Let $L_1(\lambda)$ be the simple $\mathfrak{l}^{\underline{k}}_{{\bf b},1}$-module with highest weight $\lambda$. 
It can be extended trivially to a  $\mathfrak{p}^{\underline{k}}_{{\bf b},1}$-module.  Similarly we can define the parabolic Verma module
\[
M^{\ul k}_{{\bf b},1}(\lambda) :=\text{Ind}^{\osp(2m+1|2n+2k)}_{\mathfrak{p}^{\underline{k}}_{{\bf b},1}}L_1(\lambda).
\]
\begin{definition}

For $k \in \N \cup \{\infty\}$, let  $\mathcal{O}^{\ul k}_{{\bf b},1}$ be the category of  
$\h_{2m+1|2n+2k}$-semisimple $\mathfrak{osp}(2m+1|2n+2k)$-modules $M$ such that
\begin{itemize}
\item[(i)]
$M=\bigoplus_{\mu}M_\mu$ and $\dim M_\mu<\infty$;

\item[(ii)]
M decomposes over $\mathfrak{p}^{\underline{k}}_{{\bf b},1}$ into a direct sum of $L_1(\lambda)$ for $\lambda \in \wtlWk$;

\item[(iii)]
there exist finitely many weights ${}^1\la,{}^2\la,\ldots,{}^k\la\in \wtlWk$
(depending on $M$) such that if $\mu$ is a weight in $M$, then
$\mu\in{{}^i\la}-\sum_{\alpha\in{\Pi_{({\bf b},1^{k})}}}\N\alpha$, for
some $i$.
\end{itemize}
The morphisms in $\mathcal{O}^{\ul k}_{{\bf b},1}$ are all (not necessarily even)
homomorphisms of $\mathfrak{osp}(2m+1|2n+2k)$-modules.
\end{definition}

For $\xi \in X^{\ul k,+}_{{\bf b},1}$, we shall denote by $L^{\ul k}_{{\bf b},1}(\xi)$ the simple module in 
$\mathcal{O}^{\ul k}_{{\bf b},1}$ with highest weight $\xi$. Following Definition~\ref{def:tilting}, 
we can define the tilting module $T^{\ul k}_{{\bf b},1}(\xi)$ in $\mathcal{O}^{\ul k}_{{\bf b},1}$.

\section{Truncation functors}

Recall the definition of $\wtlVk$ and $\wtlWk$ in \eqref{align:wtlVk} and \eqref{align:wtlWk}.
For any $\lambda  = \sum^{m+n}_{i=1}\lambda_i\ep^{b_i}_{i} 
+ \sum_{1 \leq j}{^+\lambda_{\ul j}}\ep^s_{\ul j} + d\ep_{\infty}^s \in X^{\ul \infty, +}_{{\bf b},s}$, we define
\[
\lambda^{\ul k}:=\sum^{m+n}_{i =1}\lambda_i\ep^{b_i}_i+ \sum^k_{j =1}{^+}\lambda_{\ul j}\ep^s_{\ul j}
+  d\sum^k_{j =1}\ep^s_{\ul j} \in X^{\ul k, +}_{{\bf b},s}, \quad \text{ for } s \in \{0,1\}.
\]
  Let $M^{\ul{\infty}}_{{{\bf b}, 0}} \in\OO^{\ul{\infty}}_{{{\bf b}, 0}}$
and $M^{\ul{\infty}}_{{{\bf b}, 1}} \in \OO^{\ul{\infty}}_{{{\bf b}, 1}}$. Then we have the weight
space decompositions
\begin{align*}
M^{\ul{\infty}}_{{{\bf b}, 0}}=\bigoplus_{\mu} M^{\ul{\infty}}_{{{\bf b}, 0, \mu}},\qquad
M^{\ul{\infty}}_{{{\bf b}, 1}}=\bigoplus_{\mu} M^{\ul{\infty}}_{{{\bf b}, 1,\mu}}.
\end{align*}

We define an exact functor $\mf{tr}_0:\OO^{\ul{\infty}}_{{{\bf b}, 0}} \rightarrow \OO^{\ul{k}}_{{{\bf b}, 0}}$ by
\[
\mf{tr}_0( M^{\ul{\infty}}_{{{\bf b}, 0}}):=\bigoplus_\mu M^{\ul{\infty}}_{{{\bf b}, 0,\mu}}, 
\]
$\text{ where $\mu$ satisfies } (\mu,\ep^0_{\ul{j}} - \ep^0_{\ul{j+1}})=0$, $\forall j\ge k+1\text{ and }j\in\N$.
Similarly, we define an exact functor
$\mf{tr}_1:\OO^{\ul{\infty}}_{{{\bf b}, 1}}\rightarrow \OO^{\ul{k}}_{{{\bf b}, 1}}$ by
\[
\mf{tr}_1 (M^{\ul{\infty}}_{{{\bf b}, 1}}):=\bigoplus_\mu
M^{\ul{\infty}}_{{{\bf b}, 1, \mu}},
\]
where $\mu$ satisfies $ (\mu, \ep^1_{\ul{j}} - \ep^1_{\ul{j+1}})=0$, $\forall j \geq k+1 \text{ and } j \in \N$.
The following has been known \cite{CW08, CLW11}; also see \cite[Proposition 6.9]{CW12}.
\begin{prop} 
  \label{prop:trunc:ML}
For $s=0$, $1$,
the functors $\mf{tr}_s:\OO^{\ul{\infty}}_{{{\bf b}, s}}\rightarrow \OO^{\ul{k}}_{{{\bf b}, s}}$ satisfy the following: for
$Y=M$, $L$, $T$,
$\la=\sum^{m+n}_{i=1}\lambda_i\ep^{b_i}_{i} + \sum_{1 \leq j}{^+\lambda_{\ul j}}\ep^s_{\ul j} + d\ep_{\infty}^s \in X^{\ul \infty, +}_{{\bf b},s}$,
we have
\[
\mf{tr}_s\left( Y^{\ul{\infty}}_{{{\bf b}, s}}(\la)\right)=
\begin{cases}
 Y^{\ul{k}}_{{{\bf b}, s}}(\la^{\ul k}),&\text{ if } l(^+\lambda) \leq k,
 \\
0,&\text{ otherwise}.
\end{cases}
\]
\end{prop}

\chapter{Fock spaces and Bruhat orderings}
  \label{sec:Fockspaces}

 In this chapter, we formulate the infinite-rank variants of the basic constructions in Part~1.
 We  set up various Fock spaces which are the $q$-versions of Grothendieck groups,
 and transport Bruhat ordering from the BGG categories  to the corresponding Fock spaces.
  
\section{Infinite-rank constructions}
  \label{sec:infiniterankQSP}
  
Let us first set up some notations which will be used often in Part~2.
We set 
\begin{align}
 \label{eq:III}
\I  = \cup^{\infty}_{r=0}  & \I_{2r+1} = \Z,
\qquad
\I^{\imath} = \cup^{\infty}_{r=0} \I^{\imath}_{r} = \Z_{>0},
\qquad
I  = \Z+\hf.
\end{align}

Recall from Chapter \ref{sec:qsp} the finite-rank quantum symmetric pair $(\U_{2r+1}, \U^{\imath}_r)$. 
We have the natural inclusions of $\Qq$-algebras:
\begin{align*}
\cdots \subset \U_{2r-1} \subset & \U_{2r+1} \subset \U_{2r+3} \subset \cdots ,\\
\cdots \subset \U^{\imath}_{r-1}   \subset & \U^{\imath}_{r} \subset \U^{\imath}_{r+1} \subset \cdots. 
\end{align*}
Define the following $\Qq$-algebras:
\[
\Ui := \bigcup^{\infty}_{r=0} \U^{\imath}_r  \quad \text{ and } \quad \U := \bigcup^{\infty}_{r=0} \U_{2r+1}.
\]
It is easy to see that $\Ui$ is generated by $\{\be_{\alpha_i}, \bff_{\alpha_i}, \bk^{\pm 1}_{\alpha_i}, \bt \mid i \in \I^{\imath} \}$, 
and $\U$ is generated by $\{E_{\alpha_i}, F_{\alpha_i}, K^{\pm 1}_{\alpha_i} \mid i \in \I\}$. 
The  embeddings of $\Qq$-algebras $\iota : \U^{\imath}_r \rightarrow \U_{2r+1}$ 
induce an embedding of $\Qq$-algebras, denoted also by $\iota : \Ui \longrightarrow \U$.
Again $\U$ is naturally a Hopf algebra with coproduct $\Delta$, and its restriction under $\iota$,
$
\Delta: \Ui \rightarrow \Ui \otimes \U,
$
makes $\Ui$ (or more precisely $\iota(\Ui)$) naturally a (right) coideal subalgebra of $\U$. 
The anti-linear bar involutions on $\U^{\imath}_r$ and $\U_{2r+1}$ induce anti-linear bar involutions 
on $\Ui$ and $\U$, respectively, both denoted by $\bar{\ }$ as well. 
As in Part ~1, in order to avoid confusion, we shall sometimes set $\Abar(u) := \ov{u}$ for $u \in \U$, 
and set $\Bbar(u) := \ov{u}$ for $u \in \Ui$.


Recall $\Pi_{2r+1}$ denotes the simple system of $\U_{2r+1}$. Then
$$
\Pi := \bigcup^{\infty}_{r=0} \Pi_{2r+1}
$$
is a simple system of $\U$.
Recall we denote the integral weight lattice of $\U_{2r+1}$ by $\Lambda_{2r+1}$. Then
\[
\Lambda := \oplus_{i \in \hf +\Z} \Z[\varepsilon_i] = \bigcup^{\infty}_{r =0} \Lambda_{2r+1}
\] 
is the integral weight lattice of $\U$. 
Following \S \ref{subsec:theta}, we have the quotient lattice $\Lambda_{\inv}$ of the lattice $\Lambda$.

Recall the intertwiner of the pair $(\U_{2r+1}, \U^{\imath}_r)$ in \S \ref{subsec:Upsilon}, 
which we shall denote by $\Upsilon^{(r)}$. 
We have  $\Upsilon^{(r)} = \sum_{\mu \in \N\Pi_{2r+1}} \Upsilon^{(r)}_{\mu}$ in a completion of $\U^{-}_{2r+1}$ with $\Upsilon^{(r)}_0=1$.
Following the construction of $\Upsilon^{(r)}$ in Theorem \ref{thm:Upsilon}, we see that
\[
\Upsilon^{(r+1)}_{\mu} = \Upsilon^{(r)}_{\mu}, \quad \text{ for } \mu \in \N\Pi_{2r+1}.
\]
Hence we can define an element  $\Upsilon_\mu \in \U^{-}_{\mu}$, for $\mu \in \N\Pi$ by letting
\[
\Upsilon_{\mu} := \displaystyle\lim_{r \to \infty} \Upsilon^{(r)}_{\mu}.
\]
Define the formal sum $\Upsilon$ (which lies in some completion of $\U^-$) by
\begin{equation}
  \label{pt2:eq:Upsilon}
\Upsilon := \sum_{\mu \in \N\Pi} \Upsilon_{\mu}.
\end{equation}
We shall view $\Upsilon$ as a well-defined operator on $\U$-modules that we are concerned.

\section{The Fock space $\mathbb{T}^{\bf b}$}
  \label{subsec:Fockspaces}

Let $\VV := \sum_{a \in I} \Qq v_a$ be the natural representation of $\U$, 
where the action of $\U$ on $\VV$ is defined as follows (for $i \in \I$, $a \in I$):
\[
E_{\alpha_i} v_a = \delta_{i+\hf, a} v_{a-1}, \quad F_{\alpha_i}v_a 
= \delta_{i-\hf, a}v_{a+1}, \quad K_{\alpha_i} v_a = q^{(\alpha_i, \varepsilon_a)}v_a.
\]
Let $\WW :=\VV^*$ be the restricted dual module of $\VV$ with basis $\{w_a \mid a \in I\}$ such that 
$\langle w_a, v_b \rangle =  (-q)^{-a} \delta_{a,b}$. The action of $\U$ on $\WW$ is given by the following formulas 
(for $i \in \I$, $a \in I$):
\[
E_{\alpha_i} w_a = \delta_{i-\hf, a} w_{a+1}, \quad F_{\alpha_i}w_a 
= \delta_{i+\hf, a}w_{a-1}, \quad K_{\alpha_i} w_a = q^{-(\alpha_i, \varepsilon_a)}w_a.
\]
By restriction through the embedding $\iota$, $\VV$ and $\WW$ are naturally  $\Ui$-modules.

Fix a ${0^m1^n}$-sequence ${\bf b} =(b_1,b_2,\ldots,b_{m+n})$. We have the following
tensor space over $\Q(q)$, called the {\em $\bf b$-Fock space} or simply
{\em Fock space}:
\begin{equation}  \label{eq:Fock}
{\mathbb T}^{\bf b} :={\mathbb V}^{b_1}\otimes {\mathbb
V}^{b_2}\otimes\cdots \otimes{\mathbb V}^{b_{m+n}},
\end{equation}
where we denote
$${\mathbb V}^{b_i}:=\begin{cases}
{\mathbb V}, &\text{ if }b_i={0},\\
{\mathbb W}, &\text{ if }b_i={1}.
\end{cases}$$
The tensors here and in similar settings later on are understood
to be over the field $\Q(q)$.
Note that both algebras $\U$ and $\Ui$ act on $\mathbb T^{\bf b}$ via an iterated coproduct.

For $f\in I^{m+n}$, we define
\begin{equation}  \label{eq:Mf}
M^{\bf b}_f :=\texttt{v}^{b_1}_{f(1)}\otimes
\texttt{v}^{b_2}_{f(2)}\otimes\cdots\otimes
\texttt{v}^{b_{m+n}}_{f(m+n)},
\end{equation}
where we use the notation
$\texttt{v}^{b_i}:=
\begin{cases}v,\text{ if }b_i={0},
\\
w,\text{ if }b_i={1}.
\end{cases}$ 
We refer to $\{M^{\bf b}_f \mid f\in
I^{m+n}\}$ as the {\em standard monomial basis} of ${\mathbb
T}^{\bf b}$.

For $r\in \N$, we shall denote the natural representation of $\U_{2r+1}$ by $\VV_r$ now, 
where $\VV_r$ admits a natural basis $\{v_a \mid a \in \I_{2r+2}\}$. 
Let $\WW_r$ be the dual of $\VV_r$ with basis $\{w_a \mid a \in \I_{2r+2}\}$ 
such that $\langle w_a, v_b \rangle =  (-q)^{-a} \delta_{a,b}$. 
We have natural inclusions of $\Qq$-spaces
$$
\cdots \subset \VV_{r-1} \subset \VV_r \subset \VV_{r+1} \cdots, 
\quad \text{ and } \quad
\cdots \subset \WW_{r-1} \subset \WW_r \subset \WW_{r+1} \cdots.
$$
Similarly we can define the space
\[
\mathbb{T}^{\bf b}_r :={\mathbb V}_r^{b_1}\otimes {\mathbb
V}_r^{b_2}\otimes\cdots \otimes{\mathbb V}_r^{b_{m+n}},
\]
where we denote
$$
{\mathbb V}_r^{b_i}:=\begin{cases}
{\mathbb V}_r, &\text{ if }b_i={0},\\
{\mathbb W}_r, &\text{ if }b_i={1}.
\end{cases}
$$
Then $\{M^{\bf b}_f \mid f\in
\I_{2r+2}^{m+n}\}$ forms the {\em standard monomial basis} of ${\mathbb
T}_r^{\bf b}$. In light of the standard monomial bases, we may view 
\begin{equation}
  \label{osp:eq:TrT}
\cdots \subset \mathbb{T}_r^{\bf b} \subset \mathbb{T}_{r+1}^{\bf b} \subset \cdots,
\qquad \text { and } \quad \mathbb{T}^{\bf b} = \cup_{r \in \N} \mathbb{T}_r^{\bf b}.
\end{equation}

\begin{definition}  \label{def:wtb}
For $f \in \I_{2r+2}^{m+n}$, let $\texttt{wt}_{\bf b} (f)$ be the $\Ui$-weight of $M^{\bf b}_f$, i.e., 
the image of the $\U$-weight in $\Lambda_{\inv}$. 
\end{definition}

\section{The $q$-wedge spaces}
  \label{osp:subsec:$q$-wedges}

Recall from \S \ref{sec:HeckeB} the right action on $\VV^{\otimes k}$  
on the Hecke algebra $\mc{H}_{B_k}$,  
where $\VV$ is now of infinite dimension. We take $\wedge^k\VV$ as the quotient of $\VV^{\otimes k}$ 
by the sum of the kernel of the operators $H_i-q^{-1}$, $1 \leq i \leq k-1$. The $\wedge^k\VV$ is naturally a $\U$-module, 
hence also a $\Ui$-module. For any $v_{p_1} \otimes v_{p_2} \otimes \cdots \otimes v_{p_k}$ in $\VV^{\otimes k}$, 
we denote its image in $\wedge^{k}\VV$ by $v_{p_1} \wedge v_{p_2} \wedge \cdots \wedge v_{p_k}$.

For $d \in \Z$ and $l \ge k$, consider the $\Qq$-linear maps
\begin{align*}
\wedge^{k,l}_d &: \wedge^{k}\VV \longrightarrow \wedge^{l}\VV
\\
v_{p_1}  \wedge \cdots \wedge v_{p_k} \mapsto v_{p_1} \wedge \cdots \wedge v_{p_k} 
& \wedge v_{d+\hf-k-1} \wedge v_{d+\hf-k-2} \wedge \cdots \wedge v_{d+\hf-l}.
\end{align*}
Let $\wedge_d^{\infty}\VV : =\varinjlim \wedge^{k}\VV$ be the direct limit of the $\Qq$-vector spaces 
with respect to the maps $\wedge^{k,l}_d$, which is called the $d$th sector of the semi-infinite $q$-wedge space
$\wedge^{\infty}\VV$; that is,
\[
\wedge^{\infty}\VV = \bigoplus_{d \in \Z} \wedge_d^{\infty}\VV.
\]
Note that for any fixed $u \in \U$ and fixed $d \in \Z$, we have 
\[
\wedge^{k,l}_d u = u \wedge^{k,l}_d : \wedge^{k}\VV \longrightarrow \wedge^{l} \VV, \quad \text{ for } l \ge k \gg 0.
\]
Therefore $\wedge^{\infty}_d \VV$ and hence $\wedge^{\infty}\VV$  become  both $\U$-modules and $\Ui$-modules.

We can think of elements in $\wedge^{\infty}\VV$ as linear combinations of infinite $q$-wedges of the form
\[
v_{p_1} \wedge v_{p_2} \wedge  v_{p_3} \wedge \cdots,
\]
where  $p_1 > p_2  >p_3> \cdots,$ and $p_i - p_{i+1}=1$ for $i \gg0$.
%
%
%
%
%
%
Alternatively, the space
$\wedge^\infty\mathbb V$ has a basis indexed by pairs of a partition and an integer given by
$$
|\la,d\rangle := v_{\la_1+d-\frac{1}{2}} \wedge v_{\la_2+d-\frac{3}{2}} \wedge v_{\la_3+d-\frac{5}{2}}\wedge
\cdots,
$$
where $\la =(\la_1, \la_2, \ldots)$ runs over the set $\mc P$ of all
partitions, and $d$ runs over $\Z$. Clearly we can realize $\wedge_d^{\infty}\VV$ as the subspace of 
$\wedge^\infty\VV$ spanned by $\{|\la, d \rangle \mid {\lambda \in \mc P}\}$, for $d \in \Z$.

In the rest of this paper, we shall index the q-wedge spaces by 
$[\ul k] := \{\ul 1, \ul 2, \dots, \ul k\}$ and $[\ul \infty] :=\{\ul 1, \ul 2 \dots\}$. More precisely, let
\begin{align*}
I^{k}_+ &=\{f:[\ul k] \rightarrow I \mid f(\ul 1)> f(\ul 2)>
\cdots> f(\ul k)\}, \quad  \text{ for } k\in\N,
  \\
I^{ \infty}_+ &=\{f:[\ul \infty] \rightarrow I \mid f(\ul 1)>
f(\ul 2)> \cdots; f(\ul t) - f(\ul {t+1})=1 \text{ for } t \gg 0\}.
\end{align*}
For $f\in I^{k}_+$, we denote
\begin{align*}
\mathcal{V}_f=v_{f(\ul 1)}\wedge v_{f(\ul 2)}\wedge\cdots \wedge
v_{f(\ul k)}.
\end{align*}
Then $\{\mathcal{V}_f \mid f\in I^k_+\}$ is a basis of $\wedge^k\mathbb
V$, for $k\in\Z_{>0} \cup \{ \infty\}$. 

For $k\in \Z_{>0}$, we let $w^{(k)}_0$ be the longest element in $\mathfrak{S}_{k}$. Define  
\[
L_{w^{(k)}_0} := \sum_{w \in \mathfrak{S}_k}(-q)^{l(w)-l(w^{(k)}_{0})}H_w  \in \mathcal{H}_{A_{k-1}}.
\]
It is well known \cite{KL} that $\overline{L_{w^{(k)}_0}} = L_{w^{(k)}_0}$. 
The right action by $L_{w^{(k)}_0}$ define a $\Qq$-linear map  (the $q$-skew-symmetrizer)
$\text{SkSym}_k: \VV^{\otimes k} \rightarrow \VV^{\otimes k}$. Then
the $q$-wedge space $\wedge^k\VV$ can also be regarded as a subspace $\text{Im} (\text{SkSym}_k)$ of $\VV^{\otimes k}$
while identifying 
$\mathcal{V}_{f} \equiv M^{(0^{k})}_{f \cdot w^{(k)}_{0}} L_{w^{(k)}_0}$ for $f \in I^{k}_{+}$ (cf. e.g. \cite[\S 4.1]{CLW12}).

Similar construction gives rise to $\wedge^{\infty} \WW$.
For each $d \in \Z$ and $l \ge k$, consider the $\Qq$-linear maps
\begin{align}
\label{eq:wedge} 
\wedge^{k,l}_d&: \wedge^{k}\WW \longrightarrow \wedge^{l}\WW
\\
\notag w_{p_1}  \wedge \cdots \wedge w_{p_k} \mapsto w_{p_1}  \wedge \cdots \wedge w_{p_k} 
& \wedge w_{d-\hf+k+1} \wedge w_{d-\hf+k+2} \wedge \cdots \wedge w_{d-\hf+l}.
\end{align}
Let $\wedge_d^{\infty}\WW : =\varinjlim \wedge^{k}\WW$ be the direct limit of the $\Qq$-vector 
spaces with respect to the maps $\wedge^{k,l}_d$. Define
\[
\wedge^{\infty}\WW := \bigoplus_{d \in \Z} \wedge_d^{\infty}\WW.
\]
Note that for any fixed $u \in \U$ and fixed $d \in \Z$, we have 
\[
\wedge^{k,l}_d u = u \wedge^{k,l}_d : \wedge^{k}\WW \rightarrow \wedge^{l} \WW, \quad \text{ for } l \ge k \gg 0.
\]
Therefore $\wedge^{\infty}_d \WW$ and hence $\wedge^{\infty}\WW$  become  both $\U$-modules and $\Ui$-modules.

We can think of elements in $\wedge^{\infty}\WW$ as linear combinations of infinite $q$-wedges of the form
\[
w_{p_1} \wedge w_{p_2} \wedge w_{p_3} \wedge \cdots,
\]
where $p_1<p_2<p_3<\cdots$, and $p_i - p_{i+1}= -1$, for $i\gg 0$.
%
%
%
%
%
%
%
%
Alternatively,
the space $\wedge^\infty\mathbb W$ has a basis indexed by partitions
given by
$$
|\la_*,d\rangle := w_{d +\frac{1}{2}-\la_1} \wedge w_{d +\frac{3}{2}-\la_2} \wedge
w_{d +\frac{5}{2}-\la_3}\wedge \cdots, $$ where $\la =(\la_1, \la_2, \cdots)$
runs over the set $\mc{P}$ of all partitions, and $d$ runs over $\Z$. Clearly we can realize $\wedge_d^{\infty}\WW$ 
as the subspace of $\wedge^\infty\WW$ spanned by $\{|\la_*, d \rangle \mid {\lambda \in \mc P}\}$, for $d \in \Z$. 

Let
\begin{align*} 
I^{k}_- &=\{f:[\ul k] \rightarrow I \mid f(\ul 1)<f(\ul 2)<
\cdots< f(\ul k)\},  \text{ for }  k\in\N,
  \\
I^{\infty}_- &=\{f:[\ul \infty] \rightarrow I \mid f(\ul 1)<
f(\ul 2)< \cdots; f(\ul t)-f(\ul{t+1})=-1 \text{ for }t\gg 0\}.
\end{align*}
For $f\in I^{k}_-$, we denote
\begin{align*}
\mathcal{W}_{f}=w_{f(\ul 1)}\wedge w_{f(\ul 2)}\wedge\cdots\wedge
w_{f(\ul k)}.
\end{align*}
Then $\{\mathcal{W}_f \mid f\in I^k_-\}$ is a basis of $\wedge^k\mathbb
W$, for $k\in\N \cup \{\infty\}$.

\begin{rem}
The semi-infinite $q$-wedge spaces considered in this paper will involve all sectors, 
while only the $0$th sector was considered and needed in \cite[\S 2.4]{CLW12}.  
\end{rem}

\section{Bruhat orderings}

Let ${\bf b} =(b_1, \ldots, b_{m+n})$ be an arbitrary $0^m1^n$-sequence.
We first define a partial ordering on $I^{m+n}$, which depends on the sequence ${\bf b}$.
There is a natural bijection $I^{m+n} \leftrightarrow X(m|n)$ (recall $X(m|n)$ from  \eqref{eq:Xmn}), defined as
\begin{align}
&f \mapsto \lambda^{\bf b}_f, \text{ where }  \lambda^{\bf b}_f 
= \sum_{i=1}^{m+n}(-1)^{b_i}f(i)\ep^{b_i}_i -\rho_{\bf b},  \quad \text{for } f \in I^{m+n}, 
\label{osp:eq:ftolambda}
\\
& \lambda \mapsto f^{\bf b}_{\lambda},
 \text{ where } f(i)=(\lambda + \rho_{\bf b} \vert \ep^{b_i}_i), \quad \quad \quad \text{for } \lambda \in X(m|n).
\label{osp:eq:lambdatof}
\end{align}
We transport the Bruhat ordering \eqref{eq:BrX} on $X(m|n)$ by the above bijection to $I^{m+n}$. 
\begin{definition}
  \label{osp:def:dominanceordering}
The Bruhat ordering or $\bf b$-Bruhat ordering $\preceq_{\bf b}$ on $I^{m+n}$ is defined as follows:
For $f$, $g \in I^{m+n}$,  $f \preceq_{\bf b} g$ 
if $\lambda^{\bf b}_f  \preceq_{\bf b} \lambda^{\bf b}_g$.
We also say $f \sim g$ if $\lambda^{\bf b}_f \sim \lambda^{\bf b}_g$. 
\end{definition}

The following lemma follows immediately from the definition. 
\begin{lem}
  \label{osp:lem:finitelength}
Given $f, g \in I^{m+n}$  such that $g \preceq_{\bf b} f$, then the set $\{h \in I^{m+n} \mid g \preceq_{\bf b} h \preceq_{\bf b} f \}$ is finite.
\end{lem}
Recalling the weight $\texttt{wt}_{\bf b}(\cdot)$ on $I^{m+n}$ from Definition~\ref{def:wtb},
we set 
\begin{equation}  \label{eq:wtbX}
\texttt{wt}_{\bf b}(\lambda) : = \texttt{wt}_{\bf b}(f^{\bf b}_{\lambda}),
\qquad \text{ for } \lambda \in X(m|n).
\end{equation}
We have the following analogue of \cite[Lemma 4.18]{Br03}.
\begin{lem}
  \label{lem: linkage}
For any $f, g \in I^{m+n}$, $f \sim g$  if and only if $\texttt{wt}_{\bf b}(f) =\texttt{wt}_{\bf b}(g)$.
\end{lem}

\begin{proof}
This proof is analogous to \cite[Theorem 2.30]{CW12}.
Assume $f \sim g$ at first. Recall \S \ref{subsec:osp}, this means 
\[\lambda^{\bf b}_{g} + \rho_{\bf b} 
= w(\lambda^{\bf b}_{f} + \rho_{\bf b} - \sum^l_{i=1}c_i\alpha_i), 
\quad (\lambda^{\bf b}_{f} +\rho_{\bf b} \vert \alpha_j) = 0, j=1,\dots ,l.
\]
where $\alpha_i$'s are mutually orthogonal isotropic odd roots. Recall the Weyl group of $\osp(2m+1|2n)$ 
is isomorphic to $(\Z_{2} \rtimes \mf{S}_{m}) \times (\Z_{2} \rtimes \mf{S}_{n})$. Thanks to Definition \ref{prop:coproduct} 
and the actions the $\bk_{\alpha_i}$'s on $\VV$ and $\WW$, we have 
$\texttt{wt}_{\bf b}(w(\lambda^{\bf b}_{f} + \rho_{\bf b} - \sum^l_{i=1}c_i\alpha_i)) 
 = \texttt{wt}_{\bf b}(\lambda^{\bf b}_{f} + \rho_{\bf b} - \sum^l_{i=1}c_i\alpha_i)$. 
Isotropic odd roots of $\Phi$ are of the form $\pm\ep^{b_x}_x\pm\ep^{b_y}_{y}$, where $b_x$ and  $b_y$ are distinct. 
We shall discuss one case here, as the others will be similar. 

Let $\alpha = \ep^{b_s}_{s} -\ep^{b_t}_{t}= \ep^0_{s} -\ep^1_{t}$ be an isotropic odd root such that 
$(\lambda^{\bf b}_{f} +\rho_{\bf b} \vert \alpha)  =(\sum_{i=1}^{m+n}(-1)^{b_i}f(i)\ep^{b_i}_i \vert \alpha) = 0$. 
Therefore, $f(s)=f(t)$. Hence we have 
\begin{align*}
&\texttt{wt}_{\bf b}(\lambda^{\bf b}_{f} + \rho_{\bf b} + c\alpha) 
\\
&= \texttt{wt}_{\bf b}(\dots,f(s-1),  f(s)+c, f(s+1), \dots,f(t-1),
 f(t)+c, f(t+1), \dots) 
 \\
 &= \texttt{wt}_{\bf b}(f),
\end{align*} 
where the last equality comes from the actions of $\bk_{\alpha_i}$'s on $\VV$ and $\WW$.  
Therefore $\texttt{wt}_{\bf b}(f) =\texttt{wt}_{\bf b}(g)$.
\par Now suppose $\texttt{wt}_{\bf b}(f) =\texttt{wt}_{\bf b}(g)$. We have
\begin{equation}\label{eq:weight}
\sum^{m+n}_{i=1}(-1)^{b_i}\ul{\varepsilon_{f(i)}} = \sum^{m+n}_{i=1}(-1)^{b_i}\ul{\varepsilon_{g(i)}}.
\end{equation}
 For distinct $b_{i_a}$, $b_{j_a}$ $(i_a \neq j_a)$, if 
 $f(i_a)= \pm f(j_a)$, then 
 $$(-1)^{b_{i_a}}\ul{\varepsilon_{f(i_a)}} + (-1)^{b_{j_a}}\ul{\varepsilon_{f(j_a)}} = 0$$ 
 (recall that $\ul{\varepsilon_{f(s)}} = \ul{\varepsilon_{-f(s)}}$). Similar results hold for $g$. 
 After canceling all such pairs (all $i_a$ and all $j_a$ are distinct) on both sides of \eqref{eq:weight}, 
 the survived terms match bijectively up to signs. More precisely, for any survived $f(x)$, there exists a suvived $g(y)$, 
 such that $g(y) = \pm f(x), b_x = b_y$.  Hence the same number of pairs cancelled on both sides, say $l$ pairs. 
 Therefore we have $\lambda^{\bf b}_{f} +\rho_{\bf b} - \sum^l_{a =1}c_a(\ep^{0}_{i_a} -s_a \ep^{1}_{j_a}) 
  = w(\lambda^{\bf b}_{g} +\rho_{\bf b})$ for some $w \in  (\Z_{2} \rtimes \mf{S}_{m}) \times (\Z_{2} \rtimes \mf{S}_{n})$, 
  $s_a \in \{\pm\}$. The $s_a$'s are chosen to satisfy 
\[
(\lambda^{\bf b}_{f} +\rho_{\bf b} \vert \ep^{0}_{i_a} -s_a \ep^{1}_{j_a}) = 0.
\]
Therefore $\lambda^{\bf b}_f \sim \lambda^{\bf b}_g$ by the definition in \S \ref{subsec:osp}. Hence $f \sim g$. 

This completes the proof of the lemma.
\end{proof}
%
Now let us define partial orderings on the sets $I^{m+n} \times I^{\infty}_{\pm}$, which again depend on  ${\bf b}$.
Recall \eqref{eq:wtlV} and \eqref{eq:wtlW} for the definitions of $\wtlV$ and $\wtlW$. We define a map 
\begin{equation}
  \label{osp:eq:bijectionwtlV}
\wtlV \longrightarrow I^{m+n} \times I^\infty_{+}, \quad \lambda \mapsto f^{{\bf b}0}_{\lambda},
\end{equation}
by sending each $\lambda = \sum^{m+n}_{i=1}\lambda_i\ep^{b_i}_{i} + \sum_{1 \leq j}{^+\lambda_{\ul j}}\ep^0_{\ul j} + d\epinftyV$ 
to the element $f^{{\bf b}0}_{\lambda} = f^{({\bf b},0^{\infty})}_{\lambda}$ given below
(which is consistent with the $\rho$-shift associated to a simple system of the type ($\star$) in \S\ref{subsec:osp} by Remark~\ref{rem:rhob}): 
\begin{align}
 \label{eq:fb0}
\begin{split}
&f^{{\bf b}0}_{\lambda}(i) = f^{\bf b}_{\lambda}(i), \quad \text{ if } i \in [m+n]:=\{1,2,\dots,m+n\},\\
&f^{{\bf b}0}_{\lambda}(\ul j) = {^+\lambda_{\ul j}} + d +n-m+\hf - j, \quad \text{ if } 1 \leq j.
\end{split}
\end{align}
This map is a bijection, where the inverse sends $f \in  I^{m+n} \times I^\infty_{+} $ to
$$
\lambda^{{\bf b}0}_f := \sum^{m+n}_{i=1}\lambda^{\bf b}_{f,i}\ep^{b_i}_{i} + \sum_{1 \leq j}{^+\lambda_{f,{\ul j}}}\ep^0_{\ul j} + d_f\epinftyV.
$$ 

Similarly we define a bijection
\begin{equation}
  \label{osp:eq:bijectionwtlW}
\wtlW \longrightarrow I^{m+n} \times I^\infty_{-}, \quad \lambda \mapsto f^{{\bf b}1}_{\lambda},
\end{equation}
by sending each 
$\lambda = \sum^{m+n}_{i=1}\lambda_i\ep^{b_i}_{i} + \sum_{1 \leq j}{^+\lambda_{\ul j}}\ep^1_{\ul j} + d\epinftyW$ 
to the element $f^{{\bf b}1}_{\lambda} = f^{({\bf b},1^{\infty})}_{\lambda}$ given below:
\begin{align}
 \label{eq:fb1}
\begin{split}
&f^{{\bf b}1}_{\lambda}(i) := f^{\bf b}_{\lambda}(i), \quad \text{ if } i \in [m+n],\\
&f^{{\bf b}1}_{\lambda}(\ul j) := -{^+\lambda_{\ul j}} + d +n-m-\hf +j, \quad \text{ if } 1 \leq j.
\end{split}
\end{align}
The inverse sends $f \in  I^{m+n} \times I^\infty_{-}$ to
$
\lambda^{{\bf b}1}_f 
:= \sum^{m+n}_{i=1}\lambda^{\bf b}_{f,i}\ep^{b_i}_{i} + \sum_{1 \leq j}{^+\lambda_{f,{\ul j}}}\ep^1_{\ul j} + d_f\epinftyW.
$

Note that for $s \in \{0,1\}$, the sum $\sum^{m+n}_{i=1}\lambda^{\bf b}_{f,i}\ep^{b_i}_{i} + \sum_{1 \leq j}{^+\lambda_{f,{\ul j}}}\ep^s_{\ul j}$ 
lies in the root system of a finite-rank Lie superalgebra. Hence the following definitions make sense.
\begin{definition}
  \label{def:orderinfV}
 For $f$, $g \in I^{m+n} \times I^\infty_{+}$, 
we say $f \sim g$ if 
$$
d_f = d_g \;\; \text{ and } \;\;
(\sum^{m+n}_{i =1}\lambda^{\bf b}_{f,i} \ep^{b_i}_{i} + \sum_{1 \leq j}{^+}\lambda_{f,\ul j}\ep^0_{\ul j}) 
  \sim (\sum^{m+n}_{i =1}\lambda^{\bf b}_{g,i} \ep^{b_i}_{i} + \sum_{1 \leq j}{^+}\lambda_{g,\ul j}\ep^0_{\ul j}) 
 $$ 
 in the sense of \S\ref{subsec:osp}.
 We say $f \preceq_{{\bf b}, 0} g$ if 
 $$
 f \sim g \;\; \text{ and } \;\; 
 \lambda^{{\bf b},0}_g-\lambda^{{\bf b},0}_f \in \N\Pi_{{\bf b},0}.
 $$ 
\end{definition}

We similarly define an equivalence $\sim$ and a partial ordering $\preceq_{{\bf b},1}$ on $I^{m+n} \times I^\infty_{-}$.
\begin{definition}
  \label{def:orderinfW}
 For $f$, $g \in I^{m+n} \times I^\infty_{-}$, 
 we say $f \sim g$ if 
$$
\left( d_f = d_g\; \;
\text{ and  } \;\;
\big(\sum^{m+n}_{i =1}\lambda^{\bf b}_{f,i} \ep^{b_i}_{i} + \sum_{1 \leq j}{^+}\lambda_{f,\ul j}\ep^1_{\ul j} \big)  
\sim \big(\sum^{m+n}_{i =1}\lambda^{\bf b}_{g,i} \ep^{b_i}_{i} + \sum_{1 \leq j}{^+}\lambda_{g,\ul j}\ep^1_{\ul j} \big) 
\right)
$$ 
in the sense of \S \ref{subsec:osp}. 
We say $f \preceq_{{\bf b}, 1} g$ if 
$$
f \sim g\;\; \text{ and } 
\lambda^{{\bf b}, 1}_g -\lambda^{{\bf b}, 1}_f \in \N\Pi_{{\bf b},1}.
$$
\end{definition}


The following lemma follows from 
Definition~\ref{def:orderinfV}, Definition~\ref{def:orderinfW}, and Lemma~\ref{osp:lem:finitelength}.
\begin{lem}
 \label{osp:lem:finitelengthW}
 \begin{enumerate}
 \item
Given $f, g \in I^{m+n} \times I^{\infty}_{+}$  such that $g \preceq_{{\bf b},0} f$, the set 
$\{h \in I^{m+n} \times I^{\infty}_{+}\mid g \preceq_{{\bf b},0} h \preceq_{{\bf b},0} f \}$ is finite.

\item
Given $f, g \in I^{m+n} \times I^{\infty}_{-}$ such that $g \preceq_{{\bf b},1} f$, the set 
$\{h \in I^{m+n} \times I^{\infty}_{-}\mid g \preceq_{{\bf b},1} h \preceq_{{\bf b},1} f \}$ is finite.
\end{enumerate}
\end{lem}

The following lemma is an infinite-rank analogue of Lemma~\ref{lem: linkage}.
\begin{lem}
  \label{lem:linkage at infinity}
For any $f$, $g \in I^{m+n} \times I^\infty_{+}$ (respectively, $I^{m+n} \times I^\infty_{-}$),  $f \sim g$ 
if and only if $\texttt{wt}_{{\bf b},0}(f)=\texttt{wt}_{{\bf b},0}(g)$ (respectively, $\texttt{wt}_{{\bf b},1}(f)=\texttt{wt}_{{\bf b},1}(g)$).
\end{lem}

\begin{proof}
Follows from Definition~\ref{def:orderinfV}, Definition~\ref{def:orderinfW}, and Lemma~\ref{lem: linkage}.
\end{proof}


\chapter{$\imath$-Canonical bases and Kazhdan-Lusztig-type polynomials}
  \label{osp:sec:cbanddcb}
  
In this chapter, suitably completed Fock spaces are constructed and shown to admit
$\imath$-canonical as well as dual $\imath$-canonical bases.
We introduce truncation maps to study the relations among bases for different Fock spaces,
which then allow us to formulate $\imath$-canonical bases in certain semi-infinite Fock spaces. 

\section{The $B$-completion and $\Upsilon$}
   \label{osp:subsec:cbanddcb}

Let ${\bf b}$ be a ${0^m1^n}$-sequence. 
For $r \in \N$, we let
\begin{align}\label{def:pi:lek}
\pi_r:\mathbb T^{\bf b}\longrightarrow \mathbb T^{\bf b}_{r}
\end{align}
be the natural projection map with respect to the standard basis $\{M_f^{\bf
b} \mid f \in I^{m+n}\}$ of $\mathbb T^{\bf b}$ (see \eqref{osp:eq:TrT}). 
We then let
$\wt{\mathbb T}^{\bf b}$ be the completion of $\mathbb T^{\bf b}$
with respect to the descending sequence of subspaces $\{\ker \pi_r \mid r \ge 1\}$. Formally, every element in
$\wt{\mathbb T}^{\bf b}$ is a possibly infinite linear combination
of $M_f$, with $f\in I^{m+n}$. We let $\widehat{\mathbb T}^{\bf b}$
denote the subspace of $\wt{\mathbb T}^{\bf b}$ spanned by elements
of the form 
\begin{align}
  \label{vec:in:comp}
M_f+\sum_{g\prec_{\bf b}f}c_{gf}^{\bf b}(q) M_g, \quad \text{ for } c_{gf}^{\bf b}(q) \in\Q(q).
\end{align}

\begin{definition}
  \label{osp:def:completionT}
The $\Q(q)$-vector spaces $\wt{\mathbb T}^{\bf b}$ and
$\widehat{\mathbb T}^{\bf b}$ are called the $A$-{\em completion}
and $B$-{\em completion} of $\mathbb T^{\bf b}$, respectively.
\end{definition}

\begin{rem}
  \label{osp:rem:Bcompletions}
The $B$-completion we defined here is different from the one defined in \cite{CLW12}, 
since they are based on {\em different partial orderings}. However,  observing that the partial ordering used in \cite{CLW12} 
is coarser than the partial ordering here, our $B$-completion here contains the  $B$-completion in 
\cite[Definition 3.2]{CLW12} as a subspace. 
This fact very often allows us to cite directly the results therein.
\end{rem}

\begin{lem}
  \label{osp:lem:compatibleUpsilon}
Let $f \in I_r^{m+n}$. Then we have $M_f \in \mathbb{T}^{\bf b}_{r}$, and 
\[
\pi_r ( \Upsilon^{(l)}M_f) = \Upsilon^{(r)} M_f, \quad \text{ for all }\; l \ge r.
\]
\end{lem}

\begin{proof}
Note that $\N\Pi_{2r+1} \subset \N\Pi_{2l+1}$, 
for $l \geq r$. It is clear from the construction of $\Upsilon^{(r)}$ in Theorem \ref{thm:Upsilon} that we have 
\[
\Upsilon^{(l)} = \Upsilon^{(r)} + \sum_{\mu \in \N\Pi_{2l+1}\backslash \N\Pi_{2r+1}} \Upsilon^{(l)}_\mu. 
\]

By $\U$-weight consideration, it is easy to see $\pi_r(\Upsilon^{(l)}_\mu M_f) = 0$ if $\mu \not \in \N\Pi_{2r+1}$. Therefore 
\[
\pi_r ( \Upsilon^{(l)}M_f) = \pi_r ( \Upsilon^{(r)}M_f) =\Upsilon^{(r)}M_f.
\]
The lemma follows.
\end{proof}

It follows from Lemma \ref{osp:lem:compatibleUpsilon} that $\displaystyle\lim_{r \to \infty} \Upsilon^{(r)} M_f$, 
for any $f \in I^{m+n}$, is a well-defined element in $\wt{\mathbb T}^{\bf b}$. Therefore we have 
\[
\Upsilon M_f = \lim_{r \to \infty} \Upsilon^{(r)} M_f,
\]
where $\Upsilon$ is the operator defined in \eqref{pt2:eq:Upsilon}.

\begin{lem}
  \label{osp:lem:Upsilonorder}
For $f \in I^{m+n}$, we have
\begin{equation}
  \label{osp:lem:eq:Upsilonorder}
\Upsilon M_f =
M_f + \sum_{g \prec_{\bf b} f}r'_{gf}(q)M_g, \quad \text{ for }r'_{gf}(q) \in \mA.
\end{equation}
In particular, we have $\Upsilon: \mathbb{T}^{\bf b} \rightarrow \widehat{\mathbb{T}}^{\bf b}$.
\end{lem}

\begin{proof}
For any $u \in \U^-$ with $\Ui$-weight $0$, $f \in I^{m+n}$, let $u M_f = \sum_{g} c_{gf}M_{g}$.
Fix any $g$ with $c_{gf} \neq 0$.  
Since $u$ has $\Ui$-weight $0$, we know by Lemma \ref{lem: linkage}  that $g \sim f$ and so
$\lambda^{\bf b}_g \sim \lambda^{\bf b}_f$. By a direct computation (by writing $u$ in terms of Chevalley generator $F$'s), 
it is easy to see that $u \in \U^-$ implies that $\lambda^{\bf b}_f - \lambda^{\bf b}_g \in \N\Pi_{\bf b}$. Hence we have $g \preceq_{\bf b} f$.

Recall that $\Upsilon_\mu \in \U^-$ for all $\mu$ and $\Upsilon_\mu \neq 0$ only if $\mu = \mu^{\inv}$, i.e., $\mu$ is of $\Ui$-weight $0$.
Hence we have the identity \eqref{osp:lem:eq:Upsilonorder}, where $r'_{gf}(q) \in \mA$ follows from
Theorem \ref{thm:UpsiloninZ}. The lemma follows.
\end{proof}

\begin{lem}
  \label{osp:lem:Upsilonhattohat}
The map 
$\Upsilon : \mathbb{T}^{\bf b} \rightarrow \widehat{\mathbb T}^{\bf b}$
 extends uniquely to a $\Qq$-linear map 
 $\Upsilon: \widehat{\mathbb T}^{\bf b} \rightarrow \widehat{\mathbb T}^{\bf b}$.
\end{lem}

\begin{proof}
We adapt the proof of \cite[Lemma 3.7]{CLW12} here.
To show that the map $\Upsilon$ extends to $\widehat{\mathbb T}^{\bf b}$ we
need to show that if $y=M_f+\sum_{g\prec_{\bf b}f}r_{g}(q) M_g\in
\widehat{\mathbb T}^{\bf b}$ for $r_{g}(q) \in\Q(q)$ then
$\Upsilon{y}\in\widehat{\mathbb T}^{\bf b}$. By Lemma~
\ref{osp:lem:Upsilonorder} and the definition of $\widehat{\mathbb T}^{\bf
b}$, it remains to show that $\Upsilon{y}\in\wt{\mathbb T}^{\bf b}$. To
that end, we note that if the coefficient of $M_h$ in $\Upsilon{y}$ is
nonzero, then there exists $g\preceq_{\bf b}f$ such that $r'_{hg}(q)\not=0$. Thus we have $h\preceq_{\bf b
}g\preceq_{\bf b}f$.  However, by Lemma \ref{osp:lem:finitelength} there
are only finitely many such $g$'s. Thus, only finitely many $g$'s
can contribute to the coefficient of $M_h$ in $\Upsilon{y}$, and hence
$\Upsilon{y}\in\wt{\mathbb T}^{\bf b}$.
\end{proof}

\section{$\imath$-Canonical bases}
 \label{osp:subsec:cbdcb}

Anti-linear maps $\Abar: {\mathbb{T}}^{\bf b}_r \rightarrow {\mathbb{T}}^{\bf b}_r$ and 
$\Abar: {\mathbb{T}}^{\bf b} \rightarrow \widehat{\mathbb{T}}^{\bf b}$ were defined in \cite[\S 3.3]{CLW12}  
(recall Remark \ref{osp:rem:Bcompletions} that our $B$-completion contains the one therein as a subspace,
so $\wTb$ here can and will be understood in the sense of this paper). 
We define the map $\Bbar : \mathbb{T}^{\bf b} \rightarrow \widehat{\mathbb{T}}^{\bf b}$ by 
\begin{equation}
 \label{osp:Upsilon}
\Bbar(M_f) := \Upsilon \Abar(M_f).
\end{equation}
Recall from \S \ref{subsec:bars} that $\mathbb{T}_r^{\bf b}$ is an $\imath$-involutive $\U^{\imath}_r$-module with anti-linear involution $\Bbar^{(r)}$. 
\begin{lem}
We have 
$
\pi_r ( \Bbar(M_f)) = \Bbar^{(r)} (M_f)$, for $f \in I_r^{m+n}$.
\end{lem}

\begin{proof}
Recall that $\Bbar^{(r)} = \Upsilon^{(r)} \Abar^{(r)}$. By a variant of Lemma \ref{osp:lem:compatibleUpsilon}, we have 
\[
\pi_r(\Bbar(M_f)) = \pi_r( \Upsilon^{(r)} \Abar(M_f))
\]
by a $\U$-weight consideration. Therefore we have 
\[
\pi_r(\Bbar(M_f)) = \Upsilon^{(r)} \pi_r (\Abar(M_f)) = \Upsilon^{(r)} \Abar^{(r)}(M_f),
\]
where the last identity follows from \cite[Lemma 3.4]{CLW12}. The lemma follows.
\end{proof}

It follows immediately that we have 
\begin{equation}\label{osp:eq:Bbarlimit}
\Bbar(M_f) = \lim_{r \to \infty}\Bbar^{(r)} (M_f), \quad \text{ for } f \in I^{m+n}.
\end{equation}

\begin{lem}
  \label{lem:Mfbar}
Let $f \in I^{m+n}$. 
Then we have 
\[\Bbar(M_f) = M_f + \sum_{g \prec_{\bf b} f} r_{gf}(q)M_g, \quad \text{ for } r_{gf}(q) \in \mA.
\]
Hence the anti-linear map $\Bbar : {\mathbb{T}}^{\bf b} \rightarrow \widehat{\mathbb{T}}^{\bf b}$ 
extends to a map $\Bbar : \widehat{\mathbb{T}}^{\bf b} \rightarrow \widehat{\mathbb{T}}^{\bf b}$. 
Moreover $\Bbar$ is independent of the bracketing orders for the tensor product $\mathbb{T}^{\bf b}$.
\end{lem}

\begin{proof}
Following \cite[Proposition 3.6]{CLW12} and Remark \ref{osp:rem:Bcompletions}, we have
\[\Abar(M_f) = M_f + \sum_{g \prec_{\bf b} f} r''_{gf}(q)M_g, \quad \text{ for }r''_{gf}(q) \in \mA.
\]
Hence the first part of the lemma follows from Lemma \ref{osp:lem:Upsilonorder}.

We can show that the map $\Bbar : {\mathbb{T}}^{\bf b} \rightarrow \widehat{\mathbb{T}}^{\bf b}$ extends 
to a map $\Bbar : \widehat{\mathbb{T}}^{\bf b} \rightarrow \widehat{\mathbb{T}}^{\bf b}$ by applying the 
same argument used in the proof of Lemma \ref{osp:lem:Upsilonhattohat}.
Since $\Abar$ is independent from the bracketing orders for the tensor product 
$\mathbb{T}^{\bf b}$ by \cite[Proposition 3.5]{CLW12},  so is $\Bbar$.
\end{proof}

\begin{lem}
  \label{osp:lem:Bbarinvolution}
The map $\Bbar : \widehat{\mathbb{T}}^{\bf b} \longrightarrow \widehat{\mathbb{T}}^{\bf b}$ is an anti-linear involution.
\end{lem}

\begin{proof}
In order to prove the lemma, we need to prove that for fixed $f, h \in I^{m+n}$ with $h \prec_{\bf b} f$, we have 
\begin{equation*}
\sum_{h\preceq_{\bf b} g \preceq_{\bf b}f} r_{hg}(q)\ov{r_{gf}(q)} = \delta_{hf}.
\end{equation*}
By Lemma \ref{osp:lem:finitelength}, there is only finitely many $g$ such that $h\preceq_{\bf b} g \preceq_{\bf b}f$.
Recall \S \ref{subsec:bars}. We know $\Bbar^{(r)}$ is an involution. By \eqref{osp:eq:Bbarlimit}, this is equivalent to the same identities in the finite-dimensional space $\mathbb{T}^{\bf b}_r$ with $r \gg 0$. Then the lemmas follows from Proposition~\ref{prop:compatibleBbar}.
\end{proof}

Thanks to Lemmas \ref{lem:Mfbar} and \ref{osp:lem:Bbarinvolution}, 
we are in a position to apply \cite[Lemma 24.2.1]{Lu94} to 
the anti-linear involution $\Bbar : \widehat{\mathbb{T}}^{\bf b} \rightarrow \widehat{\mathbb{T}}^{\bf b}$ to establish the following.

\begin{thm} 
 \label{thm:iCBb}
The $\Qq$-vector space $\widehat{\mathbb{T}}^{\bf b}$ has unique $\Bbar$-invariant topological bases
\[
\{T^{\bf b}_f \mid f \in I^{m+n}\} \text{ and } \{L^{\bf b}_f \mid f \in I^{m+n}\}
\]
such that 
\[
T^{\bf b}_f = M_f + \sum_{g \preceq_{\bf b} f }t^{\bf b}_{gf}(q)M^{\bf b}_g, 
\quad 
L^{\bf b}_f  = M_f + \sum_{g \preceq_{\bf b} f }\ell^{\bf b}_{gf}(q)M^{\bf b}_g,
\]
with $t^{\bf b}_{gf}(q) \in q\Z[q]$, and $\ell^{\bf b}_{gf}(q) \in q^{-1}\Z[q^{-1}]$, for $g \preceq_{\bf b}f $. 
(We shall write $t^{\bf b}_{ff}(q) = \ell^{\bf b}_{ff}(q) = 1$, $t^{\bf b}_{gf}(q)=\ell^{\bf b}_{gf}(q)=0$ for $ g \not\preceq_{\bf b} f$.)
\end{thm}

\begin{definition}
$\{T^{\bf b}_f \mid f \in I^{m+n}\} \text{ and } \{L^{\bf b}_f \mid f \in I^{m+n}\}$ are call the 
{\em $\imath$-canonical basis} and {\em dual $\imath$-canonical basis} of $\widehat{\mathbb{T}}^{\bf b}$, respectively. 
The polynomials $t^{\bf b}_{gf}(q)$ and $\ell^{\bf b}_{gf}(q)$ are called {\em $\imath$-Kazhdan-Lusztig (or $\imath$-KL) polynomials}.
\end{definition}

\begin{thm} 
\begin{enumerate}
\item
{(Positivity)}
We have $t^{\bf b}_{gf}(q) \in \N [q]$.

\item  
The sum $T^{\bf b}_f = M_f + \sum_{g \preceq_{\bf b} f }t^{\bf b}_{gf}(q)M^{\bf b}_g$ is finite, for all $f \in I^{m+n}$. 
\end{enumerate}
\end{thm}

\begin{proof}
Note that  the finite sum claim in (2) at the $q=1$ specialization holds by Theorem~\ref{thm:iKL}
(the proof of Theorem~\ref{thm:iKL} does not use the claim (2); we decided to list such an algebraic statement (2) here 
rather than as a corollary to Theorem~\ref{thm:iKL}).
Hence, the validity of the positivity (1) implies the validity of (2).

It remains to prove (1).  Actually the same strategy as for type $A$ 
(see  \cite{BLW} and \cite[proof of Theorem~3.12, Remark~3.14]{CLW12}) works here, and so we shall be brief.
Fix $f, g \in I^{m+n}$. 
Choose a half-integer $k\gg 0$ (relative to $f, g$), and consider the subspaces $\VV_{[k]}$ of $\VV$ spanned by $v_i$ for $i\in [-k,k] \cap I \subset I$
and an analogous subspace $\WW_{[k]}$ of $\WW$. We then define $\mathbb{T}^{\bf b}_{[k]}$ to be the subspace
of $\mathbb{T}^{\bf b}$ spanned by the elements $T^{\bf b}_f$ for $f \in ([-k,k] \cap I)^{m+n}$. 
Via the natural identification $\WW_{[k]} \cong \wedge^{2k} \VV_{[k]}$, we can identify $\mathbb{T}^{\bf b}_{[k]}$
with a tensor product of  copies of $\VV_{[k]}$ and $\wedge^{2k} \VV_{[k]}$ (such an identification in type $A$ setting
appeared first in \cite{CLW12}).
The latter provides a reformulation of the parabolic KL conjecture of type $B$ thanks to Remark~\ref{rem:samebar}
(which was in turn based on Theorem~\ref{thm:iCBtensor} and Theorem~ \ref{thm:samebar}); hence
$t^{\bf b}_{gf}(q)$ can be identified with a (non-super) type $B$ KL polynomial.
The positivity of type $B$ KL polynomials is well known (\cite{KL80, BGS}), whence the positivity  (1). 
\end{proof}

\begin{rem}
We expect a positivity property of the coefficients in the expansion of the $\imath$-canonical basis elements here
with respect to the (type $A$) canonical basis on $\wTb$ in \cite{CLW12} (compare with the remark after Theorem~\ref{thm:iCBtensor}). 
\end{rem}

\section{Bar involution and  $q$-wedges of $\WW$}
\label{subsec:barWW}

Let $k \in\N \cup \{\infty\}$. For $f = (f_{[m+n]}, f_{[\ul{k}]}) \in I^{m+n} \times I^{k}_+$, set
\[
M^{{\bf b}, 0}_f := M^{{\bf b}}_{f_{[m+n]}}  \otimes \mc{V}_{f_{[\ul{k}]}} \in \mathbb{T}^{\bf b} \otimes \wedge^{k} \VV.
\]
Then $\{M^{{\bf b}, 0}_f  \mid f  \in I^{m+n} \times I^{k}_+\}$ forms a basis, called the {\em standard monomial basis}, 
of the $\Qq$-vector space $\mathbb{T}^{\bf b} \otimes \wedge^{k} \VV$.
Similarly,  $\mathbb{T}^{\bf b} \otimes \wedge^{k} \WW$ admits a {\em standard monomial basis} given by 
\[
	M^{{\bf b}, 1}_g := M^{{\bf b}}_{g_{[m+n]}}  \otimes \mc{W}_{g_{[\ul{k}]}} \in \mathbb{T}^{\bf b} \otimes \wedge^{k} \WW,
\]
where $g = (g_{[m+n]}, g_{[\ul{k}]}) \in I^{m+n} \times I^{k}_-$. Following \cite[\S 4]{CLW12}, 
here we shall focus on the case $\mathbb{T}^{\bf b} \otimes \wedge^{k} \WW$, 
while the case $\mathbb{T}^{\bf b} \otimes \wedge^{k} \VV$ is similar. 

Let us consider $k \in \N$ first. As in \cite[\S 4]{CLW12}, $\mathbb{T}^{\bf b} \otimes \wedge^{k}\WW$ 
can be realized as a subspace of $\mathbb{T}^{\bf b} \otimes \WW^{\otimes k} = \mathbb{T}^{({\bf b}, 1^k)}$.
Therefore we can define a $B$-completion of $\mathbb{T}^{\bf b} \otimes \wedge^{k}\WW$, 
denoted by $\mathbb{T}^{\bf b} \widehat{\otimes} \wedge^{k}\WW$, as the closure of the subspace 
$\mathbb{T}^{\bf b} \otimes \wedge^{k} \WW \subset \mathbb{T}^{\bf b} \widehat{\otimes} \WW^{\otimes k}
 = \widehat{ \mathbb{T}}^{({\bf b}, 1^k)}$ with respect to the linear topology $\{\ker \pi_r \mid r \ge 1\}$ 
 defined in \S \ref{osp:subsec:cbanddcb}. 
By construction $\mathbb{T}^{\bf b} \widehat{\otimes} \wedge^{k}\WW$ is invariant under the involution $\Bbar$, i.e., we have 
\[
	\Bbar({M}^{{\bf b},1}_f) = M^{{\bf b},1}_f  + \sum_{g \prec_{({\bf b}, 1^k)}f}r_{gf}(q) M^{{\bf b},1} _g,
\]
where $r_{gf}(q) \in \mA$, and the sum running over $g \in I^{m+n} \times I^{k}_{-}$ is possibly infinite.

\begin{rem}
If $k=0$, $M^{{\bf b}, 0}_f$ and $M^{{\bf b}, 1}_g$ are understood as $M^{\bf b}_f$ and $M^{\bf b}_g$, respectively;
also, $\mathbb{T}^{\bf b} \widehat{\otimes} \wedge^0 \WW$ and $\mathbb{T}^{\bf b} \widehat{\otimes} \wedge^0 \VV$ 
are understood as $\widehat{\mathbb{T}}^{\bf b}$. 
\end{rem}

Recall the linear maps $\wedge^{k,l}_d$ defined in \eqref{eq:wedge}.
For $l \ge k$ and each $d \in \Z$, define the $\Qq$-linear map
\begin{align*}
\id \otimes \wedge^{k,l}_d : \mathbb{T}^{\bf b} \otimes \wedge^{k} \WW &\longrightarrow \mathbb{T}^{\bf b} \otimes \wedge^{l} \WW.
\end{align*}
It is easy to check that the map $\id \otimes \wedge^{k,l}_d$ extends to the $B$-completions; that is, we have 
\[
\id \otimes \wedge^{k,l}_d : \mathbb{T}^{\bf b} \widehat{\otimes} \wedge^{k} \WW
 \longrightarrow \mathbb{T}^{\bf b} \widehat{\otimes} \wedge^{l} \WW.
\]
Let $ \mathbb{T}^{\bf b} \widehat{\otimes} \wedge_d^{\infty} \WW := \varinjlim \mathbb{T}^{\bf b} \widehat{\otimes} \wedge^{k} \WW$ 
be the direct limit of the $\Qq$-vector spaces with respect to the linear maps $\id \otimes \wedge^{k,l}_d$. 
It is easy to see that $\mathbb{T}^{\bf b} {\otimes} \wedge_d^{\infty} \WW \subset \mathbb{T}^{\bf b} \widehat{\otimes} \wedge_d^{\infty} \WW$. 
Define the $B$-completion of  $\mathbb{T}^{\bf b} {\otimes} \wedge^{\infty} \WW$ as follows:
\begin{equation}
  \label{osp:eq:Bcompletioninfinity}
\mathbb{T}^{\bf b} \widehat{\otimes} \wedge^{\infty} \WW 
:= \bigoplus_{d \in \Z} \mathbb{T}^{\bf b} \widehat{\otimes} \wedge_d^{\infty} \WW.
\end{equation}
By the same argument as in \S\ref{osp:subsec:$q$-wedges}, we see that $\mathbb{T}^{\bf b} \widehat{\otimes} \wedge_d^{\infty} \WW$ 
and $\mathbb{T}^{\bf b} \widehat{\otimes} \wedge^{\infty} \WW$ are (topological) $\U$-modules, hence (topological) $\Ui$-modules.

Following the definitions of the partial orderings in Definition~\ref{osp:def:dominanceordering} and 
Definition~\ref{def:orderinfW}, we see that $\mathbb{T}^{\bf b} \widehat{\otimes} \wedge^{\infty} \WW$ is spanned by elements of the form
\[
M^{{\bf b},1}_f + \sum_{g \prec_{{\bf b},1} f} c_{gf}(q)M^{{\bf b},1}_g, \quad \text{ for } g,f \in I^{m+n} \times I^{\infty}_-.
\]
Following \cite[\S 4.1]{CLW12}, we can extend the anti-linear involution 
$\Abar : \mathbb{T}^{\bf b} \widehat{\otimes} \wedge^{k} \WW \rightarrow \mathbb{T}^{\bf b} \widehat{\otimes} \wedge^{k} \WW$ 
to an anti-linear involution $\Abar: \mathbb{T}^{\bf b} \widehat{\otimes} \wedge^{\infty} \WW
\rightarrow \mathbb{T}^{\bf b} \widehat{\otimes} \wedge^{\infty} \WW$ such that 
\[
\Abar (M^{{\bf b},1}_f) =  M^{{\bf b},1}_f + \sum_{g \prec_{{\bf b},1} f} r''_{gf}(q)M^{{\bf b},1}_f, \quad \text{ for } r''_{gf}(q) \in \mA.
\]
Here we have used the fact that our $B$-completion contains the $B$-completion
in {\em loc. cit.}  as a subspace (see Remark~\ref{osp:rem:Bcompletions}).

Following the definition of the $B$-completion $\mathbb{T}^{\bf b} \widehat{\otimes} \wedge^{\infty} \WW$, 
we have $\Upsilon$ as a well-defined operator on $\mathbb{T}^{\bf b} \widehat{\otimes} \wedge^{\infty} \WW$ such that 
\[
\Upsilon (M^{{\bf b},1}_f) =  M^{{\bf b},1}_f + \sum_{g \prec_{{\bf b},1} f} r'_{gf}(q)M^{{\bf b},1}_f, \quad \text{ for } r'_{gf}(q) \in \mA.
\]
Therefore we can define the anti-linear map 
\[
\Bbar :=\Upsilon\Abar : \mathbb{T}^{\bf b} \widehat{\otimes} \wedge^{\infty} \WW
\longrightarrow \mathbb{T}^{\bf b} \widehat{\otimes} \wedge^{\infty} \WW,
\]
such that 
\[
\Bbar (M^{{\bf b},1}_f) =  M^{{\bf b},1}_f + \sum_{g \prec_{{\bf b},1} f} r_{gf}(q)M^{{\bf b},1}_f, \quad \text{ for } r_{gf}(q) \in \mA.
\]

\begin{lem}
Let $k \in \N \cup \{\infty\}$. The map 
$
\Bbar : \mathbb{T}^{\bf b} \widehat{\otimes} \wedge^{k} \WW \longrightarrow \mathbb{T}^{\bf b} \widehat{\otimes} \wedge^{k} \WW
$
is an involution.
\end{lem}

\begin{proof} 
For $k \in \N$, the lemma was already established. For $k = \infty$,
the lemma can be proved in the same way as 
Lemma~\ref{osp:lem:Bbarinvolution} with the help of Lemma~\ref{osp:lem:finitelengthW}.
\end{proof}

\section{Truncations}
 \label{subsec:trun}

In this section we shall again only focus on $\mathbb{T}^{\bf b} \otimes \wedge^{k} \WW$ for $k \in \N \cup \{\infty\}$. 
We shall use $f^{\ul k} \in I^{m+n} \times I^{k}_{\pm}$ as a short-hand notation for the restriction of 
$f_{[m+n] \cup [\ul k]}$ of a function $f \in I^{m+n} \times I^{\infty}_{\pm}$.

Now let us define the {\em{truncation map}} ${\texttt{Tr}} : \mathbb{T}^{\bf b} \otimes \wedge^{\infty}\mathbb{W} 
\rightarrow \mathbb{T}^{\bf b}\otimes \wedge^k\mathbb{W}$, for  $k \in \N$, as follows:
\[
{\texttt{Tr}}(m \otimes \mathcal{W}_h) =
\begin{cases}
 m \otimes \mathcal{W}_{h_{[\underline{k}]}},& \text{ if } h(\ul{i}) -h(\ul{i+1}) = -1, \text{ for } i \geq k+1,
 \\
0, &\text{ otherwise}. 
\end{cases}
\]

\begin{lem}
  \label{pt2:lem:burhatinfinity}
Let $k \in \N$. The truncation map ${\texttt{Tr}} : \mathbb{T}^{\bf b} \otimes \wedge^{\infty}\mathbb{W} 
\rightarrow \mathbb{T}^{\bf b}\otimes \wedge^k\mathbb{W}$ is compatible with the partial orderings, 
and hence extends naturally to a $\Qq$-linear map $\texttt{Tr} : \mathbb{T}^{\bf b}\widehat{\otimes} 
\wedge^\infty \mathbb{W} \rightarrow \mathbb{T}^{\bf b}\widehat{\otimes} \wedge^k\mathbb{W}$.
\end{lem}

\begin{proof}
Let $f$, $g \in I^{m+n} \times I^{\infty}_{-}$ with $g \preceq_{{\bf b},1} f$. 
According to Definition \ref{def:orderinfW}, this means $f(\ul i) = g(\ul i)$ for all $i \gg 0$. 
If $\texttt{Tr}(M^{{\bf b} , 1}_f) \neq 0$ and $\texttt{Tr}(M^{{\bf b} , 1}_g) \neq 0$, 
we must have $g(\ul i) =f(\ul i)$, $\forall i \geq k+1$. 
Hence we have $\lambda_{g^{\ul k}}^{({\bf b},1^{k})} \preceq_{({\bf b},1^{k})} \lambda_{f^{\ul k}}^{({\bf b},1^{k})}$ 
by comparing Definition \ref{def:orderinfW} with Definition \ref{osp:def:dominanceordering}.  
Thanks to Lemma \ref{lem: linkage} and Lemma \ref{lem:linkage at infinity}, 
we have $g^{\ul k} \sim f^{\ul k}$ as well. Therefore we have $g^{\ul k} \preceq_{({\bf b}, 1^k)} f^{\ul k}$.
\par Now suppose $\texttt{Tr}(M^{{\bf b},1}_f) = 0$ and $ g \preceq_{{\bf b},1} f$. 
If $f_{[\ul \infty]} = g_{[\ul \infty]}$, then $\texttt{Tr}(M^{{\bf b},1}_g) = 0$. 
If not, choose $\ul i$ with $i$ maximal such that $f(\ul i) \neq g(\ul i)$. 
If $i \leq k$, then again we have $\texttt{Tr}(M^{{\bf b},1}_g) = 0$. 
So suppose $i \geq k+1$. Since $ g \preceq_{{\bf b},1} f$, we have $g(\ul j ) = f(\ul j)$ for $j \gg 0$ 
and $g(\ul i) < f(\ul i) $. Hence there must be some $t \geq k+1$ such that $g(\ul t) -g(\ul{t+1}) \geq f(\ul t) -f(\ul{t+1}) >-1$. 
Therefore $ \texttt{Tr}(M^{{\bf b},1}_g) = 0$. The lemma follows.
\end{proof}

\begin{lem} 
  \label{pt2:lem:baroninfty}
The truncation map $\texttt{Tr} : \mathbb{T}^{\bf b}\widehat{\otimes} \wedge^\infty \mathbb{W} 
\rightarrow \mathbb{T}^{\bf b}\widehat{\otimes} \wedge^k\mathbb{W}$ commutes with the anti-linear involution $\Bbar$, that is,
\[
\Bbar({\texttt{Tr}(M^{{\bf b},1}_f)}) = {\texttt{Tr}(\Bbar({M^{{\bf b},1}_f}))}, \quad \text{ for } f \in I^{m+n} \times I^{\infty}_{-}.
\]
\end{lem}

\begin{proof}
Following \cite[Lemma 4.2]{CLW12}, we know $\texttt{Tr}$ commutes with $\Abar$. 
As shown in the proof of \cite[Lemma 4.2]{CLW12}, $\texttt{Tr}$ is a homomorphism of $\U^{-}$-modules. 
By \eqref{pt2:eq:Upsilon}, we have $\Upsilon = \sum_{\mu \in \Lambda}\Upsilon_{\mu}$, 
where $\Upsilon_{\mu} \in \U^{-}$. The lemma follows.
\end{proof}

\begin{prop}
Let $k \in \N \cup \{\infty\}$. The anti-linear map $\Bbar: \mathbb{T}^{\bf b}\widehat{\otimes}
 \wedge^k\mathbb{W} \rightarrow \mathbb{T}^{\bf b}\widehat{\otimes} \wedge^k\mathbb{W}$ is an involution. 
 Moreover, the space $\mathbb{T}^{\bf b}\widehat{\otimes} \wedge^k\mathbb{W}$ has unique $\Bbar$-invariant topological bases
\[
\{T^{{\bf b},1}_f \mid f \in I^{m+n} \times I^{k}_- \} \quad \text{ and } \quad \{L^{{\bf b},1}_f \mid f \in I^{m+n} \times I^{k}_-\}
\]
such that
\[
T^{{\bf b},1}_f = M^{{\bf b}, 1}_f + \sum_{g \prec_{({\bf b}, 1^{k}) }f} t^{{\bf b},1}_{gf}(q)M^{{\bf b}, 1}_g, \quad 
L^{{\bf b},1}_f = M^{{\bf b}, 1}_f + \sum_{g \prec_{({\bf b}, 1^{k}) }f} \ell^{{\bf b},1}_{gf}(q)M^{{\bf b}, 1}_g
\]
with $t^{{\bf b},1}_{gf} \in q\Z[q]$, and $\ell^{{\bf b},1}_{gf}(q) \in q^{-1}\Z[q^{-1}]$. 
(We shall write $t^{{\bf b},1}_{ff} = \ell^{{\bf b},1}_{ff}(q) =1$, and $t^{{\bf b},1}_{gf}  
 = \ell^{{\bf b},1}_{gf} = 0$, for $ g \not\preceq_{({\bf b}, 1^{k}) }f$.)
\end{prop}

We call $\{T^{{\bf b},1}_f\}$ and $\{L^{{\bf b},1}_f\}$ the {\em $\imath$-canonical} 
and {\em dual $\imath$-canonical bases} of $\mathbb{T}^{\bf b}\widehat{\otimes} \wedge^k\mathbb{W}$. 
We conjecture that $t^{{\bf b},1}_{gf} \in \N [q]$.

\begin{prop}
   \label{prop:can:truncW}
Let $k\in\N$. The truncation map $\texttt{Tr}:\mathbb T^{\bf
b}\wotimes\wedge^\infty\mathbb W\rightarrow \mathbb T^{\bf
b}\wotimes\wedge^k\mathbb W$ preserves the standard, $\imath$-canonical, and
dual $\imath$-canonical bases in the following sense: for $Y=M,L,T$ and
$f\in I^{m+n} \times I^{\infty}_{-}$ we have
\begin{align*}
\texttt{Tr}\left( Y^{{\bf b},1}_{f} \right)=
\begin{cases}
Y^{{\bf b},1}_{f^{\ul k}},
&\text{ if }f(\ul{i}) - f(\ul{i+1})=-1,\text{ for }i\ge k+1,\\
0,&\text{ otherwise}.
\end{cases}
\end{align*}
Consequently, we have $t^{{\bf b},1}_{gf}(q)=t^{{\bf
b},1}_{g^{\ul k}f^{\ul k}}(q)$ and $\ell^{{\bf b},1}_{gf}(q)=\ell^{{\bf
b},1}_{g^{\ul k}f^{\ul k}}(q)$, for $g, f\in I^{m+n} \times I^{\infty}_{-}$ such
that $f(\ul{i}) -f(\ul{i+1})=g(\ul{i}) -g(\ul{i+1})=-1,$ for $i\ge k+1$.
\end{prop}

\begin{proof}
The statement is true for $Y=M$ by definition. Lemma \ref{pt2:lem:burhatinfinity} and Lemma \ref{pt2:lem:baroninfty}
now imply the statement for $Y=T$, $L$.
\end{proof}

\section{Bar involution and $q$-wedges of $\VV$}

The constructions and statements in \S\ref{subsec:barWW} and \S \ref{subsec:trun} 
have counterparts for $\mathbb{T}^{\bf b} \otimes \wedge^{k} \VV$, $k \in \N \cup \{\infty\}$. 
We shall state them without proofs. Let $\mathbb{T}^{\bf b} \widehat{\otimes} \wedge^{k} \VV$ 
be the $B$-completion of $\mathbb{T}^{\bf b} \otimes \wedge^{k} \VV$. For $k \in \N$, we define the truncation map 
$\texttt{Tr} : \mathbb{T}^{\bf b} \otimes \wedge^{\infty} \VV \rightarrow  \mathbb{T}^{\bf b} \otimes \wedge^{k} \VV$ by 
\[
{\texttt{Tr}}(m \otimes \mathcal{V}_h) =
\begin{cases}
 m \otimes \mathcal{V}_{h_{[\underline{k}]}},& \text{ if } h(\ul{i}) -h(\ul{i+1}) = 1, \text{ for } i \geq k+1,\\
0, &\text{ otherwise }. 
\end{cases}
\]
The truncation map $\texttt{Tr}$ extends to the $B$-completions. 

\begin{prop} 
Let $k\in\N\cup\{\infty\}$. The bar map $\Bbar: {\mathbb T}^{\bf
b}\wotimes\wedge^k\mathbb V \rightarrow {\mathbb T}^{\bf
b}\wotimes\wedge^k\mathbb V$ is an involution. Moreover, the space ${\mathbb
T}^{\bf b}\wotimes\wedge^k\mathbb V$ has unique $\Bbar$-invariant topological bases
\begin{align*}
\{T^{{\bf b},0}_f \mid f \in I^{m+n} \times I^{k}_{+}\} \quad \text{ and } \quad \{L^{{\bf
b},0}_f \mid f\in I^{m+n} \times I^{k}_{+}\}
\end{align*}
such that
\begin{align*}
T^{{\bf b},0}_f &=M^{{\bf b},0}_f+\sum_{g\prec_{({\bf b},{0^k})} f}
t^{{\bf b},0}_{gf}(q) M^{{\bf b},0}_g,
  \\
L^{{\bf b},0}_f &=M^{{\bf b},0}_f+\sum_{g\prec_{({\bf b},{0^k})} f}
\ell^{{\bf b},0}_{gf}(q) M^{{\bf b},0}_g,
\end{align*}
with $t^{{\bf b},0}_{gf}(q)\in q\Z[q]$, and $\ell^{{\bf
b},0}_{gf}(q)\in q^{-1}\Z[q^{-1}]$. (We will write $t_{ff}^{{\bf
b},0}(q)=\ell_{ff}^{{\bf b},0}(q)=1$, $t_{gf}^{{\bf
b},0}=\ell_{gf}^{{\bf b},0}=0$, for $g \not\preceq_{({\bf
b},{0^k})}f$.)
\end{prop}

We shall refer to the basis $\{T^{{\bf b},0}_f\}$  
as the {\em $\imath$-canonical basis} 
and refer to the basis $\{L^{{\bf b},0}_f\}$  the {\em dual $\imath$-canonical basis} for ${\mathbb T}^{\bf
b}\wotimes\wedge^k\mathbb V$. 
Also we shall call the polynomials $t^{{\bf b},0}_{gf}(q)$, $t^{{\bf b},1}_{gf}(q)$, $\ell^{{\bf b},0}_{gf}(q)$ 
and $\ell^{{\bf b},1}_{gf}(q)$ the {\em  $\imath$-KL polynomials}.

\begin{prop}
    \label{prop:can:truncV}
Let $k\in\N$. The truncation map $\texttt{Tr}:\mathbb T^{\bf
b}\wotimes\wedge^\infty\mathbb V\rightarrow \mathbb T^{\bf
b}\wotimes\wedge^k\mathbb V$ preserves the standard, $\imath$-canonical, and
dual $\imath$-canonical bases in the following sense: for $Y=M,L,T$ and
$f\in I^{m+n} \times I^{\infty}_{+}$ we have
\begin{align*}
\texttt{Tr}\left( Y^{{\bf b},0}_{f} \right)=
\begin{cases}
Y^{{\bf b},0}_{f^{\ul k}},
&\text{ if }f(\ul{i}) -f(\ul{i+1})=1,\text{ for }i\ge k+1,\\
0,&\text{ otherwise}.
\end{cases}
\end{align*}
Consequently, we have $t^{{\bf b},0}_{gf}(q)=t^{{\bf
b},0}_{g^{\ul k}f^{\ul k}}(q)$ and $\ell^{{\bf b},0}_{gf}(q)=\ell^{{\bf
b},0}_{g^{\ul k}f^{\ul k}}(q)$, for $g, f\in I^{m+n} \times I^{\infty}_{+}$ such
that $f(\ul{i}) -f(\ul{i+1})=g(\ul{i}) -g(\ul{i+1})=1,$ for $i\ge k+1$.
\end{prop}

\chapter{Comparisons of $\imath$-canonical bases in different Fock spaces}
 \label{sec:compareCB}
 
 In this chapter, we study the relations of $\imath$-canonical and dual $\imath$-canonical
 bases between three different pairs of Fock spaces.

\section{Tensor versus $q$-wedges}
 \label{subsec:tensorwedge}
  
As explained in \S \ref{osp:subsec:$q$-wedges}, we can and will regard $\wedge^k\VV$ 
as a subspace of $\VV^{\otimes k}$, for a finite $k$. 

Let {\bf b} be a fixed $0^m1^n$-sequence and $k \in \N$. We shall compare the 
$\imath$-canonical and dual $\imath$-canonical bases of 
$\mathbb{T}^{\bf b}\otimes \VV^{\otimes k}$ and its subspace $\mathbb{T}^{\bf b}\otimes \wedge^k\VV$ . 

Let $f\in I^{m+n}\times I^k_+$. As before, we write the dual
$\imath$-canonical basis element $L_f^{({\bf b},0^k)}$ in ${\mathbb T}^{\bf
b}\wotimes\mathbb V^{\otimes k}$ and the corresponding dual
$\imath$-canonical basis element $L_f^{{\bf b},0}$ in ${\mathbb T}^{\bf
b}\wotimes\wedge^k\mathbb V$ as
\begin{align}
L_f^{({\bf b},0^k)} &=\sum_{g\in I^{m+n}\times I^k}\ell^{{({\bf
b},{0^k})}}_{gf}(q) M^{{({\bf b},{0^k})}}_{g},
 \label{aux:108a} \\
 \qquad
L_f^{{\bf b},0}&= \sum_{g\in I^{m+n}\times I^k_+}\ell^{{{\bf
b},{0}}}_{gf}(q) M^{{{\bf b},0}}_{g}.
 \label{aux:108b}
\end{align}
The following proposition states that the $\imath$-KL polynomials $\ell$'s
in $\mathbb T^{\bf b}\wotimes\mathbb\wedge^k \mathbb V$ coincide
with their counterparts in $\mathbb T^{\bf b}\wotimes\mathbb
V^{\otimes k}$.

\begin{prop}
  \label{prop:aux1}
Let $f,g \in I^{m+n}\times I^k_+$. Then $\ell^{{{\bf
b},{0}}}_{gf}(q) =\ell^{{({\bf b},{0^k})}}_{gf}(q)$.
\end{prop}

\begin{proof}
The same argument in \cite[Proposition 4.9]{CLW12} applies here.
%
\end{proof}

Let $f\in I^{m+n}\times I^k_+$. Similarly as before we write the canonical basis element $T^{({\bf
b},{0^k})}_f$ in ${\mathbb T}^{\bf b}\wotimes\mathbb V^{\otimes k}$
and the canonical basis element $T^{{\bf b},{0}}_f$ in ${\mathbb
T}^{\bf b}\wotimes\wedge^k\mathbb V$ respectively as
\begin{align}
T^{({\bf b},{0^k})}_f
 & =\sum_{g\in I^{m+n}\times I^k}t^{{({\bf
b},{0^k})}}_{gf}(q) M^{{({\bf b},{0^k})}}_{g},
   \label{TTtta}\\ \quad
T^{{\bf b},{0}}_f
 & =\sum_{g\in I^{m+n}\times I^k_+}t^{{{\bf
b},{0}}}_{gf}(q) M^{{{\bf b},0}}_{g}.
  \label{TTttb}
\end{align}

\begin{prop}
  \label{prop:aux2}
For $f$, $g\in I^{m+n}\times I^k_+$, we have
$$
t^{{{\bf b},{0}}}_{gf}(q)
=\sum_{\tau\in\mf{S}_k}(-q)^{\ell(w^{(k)}_0\tau )}t^{({\bf
b},0^k)}_{g\cdot\tau,f\cdot w^{(k)}_0}(q).
$$
\end{prop}

\begin{proof}
Similar proof as for \cite[Proposition 4.10]{CLW12} works there.

Via identifying $\mathcal V_{g_{[\ul{k}]}}\equiv M^{({0}^k)}_{g_{[\ul{k}]}\cdot
w^{(k)}_0} L_{w^{(k)}_0}$, we have, as in
\cite[Lemma~3.8]{Br03},
\begin{align*}
T^{{\bf b},0}_f= T^{({\bf b},0^k)}_{f\cdot w^{(k)}_0} L_{w^{(k)}_0}.
\end{align*}
A straightforward variation of \cite[Lemma 3.4]{Br03} using
\eqref{TTtta} gives us
\begin{align*}
T^{{\bf b},0}_f &= T^{({\bf b},0^k)}_{f\cdot w^{(k)}_0} L_{w^{(k)}_0}
 = \sum_{g} t^{({\bf b},0^k)}_{g,f\cdot w^{(k)}_0}M^{({\bf b},0^k)}_{g} L_{w^{(k)}_0}
 \displaybreak[0]\\
&= \sum_{\tau\in \mf{S}_k}\sum_{g\in I^{m+n}\times I^k_+}
t^{({\bf b},0^k)}_{g\cdot\tau,f\cdot w^{(k)}_0} M^{({\bf b},0^k)}_{g\cdot\tau} L_{w^{(k)}_0}
  \displaybreak[0]\\
&=\sum_{\tau\in \mf{S}_k}\sum_{g\in I^{m+n}\times I^k_+}
 t^{({\bf b},0^k)}_{g\cdot\tau,f\cdot w^{(k)}_0}
 (-q)^{\ell(\tau^{-1}w^{(k)}_0)} M^{{\bf b},0}_{g}
  \displaybreak[0] \\
&=\sum_{g\in I^{m+n}\times I^k_+} \left(\sum_{\tau\in \mf{S}_k}
t^{({\bf b},0^k)}_{g\cdot\tau,f\cdot w^{(k)}_0}
(-q)^{\ell(w^{(k)}_0\tau)}\right) M^{{\bf b},0}_{g}.
\end{align*}
The proposition now follows by comparing with \eqref{TTttb}.
\end{proof}

\begin{rem}
The counterparts of Propositions~\ref{prop:aux1} and \ref{prop:aux2}
 hold if we replace $\VV$ by $\WW$.
\end{rem}

\section{Adjacent $\imath$-canonical bases}
\label{subsec:adjCB}

Two $0^m1^n$-sequences ${\bf b}$, ${\bf b'}$ of the form  
${\bf b}=({\bf b}^1,{0},{1},{\bf b}^2)$ and ${\bf b'}=({\bf b}^1,{1},{0},{\bf b}^2)$ are called {\em adjacent}. 
Now we compare the $\imath$-canonical 
as well as dual $\imath$-canonical bases in Fock spaces $\widehat{\mathbb T}^{\bf b}$
and $\widehat{\mathbb T}^{\bf b'}$, for adjacent $0^m1^n$-sequences $\bf b$
and $\bf b'$. 

In type $A$ setting, a strategy was developed in \cite[\S 5]{CLW12}
for such a comparison of canonical basis in adjacent Fock spaces.
We observe that the strategy applies to our current setting essentially without any change,
{\em under the assumption that ${\bf b}^1$ is nonempty}. So we will need not copy over all the details from {\em loc. cit.}
to this paper. 

Let us review the main ideas in type $A$ from \cite[\S 5]{CLW12}. 
We will restrict the discussion here to the case of canonical basis while the case of dual canonical basis is entirely similar.
The starting point is to start with the rank two setting and compare the canonical bases in the $B$-completions of 
$\VV \otimes \WW$ and $\WW \otimes \VV$. These canonical bases can be easily computed:
they are either standard monomials or a sum of two standard monomials with some $q$-power coefficients.
The problem is that the partial orderings
on $\VV \widehat{\otimes} \WW$ and $\WW \widehat{\otimes} \VV$ are not compatible.
This problem is overcome by a simple observation that matching up the canonical bases directly
is actually a $\U$-module isomorphism of their respective linear spans, which is denoted by 
$\mc R: \mathbb U \stackrel{\cong}{\rightarrow} \mathbb U'$. 
So the idea is to work with these smaller spaces $\mathbb U$ and $\mathbb U'$ instead of the $B$-completions directly.
We use $\mathbb U$ and $\mathbb U'$ to build up smaller completions of the {\em adjacent} ${\mathbb T}^{\bf b}$
and ${\mathbb T}^{\bf b'}$, 
which are used to match the canonical bases by $T_f^{\bf b} \mapsto T_{f^{\mathbb U}}^{\bf b'}$.
Here the index shift $f \mapsto f^{\mathbb U}$ is shown to correspond exactly under the bijection
$I^{m+n} \leftrightarrow X(m|n)$ 
to the shift $\la \mapsto \la^{\mathbb U}$ on $X(m|n)$ in Remark~\ref{rem:adjLT} below 
(which occurs when comparing the tilting modules relative to adjacent Borel subalgebras of type $\bf b$ and $\bf b'$).

Now we restrict ourselves to two adjacent sequences ${\bf b}$ and  ${\bf b'}$, where ${\bf b}^1$ is nonempty;
this is sufficient for the main application of determining completely the irreducible and tilting
characters in category $\mc O_{\bf b}$ for $\osp(2m+1|2n)$-modules 
(see however Remark~ \ref{rem:osprank2} below for the removal of the restriction). 
We will compare two Fock spaces of the form 
$\mathbb{T}^{{\bf b}^1} \otimes \VV \otimes \WW \otimes \mathbb{T}^{{\bf b}^2}$ and 
$\mathbb{T}^{{\bf b}^1} \otimes \WW \otimes \VV \otimes \mathbb{T}^{{\bf b}^2}$, 
where ${\bf b}^1$ is nonempty. The coideal property of the coproduct of the algebra $\Ui$ in Proposition~ \ref{prop:coproduct}
allows us to consider $\VV \otimes \WW$ and $\WW \otimes \VV$ as $\U$-modules
(not as $\Ui$-modules),
and so the type $A$ strategy of \cite[\S 5]{CLW12} applies  verbatim to our setting. 

\begin{rem}
 \label{rem:osprank2}
Now we consider $\VV \otimes \WW$ and $\WW \otimes \VV$ as $\Ui$-modules (instead of $\U$-modules).
The $\imath$-canonical bases on their respective $B$-completions can be computed explicitly,
though the computation in this case (corresponding to the BGG category of $\osp(3|2)$) is much more demanding;
the formulas
are much messier  and many more cases need to be considered, in contrast to the easy type $A$ case of $\gl(1|1)$.
Denote by $\mathbb U_\flat$ and $\mathbb U'_\flat$ the linear spans of these canonical bases respectively.
We are able to verify by a direct computation that matching the canonical bases suitably produces
 a $\Ui$-module isomorphism $\mathbb U_\flat \rightarrow \mathbb U'_\flat$.
 (The details will take quite a few pages and hence will be skipped.)
 Accepting this, the strategy of  \cite[\S 5]{CLW12} is adapted to work equally well for comparing
 the (dual) $\imath$-canonical bases between arbitrary adjacent Fock spaces $\widehat{\mathbb T}^{\bf b}$
and $\widehat{\mathbb T}^{\bf b'}$. 
\end{rem}

\begin{rem}  \label{rem:adjLT}
Let ${\bf b} = ({\bf b}^1, 0, 1, {\bf b}^2)$ 
and ${\bf b}' = ({\bf b}^1, 1, 0, {\bf b}^2)$ be adjacent  $0^m1^n$-sequences.  Let $\alpha$
be the isomorphic simple root of $\osp(2m+1|2n)$ corresponding to the pair $0,1$ in $\bf b$.
Following \cite[\S 6]{CLW12}, we introduce the notation associated to $\la \in X(m|n)$:
\begin{align*}
\la^{\mathbb L} =
\begin{cases}
 \la,         &\text{ if }(\la,\alpha)=0 \\
\la-\alpha,   &\text{ if }(\la,\alpha)\neq 0,
\end{cases}
\qquad\quad
\la^{\mathbb U} =
\begin{cases}
 \la-2\alpha,         &\text{ if }(\la,\alpha)=0\\
\la-\alpha,   &\text{ if }(\la,\alpha)\neq 0.
\end{cases} 
\end{align*}
Then we have the following identification of  simple and tilting modules (see \cite{PS}  
and  \cite[Lemma 6.2, Theorem 6.10]{CLW12}):
\[
L_{\bf b} (\lambda ) =L_{\bf b'}(\lambda^{\mathbb L}),\quad T_{\bf b}(\lambda) = T_{\bf b'}(\lambda^{\mathbb U}),\quad \text{ for  } \lambda \in X(m|n).
\] 
\end{rem}


\section{Combinatorial super duality}
  \label{sec:superduality}

For a partition $\mu = (\mu_1, \mu_2, \ldots)$, we denote its conjugate partition by 
$\mu' = (\mu'_1, \mu_2', \ldots)$. We define a $\Q(q)$-linear isomorphism
$\natural: \wedge^\infty_d \mathbb V\longrightarrow \wedge^\infty_d \mathbb W$ (for each $d\in \Z$),
or equivalently define
$\natural:\wedge^\infty\mathbb V\rightarrow \wedge^\infty\mathbb W$
by
\begin{align*}
\natural(|\la,d\rangle) =|\la'_*, d\rangle,\quad  \text{ for }
\la\in\mc{P}, d \in \Z.
\end{align*}
The following is a straightforward generalization of   \cite[Theorem 6.3]{CWZ}.

\begin{prop}
  \label{wedgeV:isom:wedgeW} 
The map  
$\natural: \wedge^\infty_d \mathbb V\longrightarrow \wedge^\infty_d \mathbb W$ (for each $d\in \Z$) or
$\natural:\wedge^\infty\mathbb V\longrightarrow \wedge^\infty\mathbb W$ 
is an isomorphism of $\U$-modules.
\end{prop}

\begin{proof}
It is a well-known fact that $\wedge_d^{\infty}\VV$ and $\wedge_d^{\infty}\WW$ as $\U$-modules
are both isomorphic to the level one integrable module associated to the $d$th fundamental weight
(by the same proof as for \cite[Proposition 6.1]{CWZ}; also see  the references therein). 

Now the proof of the proposition is the same as for \cite[Theorem 6.3]{CWZ}, which is our special case with $d=0$.
\end{proof}

This isomorphism of $\U$-modules 
$\natural:\wedge^\infty\mathbb
V \rightarrow\wedge^\infty\mathbb W$ induces
an isomorphism of $\U$-modules
\[
\natural_{\bf b} := {\id \otimes\natural}:
 {{\mathbb T}^{\bf b}}\otimes\wedge^\infty\mathbb V {\longrightarrow}
{{\mathbb T}^{\bf b}}\otimes\wedge^\infty\mathbb W.
\]
Let $f\in I^{m+n} \times I^{\infty}_{+}$. There exists unique
$\la\in\mc{P}$ and $d \in \Z$ such that $|\la, d\rangle=\mc V_{f_{[\ul{\infty}]}}$. We
define $f^\natural$ to be the unique element in
$I^{m+n} \times I^{\infty}_{-}$ determined by $f^\natural(i)=f(i)$,
for $i\in[m+n]$, and $\mc
W_{f^\natural_{[\ul{\infty}]}}=|\la'_*, d\rangle$. The assignment
$f\mapsto f^\natural$ gives a bijection (cf. \cite{CWZ})
\begin{equation}  
  \label{naturalbi}
\natural: I^{m+n} \times I^{\infty}_{+} \longrightarrow
I^{m+n} \times I^{\infty}_{-}.
\end{equation}
If we write $\lambda^{{\bf b},0}_f = \sum^{m+n}_{i=1}\lambda^{\bf b}_{f,i}\ep^{b_i}_{i} 
 + \sum_{1 \leq j}{^+\lambda_{f,{\ul j}}}\ep^0_{\ul j} + d_f\epinftyV \in \wtlV$ under the bijection defined in \eqref{osp:eq:bijectionwtlV}, then we have 
\begin{equation}
 \label{aux}
\lambda^{{\bf b},1}_{f^\natural} = \sum^{m+n}_{i=1}\lambda^{\bf b}_{f,i}\ep^{b_i}_{i} + \sum_{1 \leq j}{^+\lambda'_{f,{\ul j}}}\ep^1_{\ul j} + d_f\epinftyW \in \wtlW.
\end{equation}

The following is the combinatorial counterpart of the super duality
on representation theory in Theorem~\ref{theorem:super duality} .
We refer to \cite[Theorem~4.8]{CLW12} for a type $A$ version, on which our proof
below is based. 

\begin{thm}\label{TwedgeV:isom:TwedgeW}
Let $\bf{b}$ be a ${0^m1^n}$-sequence.
\begin{enumerate}
\item
The isomorphism $\natural_{\bf b}$ respects the Bruhat orderings
and hence extends to an isomorphism of the $B$-completions
$\natural_{\bf b}:{{\mathbb T}^{\bf b}}\wotimes\wedge^\infty\mathbb
V \rightarrow {{\mathbb T}^{\bf b}}\wotimes\wedge^\infty\mathbb W$.

\item
 The map $\natural_{\bf b}$ commutes with the bar involutions.

 \item
The map $\natural_{\bf b}$ preserves the $\imath$-canonical and dual $\imath$-canonical bases.
More precisely, for $f\in I^{m+n}\times I^\infty_+$, we have 
\[
\natural_{\bf b}(M^{{\bf b},0}_f)=M^{{\bf
b},1}_{f^\natural},
 \;\;
\natural_{\bf b}(T^{{\bf b},0}_f)=T^{{\bf b},1}_{f^\natural},
 \;\;
\natural_{\bf b}(L^{{\bf b},0}_f)=L^{{\bf b},1}_{f^\natural}.
\]

\item
We have the following identifications of $\imath$-KL polynomials, for all
$g$, $f\in I^{m+n} \times I^{\infty}_{+}$:
\[
\ell^{{\bf b},0}_{gf}(q)=\ell^{{\bf b},1}_{{g^\natural}{f^\natural}}(q),
\quad \text{ and}\quad
t^{{\bf b},0}_{gf}(q)=t^{{\bf b},1}_{{g^\natural}{f^\natural}}(q).
\]
\end{enumerate}
\end{thm}

\begin{proof}
The statements (2)-(4) follows from (1) by  the same argument as \cite[Theorem~4.8]{CLW12}. 
It remains to prove (1).

Recall the definition of the partial orderings in Definitions \ref{def:orderinfV} and  \ref{def:orderinfW}. 
To prove (1), we need to show for any $f$, $g \in I^{m+n} \times I^{\infty}_+$, $g \preceq_{{\bf b}, 0} f$ if and only if 
$g^\natural \preceq_{{\bf b}, 1} f^\natural$. This is equivalent to say that $f \sim g$ and 
$ \lambda^{{\bf b},0}_g \preceq_{{\bf b},0} \lambda^{{\bf b},0}_f$ if and only if $f^{\natural} \sim g^{\natural}$ 
and $\lambda^{{\bf b},1}_{g^{\natural}} \preceq_{{\bf b},1} \lambda^{{\bf b},1}_{f^{\natural}}$ by 
Definitions~ \ref{def:orderinfV} and \ref{def:orderinfW}.

Since $\natural_{\bf b} : {{\mathbb T}^{\bf b}}\otimes\wedge^\infty\mathbb V {\rightarrow}
{{\mathbb T}^{\bf b}}\otimes\wedge^\infty\mathbb W$ is an isomorphism of $\Ui$-modules, by 
Lemma~\ref{lem:linkage at infinity}, we have $f \sim g$ if and only if $f^\natural \sim g^\natural$. 
We shall assume that $f \sim g$, hence $f^\natural \sim g^\natural$ for the rest of this proof. 

We shall only prove that
$\lambda^{{\bf b},0}_g \preceq_{{\bf b},0} \lambda^{{\bf b},0}_f$
 implies 
 $\lambda^{{\bf b},1}_{g^{\natural}} \preceq_{{\bf b},1} \lambda^{{\bf b},1}_{f^{\natural}}$
here, as the converse is entirely similar. 
We write
\begin{align*}
\lambda^{{\bf b},0}_f- \lambda^{{\bf b},0}_g 
&= a(-\ep^{b_1}_1) + \sum^{m+n-1}_{i=1}a_i (\ep^{b_i}_i - \ep^{b_{i+1}}_{i+1}) 
\\
& \qquad\quad \quad 
+ a_{m+n}(\ep^{b_{m+n}}_{m+n} - \ep^0_{\ul 1}) + \sum_{i \ge 1} a_{\ul i}(\ep^0_{\ul i} - \ep^0_{\ul {i+1}}),
\end{align*}
where all coefficients are in $\N$ and $a_{\ul i} = 0$ for all but finitely many $i$. Set
\begin{equation*}
\lambda^{{\bf b},0}_h := \lambda^{{\bf b},0}_f - a(-\ep^{b_1}_1)
\end{equation*}
for some $h \in I^{m+n} \times I^{\infty}_+$. Apparently we have  
$\lambda^{{\bf b},0}_g \preceq_{{\bf b},0} \lambda^{{\bf b},0}_h \preceq_{{\bf b},0} \lambda^{{\bf b},0}_f $. 

Note that $\lambda^{{\bf b},0}_h$ actually dominates $\lambda^{{\bf b},0}_g$ with respect to the Bruhat ordering 
of type $A$ defined in \cite[\S2.3]{CLW12}. 
Therefore following \cite[Theorem~4.8]{CLW12} and Remark~\ref{osp:rem:Bcompletions}, we have 
\begin{equation}
  \label{osp:eq:comSD1}
\lambda^{{\bf b},1}_{g^{\natural}} \preceq_{{\bf b},1} \lambda^{{\bf b},1}_{h^{\natural}}.
\end{equation}
On the other hand, by definitions of $\lambda^{{\bf b},0}_h$ and the isomorphism of $\natural$, we have  
$\lambda^{{\bf b},1}_{h^{\natural}} = \lambda^{{\bf b},1}_{f^{\natural}} - a(-\ep^{b_1}_1)$, and hence
$
\lambda^{{\bf b},1}_{h^{\natural}} \preceq_{{\bf b},1} \lambda^{{\bf b},1}_{f^{\natural}}.
$ 
Combining this with \eqref{osp:eq:comSD1} implies that
$\lambda^{{\bf b},1}_{g^{\natural}} \preceq_{{\bf b},1} \lambda^{{\bf b},1}_{f^{\natural}}.
$
The statement (1) is proved.
\end{proof}

\chapter{Kazhdan-Lusztig theory of type $B$ and $\imath$-canonical basis}
  \label{sec:b-KL}
  
 In this chapter, we formulate connections between Fock spaces and 
 Grothendieck groups of various BGG categories. We establish relations of simple as well as tilting modules
 between a BGG category and its parabolic subcategory. 
 We show that $\Ui$ at $q=1$ are realized as translation functors in the BGG category.
 Finally, we establish the Kazhdan-Lusztig theory for $\osp(2m+1|2n)$, which is the main goal of the paper.

\section{Grothendieck groups and Fock spaces}
 \label{subsec:G=Fock}
 
Recall the Fock space $\mathbb{T}^{\bf b}$ in \S\ref{subsec:Fockspaces}. Starting with an $\mA$-lattice $\Tb_\mA$
spanned by the standard monomial basis of the $\Qq$-vector space $\mathbb{T}^{\bf b}$, 
we define $\mathbb{T}^{\bf b}_{\Z} =\Z \otimes_\mA \Tb_{\mA}$
where $\mA$ acts on $\Z$ with $q=1$. For any $u$ in the $\mA$-lattice $\Tb_\mA$, 
we denote by $u(1)$ its image in $\mathbb{T}^{\bf b}_{\Z}$.

Recall the category $\mathcal{O}_{\bf b}$  from \S\ref{subsec:cat}.
Let $\mathcal{O}^{\Delta}_{\bf b}$ be the full subcategory of $\mathcal{O}_{\bf b}$ consisting of all modules possessing 
a finite ${\bf b}$-Verma flag. Let $[\mathcal{O}^{\Delta}_{\bf b}]$ be its Grothendieck group. 
The following lemma is immediate from the bijection $I^{m+n} \leftrightarrow X(m|n)$ ($\la \leftrightarrow f_\la^{\bf b}$) given by
\eqref{osp:eq:ftolambda} and \eqref{osp:eq:lambdatof}.

\begin{lem} 
  \label{osp:lem:OtoT}
The map
\[
\Psi : [\mathcal{O}^{\Delta}_{\bf b}] \longrightarrow \mathbb{T}_{\Z}^{\bf b}, 
\quad \quad \quad [M_{\bf b}(\lambda)] \mapsto M^{\bf b}_{f^{\bf b}_{\lambda}}(1),
\]
defines an isomorphism of $\Z$-modules. 
\end{lem}

Recall the category $\mathcal{O}^{\ul k}_{{\bf b},0}$  from \S\ref{subsec:cat}.
We shall denote $\CatO^{\underline{k}, \Delta}_{{\bf b},0}$ the full subcategory of $\mathcal{O}^{\ul k}_{{\bf b},0}$ 
consisting of all modules possessing finite parabolic Verma flags. Recall in \S \ref{osp:subsec:$q$-wedges}, 
we defined the $q$-wedge spaces $\wedge^{k}\VV$ and $\wedge^{k}\WW$.
Recall a bijection   
$\wtlV \rightarrow I^{m+n} \times I^\infty_{+}, \lambda \mapsto f^{{\bf b}0}_{\lambda}$
from \eqref{osp:eq:bijectionwtlV}.
Similarly, we have a bijection
\begin{align*}
\wtlVk \longrightarrow I^{m+n} \times I^k_+, \quad
\lambda \mapsto f^{{\bf b}0}_{\lambda}.
\end{align*}
(Here $f^{{\bf b}0}_{\lambda}$ is understood as the natural restriction to the part $[m+n] \times \ul{k}$.)
Now the following lemma is clear.

\begin{lem}
For $k \in \N\cup\{\infty\}$, the map 
\[
\Psi : [\CatO^{\underline{k}, \Delta}_{{\bf b},0}] \longrightarrow \mathbb{T}_{\Z}^{\bf b} \otimes \wedge^{k}\VV_{\Z}, 
\quad \quad \quad [M^{\underline{k}}_{{\bf b},0}(\lambda)] \mapsto M^{{\bf b}, 0}_{f^{{\bf b}0}_{\lambda}}(1),
\]
defines an isomorphism of $\Z$-modules. 
\end{lem}
We have abused the notation $\Psi$ for all the isomorphisms unless otherwise specified, 
since they share the same origin.
For $k \in \N \cup \{\infty\}$, we define $[ [\CatO^{\underline{k}, \Delta}_{{\bf b},0}] ]$ as the completion of 
$[\CatO^{\underline{k}, \Delta}_{{\bf b},0}] $ such that the extensions of $\Psi$ 
\begin{align}
\Psi : & [ [\CatO^{\underline{k}, \Delta}_{{\bf b},0}] ] \longrightarrow \mathbb{T}_{\Z}^{\bf b} \widehat{ \otimes} \wedge^{k}\VV_{\Z} 
  \label{eq : BKLiso}
\end{align}
are isomorphism of $\Z$-modules. 
Recall the category $\mathcal{O}^{\ul k}_{{\bf b},1}$  from \S\ref{subsec:cat}.
We shall denote $\CatO^{\underline{k}, \Delta}_{{\bf b},1}$ the full subcategory of 
$\mathcal{O}^{\ul k}_{{\bf b},1}$ consisting of all modules possessing parabolic Verma flags.
Recall a bijection   
$\wtlW \longrightarrow I^{m+n} \times I^\infty_{-},  \lambda \mapsto f^{{\bf b}1}_{\lambda}$
from \eqref{osp:eq:bijectionwtlW}.
Similarly, we have a bijection
\begin{align*}
\wtlWk \longrightarrow I^{m+n} \times I^k_-, \quad
\lambda \mapsto f^{{\bf b}1}_{\lambda}.
\end{align*}
(Here $f^{{\bf b}1}_{\lambda}$ is understood as the natural restriction to the part $[m+n] \times \ul{k}$.)
Now the following lemma is clear.

\begin{lem}
 For $k \in \N \cup \{ \infty\}$, the map
\[
\Psi : [\CatO^{\underline{k}, \Delta}_{{\bf b},1}] \longrightarrow \mathbb{T}_{\Z}^{\bf b} \otimes 
\wedge^{k}\WW_{\Z}, \quad \quad \quad [M^{\underline{k}}_{{\bf b},1}(\lambda)] \mapsto M^{{\bf b}, 1}_{f^{{\bf b}1}_{\lambda}}(1),
\]
is an isomorphism of $\Z$-modules. 
\end{lem}

For  $k \in \N \cup \{\infty\}$, we define $[ [\CatO^{\underline{k}, \Delta}_{{\bf b},1}] ]$ as the completion of 
$[\CatO^{\underline{k}, \Delta}_{{\bf b},1}] $ such that the extensions of $\Psi$
\begin{align}
 \label{eq : BKLisoW}
\Psi : &[ [\CatO^{\underline{k}, \Delta}_{{\bf b},1}] ] \longrightarrow \mathbb{T}_{\Z}^{\bf b} \widehat{ \otimes} \wedge^{k}\WW_{\Z}
\end{align}
are isomorphism of $\Z$-modules.

\begin{prop}
  \label{prop:truncation}
The truncation maps defined here are compatible under the isomorphism $\psi$ with the truncations in Propositions \ref{prop:can:truncW} 
and \ref{prop:can:truncV}. More precisely, we have the following commutative diagrams,
\[
\xymatrix{[[\OO^{\ul \infty, \Delta}_{{\bf b}, 0}]] \ar[r]^-{\Psi} \ar[d] ^{\mf{tr_0}}
 & \mathbb{T}^{\bf b}_{\Z} \widehat{\otimes} \wedge^{\infty}\VV_{\Z} \ar[d]^{\texttt{Tr}} \\ 
[[\OO^{\ul k, \Delta}_{{\bf b}, 0}]] \ar[r]^-{\Psi} & \mathbb T^{\bf b}_{\Z} \widehat{\otimes} \wedge^k\VV_{\Z} } \qquad \qquad
\xymatrix{[[\OO^{\ul \infty, \Delta}_{{\bf b}, 1}]] \ar[r]^-{\Psi} \ar[d] ^{\mf{tr_1}}
 & \mathbb{T}^{\bf b}_{\Z} \widehat{\otimes} \wedge^{\infty}\WW_{\Z} \ar[d]^{\texttt{Tr}} \\ 
[[\OO^{\ul k, \Delta}_{{\bf b}, 1}]] \ar[r]^-{\Psi} & \mathbb T^{\bf b}_{\Z} \widehat{\otimes} \wedge^k\WW_{\Z} } 
\]
\end{prop}

\begin{proof}
The proposition follows by a direct computation using the respective standard bases $\{[M^{\ul \infty}_{{\bf b}, 0}(\lambda)]\}$ 
and $\{[M^{\ul \infty}_{{\bf b}, 1}(\lambda)]\}$, and applying Propositions \ref{prop:can:truncW}, \ref{prop:can:truncV}, and 
\ref{prop:trunc:ML}.
\end{proof}

\section{Comparison of characters}
 \label{subsec:comparisonCh}

Let ${\bf b}$ be a fix $0^m1^n$-sequence. For $k \in \N$, consider the extended  sequences $({\bf b}, 0^k)$ and $({\bf b}, 1^k)$. 
Associated to the extended sequences, we introduced in Chapter~\ref{sec:BGG} the categories $\mathcal{O}^{m+k|n}_{({\bf b},0^k)}$ 
and $\mathcal{O}^{m|n+k}_{({\bf b},1^k)}$, as well as the parabolic categories 
$\mathcal{O}^{\ul k}_{{\bf b}, 0}$ and $\mathcal{O}^{\ul k}_{{\bf b}, 1}$, respectively. 

For $\lambda \in X^{{\ul k},+}_{{\bf b}, 0}$, we can express the simple module $[L_{({\bf b}, 0^k)}(\lambda)]$ 
in terms of Verma modules as follows:
\[
[L_{({\bf b}, 0^k)}(\lambda)] = \sum_{\mu \in X(m+k|n)}a_{\mu \lambda} [M_{({\bf b},0^k)}(\mu)], \quad \text{ for } a_{\mu \lambda} \in \Z.
\]
Since the simple modules $\{L_{({\bf b}, 0^k)}(\lambda) = L^{\ul k}_{{\bf b},0} (\lambda) \mid \lambda \in X^{{\ul k},+}_{{\bf b}, 0}\}$ 
also lie in the parabolic category $\mathcal{O}^{\ul k}_{{\bf b}, 0}$, we can express them in terms of parabolic Verma modules as follows:
\[
[L_{({\bf b}, 0^k)}(\lambda)] = \sum_{\nu \in X^{{\ul k},+}_{{\bf b}, 0}} b_{\nu \lambda} [M^{\ul k}_{{\bf b},0}(\nu)], \quad \text{ for } b_{\nu \lambda} \in \Z.
\]

Recall that $M^{\ul k}_{{\bf b},0}(\lambda) =\text{Ind}^{\osp(2m+1|2n|2k)}_{\mathfrak{p}^{\underline{k}}_{{\bf b},0}}L_0(\lambda)$. 
By the Weyl character formula applied to $L_0(\lambda)$, we obtain that
$
a_{\nu \lambda} = b_{\nu \lambda}, \text{ for } \nu, \lambda \in  X^{{\ul k},+}_{{\bf b}, 0}$.  
This proves the following.

\begin{prop}
  \label{osp:prop:simple1}
 Let $\lambda \in X^{{\ul k},+}_{{\bf b}, 0}$ and let $\xi \in X^{{\ul k},+}_{{\bf b}, 1}$. Then we have 
  \begin{align*}
[L_{({\bf b}, 0^k)}(\lambda)] 
&= \sum_{\mu \in X(m+k|n)}a_{\mu \lambda} [M_{({\bf b},0^k)}(\mu)] 
= \sum_{\nu \in X^{{\ul k},+}_{{\bf b}, 0}} a_{\nu \lambda} [M^{\ul k}_{{\bf b},0}(\nu)].
 \\
[L_{({\bf b}, 1^k)}(\xi)] 
&= \sum_{\mu \in X(m|n+k)}a'_{\mu \xi} [M_{({\bf b},1^k)}(\mu)] 
= \sum_{\eta \in X^{{\ul k},+}_{{\bf b}, 1}} a'_{\eta \xi} [M^{\ul k}_{{\bf b},1}(\eta)].
\end{align*}
\end{prop}

Now we proceed with the tilting modules. Let $\lambda \in X^{{\ul k},+}_{{\bf b}, 0}$ and $\xi \in X^{{\ul k},+}_{{\bf b}, 1}$. 
We can express the tilting modules $T_{({\bf b},0^k)}(\lambda)$ and $T_{({\bf b},0^k)}(\xi)$ in terms of Verma modules as follows:
\begin{align*}
[T_{({\bf b},0^k)}(\lambda)] &= \sum_{\mu \in X(m+k|n)} c_{\mu \lambda}[M_{({\bf b},0^k)}(\mu)], \quad \text{ for } c_{\mu \lambda} \in \Z,
\\
[T_{({\bf b},1^k)}(\xi)] &= \sum_{\eta \in X(m|n+k)} c'_{\eta \xi}[M_{({\bf b},1^k)}(\eta)], \quad \text{ for } c'_{\eta \xi} \in \Z.
\end{align*}
Recall the tilting modules $T^{\ul k}_{{\bf b},0}(\la)$ and $T^{\ul k}_{{\bf b},1}(\xi)$ 
in the parabolic categories $\mathcal{O}^{\ul k}_{{\bf b}, 0}$ and $\mathcal{O}^{\ul k}_{{\bf b}, 1}$. 
The following proposition is a counterpart of \cite[Proposition~8.7]{CLW12} with the same proof, which is based on
\cite{So98, Br04}.
Recall $w_0^{(k)}$ denotes the longest element in $\mf S_k$.
\begin{prop}
  \label{osp:prop:tilting}
  \quad
\begin{enumerate}
\item 
Let $\lambda \in X^{{\ul k},+}_{{\bf b}, 0}$, and write
$
T^{\ul k}_{{\bf b},0} (\la) = \sum_{\nu \in X^{\ul k,+}_{{\bf b},0}}d_{\nu \lambda} M^{\ul k}_{{\bf b}, 0}(\nu).
$
Then we have $d_{\nu \lambda} = \sum_{\tau \in \mathfrak{S}_k} (-1)^{\ell(\tau w_0^{(k)})}c_{\tau \cdot \nu, w_0^{(k)} \cdot \lambda}$.

\item 
Let $\xi \in X^{{\ul k},+}_{{\bf b}, 1}$, and write
$
T^{\ul k}_{{\bf b},1}(\xi) = \sum_{\eta \in X^{\ul k,+}_{{\bf b},1}}d'_{\eta \xi} M^{\ul k}_{{\bf b}, 1}(\eta).
$
Then we have 

\noindent $
d'_{\eta \xi} = \sum_{\tau \in \mathfrak{S}_k}(-1)^{\ell(\tau w_0^{(k)})}c'_{\tau \cdot \eta, w_0^{(k)} \cdot \lambda}.
$
\end{enumerate}
\end{prop}

\section{Translation functors} 

In \cite{Br03}, Brundan established a $\U$-module
isomorphism between the Grothendieck group of the category $\mc O$ of $\gl(m|n)$ 
and a Fock space (at $q=1$), where some properly defined translation functors act as Chevalley generators of $\U$ at $q=1$. 
Here we shall develop a type $B$ analogue in the setting of $\osp(2m+1|2n)$.

\par Let $V$ be the natural $\mathfrak{osp}(2m+1|2n)$-module. Notice that $V$ is self-dual. 
Recalling \S \ref{subsec:osp}, we have the following decomposition of $\mathcal{O}_{\bf b}$:
$$
\mathcal{O}_{\bf b} = \displaystyle\bigoplus _{\chi_\lambda} \mathcal{O}_{{\bf b}, \chi_{\lambda}},
$$ 
where $\chi_\lambda$ runs over all the integral central characters.
Thanks to Lemma \ref{lem: linkage}, we can set 
$\mathcal{O}_{{\bf b}, \gamma} := \mathcal{O}_{{\bf b}, \chi_\lambda}$, if $\text{wt}_{\bf b}(\lambda) = \gamma$ (recall $\text{wt}_{\bf b}$ from \eqref{eq:wtbX}). 
For $r \geq 0$, let $S^r V$ be the $r$th supersymmetric power of $V$. 
For $i \in \I^{\imath}$, $M \in \mathcal{O}_{{\bf b}, \gamma}$, 
we define the following translation functors in $\mathcal{O}_{\bf b}$: 
\begin{align}
\bff^{(r)}_{\alpha_{i}} M &:= \text{pr}_{\gamma - r(\varepsilon_{i-\hf}-\varepsilon_{i+\hf})}(M \otimes S^rV),
\\
\be^{(r)}_{\alpha_{i}} M &:= \text{pr}_{\gamma + r(\varepsilon_{i-\hf}-\varepsilon_{i+\hf})}(M \otimes S^rV),
\label{eq:eft}
\\
\bt M &:=\text{pr}_{\gamma}(M \otimes V),
\end{align}
where $\text{pr}_{\mu}$ is the natural projection from $\mathcal{O}_{\bf b}$ to $\mathcal{O}_{{\bf b}, \mu}$.

Note that the translation functors naturally induce operators on the Grothendieck group $[\CatO^{\Delta}_{\bf b}]$, 
denoted by $\bff^{(r)}_{\alpha_{i}}$, $\be^{(r)}_{\alpha_{i}}$, and $\bt$ as well.
The following two lemmas are analogues of \cite[Lemmas 4.23 and 4.24]{Br03}. 
Since they are standard, we shall skip the proofs.
\begin{lem}
On the category $\mathcal{O}_{\bf b}$, the translation functors $\bff^{(r)}_{\alpha_{i}}$, $\be^{(r)}_{\alpha_{i}}$, 
and $\bt$ are all exact. They commute with the $\tau$-duality.
\end{lem}

\begin{lem}
  \label{osp:lem:MOSV}
Let $\nu_1$, $\dots$, $\nu_N$ be the set of weights of $S^rV$ ordered so that $v_i > v_j$ if and only if $ i <j$. 
Let $\lambda \in X(m|n)$. Then $M_{\bf b}(\lambda) \otimes S^rV$ has a multiplicity-free Verma flag with 
subquotients isomorphic to $M_{\bf b}(\lambda + \nu_1)$, $\dots$, $M_{\bf b}(\lambda + \nu_N)$ in the order from bottom to top.
\end{lem}

Denote by $\U_\Z = \Z \otimes_{\mA} \U_{\mA}$ the specialization of the $\mA$-algebra ${\U_{\mc A}}$ at $q=1$. 
Hence we can view $\mathbb{T}_{\Z}^{\bf b}$ as a $\U_\Z$-module. 
Thanks to \eqref{eq:beZ} and \eqref{eq:bffZ}, we know $\iota(\bff^{(r)}_{\alpha_{i}})$ 
and $\iota(\be^{(r)}_{\alpha_{i}})$ lie in $\U_{\mc A}$, hence their specializations at $q=1$ in $\U_\Z$ act on $\mathbb{T}_{\Z}^{\bf b}$. 

\begin{prop}
 \label{prop:translation}
Under the identification $[\CatO^{\Delta}_{\bf b}]$ and $\mathbb{T}_{\Z}^{\bf b}$ via the isomorphism $\Psi$, 
the translation functors $\bff^{(r)}_{\alpha_{i}}$, $\be^{(r)}_{\alpha_{i}}$, and $\bt$ act in the same way 
as the specialization of  $\bff^{(r)}_{\alpha_{i}}$, $\be^{(r)}_{\alpha_{i}}$, and $\bt$ in $\Ui$. 
\end{prop}

\begin{proof}
Let us show in detail that the actions match for $r=1$ (i.e., ignoring the higher divided powers). 
Set 
$$
\lambda + \rho_{\bf b} = \sum^{m+n}_{j=1} a_j \epsilon^{b_j}_{j} \in X(m|n)
\;\text{ and }\;
\gamma =\texttt{wt}_{\bf b}(\lambda).
$$ 
Then we have $M_{\bf b}(\lambda) \in \mathcal{O}_{{\bf b}, \gamma}$. 
By Lemma~\ref{osp:lem:MOSV},  $M_{\bf b}(\lambda) \otimes V$ has a multiplicity-free Verma flag 
with subquotients isomorphic to $M_{\bf b}(\lambda+\ep_{1})$, $\dots$, $M_{\bf b}(\lambda+\ep_{m+n})$, $M_{\bf b}(\lambda)$, 
$M_{\bf b}(\lambda-\ep_{m+n})$, $\dots$, $M_{\bf b}(\lambda-\ep_{1})$. 
Applying the projection $\text{pr}_{\gamma - (\varepsilon_{i-\hf} -\varepsilon_{i+\hf})}$ to the filtration, 
we obtain that $\bff_{\alpha_{i}} M_{\bf b}(\lambda)$ 
has a multiplicity-free Verma flag with subquotients isomorphic to $M_{\bf b}(\lambda \pm \epsilon_{j})$ 
such that $a_j =\pm(i-\hf)$ respectively. 

On the other hand, we have $\Psi(M_{\bf b}(\lambda)) = M^{\bf b}_{f^{\bf b}_{\lambda}}(1)$. Recall the formulas
for the embedding $\imath$ from Proposition~\ref{prop:embedding}.
Suppose $\iota(\bff_{\alpha_i})M^{\bf b}_{f^{\bf b}_{\lambda}}(1) = \sum_g M^{\bf b}_g(1)$, for $i\in \Ihf$.
It is easy to see that for $M^{\bf b}_g$ to appear in the summands, we must have 
$\lambda^{\bf b}_g +\rho_{\bf b} = \lambda + \rho_{\bf b}\pm \ep_{j}$ such that $a_j = \pm(i-\hf)$ respectively.  
Hence the action of $\iota(\bff_{\alpha_i})$ on $\mathbb{T}^{\bf b}_\Z$ 
matchs with the translation functor $\bff_{\alpha_i}$ on $\left[\mathcal{O}^{\Delta}_{\bf b}\right]$ under $\Psi$.

Similar argument works for the translation functor $\be_{\alpha_i}$.

Applying the projection $\text{pr}_{\gamma}$ to the Verma flag filtration of $M_{\bf b}(\lambda) \otimes V$, we obtain that 
$t M_{\bf b}(\lambda)$  from \eqref{eq:eft} has a multiplicity-free Verma flag with subquotients 
isomorphic to $M_{\bf b}(\lambda)$ and $M_{\bf b}(\lambda \pm \epsilon_{j})$ such that $a_j = \mp \hf$ respectively. 
Then one checks that the action of $\iota(t)$ on $\mathbb{T}^{\bf b}_\Z$ matchs with the translation functor $t$ on 
$\left[\mathcal{O}^{\Delta}_{\bf b}\right]$ under $\Psi$.

For the general divided powers, the proposition follows from a direct computation using Lemma~\ref{osp:lem:MOSV} , 
\cite[Corollary~4.25]{Br03}, and the expressions of $\iota(\bff^{(r)}_{\alpha_{i}})$ and $\iota(\be^{(r)}_{\alpha_{i}})$ 
in \eqref{eq:beZ} and \eqref{eq:bffZ}. We leave the details to the reader. 
\end{proof}

\section{Classical KL theory reformulated}

The following is a reformulation of the Kazhdan-Lusztig theory for Lie algebra of type $B$,
which was established by \cite{BB, BK, So90, So98}; also see \cite{V}.
Recall for ${\bf b}=(0^m)$ we have $\mathbb{T}_{\Z}^{\bf b} =\VV_\Z^{\otimes m}$.

\begin{thm}
  \label{thm:KL}
Let ${\bf b}=(0^m)$ and  let $k \in \N \cup \{\infty\}$.
Then the isomorphism 
$\Psi : [[\CatO^{ \ul{k}, \Delta}_{{\bf b},0}]] \rightarrow  \mathbb{T}_{\Z}^{\bf b} \widehat{\otimes} \wedge^{k}\VV_{\Z}$ 
in \eqref{eq : BKLiso} satisfies 
\[
\Psi([L^{\ul{k}}_{{\bf b},0}(\lambda)]) = L^{{\bf b}, 0}_{f^{{\bf b}0}_{\lambda}}(1), 
\quad \quad \quad \Psi([T^{\ul{k}}_{{\bf b},0}(\lambda)]) 
= T^{{\bf b}, 0}_{f^{{\bf b}0}_{\lambda}}(1), \quad \quad \text{ for } \lambda \in \wtlVk.
\]
\end{thm}

\begin{proof} 
For $k \in \N$, the theorem follows from Remark~\ref{rem:samebar}
that the parabolic Kazhdan-Lusztig basis is matched 
with the $\imath$-canonical basis. 
The case with $k = \infty$ follows from Proposition~ \ref{prop:can:truncV}
and Proposition~\ref{prop:trunc:ML}.
\end{proof}

\section{Super duality and Fock spaces}

\begin{thm} \cite[Theorem 7.2]{CLW12}
  \label{theorem:super duality} 
There is an equivalence of categories (called super duality) $\mathsf{SD} : 
\CatO^{\underline{\infty}, \Delta}_{{\bf b},0}\rightarrow \CatO^{\underline{\infty}, \Delta}_{{\bf b},1}$ 
such that the induced map 
$\mathsf{SD} : [[\CatO^{\underline{\infty}, \Delta}_{{\bf b},0}]] \rightarrow [[\CatO^{\underline{\infty}, \Delta}_{{\bf b},1}]]$ 
satisfies, for any $Y = M$, $L$, or $T$,
\[
\mathsf{SD}[Y^{\ul \infty}_{{\bf b},0}(\lambda)] = [Y^{\ul \infty}_{{\bf b},1}(\lambda^\natural)], \quad \text{ for } \lambda \in \wtlV.
\]
\end{thm}

\begin{prop}
 \label{VtoW} 
 Let $\bf b$ be  any $0^m1^n$-sequence.
Assume that the isomorphism 
$\Psi : [[\CatO^{ \ul{\infty}, \Delta}_{{\bf b},0}]] \rightarrow  \mathbb{T}_{\Z}^{\bf b} \widehat{\otimes} \wedge^{\infty}\VV_{\Z}$ 
in \eqref{eq : BKLiso} satisfies 
\[
\Psi([L^{\ul{\infty}}_{{\bf b},0}(\lambda)]) = L^{{\bf b}, 0}_{f^{{\bf b}0}_{\lambda}}(1), 
\quad \quad  \Psi([T^{\ul{\infty}}_{{\bf b},0}(\lambda)]) 
= T^{{\bf b}, 0}_{f^{{\bf b}0}_{\lambda}}(1), \quad   \text{ for } \lambda \in \wtlV.
\]
Then the isomorphism 
$\Psi : [ [\CatO^{\underline{\infty}, \Delta}_{{\bf b},1}] ] \rightarrow \mathbb{T}_{\Z}^{\bf b} \widehat{ \otimes} \wedge^{\infty}\WW_{\Z}$ satisfies
\[
\Psi([L^{\ul{\infty}}_{{\bf b},1}(\lambda)]) = L^{{\bf b}, 1}_{f^{{\bf b}1}_{\lambda}}(1), 
\quad \quad \Psi([T^{\ul{\infty}}_{{\bf b},1}(\lambda)]) = T^{{\bf b}, 1}_{f^{{\bf b}1}_{\lambda}}(1), \quad  \text{ for } \lambda \in \wtlW.
\]
\end{prop}

\begin{proof}
By the combinatorial super duality in Theorem \ref{TwedgeV:isom:TwedgeW}, we have the following isomorphism
\[
\natural_{\bf b} : \mathbb{T}_{\Z}^{\bf b} \widehat{\otimes} \wedge^{\infty}\VV_{\Z} \longrightarrow
  \mathbb{T}_{\Z}^{\bf b} \widehat{ \otimes} \wedge^{\infty}\WW_{\Z},
\]
which preserves the $\imath$-canonical and dual $\imath$-canonical bases. Combining this with the super duality, we have the following diagram:
\begin{equation}
  \label{osp:eq:superduality}
\xymatrix{[[\CatO^{\underline{\infty}, \Delta}_{{\bf b},0}]] \ar[r]^-{\Psi} \ar[d]^{\mathsf{SD}} & \mathbb{T}_{\Z}^{\bf b} \widehat{\otimes} \wedge^{\infty}\VV_{\Z} \ar[d]^{\natural_{\bf b}} 
\\ 
 [ [\CatO^{\underline{\infty}, \Delta}_{{\bf b},1}] ] \ar[r]^-{\Psi} & \mathbb{T}_{\Z}^{\bf b} \widehat{ \otimes} \wedge^{\infty}\WW_{\Z}}
 \end{equation}
 where $\mathsf{SD}$ is the super duality from Theorem \ref{theorem:super duality}.

With the help of the basis $\{[M^{\ul \infty}_{{\bf b},0}(\lambda)]\}$, 
it is easy to check that the diagram \eqref{osp:eq:superduality} commutes. Hence we have the following two commutative diagrams:
 \[
 \xymatrix{[L^{\ul{\infty}}_{{\bf b},0}(\lambda)] \ar@{|->}[d] \ar@{|->}[r] & L^{{\bf b}, 0}_{f^{{\bf b}0}_{\lambda}}(1) \ar@{|->}[d] 
  \\
 [L^{\ul{\infty}}_{{\bf b},1}(\lambda^{\natural})] \ar@{|->}[r] & L^{{\bf b}, 1}_{f^{{\bf b}1}_{\lambda^{\natural}}}(1)} 
 \qquad \qquad 
 \xymatrix{[T^{\ul{\infty}}_{{\bf b},0}(\lambda)] \ar@{|->}[d] \ar@{|->}[r] & T^{{\bf b}, 0}_{f^{{\bf b}0}_{\lambda}}(1) \ar@{|->}[d] 
  \\
 [T^{\ul{\infty}}_{{\bf b},1}(\lambda^{\natural})] \ar@{|->}[r] & T^{{\bf b}, 1}_{f^{{\bf b}1}_{\lambda^{\natural}}}(1)} 
 \]  
The two horizontal arrows on the bottom give us the proposition.
\end{proof}

\section{$\imath$-KL theory for $\osp$}

We can now formulate and prove the main result of Part~ 2, which is a generalization of \cite[Theorem~8.11]{CLW12} (Brundan's conjecture \cite{Br03})
to the ortho-symplectic Lie superalgebra $\mathfrak{osp}(2m+1|2n)$. 

\begin{thm}
   \label{thm:iKL}
For any $0^m1^n$-sequence ${\bf b}$ starting with $0$, the isomorphism 
$\Psi : [ [\CatO^{\Delta}_{\bf b}]] \rightarrow \widehat{\mathbb{T}}_{\Z}^{\bf b}$ in \eqref{eq : BKLiso} (with $k=0$) satisfies 
\[
\Psi([L_{{\bf b}}(\lambda)]) = L^{{\bf b}}_{f^{{\bf b}}_{\lambda}}(1), \quad \quad \Psi([T_{{\bf b}}(\lambda)]) 
= T^{{\bf b}}_{f^{{\bf b}}_{\lambda}}(1),  \quad \text{ for } \lambda \in X(m|n).
\]
\end{thm}

The following proposition is a counterpart of \cite[Theorem 8.8]{CLW12}.
It can now be proved in the same way as in {\em loc. cit.} as we have done all the suitable preparations in \S\ref{subsec:adjCB}  
(as in \cite[\S6]{CLW12}).
We will skip the details.  

\begin{prop} \label{Prop : switching borels}
Let ${\bf b} = ({\bf b}^1, 0 , 1 , {\bf b}^2)$ and ${\bf b}' = ({\bf b}^1, 1, 0 , {\bf b}^2)$ be  adjacent $0^m1^n$-sequences 
with nonempty ${\bf b}^1$ starting with $0$. Then Theorem~\ref{thm:iKL} holds for ${\bf b}$ if and only if it  holds for ${\bf b'}$.
\end{prop}

\begin{rem}
\label{rem:remove0}
The assumption ``nonempty ${\bf b}^1$ starting with $0$" in Proposition~\ref{Prop : switching borels} is removable, if we apply
the observation in Remark~\ref{rem:osprank2}. Subsequently, we can also remove a similar assumption 
on $\bf b$ from Proposition~\ref{VtoW} and Theorem~\ref{thm:iKL}. 
Theorem~\ref{thm:iKL} in its current form already solves completely the irreducible and tilting character problem on $\mc O_{\bf b}$
for an arbitrary $\bf b$,
since $\mc O_{\bf b}$ is independent of $\bf b$ and the relations between the simple/tilting characters in $\mc O_{\bf b}$ 
for different $\bf b$ are understood (see Remark~\ref{rem:adjLT}).
\end{rem}

\begin{proof}[Proof of Theorem~\ref{thm:iKL}]
The overall strategy of the proof is by
induction on $n$, following the proof of Brundan's KL-type conjecture  in \cite{CLW12}. 
The inductive procedure, denoted by
 $ \imath\texttt{KL}(m|n) \, \forall m \ge 1 \Longrightarrow \imath\texttt{KL}(m|n+1)$,
 is divided into the following steps:
{\allowdisplaybreaks
\begin{align}
 \imath\texttt{KL}(m+k|n) \;\; \forall k
 & \Longrightarrow  \imath\texttt{KL}(m|n|k) \;\; \forall k, \text{ by changing Borels}
  \label{ind:oddref} \\
 &\Longrightarrow \imath\texttt{KL}(m|n|\underline{k})
  \;\; \forall k,  \text{ by passing to parabolic}
  \label{ind:para} \\
 &\Longrightarrow \imath\texttt{KL}(m|n|\underline{\infty}),  \text{ by taking $k\mapsto \infty$}
  \label{ind:infty} \\
 & \Longrightarrow \imath\texttt{KL}(m|n+\underline{\infty}), \text{ by super duality}
  \label{ind:SD}  \\
 & \Longrightarrow \imath\texttt{KL}(m|n+1)\;\; \forall m, \text{ by truncation}.
  \label{ind:trunc}
\end{align}
 }
It is instructive to write down the Fock spaces corresponding to the
steps above:
{\allowdisplaybreaks
\begin{align*}
 \VV^{\otimes (m+k)} \otimes \WW^{\otimes n}\;\;  \forall k
 &\Longrightarrow  \VV^{\otimes m} \otimes \WW^{\otimes n} \otimes
 \VV^{\otimes k}\;\;  \forall k
  \\
 &\Longrightarrow \VV^{\otimes m} \otimes \WW^{\otimes n}\otimes
 \wedge^k\VV\;\; \forall k
   \\
 &\Longrightarrow \VV^{\otimes m} \otimes \WW^{\otimes n}\otimes
 \wedge^\infty \VV
   \\
 & \Longrightarrow \VV^{\otimes m} \otimes \WW^{\otimes n}\otimes
 \wedge^\infty\WW
   \\
 & \Longrightarrow
 \VV^{\otimes m} \otimes \WW^{\otimes (n+1)}\;\;\; \forall m \ge 1.
\end{align*}
}

A complete proof would be simply a copy  from the proof of 
\cite[Theorem 8.10]{CLW12}, as we are in a position to take care of each step of 
\eqref{ind:oddref}--\eqref{ind:trunc}. Here we will be contented with specifying how each step follows
and refer the reader to the proof of \cite[Theorem~ 8.10]{CLW12} for details. 

Thanks to Theorem~\ref{thm:samebar}, the base case for the
induction, $ \imath\texttt{KL}(m|0)$, is equivalent to the original
Kazhdan-Lusztig conjecture \cite{KL} for $\mf{so}(2m+1)$, which is a
theorem of \cite{BB} and \cite{BK} (and extended to all singular weights by \cite{So90}); 
The tilting module characters were due to \cite{So98}.

Step~\eqref{ind:oddref} is a special case of Proposition~\ref{Prop : switching borels}. 

Step~\eqref{ind:para} is based on \S \ref{subsec:tensorwedge} (Propositions~\ref{prop:aux1}  and \ref{prop:aux2})
and \S \ref{subsec:comparisonCh} (Propositions~\ref{osp:prop:simple1} and  \ref{osp:prop:tilting}).

Step~\eqref{ind:infty} is based on Proposition~\ref{prop:truncation}.

Step~\eqref{ind:SD}  is based on Proposition \ref{VtoW}.

Step~\eqref{ind:trunc} is based on  Propositions \ref{prop:trunc:ML}, \ref{prop:truncation}, and \ref{prop:can:truncW} (with $k=1$ therein).
The theorem is proved.
\end{proof}

\begin{rem}
There is a similar Fock space formulation of Kazhdan-Lusztig theories for various parabolic subcategories of $\osp(2m+1|2n)$-modules, which 
can be derived as a corollary to Theorem~\ref{thm:iKL}. 
\end{rem}
\begin{rem}
The establishment of a KL theory in Theorem~\ref{thm:iKL} naturally leads to the expectation on a Koszul graded lift for $\mc O_{\bf b}$; cf. \cite{BGS}. 
\end{rem}


\chapter{BGG category of $\mathfrak{osp}(2m+1|2n)$-modules of half-integer weights}
\label{sec:BGGj}

In this chapter we shall deal with a version of BGG category for $\mathfrak{osp}(2m+1|2n)$ associated with 
a {\em half-integer} weight set $'{X}(m|n)$. The relevant quantum symmetric pair turns out to be the $r\mapsto \infty$ limit of
$(\U_{2r}, \U^{\jmath}_{r})$ established in Chapter~\ref{sec:QSPc}. 
 This chapter is a variant of Chapters~\ref{sec:BGG}-\ref{sec:b-KL}, in which we will formulate the main theorems while
 skipping the identical proofs. 
 
\section{Setups for half-integer weights}
Let us first set up some notations. 
Switching the sets of integers and half-integers in \eqref{eq:III}, we set 
\begin{align}
 \label{int:eq:III}
\I  = \cup^{\infty}_{r=0}  & \I_{2r} = \Z+\hf,
\qquad
\I^{\jmath} = \cup^{\infty}_{r=0} \I^{\jmath}_{r} = \N +\hf,
\qquad
I = \Z.
\end{align}

Recall from Chapter~\ref{sec:QSPc} the finite-rank quantum symmetric pairs $(\U_{2r}, \U^{\jmath}_r)$ with embedding
$\jmath: \Uj_r \rightarrow \U_{2r}$.  Let 
\[
\U^{\jmath} : = \bigcup^{\infty}_{r=0} \U^{\jmath}_r,  \qquad\ \quad  \U := \bigcup_{r=0}^{\infty} \U_{2r}.
\]
The pair $(\U, \U^{\jmath})$ forms a quantum symmetric pair as well, with the obvious induced embedding $\jmath: \U^{\jmath} \rightarrow \U$. Let
$
\Pi := \bigcup_{r=0}^{\infty}\Pi_{2r}
$
be the simple system of $\U$. Recall the intertwiner $\Upsilon^{(r)}$ of the pair $(\U_{2r}, \U^{\jmath}_r)$. 
Note that 
$
\Upsilon^{(r+1)}_\mu = \Upsilon^{(r)}_{\mu}, \text{ for } \mu \in \N\Pi_{2r},
$
and this allows us to define 
\[
\Upsilon_{\mu} = \lim_{r \to \infty} \Upsilon^{(r)}_{\mu}, \quad \text{ for } \mu \in \N\Pi.
\]
We then define the formal sum (which lies in some completion of $\U^-$) 
\begin{equation}
  \label{int:eq:Upsilon}
\Upsilon : = \sum_{\mu \in \N\Pi} \Upsilon_{\mu},
\end{equation}
which shall be viewed as a well-defined operator on $\U$-modules that we are concerned.


Introduce the following set of half-integer weights
  \begin{equation}
   \label{eq:X'}
  'X(m|n) : =  \sum^m_{i=1} (\Z +\hf)\ep_i + \sum^n_{j=1}(\Z+\hf)\ep_{\ov j}.
  \end{equation}
  Let ${\bf b} =(b_1, \ldots, b_{m+n})$ be an arbitrary $0^n1^m$-sequence.
We first define a partial ordering on $I^{m+n}$, which depends on the sequence ${\bf b}$.
There is a natural bijection $I^{m+n} \leftrightarrow {}'X(m|n)$,
$f \mapsto \lambda^{\bf b}_f$ and $\lambda \mapsto f^{\bf b}_{\lambda},$  defined 
formally by the same formulas \eqref{osp:eq:ftolambda}-\eqref{osp:eq:lambdatof} for 
the bijection $I^{m+n} \leftrightarrow X(m|n)$ therein,
though $I$ here has a different meaning.

%

Recall the Bruhat ordering $\preceq_{\bf b} $ given by 
\eqref{eq:BrX} on $\mathfrak h_{m|n}$ and hence on $'X(m|n)$.
We now transport the ordering on $'X(m|n)$ by the above bijection to $I^{m+n}$. 

\begin{definition}
  \label{int:def:dominanceordering}
The Bruhat ordering or $\bf b$-Bruhat ordering $\preceq_{\bf b}$ on $I^{m+n}$ is defined as follows:
For $f$, $g \in I^{m+n}$,  $f \preceq_{\bf b} g \Leftrightarrow
\lambda^{\bf b}_f  \preceq_{\bf b} \lambda^{\bf b}_g$. 
We also say $f \sim g$ if $\lambda^{\bf b}_f \sim \lambda^{\bf b}_g$. 
\end{definition}

A BGG category $'\mathcal{O}_{\bf b}$ of $\mathfrak{osp}(2m+1|2n)$-modules
with weight set $'X(m|n)$ is defined in the same way as in Definition~\ref{def:catO},
where the weight set was taken to be $X(m|n)$.
Again, the category $'\mathcal{O}_{\bf b}$ contains several distinguished modules:
the $\bf b$-Verma modules $M_{\bf b}(\lambda)$,
simple modules $L_{\bf b}(\lambda)$, and tilting modules $T_{\bf b}(\lambda)$,
for $\la \in 'X(m|n)$.

 \section{Fock spaces and $\jmath$-canonical bases}
 
Let $\VV : = \sum_{a \in I} \Qq v_a$ be the natural representation of $\U$. Let $\WW : = \VV^*$ be the restricted dual 
module of $\VV$ with the basis $\{w_a \mid a \in I\}$ such that $\langle w_a, v_b\rangle = (-q)^{-a} \delta_{a,b}$. 
By restriction through the embedding $\jmath$, $\VV$ and $\WW$ are naturally $\U^{\jmath}$-modules. 
For a given ${0^m1^n}$-sequence ${\bf b} =(b_1,b_2,\ldots,b_{m+n})$, we again define the Fock space $\Tb$ by the formula \eqref{eq:Fock}
and the standard monomial basis $M_f$, for $f\in I^{m+n}$, by the formula \eqref{eq:Mf}.
Following \S\ref{osp:subsec:cbanddcb}, we define the $B$-completion of the Fock space $\mathbb{T}^{\bf b}$ 
with respect to the Bruhat ordering defined in Definition~\ref{int:def:dominanceordering}. 

Following \S\ref{osp:subsec:cbanddcb} and \S\ref{osp:subsec:cbdcb}, we define an anti-linear involution 
\[
\Cbar : = \Upsilon \Abar : \widehat{\mathbb{T}}^{\bf b} \longrightarrow \widehat{\mathbb{T}}^{\bf b},
\]
where $\Upsilon$ is the operator defined in \eqref{int:eq:Upsilon}, such that 
\[
\Cbar(M_f) = M_f + \sum_{g \prec_{\bf b} f } r_{gf}(q) M_g, \quad \text{ for } r_{gf}(q) \in \mA.
\]

Therefore we have the following counterpart of Theorem~\ref{thm:iCBb} (here we emphasize 
that the index set $I$ here is different from the same notation used therein and $\Uj$ is a different algebra than $\Ui$).

\begin{thm} 
The $\Qq$-vector space $\widehat{\mathbb{T}}^{\bf b}$ has unique $\Cbar$-invariant topological bases
\[
\{T^{\bf b}_f \mid f \in I^{m+n}\} \quad \text{ and } \quad \{L^{\bf b}_f \mid f \in I^{m+n}\}
\]
such that 
\[
T^{\bf b}_f = M_f + \sum_{g \preceq_{\bf b} f }t^{\bf b}_{gf}(q)M^{\bf b}_g, \quad L^{\bf b}_f 
= M_f + \sum_{g \preceq_{\bf b} f }\ell^{\bf b}_{gf}(q)M^{\bf b}_g,
\]
with $t^{\bf b}_{gf}(q) \in q\Z[q]$, and $\ell^{\bf b}_{gf}(q) \in q^{-1}\Z[q^{-1}]$, for $g \preceq_{\bf b}f $. 
(We shall write $t^{\bf b}_{ff}(q) = \ell^{\bf b}_{ff}(q) = 1$, $t^{\bf b}_{gf}(q)=\ell^{\bf b}_{gf}(q)=0$ for $ g \not\preceq_{\bf b} f$.)
\end{thm}

$\{T^{\bf b}_f \mid f \in I^{m+n}\} \text{ and } \{L^{\bf b}_f \mid f \in I^{m+n}\}$ are call the 
{\em $\jmath$-canonical basis} and {\em dual $\jmath$-canonical basis} of $\widehat{\mathbb{T}}^{\bf b}$, respectively. 
The polynomials $t^{\bf b}_{gf}(q)$ and $\ell^{\bf b}_{gf}(q)$ are called 
{\em $\jmath$-Kazhdan-Lusztig (or $\jmath$-KL) polynomials}.

\section{KL theory and $\jmath$-canonical basis}

Starting with an $\mA$-lattice $\Tb_\mA$
spanned by the standard monomial basis of the $\Qq$-vector space $\mathbb{T}^{\bf b}$, 
we define $\mathbb{T}^{\bf b}_{\Z} =\Z \otimes_\mA \Tb_{\mA}$
where $\mA$ acts on $\Z$ with $q=1$. For any $u$ in the $\mA$-lattice $\Tb_\mA$, 
we denote by $u(1)$ its image in $\mathbb{T}^{\bf b}_{\Z}$.

Let $'\mathcal{O}^{\Delta}_{\bf b}$ be the full subcategory of $'\mathcal{O}_{\bf b}$ consisting of all modules possessing 
a finite ${\bf b}$-Verma flag. Let $\left['\mathcal{O}^{\Delta}_{\bf b}\right]$ be its Grothendieck group. 
The following lemma is immediate from the bijection $I \leftrightarrow {}'X(m|n)$.

\begin{lem} 
  \label{int:lem:OtoT}
The map
\[
\Psi : \left['\mathcal{O}^{\Delta}_{\bf b}\right] \longrightarrow \mathbb{T}_{\Z}^{\bf b}, 
\quad \quad \quad [M_{\bf b}(\lambda)] \mapsto M^{\bf b}_{f^{\bf b}_{\lambda}}(1),
\]
defines an isomorphism of $\Z$-modules. 
\end{lem}

Denote by $\Uj_\mA$ the $\mA$-form of $\Uj$ generated by the divided powers,
and set $\Uj_\Z  =\Z \otimes_\mA \Uj_\mA$.

\begin{rem}
The map $\Psi$ is actually a $\Uj_\Z$-module isomorphism, where $\Uj_\Z$
acts on $\left['\mathcal{O}^{\Delta}_{\bf b}\right]$ via translation functors
analogous to Proposition~\ref{prop:translation}. 
\end{rem}

We define $\left[\left[ '\mathcal{O}^{\Delta}_{\bf b} \right]\right]$ as the completion of 
$\left['\mathcal{O}^{\Delta}_{\bf b}\right]$ such that the extension of $\Psi$
\begin{equation*}
\Psi: \left[\left[ '\mathcal{O}^{\Delta}_{\bf b} \right]\right] \longrightarrow \widehat{\mathbb{T}}^{\bf b} 
\end{equation*}
is an isomorphism of $\Z$-modules.
We have the following counterpart of Theorem~\ref{thm:BKL} with the same proof.
\begin{thm}
   \label{thm:BKL}
For any $0^m1^n$-sequence ${\bf b}$ starting with $0$, the isomorphism 
$\Psi : \left[ \left['\CatO^{\Delta}_{\bf b}\right]\right] \rightarrow \widehat{\mathbb{T}}_{\Z}^{\bf b}$ 
satisfies 
\[
\Psi([L_{{\bf b}}(\lambda)]) = L^{{\bf b}}_{f^{{\bf b}}_{\lambda}}(1), \quad \quad \quad \Psi([T_{{\bf b}}(\lambda)]) 
= T^{{\bf b}}_{f^{{\bf b}}_{\lambda}}(1), \quad \quad \text{ for } \lambda \in {}'X(m|n).
\]
\end{thm}


\end{document}